\documentclass{amsart}
\usepackage{tikz-cd}
\usepackage{enumitem}

\usepackage{xr-hyper}
\usepackage{
	amssymb,amsthm,amsrefs,amsxtra,
	mathtools,
}
\usepackage[all]{xy}
\usepackage{
	mathrsfs,
}

\usepackage{hyperref}

\let\narrowtilde=\tilde
\let\tilde\widetilde

\newcommand\At{{\widetilde A}}
\newcommand\at{{\widetilde a}}
\newcommand\alphat{{\widetilde\alpha}}
\newcommand\Bt{{\widetilde B}}
\newcommand\bt{{\widetilde b}}
\newcommand\betat{{\widetilde\beta}}
\newcommand\Deltat{{\widetilde\Delta}}
\newcommand\Gt{{\widetilde G}}
\newcommand\gt{{\narrowtilde g}}
\newcommand\Ht{{\widetilde H}}
\newcommand\Mt{{\widetilde M}}
\newcommand\Nt{{\widetilde N}}

\newcommand\Pt{{\widetilde P}}
\newcommand\Phit{{\widetilde\Phi}}
\newcommand\Psit{{\widetilde\Psi}}

\newcommand\St{{\widetilde S}}
\newcommand\st{{\widetilde s}}
\newcommand\Tt{{\widetilde T}}
\renewcommand\tt{{\widetilde t}}
\newcommand\Ut{{\widetilde U}}
\newcommand\ut{{\narrowtilde u}}
\newcommand\Vt{{\widetilde V}}

\newcommand\Zt{{\widetilde Z}}
\newcommand{\bX}{{\mathbf{X}}}
\newcommand{\Z}{{\mathbb{Z}}}
\newcommand{\sfG}{{\sfyes{G}}}
\newcommand{\sfGt}{\widetilde\sfG}

\newcommand{\R}{\mathbb{R}}

\let\connsup=\circ
\newcommand{\conn}{^\connsup}
\DeclareMathOperator{\stab}{stab}
\newcommand\algsup{\textup a}
\newcommand\alg{\textsup\algsup}
\newcommand\sepsup{\textup s}
\newcommand\sep{\textsup\sepsup}

\newcommand{\ka}{{k\alg}}
\newcommand{\ks}{{k\sep}}

\let\sfyes\mathsf

\newcommand{\Q}{\mathbb{Q}}

\DeclareMathOperator{\coker}{coker}
\DeclareMathOperator{\Hom}{Hom}
\DeclareMathOperator{\Int}{Int}
\DeclareMathOperator{\GL}{GL}

\DeclareMathOperator\Or{O}
\DeclareMathOperator{\SL}{SL}
\DeclareMathOperator{\SO}{SO}
\DeclareMathOperator\Sp{Sp}
\DeclareMathOperator{\Gal}{Gal}

\DeclareMathOperator\Aut{Aut}
\DeclareMathOperator\Ind{Ind}
\DeclareMathOperator\PGL{PGL}

\DeclareMathOperator\WRes{R}
\DeclareMathOperator\SU{SU}
\DeclareMathOperator\im{im}

\renewcommand{\AA}{\mathscr{A}}
\newcommand{\BB}{\mathscr{B}}
\renewcommand{\SS}{\mathscr{S}}
\newcommand{\ff}{\mathfrak{f}}
\newcommand{\inv}{^{-1}}
\let\multsup=\times
\newcommand{\mult}{^{\multsup}}
\newcommand\semi{\textsub{semi}}
\newcommand\unip{\textsub{unip}}

\makeatletter
\DeclarePairedDelimiter\@sset\{\}
\let\smashsset=\@sset
\newcommand\sset{\@sset*}
\DeclarePairedDelimiterX\@set[2]\{\}{#1\,\delimsize\vert\,\mathopen{}#2}
\let\smashset=\@set
\newcommand\set{\@set*}

\DeclarePairedDelimiterX\@sett[2]\{\}{#1\,\delimsize\vert\,\mathopen{}\text{#2}}
\let\smashsett=\@sett
\newcommand\sett{\@sett*}
\makeatother

\newcommand{\lsup}[1]{{}^{#1}}
\newcommand\lsub[1]{{}_{#1}}
\newcommand\smashsub[1]{_{\smash{#1}}}

\DeclareMathOperator\Lie{Lie}

\makeatletter
\DeclarePairedDelimiter\@card\lvert\rvert
\let\smashcard=\@card
\newcommand\card{\@card*}
\DeclarePairedDelimiterX\@pair[2]\langle\rangle{#1, #2}
\let\smashpair=\@pair
\newcommand\pair{\@pair*}
\DeclarePairedDelimiter\@sgen\langle\rangle
\let\smashsgen=\@sgen
\newcommand\sgen{\@sgen*}
\makeatother

\newcommand\ldef{\mathrel{:=}}

\newcommand\adsub{\subtext{ad}}
\newcommand\adform{\textsub\adsub}
\newcommand\dersub{\subtext{der}}
\newcommand\der{\textsub\dersub}
\newcommand\scsub{\subtext{sc}}
\newcommand\scform{\textsub\scsub}
\newcommand\smoothsub{\subtext{sm}}
\newcommand\smooth{\textsub\smoothsub}

\newcommand\pin{\mathcal X}
\newcommand\tildepin{\tilde\pin}
\let\subtext=\textup
\let\suptext=\textup
\newcommand\textsub[1]{_{\subtext{#1}}}
\newcommand\textsup[1]{^{\suptext{#1}}}
\def\fix#1^#2{(#1^{#2})\smooth\conn}
\def\res#1|_{#1\rvert_}

\let\maparrow=\longrightarrow

\let\maptoarrow=\longmapsto

\let\inarrow=\longhookrightarrow
\newcommand\biarrow{\stackrel{\sim}\maparrow}
\newcommand{\maaap}[4]{\ensuremath{{#2}\colon\abmaaap{#1}{#3}{#4}}}
\newcommand{\map}{\maaap\maparrow}

\newcommand{\mapto}{\maaap\maptoarrow}   

\newcommand{\bimap}{\maaap\biarrow}
\newcommand\abmaaap[3]{\ensuremath{{#2}#1{#3}}}
\newcommand{\abmap}{\abmaaap\maparrow}

\newcommand{\abmapto}{\abmaaap\maptoarrow}
\newcommand{\abinmap}{\abmaaap\inarrow}

\newcommand{\abbimap}{\abmaaap\biarrow}

\DeclareMathOperator\Spec{Spec}

\numberwithin{equation}{section}
\newtheorem{thm}[equation]{Theorem}
\newtheorem*{adhocthm}{Theorem \theadhocthm}
\newtheorem{mainthm}{Theorem}
\newtheorem{prop}[equation]{Proposition}
\newtheorem{lem}[equation]{Lemma}
\newtheorem{cor}[equation]{Corollary}
\theoremstyle{definition}
\newtheorem{defn}[equation]{Definition}
\newtheorem{conj}[equation]{Conjecture}
\theoremstyle{remark}
\newtheorem{rem}[equation]{Remark}
\newtheorem*{note}{Note}
\newtheorem{example}[equation]{Example}
\newtheorem{notation}[equation]{Notation}

\makeatletter
  \renewenvironment{proof}[1][\proofname]{\par
    \pushQED{\qed}%
    \normalfont \topsep6\p@\@plus6\p@\relax
    \trivlist
    \item\relax
{\itshape #1\@addpunct{.}}\hspace\labelsep\ignorespaces
}{%
    \popQED\endtrivlist\@endpefalse
}
\makeatother

\let\setminus\smallsetminus
\newcommand\uAut{\underline\Aut}
\DeclareMathOperator\Fix{fix}
\DeclareMathOperator\Frob{Frob}
\DeclareMathOperator\Grp{Grp}
\newcommand\uHom{\underline\Hom}
\DeclareMathOperator\Inn{Inn}
	\newcommand\uInn{\underline\Inn}
	\newcommand\uLie{\underline\Lie}
\DeclareMathOperator\Isom{Isom}
\DeclareMathOperator\Mor{Mor}
	\newcommand\uMor{\underline\Mor}
\DeclareMathOperator\Out{Out}
	\newcommand\uOut{\underline\Out}
	\newcommand\AffSch{\mathrm{AffSch}}
\DeclareMathOperator\Sh{Sh}
	\DeclareMathOperator\SubSh{SubSh}
	\DeclareMathOperator\PSh{PSh}
\newcommand\Sets{\mathrm{Sets}}
\DeclareMathOperator\Skew{Skew}
\newcommand\gpon[2]{{#1}\text-{#2}}
\newcommand\ct{{\widetilde c}}
\newcommand\ft{{\narrowtilde f}}
\newcommand\lambdat{{\narrowtilde\lambda}}
\newcommand\vt{{\narrowtilde v}}
\newcommand\wt{{\narrowtilde w}}
\newcommand\Xt{{\widetilde X}}
\newcommand\Yt{{\widetilde Y}}
\newcommand\htilde{{\narrowtilde h}}
\newenvironment{smallpmatrix}
	{\left(\begin{smallmatrix}}
	{\end{smallmatrix}\right)}

\title[On smooth-group actions]%
{On smooth-group actions on reductive groups and spherical buildings}

\date{\today}
\author{Jeffrey D.~Adler}
\email{jadler@american.edu}
\address{Department of Mathematics and Statistics\\
American University\\
Washington, DC 20016-8050}
\author{Joshua M.~Lansky}
\email{lansky@american.edu}
\address{Department of Mathematics and Statistics\\
American University\\
Washington, DC 20016-8050}
\author{Loren Spice}
\email{l.spice@tcu.edu}
\address{Department of Mathematics \\
Texas Christian University \\
Fort Worth, TX 76129}
\makeatletter\let\@wraptoccontribs\wraptoccontribs\makeatother
\contrib[with an appendix by]{Sean Cotner}
\address{Department of Mathematics \\
University of Michigan \\
Ann Arbor, MI 48109-1043}
\email{stcotner@umich.edu}
\contrib[]{Joshua M.~Lansky}
\contrib[]{Loren Spice}



\begin{document}

\begin{abstract}
Let $k$ be a field, and suppose that $\Gamma$
is a smooth $k$-group that acts on a connected, reductive $k$-group $\Gt$.
Let $G$ denote the maximal smooth, connected subgroup
of the group of $\Gamma$-fixed points in $\Gt$.
Under fairly general conditions, we show that $G$ is a reductive $k$-group,
and that the image of the functorial embedding
$\abmap{\SS(G)}{\SS(\Gt)}$
of spherical buildings
is the set of ``$\Gamma$-fixed points in $\SS(\Gt)$'',
in a suitable sense.
In particular, we do not need to assume that $\Gamma$ has order
relatively prime to the characteristic of $k$ (nor even that $\Gamma$ is finite),
nor that the action of $\Gamma$ preserves a Borel--torus pair in $\Gt$.
\end{abstract}

\maketitle

\setcounter{tocdepth}{2}
\tableofcontents

\numberwithin{equation}{section}
\section{Introduction}

Throughout this paper, $k$ will denote an arbitrary field
of characteristic exponent $p \geq 1$.
Let $\Gt$ denote a connected, reductive $k$-group.
We will let
\(\Gamma\) be a smooth \(k\)-group that
acts on \(\Gt\), but, for simplicity,
in this introduction we will just take \(\Gamma\) to be
an abstract, finite group.

If $\Gamma$ has order prime to $p$,
then Prasad and Yu \cite{prasad-yu:actions}*{Theorem 2.1}
have shown that the connected
part of the group of fixed points $G:= (\Gt^\Gamma)\conn$
is a reductive group.

In the case where $k$ is finite or local,
one can then ask if there is a natural lifting
from representations of $G(k)$ to those of $\Gt(k)$.
For an appropriate definition of `natural',
base change would be a special case of this phenomenon.

Earlier work \cites{adler-lansky:lifting1,adler-lansky:lifting2}
by some of the present authors
accomplishes a step toward such a lifting
when $k$ is finite,
but for our intended applications we imposed
a different hypothesis on $\Gamma$,
namely, that for some extension $E$ of $k$,
the action of $\Gamma$ on $\Gt_E$
preserves a Borel--torus pair.
That is, $\Gamma$ acts \emph{quasisemisimply} on $\Gt_E$.
(In fact, we deal there with a setting that is still more general,
where $\Gamma$ acts only on the root datum of $\Gt$,
and $G$ need not be a fixed-point subgroup.
But we don't pursue that setting here.)

In the case where $k$ is a nonarchimedean local field
with residue field $\ff$,
we will show
in another work \cite{adler-lansky-spice:actions3} that,
under reasonable tameness hypotheses,
if $\Gamma$ acts quasisemisimply on $\Gt_E$,
then we can identify the Bruhat--Tits bulding $\BB(G,k)$
with the fixed-point set $\BB(\Gt,k)^\Gamma$.
Moreover, for such a fixed point $x\in \BB(\Gt,k)^\Gamma$,
one has a quasisemisimple action of $\Gamma$
on the associated $\ff$-group $\sfGt_x$,
whose rational points are the reductive quotient
of the parahoric subgroup $\Gt(k)_x$ of $\Gt(k)$.
We will also examine the relationship between
$\sfG_x$ and
the maximal smooth, connected subgroup of \(\sfGt_x^\Gamma\).
Since many representations of $\Gt(k)$ are constructed
out of representations of $\sfGt_x(\ff)$,
our lifting for representations of finite groups implies
a lifting for some of the data used to construct representations
of $p$-adic groups.

But in order to accomplish any of the above, we first need to know
that $G$ is a reductive group.

We already know that $G$ is reductive in two situations:
the order of $\Gamma$ is prime to $p$ (from
\cite{prasad-yu:actions}*{Theorem 2.1}, as mentioned above)
or
$\Gamma$ acts quasisemisimply
(from \cite{adler-lansky:lifting1}*{Proposition 3.5}).
Comparing these two hypotheses, one sees
that neither one implies the other,
suggesting that a common generalization exists.

In the present paper, we provide three overlapping results:
Theorem \ref{thm:quass}, which proves our strongest conclusions about
quasisemisimple actions under
a rationality hypothesis;
Theorem \ref{thm:ka-quass}, which removes the rationality hypothesis,
is closely related to
\cite{adler-lansky:lifting1}*{Proposition 3.5}, and corrects
an error in it (see Remark \ref{rem:error});
and Theorem \ref{thm:loc-quass}, which is a common generalization of
\cite{adler-lansky:lifting1}*{Proposition 3.5} and
\cite{prasad-yu:actions}*{Theorem 2.1}, and which
moreover comes close to generalizing Theorem \ref{thm:quass}.
Obviously, one would prefer to have a single result, but
it turns out that there is a trade-off:
we must either impose a rationality assumption,
as we do in Theorem \ref{thm:quass}; or
relax our detailed control over the structure of the fixed-point group,
as we do in Theorem \ref{thm:ka-quass}; or
impose a smoothness assumption,
as we do in Theorem \ref{thm:loc-quass}.
Since Corollary \ref{cor:G-smooth} shows that
the smoothness assumption is only an issue in
certain specific circumstances in characteristic \(2\),
we regard this as only a small imperfection.
Moreover, these are genuine restrictions, not just
an artifact of our proof.
Examples
\ref{ex:sl-not-reductive} and
\ref{ex:sl-accidentally-reductive}
show that the
equivalent statements of
Theorem \ref{thm:ka-quass}(\ref{subthm:ka-quass-quass})
do not always hold;
and
Example \ref{ex:vust:cones:prop:5:cor} is a counterexample to
Theorem \ref{thm:loc-quass}(\ref{subthm:loc-quass-reductive})
if we drop the smoothness hypothesis.

Our main result improves on
\cite{adler-lansky:lifting1}*{Proposition 3.5}
in a few additional ways.
First, we allow $\Gamma$ to be any smooth algebraic group,
rather than just an abstract finite group.
Second, in order to handle the case of certain groups
over imperfect fields of characteristic two,
we previously needed to cite an unpublished result of Lemaire
\cite{lemaire:twisted-characters}*{Th\'eor\`eme 4.6}.
Theorem \ref{thm:ka-quass}(\ref{subthm:ka-quass-quass}),
specifically the equivalence of
(\ref{case:ka-quass-quass}) with (\ref{case:ka-quass-Borel}),
replaces our use of that result.
Third,
as in Prasad--Yu \cite{prasad-yu:actions}*{Proposition 3.4},
but under our weaker hypotheses,
we show that the image of the functorial embedding
$\abmap{\SS(G)}{\SS(\Gt)}$
of spherical buildings
is the set $\SS(\Gt)^\Gamma$ of $\Gamma$-fixed points in $\SS(\Gt)$.

This paper and a subsequent one \cite{adler-lansky-spice:actions3}
are inspired by
work of Prasad and Yu \cite{prasad-yu:actions}.
Although the analogous results of Prasad and Yu are a special
case of our Theorem \ref{thm:loc-quass},
we cannot claim to have ``recovered'' the former,
as we use them in our proofs in an essential way.

The outline of the paper is as follows.
In \S\ref{sec:notation}, we fix notation, and
recall some basic structural results about
root systems and algebraic groups.
In \S\ref{sec:main}, we state our main theorems.
In \S\ref{sec:generalities}, we discuss
general results on fixed-point groups and spherical buildings, and
define a notion of `induction' of
groups with action.
This latter requires some foundations from
algebraic geometry, which are discussed in Appendix \ref{app:ind},
written jointly with Sean Cotner.

The bulk of the paper, \S\ref{sec:rd-quass}--\ref{sec:quass},
is devoted to the quasisemisimple case, i.e., the case where,
at least after base change,
there is a single Borel--torus pair that is
preserved by all elements of \(\Gamma\).
In \S\ref{sec:rd-quass}, we recapitulate and flesh out
some abstract results about quasisemisimple actions on root data,
first systematically discussed in
\cite{adler-lansky:data-actions}.

In \S\ref{sec:quass-smooth}, we observe
(Proposition \ref{prop:quass-rough})
that there is a close connection between
maximal tori in \(G\) and \(\Gt\), which allows us
to build a
``\(\Gamma\)-equivariant structure theory for \(\Gt\)'',
proving analogues of results in \cite{borel:linear}*{\S14},
especially involving root groups.
This culminates in the proof of
Theorem \ref{thm:quass}(\ref{subthm:quass-smoothable}), where we prove that
\((\Gt^\Gamma)\conn\) is smoothable
(the phrase commonly written in the literature as
``\((\Gt^\Gamma)\conn\) is defined over \(k\)'') by
pasting together root subgroups to
exhibit a large smooth subscheme that
remains large after base change.
This
allows us to prove many facts over \(k\) by
first base changing to its algebraic closure.
For example, the proof of
Theorem \ref{thm:quass}(\ref{subthm:quass-reductive}), which says that
the connected, smoothed fixed-point group
\(\fix\Gt^\Gamma\) is reductive, becomes an easy consequence of
\cite{adler-lansky:lifting1}*{Proposition 3.5}.

In \S\ref{sec:quass-smooth}, we use
Theorem \ref{thm:quass}(\ref{subthm:quass-smoothable}) to
study the relation between
Borel subgroups, then
parabolic subgroups, then finally
spherical buildings
for \(G\) and \(\Gt\), deducing
Theorem \ref{thm:quass}(\ref{subthm:quass-spherical-bldg})
as, essentially, a re-statement of
Proposition \ref{prop:parabolic-Borus}.
We then analyze the difference between the root systems
denoted, in the notation of \S\ref{sec:quass-smooth}, by
\(\Phi(\Gt, S)\) and \(\Phi(G, S)\)
in Proposition \ref{prop:GS-vs-GtS}, and, after deducing
the intermediate result Corollary \ref{cor:G-smooth}, use it to prove
Theorem \ref{thm:quass}(\ref{subthm:quass-smooth}).

There are a few places in the paper where
our general investigations turn on a detailed understanding of
a very specific case.  The first of these is
\S\ref{sec:sl-outer}, where we handle the case that
(by Corollary \ref{cor:G-smooth}) is
the only nontrivial obstruction to smoothness of
fixed-point groups for quasisemisimple actions.
(The others are
Lemma \ref{lem:A_{2n}-quass},
Proposition \ref{prop:E_6-outer}, and
\S\ref{subsec:D_4-outer}.)
Since this is also the case where
imperfect descent creates the most trouble, we actually
assume that we only have quasisemisimplicity after
(possibly inseparable) base change.
Proposition \ref{prop:sl-smoothable} discusses how close
\((\Gt^\Gamma)\conn\) comes to being smoothable, a question which,
as we will show when we prove
Theorem \ref{thm:ka-quass}(\ref{subthm:ka-quass-quass}) in
the next section, is
closely related to whether the action was quasisemisimple
before before change.

The proof of Theorem \ref{thm:ka-quass} in
\S\ref{sec:thm:ka-quass} is now almost routine, although
still involved.
It involves mostly the results of \S\ref{sec:quass-smooth}, using
the special case handled in \S\ref{sec:sl-outer} to show that
Theorem \ref{thm:ka-quass}%
	(\ref{subthm:ka-quass-quass})%
	(\ref{case:ka-quass-reductive})
implies
Theorem \ref{thm:ka-quass}%
	(\ref{subthm:ka-quass-quass})%
	(\ref{case:ka-quass-smoothable}).
Our work here allows us to prove the pleasant
Corollary \ref{cor:ss-is-quass}, which
upgrades the classical result
\cite{steinberg:endomorphisms}*{Theorem 7.5}
over an algebraically closed field to handle separably closed fields, and
Corollary \ref{cor:quass-imperfect}, which is quite close to
\cite{lemaire:twisted-characters}*{Th\'eor\`eme 4.6}.

Finally, in \S\ref{sec:thm:loc-quass},
we are ready to prove Theorem \ref{thm:loc-quass}.
Thanks to the work of Prasad and Yu \cite{prasad-yu:actions},
the only cases that can cause real difficulty are when
an absolutely almost-simple group has
an outer-automorphism group
whose order is divisible by \(p\), or
that is not cyclic.
We handle the former case in
\S\ref{subsec:quass-unip}, which proves the important reduction result
Proposition \ref{prop:inductive-step-p}, and
the finite-group-theoretic result
Proposition \ref{prop:inner-solvable}.
Surprisingly to us, this latter result turned out to involve
the Feit--Thompson theorem, and
part of the classification of finite simple groups (though not the full force of it).
The latter case (of a non-cyclic outer-automorphism group),
which can only occur for groups of type \(\mathsf D_4\) and
is only really an issue in characteristic \(p = 2\),
is handled in \S\ref{subsec:D_4-outer}.
Finally, \S\ref{subsec:thm:loc-quass} combines the reductions in
the rest of \S\ref{sec:thm:loc-quass} to prove
Theorem \ref{thm:loc-quass}.

\textbf{Acknowledgments}.
The first- and second-named authors
were partially supported by the National Science Foundation (DMS-0854844).
The first-named author was also partially supported by
the American University College of Arts and Sciences Faculty Research Fund.
The third-named author was partially supported by Simons Foundation grant 636151, and Simons Foundation gift MPS-TPM-00007390.
The first- and third-named authors thank the American Institute of Mathematics
for its hospitality.
We thank
John Stembridge for the proof of Lemma \ref{lem:keep-it-simple},
and
Alex Bauman and Sean Cotner for
Example \ref{ex:sl-not-reductive}.
The latter was communicated to us before we knew
the full statement of Proposition \ref{prop:sl-even}, and
helped us to formulate it.

\numberwithin{equation}{subsection}
\section{Notation and recollections}
\label{sec:notation}

Throughout the paper,
\(k\) is an arbitrary field.

We write \(\ka\) for a fixed (but arbitrary) algebraic closure of \(k\),
\(\ks\) for the maximal separable extension field of \(k\) inside \(\ka\),
and
\(\Gal(k)\) for the automorphism group of \(\ka/k\),
which we may, and do, identify with the Galois group of \(\ks/k\) by
restriction to \(\ks\).
When we refer to an algebraic extension \(E/k\), we will always
assume that \(E\) is contained in the fixed algebraic closure \(\ka\).

We will use ``\(k\)-scheme'' as shorthand for
``affine scheme of finite type over \(k\)'',
and ``\(k\)-group'' as shorthand for
``affine group scheme of finite type over \(k\)'',
but we do not require that our \(k\)-groups be smooth or connected.

For each positive integer \(n\),
we write \(\mu_n\) for the group scheme
\(\Spec (k[X]/(x^n - 1))\) of
\(n\)th roots of unity.

When parsing a symbol involving
subscripts and superscripts not otherwise
disambiguated by parentheses,
such as, in later notation,
\(\Gt\smooth\conn\),
\(\Gt\der^\Gamma\), or
\(\Gt_\ka^{\Gamma_\ka}\),
the subscript should be parsed before the superscript.
Thus,
\(\Gt\smooth\conn\) means
\((\Gt\smooth)\conn\), not \((\Gt\conn)\smooth\), when
they differ;
\(\Gt\der^\Gamma\) means
\((\Gt\der)^\Gamma\), not
\((\Gt^\Gamma)\der\), when
they differ; and
\(\Gt_\ka^{\Gamma_\ka}\) means
\((\Gt_\ka)^{\Gamma_\ka}\), not
\((\Gt^{\Gamma_\ka})_\ka\)
(which is usually meaningless).

\subsection{Root systems, root data, and actions}
\label{subsec:roots}

In this subsection, we do not need a field;
instead, we are concerned only with
abstract root systems and root data.
Let \(\Psi = (X^*, \Phi, X_*, \Phi^\vee)\) be a root datum.

We write
\abmapto a{a^\vee} for the duality map between \(\Phi\) and \(\Phi^\vee\),
and \(W(\Phi)\) for the Weyl group of \(\Phi\).
We say that a subset \(\Phi'\) of \(\Phi\) is
\textit{integrally closed} if the intersection with \(\Phi\) of
the \(\Z\)-span of \(\Phi'\) is
again \(\Phi'\) (in which case \(\Phi'\) is itself a root system).

\begin{defn}
Two elements of \(\Phi\) are called
\textit{strongly orthogonal} if
they are not proportional, and
neither their sum nor their difference belongs to \(\Phi\).
A subset of \(\Phi\) is called
\textit{strongly orthogonal} if
every pair of distinct elements is strongly orthogonal.
\end{defn}

\begin{lem}
\label{lem:keep-it-simple}
Suppose that
	\begin{itemize}
	\item \(\Phi_1\) is a (possibly non-reduced, possibly not closed) sub-root system of \(\Phi\),
	\item \(\Phi_1^+\) is a set of positive roots for \(\Phi_1\),
and	\item \(a \in \Phi_1\) is simple with respect to \(\Phi_1^+\).
	\end{itemize}
Then there is a system of positive roots \(\Phi^+\) for \(\Phi\)
that contains \(\Phi_1^+\),
and for which either \(a\) is simple with respect to \(\Phi^+\),
or \(a\) is divisible in \(\Phi\) and \(a/2\) is simple with respect to \(\Phi^+\).
\end{lem}

\begin{proof}[Proof (John Stembridge)]
Write \(X^\vee = \R\Phi^\vee\),
and let \(C_1\) be the chamber for \(\Phi_1\) in \(X^\vee\) determined by \(\Phi_1^+\).
The \(0\)-set of \(a\) is a wall of \(C_1\).
Let \(C\) be a chamber for \(\Phi\) in \(X^\vee\) that is contained in \(C_1\),
and has the \(0\)-set of \(a\) as a wall;
and let \(\Phi^+\) be the corresponding set of positive roots for \(\Phi\).
\end{proof}

We will need to discuss Borel--de Siebenthal theory
(see \cite{borel-desiebenthal}*{\S7, p.~216, Th\'eor\`eme 6})
in some of the detailed computations of \S\ref{sec:thm:loc-quass}.
This theory describes certain full-rank subgroups, but rests on
a classification of integrally closed subsystems, which we describe now.

\begin{defn}
\label{defn:rd-BdS}
Fix a system \(\Delta\) of simple roots, and
an element \(\alpha \in \Delta\).
Write \(\varpi^\vee\) for
the fundamental coweight corresponding to \(\alpha\),
\(\alpha_0\) for the \(\Delta(B, T)\)-highest root
in the irreducible component of \(\Phi(G, T)\) containing \(\alpha\), and
\(n = \pair{\alpha_0}{\varpi^\vee}\) for
the coefficient of \(\alpha\) in \(\alpha_0\).
Put \(\Delta_\alpha =
\Delta \cup \sset{-\alpha_0} \setminus \sset\alpha\).
We call the integrally closed subsystem of \(\Phi\)
generated by \(\Delta_\alpha\)
the \textit{Borel--de Siebenthal subsystem}
associated to \((\Delta, \alpha)\), and we call
\(\Delta_\alpha\) itself
the \textit{Borel--de Siebenthal basis}.
\end{defn}

\begin{rem}
\label{rem:rd-BdS-facts}\hfill
\begin{enumerate}[label=(\alph*), ref=\alph*]
\item\label{subrem:rd-BdS-fact}
Preserve the notation of Definition \ref{defn:rd-BdS}.
The Borel--de Siebenthal subsystem associated to
\((\Delta, \alpha)\) is
precisely the set of all roots \(\beta\) in \(\Phi\) such that
\(\pair\beta{\varpi^\vee}\) is divisible by \(n\),
or, equivalently, lies in \(\sset{0, \pm n}\).
It is easily verified that
the Borel--de Siebenthal basis is
actually a system of simple roots for
the Borel--de Siebenthal subsystem associated to \(\alpha\).
\item\label{subrem:springer-steinberg:conj:sec:4.5}
With two minor corrections,
the maximal integrally closed subsystems of
\(\Phi\) are described in \cite{springer-steinberg:conj}*{\S4.5}
in terms of
what we have called Borel--de Siebenthal subsystems.
First, in the notation there,
the maximal subsystem that they denote by
\(\langle h, a_2, \dotsc, \widehat{a_i}, \dotsc, a_r\rangle\)
should actually be
\(\langle h, a_1, \dotsc, \widehat{a_i}, \dotsc, a_r\rangle\).
That is, there should be only
one simple root omitted, not two.
Second, their result classifies the
maximal such subsystems only up to conjugacy in
the Weyl group.
With this caveat, we may say the following.
If \(\Phi'\) is a maximal integrally closed subsystem of \(\Phi\)
such that
\(\Z\Phi'\) has finite index in \(\Z\Phi\), then there are
a system of simple roots \(\Delta\) for \(\Phi\) and
a root \(\alpha \in \Delta\) such that
the coefficient of \(\alpha\) in the
\(\Delta\)-highest root of the irreducible component of \(\alpha\)
containing \(\Phi\) is prime, and
\(\Phi'\) is the Borel--de Siebenthal subsystem associated to
\((\Delta, \alpha)\).
\end{enumerate}
\end{rem}

Remark \ref{rem:inspection} is immediately motivated by
Remark \ref{rem:rd-BdS-facts}%
	(\ref{subrem:springer-steinberg:conj:sec:4.5}),
and, more particularly, by our needs in
Propositions
\ref{prop:inductive-step-p-prep} and
\ref{prop:inner-solvable}.

\begin{rem}
\label{rem:inspection}
Suppose that \(\Phi\) is reduced and irreducible, and
admits an automorphism \(\gamma\) of prime order, say \(p\).
By inspection
\cite{bourbaki:lie-gp+lie-alg_4-6}*{Chapter VI, Plates I--IX},
we have that
\(\Phi\) is of type
\(\mathsf A_n\), \(\mathsf D_n\), or \(\mathsf E_6\), and
\(p\) equals \(2\); or
\(\Phi\) is of type \(\mathsf D_4\), and
\(p\) equals \(3\).
In each case, \(p^2\) does not divide
the order of the automorphism group.

We now consider two further possible pieces of information.
First,
if the group of diagram automorphisms of \(\Phi\) has
a nontrivial subgroup \(\Gamma'\) of order relatively prime to \(p\) that
is normalised by \(\gamma\), then
\(p\) equals \(2\), \(\Phi\) is of type \(\mathsf D_4\), and
the natural map from \(\sgen\gamma \ltimes \Gamma'\) to
the group of diagram automorphisms of \(\Phi\) is
an isomorphism.

Second,
if instead there is a simple root of \(\Phi\) whose
coefficient \(\ell\) in the highest root is
prime, then
\(\Phi\) is of type
\(\mathsf D_n\) or \(\mathsf E_6\), and
\(\ell\) equals \(2\); or
\(\Phi\) is of type
\(\mathsf D_4\) or \(\mathsf E_6\), and
\(\ell\) equals \(3\).
In particular, the only possibilities where
\(\ell\) is different from \(p\) are
\(\mathsf E_6\), with \(p = 2\) and \(\ell = 3\); and
\(\mathsf D_4\), with \(p = 3\) and \(\ell = 2\).
\end{rem}

\begin{defn}
\label{defn:rd-quass}
Let \(\Psit\) be a root datum,
\(\Gamma\) an abstract group,
and
\abmap\Gamma{\Aut(\Psit)}
a homomorphism.
We say that the action of $\Gamma$ on $\Psit$,
or sometimes by abuse of notation the pair \((\Psit, \Gamma)\),
is \emph{quasisemisimple} if there is a
system of positive roots in the root system of \(\Psit\) that is
preserved by \(\Gamma\).
\end{defn}

\subsection{Groups and actions}
\label{subsec:groups}

We will follow \cite{milne:algebraic-groups}*{Definition 5.5}
in calling a homomorphism of \(k\)-groups a quotient map,
or just quotient, if
it is surjective and faithfully flat.
By \cite{milne:algebraic-groups}*
	{\S A.12 and
	Proposition 5.47},
a surjective homomorphism \abmap G H of \(k\)-groups is
always a quotient map if \(H\) is smooth.
We shall frequently use this fact without explicit mention.

Let \(G\) be a \(k\)-group.

We denote by
\(G\conn\) the maximal connected subgroup of \(G\)
(i.e., its identity component).
Passing to identity components commutes with base change,
in the sense that \((G\conn)_E\) equals \((G_E)\conn\) for
every field extension \(E/k\)
\cite{milne:algebraic-groups}*{Proposition 1.34}.

We denote by
\(G\smooth\) the maximal smooth subgroup of \(G\)
\cite{conrad-gabber-prasad:prg}*{Lemma C.4.1 and Remark C.4.2}.
Many operations involving smoothing that one might expect to commute
actually do not.

First,
\((G\conn)\smooth\) (which is smooth)
contains
\((G\smooth)\conn\) (which is both smooth and connected), but
\cite{conrad-gabber-prasad:prg}*{Remark C.4.2} gives
an example showing that they need not be equal.
If \(k\) is perfect, then
\((G\conn)\smooth\) is
the maximal reduced subscheme of \(G\conn\), hence is connected,
so equals \((G\smooth)\conn\).
Remember that, when there is a difference,
we will always understand \(G\smooth\conn\)
to mean \((G\smooth)\conn\),
not \((G\conn)\smooth\).
Thus, \(G\smooth\conn\) is always
the maximal smooth, connected subgroup of \(G\).

Second,
\((G_\ka)\smooth\)
contains, but need not equal,
\((G\smooth)_\ka\).
We have by \cite{conrad-gabber-prasad:prg}*{Lemma C.4.1}
that \((G_\ks)\smooth\) equals \((G\smooth)_\ks\);
in particular, we have equality
\((G_\ka)\smooth = (G\smooth)_\ka\) if
\(k\) is perfect.

\begin{defn}
\label{defn:smoothable}
The \(k\)-group \(G\) is called \textit{smoothable}
if \((G_\ka)\smooth\) equals \((G\smooth)_\ka\).
\end{defn}

For example, every group of multiplicative type is smoothable
\cite{conrad-gabber-prasad:prg}*{Corollary A.8.2}.
If \(k\) is perfect, then every \(k\)-group is smoothable.

\begin{rem}
\label{rem:conn-smooth}
If \(G\conn\) is smoothable, then
\(((G\conn)\smooth)_\ka\) equals
\(((G\conn)_\ka)\smooth = ((G_\ka)\smooth)\conn\), and so
is connected.
That is, \((G\conn)\smooth\) is connected, hence equals
\((G\smooth)\conn\).
Therefore, \(((G\smooth)\conn)_\ka\) equals \(((G_\ka)\smooth)\conn\).

Conversely, if
\(((G\smooth)\conn)_\ka\) equals \(((G_\ka)\smooth)\conn\), then
\(((G\conn)\smooth)_\ka\), which is always contained in
\(((G\conn)_\ka)\smooth = (G_\ka)\smooth\conn\),
is contained in \(((G\smooth)\conn)_\ka\), so
\((G\conn)\smooth\) is contained in \((G\smooth)\conn\).
Since the reverse containment always holds, we have equality
\((G\conn)\smooth = (G\smooth)\conn\).
Therefore,
\(((G\conn)\smooth)_\ka\) equals
\(((G\smooth)\conn)_\ka =
((G_\ka)\smooth)\conn =
((G\conn)_\ka)\smooth\).
That is, \(G\conn\) is smoothable.

Finally, if \(G\) is smoothable, then
\(((G\smooth)\conn)_\ka = ((G\smooth)_\ka)\conn\)
equals
\(((G_\ka)\smooth)\conn\), so
\(G\conn\) is smoothable.
The converse of this statement does not hold.
We thank Sean Cotner for explaining the following example.
If \(k\) is imperfect and
\(t\) is an element of \(k\mult \setminus (k\mult)^p\), then
the group \(G = \Spec k[X]/(X^{p^2} - t X^p)\) of
\cite{milne:algebraic-groups}*{\S1.57} has smoothable
identity component
\(G\conn = \alpha_p = \Spec k[X]/(X^p)\) and
\emph{trivial} maximal smooth subgroup, while
\((G_\ka)\smooth = \Spec k[X]/(X^p - \sqrt[p]t X)\)
is nontrivial.
\end{rem}

The literature, including \cites{
	steinberg:endomorphisms,
	digne-michel:non-connected,
	adler-lansky:lifting1
},
often works exclusively with smooth group schemes,
and so does not mention passage to the maximal smooth subgroup.
Moreover, since the maximal smooth subgroup of a \(k\)-group can
be unexpectedly small
(as seen, for example, in Remark \ref{rem:conn-smooth}),
one often wants to base change to \(\ka\) before passing to
the maximal smooth subgroup.
Thus, for example, the reference to
\((\Gt^\Gamma)\conn\) in \cite{adler-lansky:lifting1}*{Lemma 3.4}
is really to the maximal smooth subgroup
\((((\Gt^\Gamma)\conn)_\ka)\smooth\) of
the base-changed fixed-point group
\(((\Gt^\Gamma)\conn)_\ka\).
Since this is now a \(\ka\)-group, it makes sense to ask
whether it is defined over \(k\), and
the content of \cite{adler-lansky:lifting1}*{Lemma 3.4} is that
it is.
Then it is easy to show that the descent to \(k\) must be
the maximal smooth subgroup \(((\Gt^\Gamma)\conn)\smooth\) of
\((\Gt^\Gamma)\conn\); that is, that
\((\Gt^\Gamma)\conn\) is smoothable, in our sense.
In this paper, we always say explicitly when
we are performing base change or
passing to the maximal smooth subgroup.

We write \(\uLie(G)\) for the Lie-algebra-valued functor
\abmapto A{\ker(\abmap{G(A[\epsilon]/(\epsilon^2))}{G(A)})}
on \(k\)-algebras
of \cite{demazure-gabriel:groupes-algebriques}*
	{Ch.~II, \S4, 1.2 and 4.2}, and put
\(\Lie(G) = \uLie(G)(k)\).
Since \(k\) is a field, the canonical map from
the vector group
\abmap A{\Lie(G) \otimes_k A}
to \(\uLie(G)\) is an isomorphism
\cite{demazure-gabriel:groupes-algebriques}*
	{Ch.~II, \S4, Proposition 4.8 and (b)}, and
we freely treat it as equality.

If \(G\) acts on a \(k\)-scheme \(Z\), and
\(X\) is a closed subscheme of \(Z\),
then we write \(\stab_G(X)\) for
the stabilizer of \(X\) in \(G\)
\cite{milne:algebraic-groups}*{Corollary 1.81}.
If \(G\) is acting on \(Z = G\) by conjugation,
we write \(N_G(X)\) in place of \(\stab_G(X)\).

If \(H\) is a \(k\)-group acting on \(G\), then we write \(G^H\) for
the \(H\)-fixed-point subgroup of \(G\), i.e.,
the maximal subgroup of \(G\) on which
\(H\) acts trivially
\cite{milne:algebraic-groups}*{Theorem 7.1}.
For all \(a \in \bX^*(H)\), we will permit ourselves to write
\(\uLie(G)_a\) for the vector group
\abmapto A{\Lie(G)_a \otimes_k A}.
By \cite{demazure-gabriel:groupes-algebriques}*
	{Ch.~II, \S4, Proposition 2.5},
we have that \(\Lie(G^H)\) equals \(\Lie(G)^H\).
If \(H\) is a subgroup of \(G\), then we write
\(C_G(H)\) for the centralizer of \(H\) in \(G\)
\cite{milne:algebraic-groups}*{Proposition 1.92}.
Thus, if \(H\) is a subgroup of \(G\) acting on \(G\) by conjugation,
then \(C_G(H)\) equals \(G^H\).
We will sometimes allow ourselves more generally
to write \(C_G(H)\) for \(G^H\)
whenever \(H\) is a \(k\)-group acting on \(G\).

Write \(\bX^*(G) = \Hom(G, \GL_1)\)
for the abstract group of characters of \(G\),
and \(\bX_*(G) = \Hom(\GL_1, G)\)
for the set of cocharacters of \(G\).

If \(G\) is a \(k\)-group equipped with
a representation on \(X\), in the sense of
\cite{jantzen:alg-reps}*{Part I, \S2.7},
then we write \(X^G\) for the \(G\)-fixed subspace of \(X\)
and, for every \(a \in \bX^*(G)\),
\(X_a\) for the \(a\)-weight space for \(G\) in \(X\)
\cite{jantzen:alg-reps}*{Part I, \S2.10(1, 1\('\))}.
If \(X'\) is a subset of \(X\), then, as in
\cite{jantzen:alg-reps}*{Part I, 2.12(1)}
(except that we use \(C\) in place of \(Z\)), we write
\(C_G(X')\) for the fixer of \(X'\) in \(G\).
We have by
\cite{jantzen:alg-reps}*{Part I, 7.11(10)} that
\(\Lie(C_G(X'))\) equals
\(C_{\Lie(G)}(X')\), i.e.,
the annihilator in \(\Lie(G)\) of \(X'\).

We write \(\Phi(X, G)\) for
the set of nonzero weights of \(G\) on \(X\), i.e.,
those nonzero \(a \in \bX^*(G)\) such that
\(X_a\) is nonzero.
This will be most interesting when \(G\) is a torus, but
we do not require this.

If \(G = S\) is a group of multiplicative type,
then \(\bX_*(S)\) is a lattice, and
it depends only on \(S\smooth\conn\).
Put $V(S) = \bX_*(S) \otimes_\Z \R$.

If \(S\) acts on \(G\), then
we write \(\Phi(G, S)\) for \(\Phi(\Lie(G), S)\), even if
\(G\) is not smooth.

As usual, we call a \(k\)-group \(G\) \textit{reductive} if
it is smooth and
there is no nontrivial smooth, connected, unipotent, normal subgroup
of \(G_\ka\).
If \(G\) is reductive, then
we write \(G\adform\) for its adjoint quotient \(G/Z(G)\), and
\(G\scform\) for the simply connected cover of \(G\der\).

If \(S\) is a torus in \(G\),
then we put \(W(G, S) = N_G(S)/C_G(S)\), even if
this is not a Coxeter group.

\begin{lem}
\label{lem:which-Weyl}
Suppose that \(G\) is connected, reductive, and quasisplit.
Let \(S\) be a maximal split torus in \(G\), and put
\(T = C_G(S)\).
Then \(W(G, S)\) is contained in \(W(G, T)\), and
\(W(G, S)(k)\) equals \(W(G, T)(k)\).
\end{lem}

\begin{proof}
We have that \(C_G(S)\) and \(C_G(T)\) both equal \(T\),
so \(N_G(S)\) normalizes \(C_G(S) = T\) and hence lies in
\(N_G(T)\).
Thus, \(W(G, S) = N_G(S)/T\) is a subgroup of
\(W(G, T) = N_G(T)/T\).

Every automorphism of \(T\) preserves
its maximal split torus \(S\).
(Remember that ``automorphism of \(T\)'' \emph{means}
``automorphism of \(T\) defined over \(k\).'')
Thus, if \(w\) belongs to \(W(G, T)(k)\) and
\(n\) is a representative of \(w\) in \(N_G(T)(\ks)\), then
\(n\) normalizes \(S_\ks\), hence belongs to
\(N_{G_\ks}(S_\ks)(\ks) = N_G(S)(\ks)\).
Therefore, the image \(w\) of \(n\) in \(W(G, T)(k)\) belongs to
\(N(G, S)(k)\).
\end{proof}

Put \(Z(G) = C_G(G)\).
Thus, for example, if the characteristic exponent
\(p\) of \(k\) is greater than \(1\), then
\(Z(\SL_p)\) is the infinitesimal group scheme \(\mu_p\), not
its underlying maximal smooth group scheme, which is trivial.

For every \(k\)-algebra \(A\),
write \(\Aut(G_A)\) for
the abstract group of automorphisms of
\(G_A \ldef \Spec (k[G] \otimes_k A)\); and then
write \(\uAut(G)\) for
the automorphism group functor,
defined by \(\uAut(G)(A) = \Aut(G_A)\)
for all \(k\)-algebras \(A\)
\cite{SGA-3.1}*{Expos\'e I, no.~1.7, p.~10}).
Write \(\Int\) for the inner-automorphism map
\abmap G{\uAut(G)}
given by
\abmapto g{(\abmapto h{g h g\inv})},
\(\uInn(G)\) for the (sheaf-theoretic) image of \(G\), so that
\(\Int\) is an isomorphism of \(G\adform = G/Z(G)\) onto \(\uInn(G)\), and
\(\Inn(G) = \uInn(G)(k)\).
Thus, \(\Int(G(k))\) is contained in, but need not equal,
\(\Inn(G)\).
We shall call an automorphism of \(G\) \textit{inner} if
it belongs to \(\Inn(G)\), and \textit{outer} if
it is trivial or does not belong to \(\Inn(G)\).

Write \([\cdot, \cdot]\) for the commutator map
\abmap{G \times G}G
given by
\abmapto{(g, h)}{g h g\inv h\inv}, and
\(G\der\) for the derived subgroup of \(G\), i.e.,
the subgroup generated by the image of \([\cdot, \cdot]\).

\begin{rem}
\label{rem:aut-inner}
If \(G\) is connected and reductive,
then the group functor \(\uAut(G)\) is
an affine group scheme over \(k\), but is
\emph{not} of finite type over \(k\) unless \(G\) is semisimple;
and
\(\uInn(G)\) is the identity component \(\uAut(G)\conn\)
\cite{SGA-3.3}*{Expos\'e XXIV, Th\'eoreme 1.3(i) and Corollaire 1.7}.
\end{rem}

\begin{rem}
\label{rem:all-ss=>lr}
If \(G\) is smooth and
all elements of \(G(\ka)\) are semisimple, then
\(G\) is linearly reductive, in the sense of
\cite{milne:algebraic-groups}*{Definition 12.52}.
This follows from
\cite{milne:algebraic-groups}*{Corollary 17.25} and
the description of linearly reductive groups in
\citelist{
	\cite{demazure-gabriel:groupes-algebriques}*
		{Ch.~II, \S7, Proposition 3.4}
	\cite{milne:algebraic-groups}*
		{Corollary 22.43}
}.
\end{rem}

If \(G\) is reductive
and
\(S\) is a maximal split torus in \(G\),
then there is a root datum
\(\Psi(G, S) \ldef
(\bX^*(S), \Phi(G, S), \bX_*(S), \Phi^\vee(G, S))\),
\cite{conrad-gabber-prasad:prg}*{Theorem C.2.15}.

A minimal
parabolic
subgroup $B$ of $G$ containing $S$ determines
a system of simple roots $\Delta(B,S)$ in $\Psi(G,S)$, and
we write $\Psi(G,B,S)$ for the associated based root datum.

\begin{defn}
\label{defn:decomposing-data}
A \textit{reductive datum} is a triple
\((\Gt, \Gamma, \varphi)\), where
\(\Gt\) is a connected, reductive \(k\)-group,
\(\Gamma\) is a smooth \(k\)-group, and
\map\varphi{\Gamma \times \Gt}\Gt\ is
an action of \(\Gamma\) on \(\Gt\).
We usually denote a datum by \((\Gt, \Gamma)\), leaving
\(\varphi\) implicit.

A morphism
from one reductive datum \((\Gt, \Gamma)\) to
another \((\Gt', \Gamma')\)
is a pair of group homomorphisms
\abmap\Gt{\Gt'} and
\abmap\Gamma{\Gamma'} with
obvious compatibility properties.
In this paper,
we will only consider the case where
\(\Gamma\) equals \(\Gamma'\) and
the morphism \abmap\Gamma{\Gamma'} is the identity, so
we omit it from the notation.
\end{defn}

Although it is perfectly fine if the group \(\Gamma\) in
a reductive datum \((\Gt, \Gamma)\) is constant,
which is the original motivating case for this paper,
we do \emph{not} assume this, and do \emph{not}
identify \(\Gamma\) with its set of \(\ka\)-points.
If \(\gamma\) is an element of \(\Gamma(k)\), then
\(\sgen\gamma\) always means
the \emph{algebraic} subgroup of \(\Gt\) generated by \(\gamma\),
\emph{not} an \emph{abstract} subgroup of \(\Gamma(k)\).
This leads to surprising notation such as
\(\sgen\gamma(k)\).
If \(\gamma\) has finite order, then this is indeed
the abstract subgroup of \(\Gt(k)\) generated by \(\gamma\);
but, otherwise, \(\sgen\gamma(k)\) is usually strictly bigger than
that abstract subgroup.

\begin{defn}
\label{defn:ordinary-cases}
Let
\((\Gt, \Gamma)\) be a reductive datum.
\begin{enumerate}[label=(\alph*), ref=\alph*]
\item
A \emph{Borel--torus pair} in \(\Gt\) is a pair \((\Bt, \Tt)\)
of a Borel subgroup \(\Bt\) of \(\Gt\)
(i.e., a subgroup such that \(\Bt_\ka\) is a maximal
smooth, connected, solvable subgroup of \(\Gt_\ka\))
and a maximal torus \(\Tt\) in \(\Bt\).
A \emph{parabolic--Levi} pair in \(\Gt\) is a pair \((\Pt, \Mt)\),
where \(\Pt\) is a parabolic subgroup of \(\Gt\)
(i.e., \(\Pt_\ka\) contains a Borel subgroup of \(\Gt_\ka\))
and
\(\Mt\) is a Levi  component of \(\Pt\)
(i.e., a subgroup such that \(\Mt_\ka\) maps isomorphically onto
the maximal reductive quotient of \(\Pt_\ka\)).
The group \(\Gt\) always admits at least one parabolic--Levi pair,
the trivial pair \((\Gt, \Gt)\),
but need not admit a Borel--torus pair; it admits such a pair
if and only if
\(\Gt\) is quasisplit.
\item\label{subdefn:gp-quass}
We say that the action, or again sometimes
by abuse of notation
the pair \((\Gt, \Gamma)\), is
\emph{quasisemisimple} if there is a
Borel--torus pair in \(\Gt\) that is
preserved by \(\Gamma\).
(This is equivalent to saying that
$\Gt$ is quasisplit,
there is a \(\Gamma\)-stable maximal torus \(\Tt\) in \(\Gt\), and
the action of
\(\Gal(k) \ltimes \Gamma(\ks)\) on the root datum
$\Psi(\Gt_\ks, \Tt_\ks)$
is quasisemisimple.
See Remark \ref{rem:how-to-restrict}.)
An automorphism \(\gamma\) of \(\Gt\) is called
\textit{quasisemisimple}, or said to
\textit{act quasisemisimply} on \(\Gt\),
if
\((\Gt, \sgen\gamma)\) is quasisemisimple.
\item
We say that \((\Gt, \Gamma)\) is \emph{exceptional} if
it is quasisemisimple, and
there is
a maximal split torus \(S\) in \(G \ldef \fix\Gt^\Gamma\) such that
some root in \(\Phi(G, S)\) is
divisible in \(\Phi(\Gt^\Gamma, S)\).
An automorphism \(\gamma\) of \(\Gt\) is called
\textit{exceptional}, or said to
\textit{act exceptionally} on \(\Gt\), if
\((\Gt, \sgen\gamma)\) is exceptional.
\end{enumerate}
\end{defn}

\begin{rem}
\label{rem:torus-quass}
If \((\Gt, \Gamma)\) is a reductive datum, then
Remark \ref{rem:aut-inner} gives that
the action of \(\Gamma\conn\) factors uniquely through
\map\Int{\Gt\adform}{\uAut(\Gt)}
to give a map \abmap{\Gamma\conn}{\Gt\adform}.
\begin{enumerate}[label=(\alph*), ref=\alph*]
\item\label{subrem:act-inner}
There is an action of
\(\pi_0(\Gamma)(k)\) on
the set of almost-simple components of \(\Gt\).
\item\label{subrem:quass-to-torus}
If \((\Bt, \Tt)\) is a Borel--torus pair in \(\Gt\) that
is preserved by \(\Gamma\), then
the image of \(\Gamma\conn\) in \(\Gt\adform\)
normalizes \(\Bt\) and \(\Tt\), hence
is contained in \(\Tt/Z(\Gt)\).
Since it is smooth and connected, the image is a torus.
More generally, if we
write \(\Gamma'\) for the subgroup of \(\Gamma\) that
acts on \(\Gt\) by inner automorphisms, then
the image of \(\Gamma'\) in \(\Gt\adform\) need not
\emph{be} a torus, but is still contained in \(\Tt/Z(\Gt)\).
\item\label{subrem:torus-to-quass}
Suppose that
\(\Gt\) is quasi-split and
\(\Gamma\) is a split torus, or,
slightly more generally, a torus whose image under
\abmap\Gamma{\Gt\adform} is split.
Then the image of \(\Gamma\) is
contained in a maximal torus in \(\Gt\adform\) that
is contained in a Borel subgroup of \(\Gt\adform\),
and this Borel--torus pair pulls back to
a Borel--torus pair in \(\Gt\) that is
preserved by \(\Gamma\).
That is, \((\Gt, \Gamma)\) is quasisemisimple.
\item
The conclusion of (\ref{subrem:torus-to-quass})
can fail if
we relax the assumption that
\(\Gamma\) is a torus to
the assumption that
it is of multiplicative type.
Indeed, if \(p\) is odd, then
the subgroup of \(\PGL_2\) generated by the images there of
\(\begin{smallpmatrix} 1 & 0 \\ 0 & -1 \end{smallpmatrix}\)
and
\(\begin{smallpmatrix} 0 & 1 \\ 1 & 0 \end{smallpmatrix}\)
is of multiplicative type, but is not contained in
any torus, and so does not preserve any Borel--torus pair
in \(\PGL_2\).
\end{enumerate}
\end{rem}

\numberwithin{equation}{section}
\section{Statements of the main theorems}
\label{sec:main}

We introduce the common
Notation \ref{notn:main} to be used in
the statement of each of our main theorems.
In the terminology of Definition \ref{defn:decomposing-data},
the conditions on \(\Gt\), and \(\Gamma\) say precisely that
\((\Gt, \Gamma)\) is a reductive datum over \(k\).

\begin{notation}
\label{notn:main}
Let
	$k$ be a field,
	\(\Gt\) a connected, reductive \(k\)-group, and
	\(\Gamma\) a smooth \(k\)-group acting on \(\Gt\).
Put \(G = \fix\Gt^\Gamma\).
\end{notation}

Recall the notion of the spherical building of a reductive group from \cite{curtis-lehrer-tits:spherical}*{\S2}.
See \S\ref{subsec:spherical} for more details.

Theorem \ref{thm:quass} proves our most comprehensive results,
including the notoriously ill-behaved case of
certain outer actions on
groups of type \(\mathsf A_{2n}\) in characteristic \(2\),
under the assumption that \((\Gt, \Gamma)\) is quasisemisimple.
Theorems \ref{thm:ka-quass} and \ref{thm:loc-quass}
generalize this in two ways.

Theorem \ref{thm:ka-quass} shows what happens if we replace
quasisemisimplicity of \((\Gt, \Gamma)\) by
the weaker hypothesis of
quasisemisimplicity of
\((\Gt_\ka, \Gamma_\ka)\).
In this case, we know that the analogue of
Theorem \ref{thm:quass}(\ref{subthm:quass-reductive})
does not hold; see
Example \ref{ex:sl-not-reductive}, due to
Alex Bauman and Sean Cotner.
We describe the possible failure of reductivity in
Theorem \ref{thm:ka-quass}(\ref{subthm:ka-quass-how-reductive}),
and provide many necessary and sufficient conditions for
the stronger hypothesis of rational quasisemisimplicity to hold in
Theorem \ref{thm:ka-quass}(\ref{subthm:ka-quass-quass}).
(It would be pleasant to be able to replace the two-part condition
Theorem \ref{thm:ka-quass}(\ref{subthm:ka-quass-quass})%
	(\ref{case:ka-quass-reductive}) by just
the condition that \(G\) is reductive, but
Example \ref{ex:sl-accidentally-reductive} shows that
this is not sufficient.)
Because of this list of necessary and sufficient conditions,
Theorem \ref{thm:ka-quass} subsumes
Theorem \ref{thm:quass}, except for
Theorem \ref{thm:quass}(\ref{subthm:quass-spherical-bldg}).
That part does not literally make sense in
the general setting of Theorem \ref{thm:ka-quass}, but
we do have Conjecture \ref{conj:ka-quass-spherical-bldg}
concerning an appropriate replacement.

Therem \ref{thm:loc-quass} still further weakens our assumption to
``local quasisemisimplicity'', where we require only that
every point of \(\Gamma(\ka)\) acts
quasisemisimply on \(\Gt_\ka\); but
this is too general for our techniques to handle, so
we must impose a smoothness requirement for
the action of each element of \(\Gamma(\ka)\).
(This assumption is vacuously satisfied unless
we are in characteristic \(2\), and moreover
encounter the difficult case mentioned above of
a certain kind of outer action on \(\mathsf A_{2n}\).)
This restriction is not a failure of our proof techniques, but
an indication of a genuine counterexample to
Theorem \ref{thm:loc-quass}(\ref{subthm:loc-quass-reductive}).
See Example \ref{ex:vust:cones:prop:5:cor}.
We do not know whether the analogue of
Theorem \ref{thm:quass}(\ref{subthm:quass-smoothable})
holds in the situation of
Theorem \ref{thm:loc-quass} if
we drop the extra smoothness assumption from the latter.
Other than this, Theorem \ref{thm:loc-quass}
subsumes Theorem \ref{thm:quass}.

Theorem \ref{thm:quass} overlaps significantly with
\cite{adler-lansky:lifting1}*{Proposition 3.5}.
Much of our proof technique is quite different
(although we do wind up citing
\cite{adler-lansky:lifting1}*{Proposition 3.5} itself in
the proof of
Proposition \ref{prop:quass-facts}(\ref{subprop:quass-split-descends}),
on which Theorem \ref{thm:quass}(\ref{subthm:quass-reductive})
relies), and
our result is somewhat more general, in that
it allows for an action by an algebraic group, rather than
a finite group.

\begin{rem}
\label{rem:error}
Theorem \ref{thm:quass} also corrects an error in
\cite{adler-lansky:lifting1}*{Proposition 3.5},
which should have included the smoothness hypothesis
of Theorem \ref{thm:loc-quass},
which, as remarked above, is automatic in most cases.
Without such a hypothesis, the group $G$ of
Notation \ref{notn:main} is not necessarily reductive,
and one instead has the weaker conclusion of
Theorem \ref{thm:ka-quass}(\ref{subthm:ka-quass-how-reductive}).
The error comes from the citation of
\cite{lemaire:twisted-characters}*{Th\'eor\`eme 4.6},
whose proof rests on
Proposition 4.5 and thence on Lemme 4.5 of the same reference.
The last result asserts in particular that,
over a separably closed field,
a unipotent element of a reductive group lies in
a (rational) Borel subgroup.
This can fail for the so called
\textit{bad} unipotent elements, in the sense of
\cite{tits:unipotent-2}*{\S3.1}.
For example, if
\(k\) is an imperfect field of characteristic \(2\) and
\(t \in k\mult\) is a nonsquare in \(k\mult\), then
the image of
\(\begin{smallpmatrix}
0 & 1 \\
t & 0
\end{smallpmatrix}\) in \(\operatorname{PGL}_2(k)\) is
bad
\cite{tits:unipotent-2}*{Example 3.5}, and so
is a counterexample to
\cite{lemaire:twisted-characters}*{Lemme 4.5}.
\end{rem}

We will prove Theorem \ref{thm:quass} in
\S\S\ref{sec:quass-smooth}, \ref{sec:quass}, as we
build up a considerable `equivariant' structure theory for
reductive groups with quasisemisimple action.
Specifically,
Theorem \ref{thm:quass}%
	(\ref{subthm:quass-smoothable},%
	\ref{subthm:quass-reductive})
is proven at the end of
\S\ref{sec:quass-smooth};
Theorem \ref{thm:quass}(\ref{subthm:quass-smooth})
is proven after
Corollary \ref{cor:G-smooth}; and
Theorem \ref{thm:quass}(\ref{subthm:quass-spherical-bldg})
is proven after
Proposition \ref{prop:parabolic-Borus}.
We are then ready for Theorem \ref{thm:loc-quass},
the essential idea of which is to combine
Theorem \ref{thm:quass} with results of
\cite{prasad-yu:actions}.

\begin{mainthm}
\label{thm:quass}
Let \(k\), \(\Gamma\), \(\Gt\), and \(G\) be as in
Notation \ref{notn:main}.
Suppose that \((\Gt, \Gamma)\) is quasisemisimple.
	\begin{enumerate}[label=(\arabic*), ref=\arabic*]
	\addtocounter{enumi}{-1}
	\item\label{subthm:quass-smoothable}
	\((\Gt^\Gamma)\conn\) is smoothable.
	\item\label{subthm:quass-smooth}
	\((\Gt^\Gamma)\conn\) equals
\((Z(\Gt)^\Gamma)\conn\cdot\fix\Gt^\Gamma\) unless
\(p\) equals \(2\) and \((\Gt_\ks, \Gamma_\ks)\) is exceptional.
	\item\label{subthm:quass-reductive}
	$G$ is connected and reductive.
	\item\label{subthm:quass-spherical-bldg}
	The functorial map from the spherical building \(\SS(G)\) of \(G\)
	to the spherical building \(\SS(\Gt)\) of \(\Gt\)
	identifies \(\SS(G)\) with
	\(\SS(\Gt) \cap \SS(\Gt_\ka)^{\Gamma(\ka)}\).
	\end{enumerate}
\end{mainthm}

\begin{mainthm}
\label{thm:ka-quass}
Let \(k\), \(\Gamma\), \(\Gt\), and \(G\) be as in
Notation \ref{notn:main}.
Suppose that \((\Gt_\ka, \Gamma_\ka)\) is quasisemisimple.
	\begin{enumerate}[label=(\arabic*), ref=\arabic*]
	\item\label{subthm:ka-quass-how-reductive}
	\(G\) is an extension of
a reductive group by
a split unipotent group.
	\item\label{subthm:ka-quass-quass}
	The following statements are equivalent.
		\begin{enumerate}[label=(\alph*), ref=\alph*]
		\item\label{case:ka-quass-quass}
		\((\Gt_\ks, \Gamma_\ks)\) is quasisemisimple.
		\item\label{case:ka-quass-smoothable}
		\((\Gt^\Gamma)\conn\) is smoothable.
		\item\label{case:ka-quass-reductive}
		\(G\) is reductive, and
\(C_\Gt(G)\) is of multiplicative type.
		\item\label{case:ka-quass-torus}
		There is a torus \(T\) in \(G\) such that
\(T_\ka\) is a maximal torus in \(\fix\Gt_\ka^{\Gamma_\ka}\).
		\item\label{case:ka-quass-Borel}
		There are a
\(\Gamma_\ks\)-stable maximal torus \(\Tt\) in \(\Gt_\ks\), and a
\(\Gamma_\ka\)-stable Borel subgroup of \(\Gt_\ka\) containing
\(\Tt_\ka\).
		\end{enumerate}
	\end{enumerate}
\end{mainthm}

Since quasisemisimplicity of an action is
preserved under base change,
the assumption of quasisemisimplicity of \((\Gt_E, \Gamma_E)\) is
weaker the larger \(E\) is.
Thus, one might wonder if Theorem \ref{thm:ka-quass} could
be made stronger by weakening its hypothesis to require
quasisemisimplicity of \((\Gt_E, \Gamma_E)\) for
some field extension \(E/k\), not necessarily algebraic.
In this case, we would have that \(\mathscr B^\Gamma(E)\)
was nonempty, where \(\mathscr B\) is the variety of
Borel--torus pairs in \(\Gt\)
(a homogeneous variety for which point stabilisers are
maximal tori); so the Nullstellensatz would give that
\(\mathscr B^\Gamma(\ka)\) was also nonempty, hence that
\((\Gt_\ka, \Gamma_\ka)\) was also quasisemisimple.
That is, we would not gain any additional power from
such a replacement.

\begin{mainthm}
\label{thm:loc-quass}
Let \(k\), \(\Gamma\), \(\Gt\), and \(G\) be as in
Notation \ref{notn:main}.
Suppose, for every \(\gamma \in \Gamma(\ka)\), that
\(\gamma\) acts quasisemisimply on \(\Gt_\ka\) and
\((\Gt\adform)_\ka^\gamma\) is smooth.
	\begin{enumerate}[label=(\arabic*), ref=\arabic*]
	\item\label{subthm:loc-quass-smoothable}
	\((\Gt^\Gamma)\conn\) equals
\((Z(\Gt)^\Gamma)\conn\cdot(\Gt^\Gamma)\smooth\conn\).
	\item\label{subthm:loc-quass-reductive}
	$G$ is reductive.
	\item\label{subthm:loc-quass-spherical-bldg}
	The functorial map from the spherical building \(\SS(G)\) of \(G\)
	to the spherical building \(\SS(\Gt)\) of \(\Gt\)
	identifies \(\SS(G)\) with
	\(\SS(\Gt) \cap \SS(\Gt_\ka)^{\Gamma(\ka)}\).
	\end{enumerate}
\end{mainthm}

Lemma \ref{lem:density-bldg} will give a version of
Theorems
\ref{thm:quass}(\ref{subthm:quass-spherical-bldg})
and
\ref{thm:loc-quass}(\ref{subthm:loc-quass-spherical-bldg})
that
does not require passing to \(\ka\) to identify the image of \(\SS(G)\).

\begin{rem}
\label{rem:ss-loc-quass}
Let \((\Gt, \Gamma)\) be a reductive datum over \(k\) such that
every element of \(\Gamma(\ka)\) is semisimple.
Then the hypotheses of Theorem \ref{thm:loc-quass}
are satisfied
\cite{steinberg:endomorphisms}*{Theorem 7.5}.

In this case, \(\Gamma\) is linearly reductive
by Remark \ref{rem:all-ss=>lr}.
Thus
Theorem \ref{thm:loc-quass}(\ref{subthm:loc-quass-smoothable})
can be strengthened to the statement that \(\Gt^\Gamma\) is
smooth, not just smoothable
\cite{conrad-gabber-prasad:prg}*{Proposition A.8.10(2)}; and
Theorem \ref{thm:loc-quass}(\ref{subthm:loc-quass-reductive}) follows from
\cite{conrad-gabber-prasad:prg}*{Proposition A.8.12}.
We do not know if
Theorem \ref{thm:loc-quass}(\ref{subthm:loc-quass-spherical-bldg})
has already appeared in the literature in this setting, but
its special case where
\(\Gamma\) is generated by a single inner automorphism is
\cite{curtis-lehrer-tits:spherical}*{Proposition 5.1}.
\end{rem}

\numberwithin{equation}{subsection}
\section{Generalities}
\label{sec:generalities}

Throughout this section, we continue with the
field \(k\) of characteristic exponent \(p\) from
\S\ref{sec:notation}.
Let \(\Gt\) be a smooth, connected \(k\)-group.
We will assume in \S\ref{subsec:ind} that
\(\Gt\) is reductive, but we do not do so yet.

\subsection{Fixed points}
\label{subsec:fixed-generalities}

We will soon
(after Lemma \ref{lem:fixed-surjective})
take \(\Gamma\) to be a smooth \(k\)-group acting on \(\Gt\),
but we do not do so quite yet.

The proof of Lemma \ref{lem:fixed-surjective} is essentially
contained in \cite{adler-lansky:lifting1}*{Proposition 3.5}.
Compare Corollary \ref{cor:fixed-surjective} to
\cite{conrad-gabber-prasad:prg}*{Proposition A.8.14(1)}.

\begin{lem}
\label{lem:fixed-surjective}
Let \(\Gamma\) be a finite abstract subgroup of \(\Aut(\Gt_\ks)\) that
is preserved by \(\Gal(k)\).
Suppose that
\(\Zt\) is a finite, central subgroup of \(\Gt\) such that
\(\Zt_\ks\) is preserved by \(\Gamma\).
Write \(\Gt^\Gamma\) and \((\Gt/\Zt)^\Gamma\) for the
descents
to \(k\) of the \(\Gal(k)\)-stable groups
\(\Gt_\ks^\Gamma\) and \((\Gt/\Zt)_\ks^\Gamma\).
Then the map
\abmap{\fix\Gt^\Gamma}{\fix(\Gt/\Zt)^\Gamma}
is an isogeny.
\end{lem}

\begin{proof}
We may, and do, assume,
upon replacing \(k\) by \(\ks\), that
\(k\) is separably closed.

It is clear that the kernel of
\abmap{\fix\Gt^\Gamma}{\fix(\Gt/\Zt)^\Gamma}
is finite, so we need only show that
the map is surjective.

Put
\(\Gt' = \Gt/\Zt\),
\(G = \fix\Gt^\Gamma\),
and
\(G' = \fix(\Gt')^\Gamma\).
Write \(\phi\) for the quotient map \abmap\Gt{\Gt'}.
The action of \(\Gt'\) on \(\Gt\) restricts to
an action of \(G'\) on \(G\),
so that \(\phi(G)\) is normal in \(G'\).
Then we need to show that \(G'/\phi(G)\) is trivial.

Since \(G'\) is smooth and connected, so is \(G'/\phi(G)\).
Thus, since \(k\) is separably closed,
it suffices to show that \((G'/\phi(G))(k)\) is finite.
Since \(\phi(G)\) is smooth and \(k\) is separably closed,
we have that
\abmap{G'(k)}{(G'/\phi(G))(k)} is surjective.
Since this map is trivial on \(\phi(G(k))\), hence factors through
\abmap{G'(k)}{G'(k)/\phi(G(k))},
it suffices to show that
\(G'(k)/\phi(G(k))\) is finite.

Since \(\Gt^\Gamma(k)\) equals \(\Gt(k)^\Gamma\), and
analogously for \(\Gt'\), we have
the exact sequence
\[\xymatrix{
	\Gt^\Gamma(k) \ar[r] &
	(\Gt')^\Gamma(k) \ar[r] &
	H^1(\Gamma, \Zt(k)).
}\]
Since \(\Gamma\) and \(\Zt(k)\) are finite,
so is
\(H^1(\Gamma, \Zt(k))\);
so \((\Gt')^\Gamma(k)/\phi(\Gt^\Gamma(k))\) is finite.
It thus suffices to show that
the kernel of
\abmap{G'(k)/\phi(G(k))}{(\Gt')^\Gamma(k)/\phi(\Gt^\Gamma(k))}
is finite.
The kernel is
\((\phi(\Gt^\Gamma(k)) \cap G'(k))/\phi(G(k))\),
which is contained in \(\phi(\Gt^\Gamma(k))/\phi(G(k))\).
This latter is the image under \(\phi\) of
\(\Gt^\Gamma(k)/G(k)
= (\Gt^\Gamma)\smooth(k)/\fix\Gt^\Gamma(k)
= ((\Gt^\Gamma)\smooth/\fix\Gt^\Gamma)(k)
= \pi_0((\Gt^\Gamma)\smooth)(k)\),
which is finite.
Thus, \(G'(k)/\phi(G(k))\) is finite, as desired.
\end{proof}

For the remainder of \S\ref{subsec:fixed-generalities},
let \(\Gamma\) be a \(k\)-group acting on \(\Gt\).

The linear reductivity hypothesis of
Corollary \ref{cor:fixed-surjective} is satisfied whenever
\((\Gt, \Gamma)\) is a quasisemisimple reductive datum over \(k\),
or even just if
\((\Gt_\ka, \Gamma_\ka)\) is quasisemisimple,
by Remark \ref{rem:torus-quass}(\ref{subrem:quass-to-torus}),
so this result may be considered an analogue of
\cite{steinberg:endomorphisms}*{Lemma 9.2(a)}.
The main difference between
Lemma \ref{lem:fixed-surjective} and
Corollary \ref{cor:fixed-surjective} is that,
in the former, \(\Gamma\) is a finite abstract group, whereas
in the latter, \(\Gamma\) is an algebraic group with
a condition imposed on its identity component.

\begin{cor}
\label{cor:fixed-surjective}
Suppose that the image of
\(\Gamma\conn\) in
\(\uAut(\Gt)\) is
linearly reductive;
\(\Zt\) is a central subgroup of \(\Gt\) that
is preserved by \(\Gamma\); and
\(\Zt\) is finite, or \(\Gt\) is reductive.
Then
\abmap{\fix\Gt^\Gamma}{\fix(\Gt/\Zt)^\Gamma}
is a quotient map.
\end{cor}

\begin{proof}
Suppose first that \(\Zt\) is finite.
We have by
\cite{conrad-gabber-prasad:prg}*{Proposition A.8.10(2)}
that \(\Gt^{\Gamma\conn}\) and \((\Gt/\Zt)^{\Gamma\conn}\) are smooth,
and by
\cite{conrad-gabber-prasad:prg}*{Proposition A.8.14(1)}
that
\abmap
	{\fix\Gt^{\Gamma\conn} = (\Gt^{\Gamma\conn})\conn}
	{((\Gt/\Zt)^{\Gamma\conn})\conn = \fix(\Gt/\Zt)^{\Gamma\conn}}
is surjective.
Since
\(\fix\Gt_\ks^{\Gamma_\ks}\) equals
\(\fix{(\fix\Gt^{\Gamma\conn})}_\ks^{\pi_0(\Gamma)(\ks)}\),
and analogously for
\(\Gt/\Zt\),
we may replace
\(\Gt\) by \(\fix\Gt^{\Gamma\conn}\),
\(\Zt\) by its intersection with \(\fix\Gt^{\Gamma\conn}\),
and
\(\Gamma\) by \(\pi_0(\Gamma)(\ks)\).
Then Lemma \ref{lem:fixed-surjective} gives the result.

Now drop the assumption that \(\Zt\) is finite, and
suppose instead that \(\Gt\) is reductive.
We use this assumption only to conclude that
\(\Gt/\Gt\der\) is a torus.
By rigidity of tori
\cite{milne:algebraic-groups}*{Corollary 12.37},
the action of \(\Gamma\) on
\(\Gt/\Gt\der\) factors through an action of \(\pi_0(\Gamma)\).
Let \(E/k\) be a finite, separable extension such that
\((\Gt/\Gt\der)_E\) is split and
\(\pi_0(\Gamma)_E\) is constant.
Since \(\Gal(E/k) \ltimes \pi_0(\Gamma)(E)\) is finite,
and since \(\bX^*((\Gt/\Zt\cdot\Gt\der)_E) \otimes_\Z \Q\) is
a (\(\Gal(E/k) \ltimes \pi_0(\Gamma)(E)\))-stable subspace of
\(\bX^*((\Gt/\Gt\der)_E) \otimes_\Z \Q\),
we have that
there is a (\(\Gal(E/k) \ltimes \pi_0(\Gamma)(E)\))-stable complement
\(V\)
to
\(\bX^*((\Gt/\Zt\cdot\Gt\der)_E) \otimes_\Z \Q\) in
\(\bX^*((\Gt/\Gt\der)_E) \otimes_\Z \Q\).
Write \(\At\) for the quotient of \(\Gt/\Gt\der\) such that
\(\bX^*(\At_E)\) is \(\bX^*((\Gt/\Gt\der)_E) \cap V\).
Since \(\At\) is a quotient of \(\Gt/\Gt\der\), it
comes equipped with a quotient map \abmap\Gt\At.
Then \abmap\Gt{\At \times \Gt/\Zt} is
a \(\Gamma\)-equivariant, central isogeny, so
Lemma \ref{lem:fixed-surjective} gives that
\abmap{\fix\Gt^\Gamma}{\fix\At^\Gamma \times \fix(\Gt/\Zt)^\Gamma}
is also a central isogeny.
The result follows.
\end{proof}

\begin{cor}
\label{cor:lift-smooth}
Preserve the hypotheses and notation of
Corollary \ref{cor:fixed-surjective}.
If \((\Gt/\Zt)^\Gamma\) is smooth, then
\((\Gt^\Gamma)\conn\) equals
\((\Zt^\Gamma)\conn\cdot\fix\Gt^\Gamma\).
\end{cor}

\begin{proof}
Corollary \ref{cor:fixed-surjective} gives that
\abmap
	{\fix\Gt^\Gamma}
	{\fix(\Gt/\Zt)^\Gamma = ((\Gt/\Zt)^\Gamma)\conn}
is surjective.
Since the image of \((\Gt^\Gamma)\conn\) in
\(\Gt/\Zt\) lies in
\(((\Gt/\Zt)^\Gamma)\conn\), it follows that
\((\Gt^\Gamma)\conn\) is contained in
\(\Zt\cdot\fix\Gt^\Gamma\), hence in
its intersection
\(\Zt^\Gamma\cdot\fix\Gt^\Gamma\) with
\(\Gt^\Gamma\), hence in its identity component
\((\Zt^\Gamma)\conn\cdot\fix\Gt^\Gamma\).
The reverse containment is obvious.
\end{proof}

\begin{cor}
\label{cor:smooth-surjective}
Preserve the hypotheses and notation of Corollary \ref{cor:fixed-surjective}.
If \((\Gt^\Gamma)\conn\) is smoothable, then
\(((\Gt/\Zt)^\Gamma)\conn\) is smoothable.
The converse holds if \((\Zt^\Gamma)\conn\) is also smoothable,
which is automatic if \(\Gt\) is reductive.
\end{cor}

\begin{proof}
Put \(\Gt' = \Gt/\Zt\).
Write \(\pi\) for the quotient map
\abmap\Gt{\Gt'}.

We make a few general observations.
If \(\Ht\) is a connected subgroup of \(\Gt\),
then \(\Ht\cdot\Zt\conn\) is connected, and
\(\Ht\cdot\Zt/\Ht\cdot\Zt\conn\) is a quotient of
the \'etale group \(\Zt/\Zt\conn\), hence \'etale.
The characterization of the identity component of a group as
the unique connected, normal subgroup with \'etale quotient
\cite{milne:algebraic-groups}*{Proposition 1.31(a)}
shows that
\(\Ht\cdot\Zt\conn\) equals \((\Ht\cdot\Zt)\conn\).
Similarly, if \(\Ht\) is a smooth subgroup of \(\Gt\),
then \(\Ht\cdot\Zt\smooth\) is smooth, and
\(\Ht\cdot\Zt/\Ht\cdot\Zt\smooth\) is a quotient of
the infinitesimal group \(\Zt/\Zt\smooth\), hence infinitesimal.
Although it is not true in general that
the maximal smooth subgroup of a group is
the unique smooth, normal subgroup with infinitesimal quotient,
this fails only in one direction
(the maximal smooth subgroup need not be normal);
if
a group has a smooth, normal subgroup with infinitesimal quotient, then
that subgroup is the maximal smooth subgroup.
Therefore, \((\Ht\cdot\Zt)\smooth\) equals \(\Ht\cdot\Zt\smooth\).
Analogous reasoning works over \(\ka\).

Corollary \ref{cor:fixed-surjective} gives that
\(\pi(\fix\Gt^\Gamma)\) equals \(\fix(\Gt')^\Gamma\) and
\(\pi(\fix\Gt_\ka^{\Gamma_\ka})\) equals
\(\fix(\Gt'_\ka)^{\Gamma_\ka}\).
Remark \ref{rem:conn-smooth} gives that
\((\Gt^\Gamma)\conn\) is smoothable if and only if
\((\fix\Gt^\Gamma)_\ka\) equals
\(\fix\Gt_\ka^{\Gamma_\ka}\),
and analogously for \(\Gt'\) and \(\Zt\).
Thus, if \((\Gt^\Gamma)\conn\) is smoothable, then we have
the equalities
\[
\fix(\Gt'_\ka)^{\Gamma_\ka} =
\pi(\fix\Gt_\ka^{\Gamma_\ka}) =
\pi((\fix\Gt^\Gamma)_\ka) =
(\fix\Gt'{}^\Gamma)_\ka,
\]
so that \((\Gt'{}^\Gamma)\conn\) is smoothable.

If \((\Gt'{}^\Gamma)\conn\) is smoothable, then we analogously have
the equality
\[
\pi(\fix\Gt_\ka^{\Gamma_\ka}) =
\pi((\fix\Gt^\Gamma)_\ka),
\]
so
\(\fix\Gt_\ka^{\Gamma_\ka}\cdot\Zt_\ka\) equals
\((\fix\Gt^\Gamma)_\ka\cdot\Zt_\ka\).
Several equalities follow:
\begin{itemize}
\item
of the groups
\(\fix\Gt_\ka^{\Gamma_\ka}\cdot\Zt_\ka^{\Gamma_\ka}\) and
\((\fix\Gt^\Gamma)_\ka\cdot\Zt_\ka^{\Gamma_\ka}\)
of \(\Gamma_\ka\)-fixed points;
\item
then of their maximal smooth subgroups
\(\fix\Gt_\ka^{\Gamma_\ka}\cdot(\Zt_\ka^{\Gamma_\ka})\smooth\) and
\((\fix\Gt^\Gamma)_\ka\cdot(\Zt_\ka^{\Gamma_\ka})\smooth\);
\item
then of their identity components
\(\fix\Gt_\ka^{\Gamma_\ka} =
\fix\Gt_\ka^{\Gamma_\ka}\cdot\fix\Zt_\ka^{\Gamma_\ka}\) and
\((\fix\Gt^\Gamma)_\ka\cdot\fix\Zt_\ka^{\Gamma_\ka}\).
\end{itemize}
If additionally \((\Zt^\Gamma)\conn\) is smoothable, then
this shows that
\(\fix\Gt_\ka^{\Gamma_\ka}\) equals
\((\fix\Gt^\Gamma)_\ka\cdot(\fix\Zt^\Gamma)_\ka =
(\fix\Gt^\Gamma)_\ka\), so that
\((\Gt^\Gamma)\conn\) is smoothable.
\end{proof}

\subsection{Semisimple elements in groups and algebras, and
	their centralizers}

Lemma \ref{lem:toral-Lie-is-Lie-torus} is related to
\cite{conrad-gabber-prasad:prg}*{Proposition A.8.10(2)}.

\begin{lem}
\label{lem:toral-Lie-is-Lie-torus}
If
\(\mathfrak s\) is a commutative subalgebra of \(\Lie(G)\)
such that all elements of \(\mathfrak s\) are semisimple,
then
\(C_G(\mathfrak s)\) is smooth, and
there is a maximal torus \(T\) in \(G\) such that
\(\mathfrak s\) is contained in \(\Lie(T)\).
\end{lem}

\begin{proof}
First suppose that \(k\) is algebraically closed, and
reason by induction on \(\dim(\mathfrak s)\).
If the dimension is \(0\), then \(\mathfrak s\), and
the result, are trivial.
Thus we may, and do, suppose that the dimension is positive,
choose a codimension-\(1\) subspace \(\mathfrak s_1\) of
\(\mathfrak s\), and assume that we have
already proven the result for \(\mathfrak s_1\).

In particular, there is a maximal torus \(T_1\) in \(G\) such that
\(\mathfrak s_1\) is contained in \(\Lie(T_1)\).
Then \(T_1\) is contained in \(C_G(\mathfrak s_1)\), so
the ranks of \(C_G(\mathfrak s_1)\) and \(G\) are equal, i.e.,
every torus that is maximal in \(C_G(\mathfrak s_1)\)
remains maximal in \(G\).
We have that \(\mathfrak s\) is contained in
\(C_{\Lie(G)}(\mathfrak s_1) = \Lie(C_G(\mathfrak s_1))\);
and
\(C_{C_G(\mathfrak s_1)}(\mathfrak s)\) equals
\(C_G(\mathfrak s)\).
Thus we may, and do, assume, upon
replacing \(G\) by \(C_G(\mathfrak s_1)\), that
\(\mathfrak s_1\) is contained in \(\Lie(G)^G\).

Let \(X_1\) be an element of \(\mathfrak s \setminus \mathfrak s_1\),
so that \(C_G(X_1)\) equals \(C_G(\mathfrak s)\).
Then \cite{borel:linear}*{Proposition 9.1(2)} gives that
\(C_G(\mathfrak s) = C_G(X_1)\) is smooth, and
\cite{borel:linear}*{Proposition 11.8} gives that
\(X_1\) belongs to the Lie algebra of
a maximal torus \(T\) in \(G\).
Since \(T\) is \(G(k)\)-conjugate to \(T_1\)
\cite{conrad-gabber-prasad:prg}*{Theorem C.2.3},
and \(\mathfrak s_1\) is contained in \(\Lie(T_1) \cap \Lie(G)^G\),
we have that
\(\mathfrak s_1\), and hence
\(\mathfrak s = \mathfrak s_1 \oplus k X_1\), is
contained in \(\Lie(T)\).

Now drop the assumption that \(k\) is algebraically closed.
Obviously \(\mathfrak s \otimes_k \ka\) is
still commutative, and
every element of it is a commuting sum of semisimple elements,
hence semisimple.
By the case of \(k\) algebraically closed, which
we have already proven, we have that
\(C_G(\mathfrak s)_\ka =
	C_{G_\ka}(\mathfrak s \otimes_k \ka)\)
is smooth.

Now we argue as in \cite{borel:linear}*{Proposition 11.8}.
Let \(T\) be a maximal torus in \(C_G(\mathfrak s)\), and write
\(C = C_{C_G(\mathfrak s)}(T)\conn\) for the corresponding
Cartan subgroup of \(C_G(\mathfrak s)\).
Then
\(C\) is nilpotent by \cite{borel:linear}*{Corollary 11.7}, so
\cite{borel:linear}*{Proposition 10.6(3, 4)} gives that
\(C_\ka\) is the direct product \(T_\ka \times U\), where
\(U\) is the unipotent radical of \(C_\ka\).
For each \(X \in \mathfrak s\), we have that
\(X \otimes_k 1 \in \Lie(C_G(\mathfrak s)_\ka)\) belongs to
\(\Lie(C_G(\mathfrak s)_\ka)^{T_\ka} = \Lie(C_\ka)\), hence
may be written as a
\emph{commuting} sum
\(X \otimes_k 1 = X_{T_\ka} + X_U\), where
\(X_{T_\ka}\) belongs to \(\Lie(T_\ka)\), hence is semisimple, and
\(X_U\) belongs to \(\Lie(U)\), hence is nilpotent.
This is therefore the Jordan decomposition of \(X \otimes_k 1\), which
we already know is semisimple, so that
\(X \otimes_k 1\) equals \(X_{T_\ka}\), hence belongs to
\(\Lie(T_\ka) = \Lie(T) \otimes_k \ka\),
so \(X\) belongs to \(\Lie(T)\).
\end{proof}

Corollary \ref{cor:toral-Lie-is-Lie-torus} is related to
\cite{conrad-gabber-prasad:prg}*
	{Propositions A.8.12 and A.8.14(1)}.

\begin{cor}
\label{cor:toral-Lie-is-Lie-torus}
Suppose that \(G\) is reductive.
If \abmap G{G'} is a central quotient and
\(\mathfrak s'\) is a commutative subalgebra of \(\Lie(G')\)
such that all elements of \(\mathfrak s'\) are semisimple,
then the group
\(C_G(\mathfrak s')\conn\) is reductive, and
the restriction of the quotient \abmap G{G'}
to \(C_G(\mathfrak s')\conn\) is a quotient
\abmap{C_G(\mathfrak s')\conn}{C_{G'}(\mathfrak s')\conn}.
\end{cor}

\begin{proof}
We may, and do, assume, upon replacing
\(k\) by \(\ka\), that
\(k\) is algebraically closed.
By Lemma \ref{lem:toral-Lie-is-Lie-torus},
there is a torus, hence a maximal torus \(T'\), in \(G'\) such that
\(\mathfrak s'\) is contained in \(\Lie(T')\).

Write \(T\) for the pre-image of \(T'\) in \(G\).
Then \(T\) is a maximal torus, and, since
the restriction of
\abmap{\Lie(G)}{\Lie(G')}
to any root space for \(T\) in \(\Lie(G)\) is
an embedding in \(\Lie(G')\), we have that
inflation to \(T\) provides
a bijection of
\(\Phi(C_G(\mathfrak s'), T) = \Phi(C_{\Lie(G)}(\mathfrak s'), T)\) with
\(\Phi(C_{\Lie(G')}(\mathfrak s'), T') = \Phi(C_{G'}(\mathfrak s'), T')\).
If \(\alpha' \in \Phi(C_{G'}(\mathfrak s'), T')\) has
inflation \(\alpha \in \Phi(C_G(\mathfrak s'), T)\), then
the image of
the root group \(U_\alpha\) for \(T\) in \(G\) is
the root group \(U_{\alpha'}\) for \(T'\) in \(G'\), which
is contained in
\(C_{G'}(\mathfrak s')\).
Since the action of \(G\) on \(\Lie(G')\) factors through
the map \abmap G{G'}, we have that
\(U_\alpha\) is contained in \(C_G(\mathfrak s')\), hence in
\(C_G(\mathfrak s')\conn\).

The argument of
\cite{borel:linear}*{Proposition 13.19} shows that
\(C_{G'}(\mathfrak s')\conn\) is reductive,
and \cite{borel:linear}*{Proposition 13.20} shows that
it is generated by
\(T'\) and those root groups for \(T'\) in \(G'\)
corresponding to
roots in \(\Phi(C_{G'}(\mathfrak s'), T')\).
Thus
\abmap{C_G(\mathfrak s')\conn}{C_{G'}(\mathfrak s')\conn}
is surjective, hence a quotient map.
Then \(C_G(\mathfrak s')\conn\) is a smooth
(by Lemma \ref{lem:toral-Lie-is-Lie-torus}), connected
extension of the reductive group
\(C_{G'}(\mathfrak s')\conn\) by
\(\ker(\abmap G{G'})\), which is central in \(G\) and so
diagonalizable,
so \(C_G(\mathfrak s')\conn\) is reductive.
\end{proof}

As remarked in \S\ref{subsec:roots},
we will need to discuss Borel--de Siebenthal theory in
some of the detailed computations of \S\ref{sec:thm:loc-quass}.
Although Remark \ref{rem:gp-BdS-facts} seems to be well known,
we could not find its contents stated in the form that we need them.

\begin{defn}
\label{defn:gp-BdS}
Suppose that \(G\) is quasisplit.
Fix a Borel--torus pair \((B, T)\) in \(G\),
let \(S\) be the maximal split torus in \(T\), and fix
an element \(a \in \Delta(B, S)\).
Write \(\varpi^\vee\) for
the fundamental coweight corresponding to \(a\),
\(a_0\) for the \(\Delta(B, S)\)-highest root
in the irreducible component of \(\Phi(G, S)\) containing \(a\), and
\(n = \pair{a_0}{\varpi^\vee}\) for
the coefficient of \(a\) in \(a_0\).
Then we call \(C_G(\varpi^\vee(\mu_n))\conn\)
the Borel--de Siebenthal subgroup of \(G\)
associated to \((B, T, a)\).
\end{defn}

\begin{rem}
\label{rem:gp-BdS-facts}\hfill
\begin{enumerate}[label=(\alph*), ref=\alph*]
\item\label{subrem:BdS-Z}
Preserve the notation and hypotheses of Definition \ref{defn:gp-BdS}.
In the terminology of
Definition \ref{defn:rd-BdS}, we have that
\(\Phi(C_G(\varpi^\vee(\mu_n))\conn, T)\)
is the Borel--de Siebenthal subsystem of
\(\Phi(G, S)\) associated to
\((\Delta(B, S), a)\).
In particular,
Remark \ref{rem:rd-BdS-facts}(\ref{subrem:rd-BdS-fact})
gives that
\(Z(C_{G\adform}(\varpi^\vee(\mu_n))\conn)\) equals
\(\varpi^\vee(\mu_n)\).
\item\label{subrem:maximal-ss}
Let \(H\) be a proper connected, reductive subgroup of \(G\) that
contains a maximally split, maximal torus \(T\) in \(G\), and such that
the maximal split, central torus in \(H\) is
central in \(G\).
Write \(S\) for the maximal split torus in \(T\).
Then \(\Z\Phi(H, S)\) has finite index in \(\Z\Phi(G, S)\).
If \(\Phi(H, S)\) is
integrally closed in \(\Phi(G, S)\)
(which is automatic except if
\(p\) equals \(2\) or \(3\)
\cite{borel-tits:reductive-groups}*{Remarque 2.5}), then
Remark \ref{rem:rd-BdS-facts}%
	(\ref{subrem:springer-steinberg:conj:sec:4.5})
gives that there are
a Borel subgroup \(B\) of \(G\) that contains \(T\), and
a root \(\alpha \in \Delta(B, S)\), such that
the coefficient of \(\alpha\) in the
\(\Delta(B, S)\)-highest root in
the irreducible component of \(\Phi(G, S)\) containing \(a\) is
prime, and
\(H\) is contained in
the Borel--de Siebenthal subgroup of \(G\) associated to
\((B, T, \alpha)\).
\end{enumerate}
\end{rem}

Lemma \ref{lem:gl-Levi} seems
to be well known, but we do not know a reference.

\begin{lem}
\label{lem:gl-Levi}
Let \(\Gt\) be a connected, reductive group such that
\(\Gt\adform\) is isomorphic to a product of
projective general linear groups, and
let \(\Gamma\) be a smooth, diagonalizable subgroup of
\(\Gt\adform\).
Then \((\Gt^\Gamma)\conn\) is a Levi subgroup of \(\Gt\).
\end{lem}

\begin{proof}
Suppose that
\((n_1, \dotsc, n_d)\) is a vector of positive integers such that
\(\Gt\adform\) is isomorphic to
\(\prod_{i = 1}^d \PGL_{n_i}\).
Write \(\Gamma_i\) for the projection of \(\Gamma\) on
the \(i\)th factor.
Since the pre-image of \(\Gamma_i\) in \(\GL_{n_i}\) is
contained in a split torus in \(\GL_{n_i}\)
\cite{milne:algebraic-groups}*{Theorem 12.12},
we have that \(\Gamma_i\) itself is contained in
a split torus in \(\PGL_{n_i}\), so
\(\Gamma \subseteq \prod \Gamma_i\) is contained in
a split torus \(\Tt\) in \(\prod \PGL_{n_i} = \Gt\adform\).
We may, and do, arrange, by enlarging \(\Tt\), that
it is maximal.

The restriction to \((\Gt^\Gamma)\conn\) of the adjoint quotient
\abmap\Gt{\Gt\adform}
is surjective onto \((\Gt\adform^\Gamma)\conn\)
\cite{conrad-gabber-prasad:prg}*{Proposition A.8.14(1)},
and we have shown that
\((\Gt^\Gamma)\conn\) contains the kernel \(Z(\Gt)\) of
the quotient, so
\((\Gt^\Gamma)\conn\) is the full pre-image in
\(\Gt\) of \((\Gt\adform^\Gamma)\conn\).
It thus suffices
to prove the result under the assumption that
\(\Gt\) is adjoint, hence isomorphic to
\(\prod \PGL_{n_i}\).
Since this isomorphism identifies
\((\Gt^\Gamma)\conn\) with
\(\prod (\PGL_{n_i}^{\Gamma_i})\conn\),
we may, and do, work one factor at a time, and so assume that
\(\Gt\) is (isomorphic to) \(\PGL_n\).

We have
by \cite{conrad-gabber-prasad:prg}*{Proposition A.8.14} that
\(\Mt \ldef (\Gt^\Gamma)\conn\) is reductive.
Since \(\Gamma\) is contained in \(\Tt\), hence in \(\Mt\), we have that
\(\Gamma\) is central in \(\Mt\); so
we have the containments
\(\Mt\subseteq C_\Gt(Z(\Mt))\conn\subseteq  C_\Gt(\Gamma)\conn = \Mt\),
hence equality \(\Mt =  C_\Gt(Z(\Mt))\conn\).
Put \(\St = Z(\Mt)\smooth\conn\),
so that \(C_\Gt(\St)\) is a Levi subgroup of \(\Gt\).
It suffices to show that \(\Mt\) equals \(C_\Gt(\St)\).
Note that \(Z(\Mt)/\St\) is finite as \(X^*(\St)\) and \(X^*(Z(\Mt))\) have equal rank;
let \(n = |Z(\Mt)/\St|\).

Suppose \(\alpha\) belongs to \(\Z\Phi (C_{\Gt}(\St),\Tt)\).
Then \(\alpha\) is trivial on \(\St\), so
\(n\alpha\) is trivial on \(Z(\Mt)\);
that is, \(n\alpha\) belongs to \(\bX^*(\Tt/Z(\Mt))\), which
equals \(\Z\Phi(\Mt,\Tt)\)
since \(\Gt\) is adjoint.
It follows that \(\Z\Phi (C_\Gt(\St),\Tt)/ \Z\Phi (\Mt,\Tt)\) is finite.
But since \(\Phi(\Gt, \Tt)\) is of type \(\mathsf A\),
the torsion part of \(\Z\Phi (\Gt,\Tt) / \Z\Phi(\Mt,\Tt)\),
and hence that of \(\Z\Phi (C_\Gt(\St),\Tt)/ \Z\Phi (\Mt,\Tt)\),  is trivial.
Thus \(\Z\Phi (\Mt,\Tt)\) equals \(\Z\Phi(C_\Gt(\St),\Tt)\),
so \(\Mt\) equals \(C_\Gt(\St)\).
\end{proof}

\subsection{Induction of reductive data}
\label{subsec:ind}

We define a notion of induction of
reductive data that is adjoint to
the natural notion of restriction.
Our definition is motivated by that of induction for modules;
see \cite{jantzen:alg-reps}*{Part I, \S3.3}.
Filling in the details requires some background, which
we provide in Appendix \ref{app:ind}.

After Remark \ref{rem:ind-root-datum}, we will
fix a reductive datum \((\Gt, \Gamma)\) over \(k\), but
we do not do so yet.

\begin{defn}
\label{defn:ind-data}
Let \((\Gt_1, \Gamma_1)\) be a reductive datum over \(k\), and
\(\Gamma\) a smooth \(k\)-group admitting
\(\Gamma_1\) as an open subgroup.
Write \(\Ind_{\Gamma_1}^\Gamma \Gt_1\) for
the group \(k\)-sheaf
\(\uMor_{\Gamma_1}(\Gamma, \Gt_1)\),
which is a connected, reductive \(k\)-group by
Proposition \ref{prop:ind-scheme},
equipped with the action of \(\Gamma\) described in
Remark \ref{rem:act-on-power}.
We say that \((\Gt, \Gamma)\) is \emph{induced from \(\Gamma_1\)} if
it arises in this way, up to isomorphism.
More generally, we may use the same notation for
any \(k\)-group \(\Gt_1\) with \(\Gamma_1\)-action,
even if it is not connected and reductive.
Then \(\Ind_{\Gamma_1}^\Gamma (\cdot)\) may be viewed as a functor
in a natural way; namely, if
\(\Ht_1\) is another \(k\)-group with \(\Gamma_1\)-action and
\map{f_1}{(\Gt_1, \Gamma_1)}{(\Ht_1, \Gamma_1)} is a morphism, then we define
\map{\Ind_{\Gamma_1}^\Gamma(f_1)}
	{(\Ind_{\Gamma_1}^\Gamma \Gt_1, \Gamma)}
	{(\Ind_{\Gamma_1}^\Gamma \Ht_1, \Gamma)}
 to be the map
\abmap
	{\uMor_{\Gamma_1}(\Gamma, \Gt_1)}
	{\uMor_{\Gamma_1}(\Gamma, \Ht_1)}
given by post-composition with \(f_1\).
\end{defn}

\begin{rem}
\label{rem:ind-adjoint}
Preserve the notation of Definition \ref{defn:ind-data}.
Lemma \ref{lem:induction-split-case} shows that
\(\Ind_{\Gamma_1}^\Gamma(\Gt_{1\,\dersub})\) is
the derived subgroup of \(\Ind_{\Gamma_1}^\Gamma \Gt_1\); that, if
\map\chi{\Gt_{1\,\scsub}}{\Gt_1}
is the simply connected cover of \(\Gt_{1\,\dersub}\), then
\(\Ind_{\Gamma_1}^\Gamma(\chi)\) is
the simply connected cover of
\((\Ind_{\Gamma_1}^\Gamma \Gt_1)\der\); and that, if
\map\pi{\Gt_1}{\Gt_{1\,\adsub}} is
the adjoint quotient of \(\Gt_1\), then
\(\Ind_{\Gamma_1}^\Gamma(\pi)\) is
the adjoint quotient of \(\Ind_{\Gamma_1}^\Gamma \Gt_1\).
\end{rem}

\begin{rem}
\label{rem:ind-root-datum}
Preserve the notation of Definition \ref{defn:ind-data}.
Suppose that \(\Gt_1\) has a maximal torus \(\Tt_1\) that is
preserved by \(\Gamma_1\).
Then Lemma \ref{lem:induction-split-case} gives that
\(\Tt = \Ind_{\Gamma_1}^\Gamma \Tt_1\) is
a maximal torus in \(\Gt\) that is preserved by \(\Gamma\).
Remark \ref{rem:ind-char} provides a
(\(\Gal(k) \ltimes \Gamma(\ks)\))-equivariant identification of
\(\bX^*(\Tt_\ks)\) with
\(\Z[\Gamma(\ks)] \otimes_{\Z[\Gamma_1(\ks)]}
	\bX^*(\Tt_{1\,\ks})\), and then dually of
\(\bX_*(\Tt_\ks)\) with
\(\Hom_{\Z[\Gamma_1(\ks)]}(\Z[\Gamma(\ks)], \bX_*(\Tt_{1\,\ks}))\).
Thus, for every
\(\gamma \in \Gamma(\ks)\) and
	\(\alphat \in \Phi(\Gt_\ks, \Tt_\ks)\),
it makes sense to speak of
the element
\(\gamma \otimes \alphat\) of \(\bX^*(\Tt_\ks)\), and
the element \((\gamma \otimes \alphat)^\vee\) of
\(\bX_*(\Tt_\ks)\) that vanishes at \(\gamma' \in \Gamma(\ks)\)
unless \(\gamma'\) belongs to \(\gamma\Gamma_1(\ks)\),
in which case it sends \(\gamma'\) to
\((\gamma^{\prime\,{-1}}\gamma\alphat)^\vee\).
With the notation of Remark \ref{rem:concrete-descent},
we have for each \(\gamma \in \Gamma(\ks)\) that
\(\smashset
	{\sigma(\gamma) \otimes \alphat}
	{\sigma \in \Gal(k),
		\alphat \in \Phi(\Gt_{1\,\ks}, \Tt_{1\,\ks})}\)
equals \(\Phi(\Gt_{1\,\gamma\,\ks}, \Tt_{1\,\gamma\,\ks})\).
Thus \(\Phi(\Gt_\ks, \Tt_\ks)\) equals
\(\smashset
	{\gamma \otimes \alphat}
	{\gamma \in \Gamma(\ks),
		\alphat \in \Phi(\Gt_{1\,\ks}, \Tt_{1\,\ks})}\).
Further,
\abmapto
	{\gamma \otimes \alphat}
	{(\gamma \otimes \alphat)^\vee}
realizes \(\Phi(\Gt_\ks, \Tt_\ks)\) as a root system in
the sublattice of \(\bX^*(\Tt_\ks)\) that it spans.
\end{rem}

For the remainder of \S\ref{subsec:ind}, let
\((\Gt, \Gamma)\) be a reductive datum over \(k\).

\begin{lem}
\label{lem:ind-simple-ad}
Let \(\Nt_1\) be a
smooth, connected, normal, semisimple subgroup of \(\Gt\).
Write \(\Gamma_1\) for the stabilizer of \(\Nt_1\) in \(\Gt\).
\begin{enumerate}[label=(\alph*), ref=\alph*]
\item\label{sublem:Gamma1-exact}
The subgroup \(\Gamma_1\) of \(\Gamma\) is
open, and
the functor \(\Ind_{\Gamma_1}^\Gamma(\cdot)\)
on the category of \(k\)-groups equipped with an action of \(\Gamma_1\)
is exact.
\item\label{sublem:simple-ad}
There is a unique map \abmap\Gt{\Nt_{1\,\adsub}}
that restricts to the adjoint quotient of \(\Nt_1\) and
annihilates \(C_\Gt(\Nt_1)\smooth\conn\).
It is \(\Gamma_1\)-equivariant, and
annihilates \(C_\Gt(\Nt_1)\).
\end{enumerate}
Lemma \ref{lem:ind-first-adjoint} gives a map
\map\psi\Gt{\Ind_{\Gamma_1}^\Gamma \Nt_{1\,\adsub}}
corresponding to the map
\abmap\Gt{\Nt_{1\,\adsub}} of (\ref{sublem:simple-ad}).
Write \(\Nt\) for the smallest \(\Gamma\)-stable subgroup
of \(\Gt\) containing \(\Nt_1\).
\begin{enumerate}[resume*]
\item\label{sublem:Gamma-closure}
\(\Nt\) is semisimple, and
\(\Nt_\ks\) is generated by those almost-simple components
\(\Gt_1\) of \(\Gt_\ks\) that
admit a \(\Gamma(\ks)\)-conjugate contained in
\(\Nt_1(\ks)\).
\item\label{sublem:ind-isogeny}
The restriction of \(\psi\) to \(\Nt\) is
a central isogeny onto its image.
The multiplication map
\abmap{\ker(\psi)\smooth\conn \times \Nt}\Gt\
is a central isogeny.
\end{enumerate}
\end{lem}

\begin{proof}
Remark \ref{rem:torus-quass}(\ref{subrem:act-inner}) gives that
\(\Gamma_1\) contains the identity component of \(\Gamma\), hence
is open.
Since a sequence of
\(k\)-groups with \(\Gamma_1\)-, or \(\Gamma\)-, action is
exact if and only if it is exact as a sequence of
fppf group sheaves over \(k\),
Corollary \ref{cor:ksep-ind-scheme} gives that
\(\Ind_{\Gamma_1}^\Gamma(\cdot)\) is exact.
This shows (\ref{sublem:Gamma1-exact}).

If \(\Gt_1\) is an almost-simple component of \(\Gt_\ks\) that
admits a \(\Gamma(\ks)\)-conjugate contained in \(\Nt_{1\,\ks}\),
then \(\Gt_1\) is contained in \(\Nt_\ks\).
On the other hand, the subgroup of \(\Gt_\ks\) generated by
all such almost-simple components
is preserved by \(\Gamma(\ks)\), hence by \(\Gamma_\ks\);
contains \(\Nt_{1\,\ks}\)
\cite{milne:algebraic-groups}*{Theorem 21.51}; and
is preserved by \(\Gal(k)\), hence descends to
a \(\Gamma\)-stable subgroup of \(\Gt\) that contains \(\Nt_1\).
It is therefore precisely \(\Nt_\ks\).
Thus \(\Nt_\ks\), and so \(\Nt\), is
smooth and connected.
This shows (\ref{sublem:Gamma-closure}).

The classical structure theory of reductive groups
\cite{milne:algebraic-groups}*
	{Theorem 21.51 and Proposition 21.61(c)}
shows that \(\Gt_\ks\) is the almost-direct product of
\(\Nt_{1\,\ks}\) and
\(C_\Gt(\Nt_{1\,\ks})\smooth\conn) = (C_\Gt(\Nt_1)\smooth\conn)_\ks\),
so there is a unique map
\abmap{\Gt_\ks}{\Nt_{1\,\ks\,\adsub} = \Nt_{1\,\adsub\,\ks}}
that restricts to the adjoint quotient on
\(\Nt_{1\,\ks}\), and
annihilates \((C_\Gt(\Nt_1)\smooth\conn)_\ks\).
This is a stronger uniqueness statement than that asserted in
(\ref{sublem:simple-ad}), but we still need to show existence.
Our stronger uniqueness statement implies that our map
\abmap{\Gt_\ks}{\Nt_{1\,\adsub\,\ks}} is fixed by
\(\Gal(k) \ltimes \Gamma_1(\ks)\), hence
descends to a map as in (\ref{sublem:simple-ad}).
Since
\(\Gt\) equals \(\Nt_1\cdot C_\Gt(\Nt_1)\smooth\conn\),
we have that
\(C_\Gt(\Nt_1)\) equals \(Z(\Nt_1)\cdot C_\Gt(\Nt_1)\smooth\conn\),
and so is
annihilated by this map.
This shows (\ref{sublem:simple-ad}).

For each \(\gamma \in \Gamma(\ks)\), we have the direct product
\(\prod \gamma(\Gt_{1\,\adsub})\)
over all almost-simple components \(\Gt_1\) of \(\Nt_{1\,\ks}\).
Although replacing \(\gamma\) by a right \(\Gamma_1(\ks)\)-translate
can affect the order of the factors, it does not affect
the overall product.
Thus, it makes sense to consider the product
\(\prod_{\gamma \in (\Gamma/\Gamma_1)(\ks)}
	\prod_{\Gt_1} \gamma(\Gt_{1\,\adsub})\).
Lemma \ref{lem:induction-split-case} allows us to identify
\(\psi\) with the product map
\abmap
	{\Gt_\ks}
	{\prod_{\gamma \in (\Gamma/\Gamma_1)(\ks)}
		\prod_{\Gt_1} \gamma(\Gt_{1\,\adsub})},
where each component map
\abmap\Gt{\gamma(\Gt_{1\,\adsub})} is the canonical
projection on an almost-simple component of \(\Gt_{\adsub\,\ks}\).
Again, the classical structure theory of reductive groups
shows that this map restricts to
a central isogeny of \(\Nt_\ks\) onto its image, and
annihilates all
almost-simple components of \(\Gt_\ks\) not contained in \(\Nt_\ks\),
which are therefore contained in
\(\ker(\psi_\ks)\smooth\conn =
(\ker(\psi)\smooth\conn)_\ks\).
This shows (\ref{sublem:ind-isogeny}).
\end{proof}

\begin{lem}
\label{lem:ind-simple-sc}
Preserve the notation and hypotheses of
Lemma \ref{lem:ind-simple-ad}.
Suppose further that
\(\Nt_1\) is an almost-simple component of \(\Gt\), and
\(\Gamma/\Gamma_1\) is constant.
Then there is a \(\Gamma\)-equivariant, central isogeny
\(\phi\) from
\(\Ind_{\Gamma_1}^\Gamma \Nt_1\) onto
\(\Nt\) such that
\(\psi \circ \phi\) is
the adjoint quotient of \(\Ind_{\Gamma_1}^\Gamma \Nt_1\).
\end{lem}

\begin{proof}
Regard the inclusion
\abmap{\Nt_1}\Gt\ as
an element of \(\Hom_{\Gamma_1}(\Nt_1, \Gt)\),
hence, by Lemma \ref{lem:was-cor:trans-induction},
a \(\Gamma\)-fixed element of
\(\Mor_{\Gamma_1}(\Gamma, \uHom(\Nt_1, \Gt))\).

The inclusion
\abmap{(\Gamma/\Gamma_1)(k)}{(\Gamma/\Gamma_1)(\ks)}
is a bijection, so,
for every \(\gamma, \gamma' \in \Gamma(\ks)\) such that
\(\gamma\) and \(\gamma'\) belong to distinct
\(\Gamma_1(\ks)\)-cosets, we have that
\(\gamma\Nt_{1\,\ks}\) and \(\gamma'\Nt_{1\,\ks}\) are
the base changes to \(\ks\) of
distinct almost-simple factors of \(\Gt\).
That is, the element of
\(\Mor_{\Gamma_1}(\Gamma, \uHom(\Nt_1, \Gt))\)
corresponding to the inclusion \abmap{\Nt_1}\Gt\ satisfies
the commutativity condition of
Corollary \ref{cor:was-subprop:ind-image-of-second-smooth}, which therefore
produces a \(\Gamma\)-equivariant homomorphism
\map\phi{\Ind_{\Gamma_1}^\Gamma \Nt_1}\Gt\
given by Equation \eqref{eq:ind-image-of-second}.

Write \(\iota\) for the map
\abmap{\Nt_1}{\Ind_{\Gamma_1}^\Gamma \Nt_1} of
Definition \ref{defn:ind-first-section}.
Equation \eqref{eq:ind-image-of-second} shows two things.
First, the image of \(\phi_\ks\) is
contained in the product of the \(\Gamma(\ks)\)-conjugates of
\(\Nt_{1\,\ks}\), hence is contained in \(\Nt_\ks\); but
\(\phi \circ \iota\) is the identity, so the image of \(\phi\) is
a \(\Gamma\)-stable subgroup of \(\Gt\) containing \(\Nt_1\), hence
containing \(\Nt\).
That is, the image of \(\phi\) is precisely \(\Nt\).
Second, the diagram
\[\xymatrix{
\Nt_1
	\ar[r]_-\iota
	\ar@(ul,ur)@{^(->}[rr]^{} &
\Ind_{\Gamma_1}^\Gamma \Nt_1
	\ar[r]_-\phi &
\Gt
	\ar[r]_-\psi
	\ar@(ul,ur)@{->>}[rr] &
\Ind_{\Gamma_1}^\Gamma \Nt_{1\,\adsub}
	\ar[r] &
\Nt_{1\,\adsub}
}\]
commutes, so that \(\psi \circ \phi\) is the map
corresponding by functoriality to
the adjoint quotient
\abmap
	{\Nt_1}
	{\Nt_{1\,\adsub}}.
Remark \ref{rem:ind-adjoint} shows that
\(\Ind_{\Gamma_1}^\Gamma \Nt_1\) is semisimple and
\(\psi \circ \phi\) is
the adjoint quotient of
\(\Ind_{\Gamma_1}^\Gamma \Nt_1\), hence, in particular,
is surjective.
In particular, the kernel of \(\phi\) is central in
the semisimple group \(\Ind_{\Gamma_1}^\Gamma \Nt_1\), hence
finite, so that \(\phi\) is a central isogeny onto its image.
\end{proof}

Corollary \ref{cor:ind-simple} can be re-phrased
informally as follows.
With the notation \(\widehat G\) introduced there,
if \(\Gt\) is adjoint, then
the various maps \(\psi\) of Lemma \ref{lem:ind-simple-ad}
piece together to a \(\Gamma\)-equivariant isomorphism
\abmap\Gt{\widehat G}; whereas,
if \(\Gt\) is simply connected, then
the various maps \(\phi\) of Lemma \ref{lem:ind-simple-sc}
piece together to a \(\Gamma\)-equivariant isomorphism
\abmap{\widehat G}\Gt.

\begin{cor}
\label{cor:ind-simple}
Suppose that \(\pi_0(\Gamma)\) is constant.
For each \(\pi_0(\Gamma)(k)\)-orbit \(i\) of
almost-simple components of \(\Gt\),
fix a representative \(\Gt_i\) of \(i\),
put \(\Gamma_i = \stab_\Gamma(\Gt_i)\), and
let \(\map{\psi_i}{\Gt}{ \Ind_{\Gamma_i}^\Gamma \Gt_{i\,\adsub}}\)
and \(\map{\phi_i}{ \Ind_{\Gamma_i}^\Gamma \Gt_i}{\Gt}\)
be the maps of Lemma \ref{lem:ind-simple-ad} and
Lemma \ref{lem:ind-simple-sc}.
Put \(\widehat G = \prod_i \Ind_{\Gamma_i}^\Gamma \Gt_i\).
If \(\Gt\) is adjoint, then the map \((\psi_i)_i\)
is a \(\Gamma\)-equivariant isomorphism
\abmap\Gt{\widehat G}.
If \(\Gt\) is simply connected,
then the map taking
\((\gt_i)_i\) to the product of the \(\phi_i(\gt_i)\) is a
\(\Gamma\)-equivariant isomorphism
\abmap{\widehat G}\Gt.
\end{cor}

\begin{proof}
These maps are \(\Gamma\)-equivariant by construction.
Since \(\pi_0(\Gamma)\) is constant, so that
the inclusion \abmap{\pi_0(\Gamma)(k)}{\pi_0(\Gamma)(\ks)} is
an isomorphism,
Remark \ref{rem:ind-root-datum}
shows that
each of these maps induces
an isomorphism on root data, hence is
an isomorphism.
\end{proof}

\begin{rem}
\label{rem:fixed-simple}
Preserve the notation and hypothesis of Corollary \ref{cor:ind-simple}.
If \(\Gt\) is adjoint, then
\(\map{(\psi_i)_i}{\Gt}{\widehat G}\)
restricts to an isomorphism of
\(\Gt^\Gamma\) onto \(\widehat G^\Gamma\), whose
composition with the isomorphism
\abbimap{\widehat \Gt^\Gamma}{\prod \Gt_i^{\Gamma_i}}
of Lemma \ref{lem:was-cor:trans-induction} is an isomorphism of
\(\Gt^\Gamma\) onto \(\prod_i \Gt_i^{\Gamma_i}\).
\end{rem}

\subsection{Spherical buildings}
\label{subsec:spherical}

Recall the notion of the spherical building \(\SS(G)\) of a reductive \(k\)-group
\(G\) from \cite{curtis-lehrer-tits:spherical}*{\S2}.
If \(S\) is a split \(k\)-torus,
\(E/k\) is a field extension,
\(T\) is a split \(E\)-torus,
and \abmap{S_E}T is an embedding, then
we obtain a corresponding embedding \abinmap{\bX_*(S_E)}{\bX_*(T)},
which, when pre-composed with the natural isomorphism
\abbimap{\bX_*(S)}{\bX_*(S_E)}, furnishes an embedding
\abinmap{\bX_*(S)}{\bX_*(T)}, hence
\abinmap{V(S)}{V(T)}.
This induces an embedding
\abinmap
	{\SS(S) = (V(S) \setminus \sset0)/\R_{> 0}}
	{(V(T) \setminus \sset0)/\R_{> 0} = \SS(T)}
of spherical apartments
\cite{curtis-lehrer-tits:spherical}*{\S1}.
If we take
\(S\) to be a maximal torus in \(G\)
and
\(T\) to be a maximal torus in \(C_G(S)_E\), then
we see that every apartment of \(\SS(G)\) embeds
in an apartment of \(\SS(G_E)\).
If \(b\) belongs to \(\SS(S)\), viewed as
an apartment \(\AA(S)\) in \(\SS(G)\), and
we write
\(P_G(b)\) for the corresponding parabolic subgroup of \(G\)
\cite{curtis-lehrer-tits:spherical}*{\S1} and
\(b_E\) for the corresponding element of \(\AA(T)\), then
the analogous parabolic subgroup \(P_{G_E}(b_E)\) of \(G_E\) equals
\(P_G(b)_E\)
(as can be verified on the level of Lie algebras).
Thus, two apartments that are glued in \(\SS(G)\) are
also glued in \(\SS(G_E)\).
We thus obtain a canonical map \abmap{\SS(G)}{\SS(G_E)}.
We now make three observations that, together, show that
\abmap{\SS(G)}{\SS(G_E)} is an embedding
(i.e., injection):
\begin{itemize}
\item Any two elements of \(\SS(G)\) lie in a common apartment
\cite{curtis-lehrer-tits:spherical}*{(2.3)}.
\item For any tori \(S\) and \(T\) as above,
the restriction to \(\SS(S)\) of
\abmap{\SS(G)}{\SS(G_E)} is
the embedding \abmap{\AA(S)}{\AA(T)}.
\item
The map from an apartment into the full spherical building is
an embedding
\cite{curtis-lehrer-tits:spherical}*{Lemma 2.2(ii)}.
\end{itemize}
With this in mind, we use the map \abmap{\SS(G)}{\SS(G_E)} to regard
\(\SS(G)\) as a subset of \(\SS(G_E)\).
It thus makes sense to ask if an element of \(\SS(G)\) is fixed by
an automorphism of \(G_E\) (acting on \(\SS(G_E)\)), even if
that automorphism is not the base change to \(E\) of
an automorphism of \(k\).

Recall that \(\SS\) is a functor from
the category of reductive \(k\)-groups and embeddings
to
the category of sets and injections
\cite{curtis-lehrer-tits:spherical}*{\S4}.

\begin{lem}
\label{lem:spherical-Res}
Let
\(E/k\) be a finite, separable field extension.
\begin{enumerate}[label=(\alph*), ref=\alph*]
\item\label{sublem:spherical-Res}
Let \(H_1\) be a reductive \(E\)-group.
If
\(S_1\) is a maximal (\(E\)-)split torus in \(H_1\), and
\(S\) is the maximal (\(k\)-)split torus in \(\WRes_{E/k}S_1\), then
\(S\) is a maximal split torus in \(\WRes_{E/k}H_1\), and
the Weil adjunction
\abmap{\Hom_k(\GL_{1, k}, \WRes_{E/k}S_1)}{\Hom_E(\GL_{1, E}, S_1)}
restricts to an isomorphism
\abmap{\bX_*(S)}{\bX_*(S_1)}.
The map \abmapto{S_1}S is a bijection between
the maximal split tori in \(H_1\) and \(\WRes_{E/k}H_1\).
The resulting maps \abmap{\SS(S)}{\SS(S_1)} fit together into
a (\(\Gal(E/k) \ltimes H_1(E)\))-equivariant bijection
\abmap{\SS(\WRes_{E/k}H_1)}{\SS(H_1)}.
\item\label{sublem:which-spherical-bc}
Let \(H\) be a reductive \(k\)-group.
The inclusion \abmap{\SS(H)}{\SS(H_E)} and
the functorial map \abmap{\SS(H)}{\SS(\WRes_{E/k}H_E)}
are identified by the bijection in (\ref{sublem:spherical-Res}).
\end{enumerate}
\end{lem}

\begin{proof}
We have by
\cite{borel-tits:reductive-groups}*{\S6.21(i)} that
the (\(E\)-)rank of \(H_1\) is the (\(k\)-)rank of \(\WRes_{E/k}H_1\),
so the construction in the statement produces
maximal tori in \(\WRes_{E/k}H_1\).
Since the inclusion of
\(\bX_*(S)\) in
\(\bX_*(\WRes_{E/k}S_1) = \Hom_k(\GL_{1, k}, \WRes_{E/k}S_1)\)
is an equality, we have that
\abmap{\bX_*(S)}{\bX_*(S_1)} is a bijection.
Since the Weil adjunction is given by composition with the co-unit
\abmap{(\WRes_{E/k}S_1)_E}{S_1}, we have that
\abmap{\bX_*(S)}{\bX_*(S_1)} respects addition, hence is
an isomorphism.

That the map \abmapto{S_1}S is a bijection is
\cite{conrad-gabber-prasad:prg}*{Proposition A.5.15(2)}.
We thus have a bijection between
apartments in \(\SS(\WRes_{E/k}H_1)\) and
apartments in \(\SS(H_1)\) such that
there is a bijection between corresponding pairs of apartments.
This family of bijections is
(\(\Gal(E/k) \ltimes H_1(E)\))-equivariant, in
the obvious sense.
To obtain the desired \(H_1(E)\)-equivariant bijection
\abmap{\SS(\WRes_{E/k}H_1)}{\SS(H_1)}, we need only show that
the way that apartments are glued matches.

The discussion of
\cite{borel-tits:reductive-groups}*{\S6.20}
furnishes an isomorphism
\abmap{\bX^*(S_1)}{\bX^*(S)}
(the one denoted there by \(\beta\), not by \(\alpha\))
dual to our map \abmap{\bX_*(S)}{\bX_*(S_1)}, and
\cite{borel-tits:reductive-groups}*{\S6.21(i)}
shows that it identifies
\(\Phi(H_1, S_1)\)
with
\(\Phi(\WRes_{E/k}H_1, S)\) in such a way that
the Weil restriction of
the subgroup of \(H_1\) associated to a quasi-closed
set of roots in \(\Phi(H_1, S_1)\) is
the subgroup of \(\WRes_{E/k}H_1\) associated to the corresponding
set of roots in \(\Phi(\WRes_{E/k}H_1, S)\).
We do not go into the details of this latter point, only note that
it shows that, if \(b\) belongs to \(V(S)\) and
\(b_1\) is the corresponding element of \(V(S_1)\), then
\(P_{\WRes_{E/k}H_1}(b)\) is
\(\WRes_{E/k}P_{H_1}(b_1)\).
In particular, the canonical identification of
\((\WRes_{E/k}H_1)(k)\) with \(H_1(E)\) identifies
\(P_{\WRes_{E/k}H_1}(b)(k)\) with
\(P_{H_1}(b_1)(E)\).
The gluings of the apartments thus match, as desired.
This shows (\ref{sublem:spherical-Res}).

For (\ref{sublem:which-spherical-bc}), since
a spherical building is made by gluing together spherical apartments,
it suffices to check this for \(H\) a split torus \(S\).
Write \(T\) for the maximal split torus in \(\WRes_{E/k}S_E\).
We have the map
\abmap
	{\bX_*(S) = \Hom_k(\GL_{1, k}, S)}
	{\bX_*(S_E) = \Hom_E(\GL_{1, E}, S_E)}
used to define \abmap{\SS(S)}{\SS(S_E)}, as well as the map
\abmap
	{\bX_*(S_E) = \Hom_E(\GL_{1, E}, S_E)}
	{\bX_*(\WRes_{E/k}S_E) = \Hom_k(\GL_{1, k}, \WRes_{E/k}S_E)}
used to define \abmap{\SS(S_E)}{\SS(\WRes_{E/k}S_E)}.
The composite of these maps on cocharacter lattices is
exactly the functorial map
\abmap{\bX_*(S)}{\bX_*(\WRes_{E/k}S_E) = \bX_*(T)}
used to define the functorial map
\abmap{\SS(S)}{\SS(\WRes_{E/k}S_E)}.
This shows (\ref{sublem:which-spherical-bc}).
\end{proof}

For the remainder of \S\ref{subsec:spherical},
we fix a reductive datum \((\Gt, \Gamma)\) over \(k\).

Lemma \ref{lem:density-bldg} allows us
to state the conclusions in
Theorems
\ref{thm:quass}(\ref{subthm:quass-spherical-bldg})
and
\ref{thm:loc-quass}(\ref{subthm:loc-quass-spherical-bldg})
without passing to \(\ka\), as long as
\(\Gamma(k)\) is Zariski dense in \(\Gamma\).

\begin{lem}
\label{lem:density-bldg}
Suppose that
\(\Gamma(k)\) is Zariski dense in \(\Gamma\).
Then \(\SS(\Gt) \cap \SS(\Gt_\ka)^{\Gamma(\ka)}\) equals
\(\SS(\Gt)^{\Gamma(k)}\).
\end{lem}

\begin{proof}
It is clear that \(\SS(\Gt)^{\Gamma(k)}\) contains
\(\SS(\Gt) \cap \SS(\Gt_\ka)^{\Gamma(\ka)}\).
Suppose conversely that
\(\bt\) is an element of \(\SS(\Gt)^{\Gamma(k)}\).
Although we have agreed to regard the map
\abmap{\SS(\Gt)}{\SS(\Gt_\ka)} as
an inclusion, in this proof we will write
\(\bt_\ka\) for emphasis when
we regard \(\bt\) as an element of \(\SS(\Gt_\ka)\).

For every \(\gamma \in \Gamma(k)\), we have that
\(\gamma P_\Gt(\bt) = P_\Gt(\gamma\bt)\) equals
\(P_\Gt(\bt)\).
Thus \(P_\Gt(\bt)\) is preserved by \(\Gamma(k)\), hence
by \(\Gamma\).
In particular, the image of \(\Gamma\conn\) in
\(\Gt\adform\) under the map of Remark \ref{rem:torus-quass}(\ref{subrem:act-inner}) lies in
\(N_{\Gt\adform}(P_\Gt(\bt))\), which, by
\cite{conrad-gabber-prasad:prg}*{Propositions 2.2.9 and 3.5.7},
equals \(P_{\Gt\adform}(\bt)\).
In particular, for every
\(\gamma \in \Gamma\conn(\ka)\), the action of \(\gamma\) on
\(\Gt_\ka\), hence on \(\SS(\Gt_\ka)\), is by an element of
\(P_{\Gt\adform}(\bt)(\ka)\).
Such an element lifts to
\(P_\Gt(\bt)(\ka) = P_{\Gt_\ka}(\bt_\ka)\), and so
fixes \(\bt_\ka\) by the definition of the spherical building
\cite{curtis-lehrer-tits:spherical}*{\S2}.

Since \(\Gamma(k)\) is Zariski dense in \(\Gamma\), we have that
every connected component of \(\Gamma\) contains
an element of \(\Gamma(k)\); so also
every connected component of \(\Gamma_\ka\) contains
an element of \(\Gamma(k)\).
Thus, for every \(\gamma \in \Gamma(\ka)\), we can write
\(\gamma = \gamma_0\gamma_1\), with
\(\gamma_0 \in \Gamma\conn(\ka)\) and
\(\gamma_1 \in \Gamma(k)\).
We have shown that both \(\gamma_0\) and \(\gamma_1\) fix
\(\bt_\ka\), so \(\gamma\) does as well.
That is, \(\bt_\ka\) belongs to \(\SS(\Gt_\ka)^{\Gamma(\ka)}\),
so \(\bt\) belongs to \(\SS(\Gt) \cap \SS(\Gt_\ka)^{\Gamma(\ka)}\).
\end{proof}

Lemma \ref{lem:spherical-cr}
is related to complete reducibility,
in the sense of \cite{serre:cr}*{Definition 2.2.1}.

\begin{lem}
\label{lem:spherical-cr}
Suppose that \(\Ht\) is a reductive subgroup of \(\Gt\) containing
\(\fix\Gt^{\Gamma(k)}\).
Each of a pair of points in \(\SS(\Gt)^{\Gamma(k)}\) that are
opposite in \(\SS(\Gt)\), in the sense of
\cite{curtis-lehrer-tits:spherical}*{\S3},
belongs to \(\SS(\Ht)\).
\end{lem}

\begin{proof}
Let \(\bt_\pm\) be opposite points in
\(\SS(\Gt)^{\Gamma(k)}\).
Thus \(P_\Gt(\bt_\pm)\) are
opposite parabolic subgroups of \(\Gt\).
Write \(\Mt\) for their intersection, which is
a Levi component of both.
If \(\St\) is a maximal split torus in \(\Mt\), then,
\(\bt_\pm\) belong to \(\AA(\St)\), and,
by the definition of `opposite',
they satisfy
\(\bt_- = -\bt_+\) there.
Since the intersection of \(\Mt\) with
the (solvable) radical of \(P_\Gt(\bt_\pm)\) is
the radical of \(\Mt\), i.e., its center,
we have that the maximal torus \(\At\) in
the intersection of \(\St\) with the radical of \(P_\Gt(\bt_\pm)\) is
central in \(\Mt\).
Conversely, it is clear that
the maximal split, central torus in \(\Mt\) is
contained in \(\St\), hence in \(\At\), so
we have equality.
In particular, \(\At\) is preserved by \(\Gamma(k)\).
We have by \cite{curtis-lehrer-tits:spherical}*{Lemma 1.2(ii)} that
\(\bt_\pm\) belong to \(\SS(\At)\), hence to
\(\SS(\At)^{\Gamma(k)}\).
Thus there is a homomorphism from
\(\Gamma(k)\) to
the multiplicative group \(\R_{> 0}\) that measures
the (common, because they are opposite)
scaling factor by which \(\gamma\) acts on the rays
\(\bt_\pm \in (V(\At) \setminus \sset0)/\R_{> 0}\).
By rigidity of tori \cite{milne:algebraic-groups}*{Corollary 12.37},
the action of \(\Gamma(k)\) factors through
the finite group \(\pi_0(\Gamma)(k)\).
Since \(\R_{> 0}\) has no nontrivial, finite subgroup,
the homomorphism
\abmap{\Gamma(k)}{\R_{> 0}} is trivial.
That is, \(\gamma\) acts trivially on the rays \(\bt_\pm\),
which are therefore contained in
\((\bX_*(\At) \otimes_\Z \R)^{\Gamma(k)} =
\bX_*((\At^{\Gamma(k)})\smooth\conn) \otimes_\Z \R\).
That is, \(\bt_\pm\) belong to
\(\SS((\At^{\Gamma(k)})\smooth\conn)\), and so to
\(\SS(\Ht)\).
\end{proof}

\begin{lem}
\label{lem:spherical-descent}
Let \(H\) be a reductive \(k\)-group.
Then \(\SS(H)\) equals \(\SS(H_\ks)^{\Gal(k)}\).
\end{lem}

\begin{proof}
We have that
\(\SS(H_\ks)\) equals \(\bigcup \SS(H_E)\),
the union over all finite, Galois field extensions \(E/k\), so
it suffices to prove that
\(\SS(H)\) equals \(\SS(H_E)^{\Gal(E/k)}\).
By Lemma \ref{lem:spherical-Res}(\ref{sublem:which-spherical-bc}),
it suffices to show, for every such extension \(E/k\), that
the functorial map \abmap{\SS(H)}{\SS(\WRes_{E/k}H_E)} is
a bijection onto \(\SS(\WRes_{E/k}H_E)^{\Gal(E/k)}\).

Functoriality implies that the image of \(\SS(H)\)
lies in \(\SS(\WRes_{E/k}H_E)^{\Gal(E/k)}\).
If \(b_+\) belongs to
\(\SS(\WRes_{E/k}H_E)^{\Gal(E/k)}\), then
the corresponding parabolic subgroup
\(P_{\WRes_{E/k}H_E}(b_+)\),
which by \cite{borel-tits:reductive-groups}*{Corollaire 6.19}
is of the form \(\WRes_{E/k}P_1^+\) for some
parabolic subgroup \(P_1^+\) of \(H_E\),
is preserved by the \emph{algebraic} action of \(\Gal(E/k)\),
so that \(P_1^+\) is preserved by the \emph{field} action of \(\Gal(E/k)\)
and hence is of the form \(P^+_E\) for some
parabolic subgroup \(P^+\) of \(H\).
If \(P^-\) is an opposite parabolic subgroup of \(H\), then
\(\WRes_{E/k}P^-\) is a parabolic subgroup of \(\WRes_{E/k}H_E\) that is
opposite to \(P_{\WRes_{E/k}H_E}(b_+)\) and
preserved by \(\Gal(E/k)\).
By \cite{curtis-lehrer-tits:spherical}*{\S3}, there is
a unique point \(b_- \in \SS(\WRes_{E/k}H_E)\) that is
opposite to \(b_+\) and
satisfies \(P_{\WRes_{E/k}H_E}(b_-) = \WRes_{E/k}P^-\).
By uniqueness, \(b_-\) is also fixed by \(\Gal(E/k)\), so
Lemma \ref{lem:spherical-cr} gives that
\(b_+\) belongs to the spherical building of
\((\WRes_{E/k}H_E)^{\Gal(E/k)}\).
This latter group is precisely the image of \(H\) in
\(\WRes_{E/k}H_E\), so
\(b_+\) belongs to the image of \(\SS(H)\).
\end{proof}

\numberwithin{equation}{section}
\section{Quasisemisimple actions on root systems}
\label{sec:rd-quass}

Let \(\Psit\) be a (possibly non-reduced) root datum.
In \S\ref{sec:rd-quass}, unlike in most of the rest of the paper,
we let \(\Gamma\) be an abstract group, and
\abmap\Gamma{\Aut(\Psit)} a quasisemisimple action that
factors through a finite quotient.
(Note that \(\Gamma\) is not a \(k\)-group;
indeed, there is no longer a field \(k\) in sight.)

We are only interested in the root system of \(\Psit\), so,
whenever convenient, we may
replace the root datum
\((\Xt, \Phit, \Xt^\vee, \Phit^\vee)\)
by the corresponding `adjoint' datum
with character lattice \(\Z\Phit\) and
root system \(\Phit\).
In particular, we may, and do, assume that
\(\Aut(\Psit)\) is finite.
Then,
since all of our constructions depend only on the action of \(\Gamma\),
we may always replace \(\Gamma\) by its image in
\(\Aut(\Psit)\) and so work with a finite acting group; but
occasionally it is handy not to have to do so.

Below, we present a collection of results concerning
the action of $\Gamma$ on a root system.
In applications, we will often have an action of a
smooth \(k\)-group $\Gamma$
on a connected, reductive group $\Gt$,
and we will apply these results to the natural action of
\(\Gal(k) \ltimes \pi_0(\Gamma)(\ks)\)
on the (absolute) root datum of $\Gt_\ks$.
See Notation \ref{notn:pi0-act}.

The pair \((\Psit, \Gamma)\) has an associated ``quotient root datum''
\(\Psi = (X, \Phi, X^\vee, \Phi^\vee)\),
constructed in \cite{adler-lansky:data-actions}*{Theorem 7}.
(In \cite{adler-lansky:data-actions},
they write \(\Psi\) for what we call \(\Psit\),
and \(\bar\Psi\) for what we call \(\Psi\).)
It is characterized by the facts that
\(X\) is the maximal torsion-free quotient of the module of co-invariants
\(\Xt_\Gamma\),
and \(\Phi\) is the image in \(X\) of \(\Phit\).
In particular, \(X^\vee\) is the module of invariants
\((\Xt^\vee)^\Gamma\).
We write \(i_\Gamma^*\) for the quotient morphism
\abmap\Xt X, so that \(i_\Gamma^*(\Phit)\) equals \(\Phi\).
(The behavior of the transpose map \abmap{X^\vee}{\Xt^\vee}
on \(\Phi^\vee\) is somewhat more complicated;
see
Proposition \ref{prop:rd-restriction}(\ref{subprop:coroot-nonmultipliable})
and
Lemma \ref{lem:half-pairs-a}(\ref{sublem:coroot-multipliable}).)
For \(\at \in \Phit\),
we refer to \(a = i_\Gamma^*(\at)\)
as the \emph{restriction} of \(\at\),
and to \(\at\) as an \emph{extension} of \(a\).

\begin{rem}
\label{rem:rd-subquotient}
Let \(\Phi'\) be an integrally closed subsystem of \(\Phi\), and
choose a system \(\Phi^{\prime\,{+}}\) of
positive roots for \(\Phi'\).
If we write \(\Phit'\)
(respectively, \(\Phit^{\prime\,{+}}\)) for
the set of extensions of elements of \(\Phi'\)
(respectively, \(\Phi^{\prime\,{+}}\)), then
\(\Phit'\) is
an integrally closed subsystem of \(\Phit\), and
\(\Phit^{\prime\,{+}}\) is
a system of positive roots for \(\Phit'\), so that
the action of \(\Gamma\) on
\((\Xt, \Phit', \Xt^\vee, \Phit^{\prime\,\vee})\) is
quasisemisimple.
The ``quotient root datum'' is
\((X, \Phi', X^\vee, \Phi^{\prime\,\vee})\).
\end{rem}

Proposition \ref{prop:rd-restriction} is essentially some of \cite{adler-lansky:data-actions}*{\S2}, rephrased in our language.

\begin{prop}
\label{prop:rd-restriction}\hfill
	\begin{enumerate}[label=(\alph*), ref=\alph*]
	\item\label{subprop:rd-root-system}
	\(\Phi\) is a (possibly non-reduced) root system.
	\item\label{subprop:rd-fiber}
	\abmap{\Phit/\Gamma}\Phi\ is a bijection.
	\item\label{subprop:rd-Borus}
	\abmapto{\Phit^+}{i_\Gamma^*(\Phit^+)}
	is a bijection from
	\(\Gamma\)-stable systems of positive roots in \(\Phit\)
	to systems of positive roots in \(\Phi\), with
	inverse bijection
	\abmapto{\Phi^+}{(i_\Gamma^*)\inv(\Phi^+)}.
	\end{enumerate}
Fix \(a \in \Phi\), and write
\(\Phit_a\) for the set of elements of
\(\Phit\) whose restriction is
an integer multiple of \(a\).
This is an integrally closed subsystem of \(\Phit\).
	\begin{enumerate}[resume*]
	\item\label{subprop:coroot-nonmultipliable}
	If \(a\) is not multipliable in \(\Phi\), then
	each irreducible component of \(\Phit_a\) is
	of type \(\mathsf A_1\), and
	contains exactly one extension of \(a\), which
	spans it.
	\(a^\vee\) equals
\(\sum\limits_{i^*_\Gamma(\at) = a}
	\at^\vee\).
	\item\label{subprop:rd-multipliable}
	If \(a\) is multipliable in \(\Phi\),
then one of the following holds.
		\begin{enumerate}[label=(\roman*), ref=\roman*]
		\item\label{case:rd-multipliable(nred)}
		There is some positive integer \(n\) such that
		every irreducible component of \(\Phit\) intersecting \(\Phit_a\) is
		of type \(\mathsf{BC}_n\), and
		intersects \(\Phit_a\) in
		an irreducible component of \(\Phit_a\) of type \(\mathsf{BC}_1\).
		Each such component contains exactly one extension of \(a\), which
		spans it.
		\item\label{case:rd-multipliable(red)}
		There is some positive integer \(n\) such that
		every irreducible component of \(\Phit\) intersecting \(\Phit_a\) is
		of type \(\mathsf A_{2n}\), and
		intersects \(\Phit_a\) in
		an irreducible component of \(\Phit_a\) of type \(\mathsf A_2\).
		Each such component contains exactly two extensions of \(a\), which
		form a system of simple roots for it.
		\end{enumerate}
	\end{enumerate}
\end{prop}

\begin{proof}
Parts
(\ref{subprop:rd-root-system}),
(\ref{subprop:rd-fiber}),
and
(\ref{subprop:rd-Borus})
are \cite{adler-lansky:data-actions}*{Theorem 7, Lemma 6, and Lemma 14},
respectively.
Part
(\ref{subprop:coroot-nonmultipliable}) comes from combining
\cite{adler-lansky:data-actions}*{Notation 4} with the definition
before \cite{adler-lansky:data-actions}*{Theorem 7}.
Part
(\ref{subprop:rd-multipliable}) is
\cite{adler-lansky:data-actions}*{Remark 12}.
\end{proof}

\begin{rem}
\label{rem:rd-Gamma-induced}
A system \(\Deltat\) of simple roots for \(\Phit\)
that is preserved by \(\Gamma\)
is \textit{a fortiori} a basis of \(\Z\Phit\), so that
\(\Z\Phit\) is induced as a \(\Gamma\)-module.
Since \((\Psit^\vee, \Gamma)\) is also quasisemisimple,
we also have that \(\Z\Phit^\vee\) is induced.
\end{rem}

\begin{defn}
\label{defn:rd-exceptional}
In the situation of Proposition \ref{prop:rd-restriction}(\ref{subprop:rd-multipliable}),
an \emph{exceptional (unordered) pair} for \((\Psit, \Gamma)\)
is the multiset of order 2 whose underlying set
is the intersection of \(\Phit_a\) with
the set of positive roots in an irreducible component of \(\Phit\).
We say that the exceptional pair \textit{extends \(a\)}.
\end{defn}

Thus,
in the situation of
Proposition \ref{prop:rd-restriction}%
	(\ref{subprop:rd-multipliable})(\ref{case:rd-multipliable(red)})
	(respectively,
Proposition \ref{prop:rd-restriction}%
	(\ref{subprop:rd-multipliable})(\ref{case:rd-multipliable(nred)})),
an exceptional pair consists of 2 distinct elements
(respectively, a single element of multiplicity 2).
In either case, if \(\sset{\at, \at'}\) is
an exceptional pair, then
\(\at + \at'\) belongs to \(\Phit\).

\begin{lem}
\label{lem:half-pairs-a}
Suppose that \(a \in \Phi\) is multipliable.
	\begin{enumerate}[label=(\alph*), ref=\alph*]
	\item\label{sublem:coroot-multipliable}
	\(a^\vee\) equals \(2\sum (\at + \at')^\vee\),
	the sum taken over all exceptional pairs
	\(\sset{\at, \at'}\)
	for \((\Psit, \Gamma)\) extending \(a\).
	\item\label{sublem:pairs-orbit}
	The set of exceptional pairs extending \(a\) is nonempty, and
permuted transitively by \(\Gamma\).
\item\label{sublem:all-pairs-same}
Either
all exceptional pairs extending \(a\) have 2 distinct elements,
or
all exceptional pairs extending \(a\) have 1 element with
	multiplicity 2.
\end{enumerate}
\end{lem}

\begin{proof}
Claim
(\ref{sublem:coroot-multipliable}) comes from combining
\cite{adler-lansky:data-actions}*{Notation 4} with the definition
before \cite{adler-lansky:data-actions}*{Theorem 7}.
Claim
(\ref{sublem:pairs-orbit})
follows immediately from
Proposition \ref{prop:rd-restriction}%
	(\ref{subprop:rd-fiber},%
	\ref{subprop:rd-multipliable}).
Claim (\ref{sublem:all-pairs-same})
follows immediately from
(\ref{sublem:pairs-orbit}).
\end{proof}

\begin{defn}
\label{defn:inert}
Suppose that $a\in\Phi$ is multipliable.
Say that $a$ is \emph{split} (respectively, \emph{inert})
for \((\Psit, \Gamma)\)
if some
(hence, by
Lemma \ref{lem:half-pairs-a}(\ref{sublem:pairs-orbit}),
every)
exceptional pair extending $a$ consists of 2 distinct elements
(respectively, consists of a single element of multiplicity 2).
When \((\Psit, \Gamma)\) is understood, we may just say that
\(a\) is split or inert, without further qualification.
\end{defn}

\begin{rem}
\label{rem:not-inert}
If \(\Phit\) is reduced, then
every multipliable element of \(\Phi\) is split.
On the other hand,
if the action of \(\Gamma\) is trivial, then
every multipliable element of \(\Phi = \Phit\) is inert.
\end{rem}

\begin{lem}
\label{lem:half-pairs-b}
Suppose that \(a \in \Phi\) is multipliable.
\begin{enumerate}[label=(\alph*), ref=\alph*]
	\item\label{sublem:half-pairs}
	Every extension of \(a\) belongs to
exactly one exceptional pair
extending \(a\).
	\item\label{sublem:coexc-extensions}
	The map
\abmapto{\sset{\at, \at'}}{\at + \at'}
is a \(\Gamma\)-equivariant bijection from
the set of exceptional pairs extending \(a\) onto
the set of extensions of \(2a\).
	\item\label{sublem:exc-stab}
	If \(\sset{\at, \at'}\) is an
exceptional pair, then
the stabilizers in \(\Gamma\) of
\(\sset{\at, \at'}\) and of
\(\at + \at'\) are equal.
	\end{enumerate}
\end{lem}

\begin{proof}
The claims
(\ref{sublem:half-pairs},%
\ref{sublem:coexc-extensions})
follow immediately from
Proposition \ref{prop:rd-restriction}%
	(\ref{subprop:rd-fiber},%
	\ref{subprop:rd-multipliable}).
In particular, if we write
\(\Gamma_{\sset{\at, \at'}}\) and
\(\Gamma_{\at + \at'}\) for
the appropriate stabilizers, then
it is clear that
\(\Gamma_{\sset{\at, \at'}}\)
is contained in
\(\Gamma_{\at + \at'}\).
We have a commutative diagram
\[\xymatrix{
\Gamma/\Gamma_{\sset{\at, \at'}} \ar[d]\ar[r] &
\Gamma/\Gamma_{\at + \at'} \ar[d] \\
\sset{\text{exceptional pairs extending \(a\)}} \ar[r] &
\sset{\text{extensions of \(\at + \at'\)}}.
}\]
Proposition \ref{prop:rd-restriction}(\ref{subprop:rd-fiber})
and
Lemma \ref{lem:half-pairs-a}(\ref{sublem:pairs-orbit})
show that the vertical arrows are bijections,
and
(\ref{sublem:coexc-extensions}) shows that
the bottom arrow is a bijection,
so the top must also be a bijection.
That gives (\ref{sublem:exc-stab}).
\end{proof}

\begin{cor}
\label{cor:avg-coroot}
For every \(a \in \Phi\), we have that
\(a^\vee\) equals
\(\sum_{i_\Gamma^*(\at) = a} \at^\vee\)
if \(a\) is non-multipliable or inert, and
\(a^\vee\) equals
\(2\sum_{i_\Gamma^*(\at) = a} \at^\vee\)
if \(a\) is split.
\end{cor}

\begin{proof}
If \(a\) is non-multipliable, then this is
Proposition \ref{prop:rd-restriction}%
	(\ref{subprop:coroot-nonmultipliable}).
If \(a\) is inert, then, for every exceptional pair
\(\sset{\at, \at'}\) extending \(a\), we have that
\(\at\) equals \(\at'\), so that
\(2(\at + \at')^\vee\)
equals \(\at^\vee\).
If \(a\) is split, then, for every exceptional pair
\(\sset{\at, \at'}\) extending \(a\), we have by
Proposition \ref{prop:rd-restriction}%
	(\ref{subprop:rd-multipliable})(\ref{case:rd-multipliable(red)})
that \(2(\at + \at')^\vee\) equals
\(2(\at^\vee + \at^{\prime\,\vee})\).
In either case,
Lemma \ref{lem:half-pairs-a}(\ref{sublem:coroot-multipliable})
gives the result.
\end{proof}

\begin{cor}[Corollary to
	Proposition \ref{prop:rd-restriction} and
	Lemma \ref{lem:half-pairs-a}]
\label{cor:coroots-vs-coroots}
The sublattices
\(\Z\Phi^\vee\) and
\((\Z\Phit^\vee)^\Gamma\)
of \(\Xt^\vee\)
are equal.
\end{cor}

\begin{proof}
Proposition \ref{prop:rd-restriction}%
	(\ref{subprop:coroot-nonmultipliable}) and
Lemma \ref{lem:half-pairs-a}(\ref{sublem:coroot-multipliable})
show that
\(\Phi^\vee\), hence \(\Z\Phi^\vee\), is contained in
\(\Z\Phit^\vee \cap X^\vee =
\Z\Phit^\vee \cap (\Xt^\vee)^\Gamma \subseteq
(\Z\Phit^\vee)^\Gamma\).

Conversely, we have by Remark \ref{rem:rd-Gamma-induced} that
the sums of \(\Gamma\)-orbits of elements of \(\Phit^\vee\)
span \((\Z\Phit^\vee)^\Gamma\),
so it suffices to show that such sums lie in \(\Z\Phi^\vee\).

Fix \(\at \in \Phit\), and let
\(a = i_\Gamma^*(\at)\).
If \(a\) is non-multipliable in \(\Phi\), then
Proposition \ref{prop:rd-restriction}%
	(\ref{subprop:rd-fiber},%
	\ref{subprop:coroot-nonmultipliable}) gives that
the sum of the \(\Gamma\)-orbit of \(\at^\vee\) equals
\(a^\vee\), and so
belongs to \(\Z\Phi^\vee\).
If \(\at \in \Phit\) is multipliable in \(\Phit\), then
\(2\at\) is not, so we have just shown that
the sum of the \(\Gamma\)-orbit of
\(\tfrac1 2\at^\vee = (2\at)^\vee\), and hence of
\(\at^\vee\), belongs to
\(\Z\Phi^\vee\).

Now suppose that \(\at\) is
non-multipliable in \(\Phit\), but
\(a\) is
multipliable in \(\Phi\).
Then \(a\) is split, so
Corollary \ref{cor:avg-coroot} gives that
\(\frac1 2 a^\vee\) is
the sum of the coroots corresponding to the extensions of \(a\) in \(\Phit\), which, by
Proposition \ref{prop:rd-restriction}(\ref{subprop:rd-fiber}),
is the sum of the \(\Gamma\)-orbit of \(\at^\vee\); but
\(\frac1 2 a^\vee\) equals \((2a)^\vee\), and so
belongs to \(\Z\Phi^\vee\), as desired.
\end{proof}

\begin{lem}
\label{lem:simple-witness}\hfill
\begin{enumerate}[label=(\alph*), ref=\alph*]
\item\label{sublem:simple-witness}
If \(a\) and \(b\) are elements of \(\Phi\) such that
\(a + b\) belongs to \(\Phi\), then,
for every
extension \(\at\) of \(a\), there is an
extension \(\bt\) of \(b\) such that
\(\at + \bt\) belongs to \(\Phit\).
\item\label{sublem:rd-simple}
\abmapto\Deltat{i_\Gamma^*(\Deltat)}
is a bijection from
\(\Gamma\)-stable systems of simple roots in \(\Phit\)
to systems of simple roots in \(\Phi\), with
inverse bijection
\abmapto\Delta{(i_\Gamma^*)\inv(\Delta)}.
\end{enumerate}
\end{lem}

\begin{proof}
We begin with (\ref{sublem:simple-witness}).
If \(a\) equals \(b\) and
\(\sset{\at, \at'}\) is
an exceptional pair containing \(\at\)
(which exists, by
Lemma \ref{lem:half-pairs-b}(\ref{sublem:half-pairs})), then
\(\at + \at'\) belongs to \(\Phit\).
Thus we may, and do, suppose that \(a\) and \(b\) are distinct.

Suppose first that
\(a\) and \(b\) are not orthogonal
(in addition to being distinct).
Then \(\pair\at{b^\vee} = \pair a{b^\vee}\) is negative
\cite{bourbaki:lie-gp+lie-alg_4-6}*
	{Ch.~VI, no.~1.3, p.~149, Corollaire to Th\'eor\`eme 1}.
Corollary \ref{cor:avg-coroot} gives that
there is some extension \(\bt\) of \(b\) such that
\(\pair\at{\bt^\vee}\) is negative; and another
application of
\emph{loc.~cit.}
gives that \(\at + \bt\) belongs to \(\Phit\).

Finally, suppose that \(a\) and \(b\) are orthogonal.
Then \(\pair{a + b}{a^\vee}\) equals \(2\), so
\(-a\) and \(a + b\) are \emph{not} strongly orthogonal; and,
in fact, \(-a + (a + b) = b\) belongs to \(\Phi\).
Thus, we have just shown that there is an extension
\(\ct\) of \(a + b\) such that
\(-\at + \ct\)
belongs to \(\Phit\).
Put \(\bt = -\at + \ct\).

For (\ref{sublem:rd-simple}), suppose first that
\(\Deltat\) is a \(\Gamma\)-stable
system of simple roots for \(\Phit\)
with corresponding system of positive roots
\(\Phit^+\).
Put \(\Delta = i_\Gamma^*(\Deltat)\).
We have by
Proposition \ref{prop:rd-restriction}(\ref{subprop:rd-Borus})
that \(\Phi^+ \ldef i_\Gamma^*(\Phit^+)\) is
a system of positive roots for \(\Phi\), and that
\(\Phit^+\) equals
\((i_\Gamma^*)\inv(\Phit^+)\).
If \(\Delta\) is not simple, then there exist \(a, b \in \Phi^+\) such that
\(a + b\) belongs to \(\Delta\).
By (\ref{sublem:simple-witness}), we have that there are
extensions \(\at\) and \(\bt\) of \(a\) and \(b\),
necessarily in \(\Phit^+\), such that
\(\at + \bt\) belongs to \(\Phit\).
Since \(\at + \bt\) belongs to
\((i_\Gamma^*)\inv(a + b) \in (i_\Gamma^*)\inv(\Delta)\),
we have by
Proposition \ref{prop:rd-restriction}(\ref{subprop:rd-fiber})
that
\(\at + \bt\) belongs to \(\Deltat\).
This contradicts the simplicity of \(\Deltat\), so \(\Delta\) must be simple.

Now suppose conversely that \(\Delta\) is
a system of simple roots for \(\Phi\) with
corresponding system of positive roots \(\Phi^+\).
Put \(\Deltat = (i_\Gamma^*)\inv(\Delta)\), which is
clearly preserved by \(\Gamma\).
Again by
Proposition \ref{prop:rd-restriction}(\ref{subprop:rd-Borus}),
we have that
\(\Phit^+ \ldef (i_\Gamma^*)\inv(\Phi^+)\) is
a \(\Gamma\)-stable system of positive roots for \(\Phit\).
It is clear that there do not exist
\(\at, \bt \in \Phit^+\) such that
\(\at + \bt\) belongs to \(\Deltat\).
Now fix \(\at \in \Phit^+ \setminus \Deltat\), and put
\(a = i_\Gamma^*(\at)\), which belongs to
\(\Phi^+ \setminus \Delta\).
By simplicity, there is some \(b \in \Delta\) such that
\(a - b\) belongs to \(\Phi^+\).
By (\ref{sublem:simple-witness}), there is some extension
\(\bt\) of \(b\), necessarily in \(\Deltat\), such that
\(\at - \bt\) belongs to \(\Phit^+\).
This shows (\ref{sublem:rd-simple}).
\end{proof}

\begin{lem}
\label{lem:rd-Dynkin}
The nodes of the Dynkin diagram of \(\Phi\) are
in bijection with the \(\Gamma\)-orbits of
nodes of the Dynkin diagram of \(\Phit\), and
two nodes of the Dynkin diagram of \(\Phi\) are adjacent
if and only if
they are restrictions of adjacent nodes of
the Dynkin diagram of \(\Phit\).
\end{lem}

\begin{proof}
Fix a system \(\Deltat\) of simple roots for \(\Phit\).
Lemma \ref{lem:simple-witness}(\ref{sublem:rd-simple}
gives that \(\Delta \ldef i_\Gamma^*(\Deltat)\)
is a system of simple roots for \(\Phi\), and
Proposition \ref{prop:rd-restriction}(\ref{subprop:rd-fiber}) shows that
we may identify \(\Delta\) with the set of
orbits of \(\Gamma\) on \(\Deltat\).

Now suppose that \(\at\) and \(\bt\) belong to
\(\Deltat\), and
write \(a\) and \(b\) for their respective restrictions.
Proposition \ref{prop:rd-restriction}(\ref{subprop:rd-fiber}) again,
and Corollary \ref{cor:avg-coroot}, give that
\(\pair b{a^\vee}\) is a positive multiple of
\(\sum_{\gamma \in \Gamma/{\stab_\Gamma \at}}
	\smashpair\bt{\gamma\at^\vee}\).
In particular, the sum is \(0\), so that
\(a\) and \(b\) are not adjacent, unless
some \(\Gamma\)-conjugate of \(\at\)
is adjacent to \(\bt\).

To complete the proof,
we may, and do, assume, upon replacing \(\at\) by
a \(\Gamma\)-conjugate if necessary, that
\(\at\) and \(\bt\) are adjacent.
An examination of the irreducible root systems
(and their diagram automorphisms), say in
\cite{bourbaki:lie-gp+lie-alg_4-6}*{Chapter VI, Plates I--IX},
shows that, if there is
a diagram automorphism of \(\Phit\) that
moves \(\at\) to a node of
the Dynkin diagram of \(\Phit\) that
is adjacent to \(\bt\), then
every diagram automorphism of \(\Phit\)
preserving the irreducible component to which \(\bt\) belongs
fixes \(\bt\).
Thus
\(\sum_{\gamma \in \Gamma/{\stab_\Gamma \at}}
	\smashpair\bt{\gamma\at^\vee}
	 = \sum_{\gamma \in \Gamma/{\stab_\Gamma \at}}
	\smashpair{\gamma\bt}{\at^\vee}\) is
a positive multiple of \(\smashpair\bt{\at^\vee}\),
hence so is \(\pair b{a^\vee}\).
In particular, \(a\) and \(b\) are adjacent.
\end{proof}

\begin{cor}
\label{cor:rd-irred}
The map \abmapto{\Phit_1}{i_\Gamma^*(\Phit_1)}
is a surjection from
the irreducible components of \(\Phit\) onto
the irreducible components of \(\Phi\).
It induces a bijection between
the \(\Gamma\)-orbits of irreducible components of \(\Phit\) and
the irreducible components of \(\Phi\), with inverse map sending
an irreducible component \(\Phi_1\) of \(\Phi\) to
the \(\Gamma\)-orbit of any irreducible component of
\((i_\Gamma^*)\inv(\Phi_1)\).
\end{cor}

\begin{rem}
\label{rem:red-to-nred-facts}
Suppose that \(\Phit_1\) is a reduced irreducible component of \(\Phit\)
and that the corresponding irreducible component
\(\Phi_1\) of \(\Phi\) is non-reduced.
Write
\(\Gamma_1\) for the subgroup of \(\Gamma\) that
preserves \(\Phit_1\), and
\(\Gamma_1'\) for the subgroup of \(\Gamma_1\) that
acts trivially on it.
Proposition \ref{prop:rd-restriction}%
	(\ref{subprop:rd-fiber},%
	\ref{subprop:rd-multipliable})
gives that
\(\Phit_1\) is of type \(\mathsf A_{2n}\) for some \(n\),
and \(\Gamma_1/\Gamma_1'\) has order \(2\).
Let \(\gamma_0\) be an element of \(\Gamma_1 \setminus \Gamma_1'\).

We shall use the terminology (motivated by \eqref{subrem:pre-mult-div} below) that
an element \(\at \in \Phit_1\) is
\textit{pre-multipliable} if
\(\at\) and \(\gamma_0\at\) are
neither equal nor orthogonal, and
\textit{pre-divisible} if
\(\at\) is fixed by \(\gamma_0\).
This condition is independent of the choice of \(\gamma_0\).
\begin{enumerate}[label=(\alph*), ref=\alph*]
\item\label{subrem:pre-mult-div}
An element of \(\Phit_1\) is pre-divisible
(respectively, pre-multipliable)
if and only if
its restriction is divisible (respectively, multipliable)
in \(\Phi_1\).
If \(\at \in \Phit_1\) is pre-multipliable, then
\(\sset{\at, \gamma_0\at}\) is an exceptional pair.
\item\label{subrem:red-stab}
For every \(\at \in \Phit_1\), we have that
\(\stab_\Gamma(\at)\) equals
\(\Gamma_1\) or \(\Gamma_1'\), according as
\(\at\) is or is not pre-divisible.
\end{enumerate}
\end{rem}

\numberwithin{equation}{section}
\section{Quasisemisimplicity and smoothability}
\label{sec:quass-smooth}

Proposition \ref{prop:rd-restriction} is
phrased entirely in the abstract language of actions on root data,
but, in this paper, we are most interested in actions on groups.
Recall the field \(k\) with characteristic exponent \(p\) from
\S\ref{sec:notation}.
Throughout the rest of the paper
(not just \S\ref{sec:quass-smooth}),
	\((\Gt, \Gamma)\) is a reductive datum over \(k\),
in the sense of
Definition \ref{defn:decomposing-data}.
Put \(G = \fix\Gt^\Gamma\).

Recall the definition of quasisemisimplicity from
Definition \ref{defn:ordinary-cases}(\ref{subdefn:gp-quass})
(although we do not assume that \((\Gt, \Gamma)\) is
quasisemisimple until Proposition \ref{prop:quass-rough}).
Lemma \ref{lem:quass-by-component} shows that
quasisemisimplicity can be checked on the level of
almost-simple components, at least
after passing to a sufficiently large
separable extension of \(k\).

\begin{lem}
\label{lem:quass-by-component}
If \((\Gt, \Gamma)\) is quasisemisimple, then
\((\Gt_1, \stab_\Gamma(\Gt_1))\) is quasisemisimple for every
almost-simple component \(\Gt_1\) of \(\Gt\).
If \(\pi_0(\Gamma)\) is constant, then the converse holds.
\end{lem}

\begin{proof}
If \((\Bt, \Tt)\) is a Borel--torus pair in \(\Gt\) that is
preserved by \(\Gamma\) and
\(\Gt_1\) is a smooth, connected, normal subgroup of \(\Gt\), then
\((\Bt \cap \Gt_1, \Tt \cap \Gt_1)\) is
a Borel--torus pair in \(\Gt_1\) that is
preserved by \(\stab_\Gamma(\Gt_1)\).
In particular, \((\Gt_1, \stab_\Gamma(\Gt_1))\) is quasisemisimple.

Now suppose that
\((\Gt_1, \stab_\Gamma(\Gt_1))\) is quasisemisimple for
every almost-simple component \(\Gt_1\) of \(\Gt\).
Consider the set of triples
\((\Gt_1, \Bt_1, \Tt_1)\), where
\(\Gt_1\) is an almost-simple component of \(\Gt\) and
\((\Bt_1, \Tt_1)\) is a Borel--torus pair in \(\Gt_1\) that
is preserved by \(\stab_\Gamma(\Gt_1)\).
Remark \ref{rem:torus-quass}(\ref{subrem:act-inner}) and
Remark \ref{rem:torus-quass}(\ref{subrem:quass-to-torus})
give that
\(\pi_0(\Gamma)(k)\) acts on the set of such triples.
By assumption, the natural map from such triples to
almost-simple components of \(\Gt\) is surjective, and
obviously it is \(\pi_0(\Gamma)(k)\)-equivariant.
Arbitrarily
choose a \(\pi_0(\Gamma)(k)\)-equivariant section.
Then the subgroup of \(\Gt\) generated by
\(Z(\Gt)\smooth\conn\) and the various \(\Bt_1\)
(respectively \(\Tt_1\)) arising as
a component of a triple in the image of the section is
a Borel subgroup \(\Bt\) of \(\Gt\)
(respectively a maximal torus in \(\Bt\)) that is
preserved by \(\pi_0(\Gamma)(k)\), hence by \(\Gamma\),
since \(\pi_0(\Gamma)\) is constant.
\end{proof}

Lemma \ref{lem:cochar-by-torus} is vacuous unless
\((\Gt, \Gamma)\) is quasisemisimple, but
we still find it convenient
(for Lemma \ref{lem:quass-from-Levi})
to state the lemma before
making that assumption.

\begin{lem}
\label{lem:cochar-by-torus}
If \((\Bt, \Tt)\) is a Borel--torus pair in \(\Gt\) that
is preserved by \(\Gamma\), then
there is a cocharacter \(\delta\) of the
maximal split torus in \(\fix\Tt^\Gamma\)
such that
\(C_\Gt(\delta)\) equals \(\Tt\) and
the parabolic subgroup
\(P_\Gt(\delta)\) of \(\Gt\) associated to \(\delta\)
\cite{springer:lag}*{Proposition 8.4.5}
is \(\Bt\).
\end{lem}

\begin{proof}
Consider the cocharacter
\(\sum_{\alphat \in \Phi(\Bt_\ks, \Tt_\ks)} \alphat^\vee\)
of \(\Tt_\ks\).
We abuse notation by denoting this cocharacter by
\(\delta_\ks\), even though
we have not yet defined a cocharacter \(\delta\) of which
it is the base change.
Since \(\delta_\ks\) is fixed by \(\Gal(k) \ltimes \Gamma(\ks)\),
and since \(\Gamma(\ks)\) is Zariski dense in \(\Gamma_\ks\),
we may regard it as a \(\Gamma\)-fixed cocharacter \(\delta\)
of \(\Tt\).
It is therefore a cocharacter of \(\Tt^\Gamma\), hence of
\(\fix\Tt^\Gamma\), hence of
its maximal split torus.

Since
\(\pair\alphat{\delta_\ks}\) equals \(2\) for
all \(\alphat \in \Delta(\Bt_\ks, \Tt_\ks)\),
we have that
\(C_\Gt(\delta)_\ks = C_{\Gt_\ks}(\delta_\ks)\) equals \(\Tt_\ka\) and
\(P_\Gt(\delta)_\ks = P_{\Gt_\ks}(\delta_\ks)\) equals \(\Bt_\ks\)
\cite{springer:lag}*{Proposition 8.4.5},
hence that
\(C_\Gt(\delta)\) equals \(\Tt\) and
\(P_\Gt(\delta)\) equals \(\Bt\).
\end{proof}

Lemma \ref{lem:quass-from-Levi} shows one convenient way to recognize
quasisemisimple actions.
Proposition \ref{prop:quass-rough}(\ref{subprop:quass-down}) and
Lemma \ref{lem:quass-to-centralizer}
provide converses to parts of Lemma \ref{lem:quass-from-Levi}.

\begin{lem}
\label{lem:quass-from-Levi}
Suppose that
\(S\) is a split torus in \(G\) and that
\((\Bt', \Tt)\) is a Borel--torus pair in
\(C_\Gt(S)\) that is preserved by \(\Gamma\).
Then there is a Borel--torus pair
\((\Bt, \Tt)\) in \(\Gt\) that is preserved by \(\Gamma\).
In particular, if there is a split torus \(S\) in
\(G\) such that
\(C_\Gt(S)\) is a torus, then
\(C_\Gt(S)\) is contained in
a \(\Gamma\)-stable Borel subgroup of \(\Gt\), and
\((\Gt, \Gamma)\) is quasisemismple.
\end{lem}

\begin{proof}
By Lemma \ref{lem:cochar-by-torus},
there is a cocharacter \(\lambda'\) of
\(\fix C_\Gt(S)^\Gamma = C_G(S)\) such that
\(\Bt'\) is the associated parabolic subgroup
\(P_{C_\Gt(S)}(\lambda')\) of \(C_\Gt(S)\), and
\(C_{C_\Gt(S)}(\lambda')\) is \(\Tt\).
For every
\(\alphat \in
\Phi(\Gt_\ks, \Tt_\ks) \setminus
	\Phi(C_\Gt(S)_\ks, \Tt_\ks)\),
the affine subspace \(V_\alphat\) of \(\bX_*(S_\ks) \otimes_\Z \Q\) on which
\(\alphat\) equals \(\pair\alphat{\lambda'}\) is proper;
so the complement
\((\bX_*(S_\ks) \otimes_\Z \Q) \setminus
	\bigcup_{\alphat \in
		\Phi(\Gt_\ks, \Tt_\ks) \setminus
			\Phi(C_\Gt(S)_\ks, \Tt_\ks)}
	V_\alphat\)
is nonempty.  Let \(\lambda^\perp\) be an element of the complement.
After multipling by a positive integer, we may, and do, assume that
\(\lambda^\perp\) belongs to \(\bX_*(S_\ks)\).
Since \(S\) is split, so that the natural map
\abmap{\bX_*(S)}{\bX_*(S_\ks)} is an isomorphism,
we may, and do, regard
\(\lambda^\perp\) as an element of \(\bX_*(S)\).
Put \(\lambda = \lambda' - \lambda^\perp\).
Then \(\pair\alphat{\lambda_\ks}\) is nonzero for all
\(\alphat \in \Phi(\Gt_\ks, \Tt_\ks)\), so
\(C_{\Gt_\ks}(\lambda_\ks)\) equals \(\Tt_\ks\) and
\(P_{\Gt_\ks}(\lambda_\ks)\) is a Borel subgroup of \(\Gt_\ks\).
The analogous facts without base change to \(\ks\) follow, so
\((P_\Gt(\lambda), \Tt)\) is
a Borel--torus pair in \(\Gt\) that is preserved by \(\Gamma\).
\end{proof}

Lemma \ref{lem:quass-by-iso} shows that,
when checking the quasisemisimplicity of $(\Gt,\Gamma)$,
we may always replace $\Gt$ by a simply connected or adjoint group.

\begin{lem}
\label{lem:quass-by-iso}
Let \(\Nt\) be a normal subgroup of \(\Gt\) that
is preserved by \(\Gamma\).
If \((\Gt, \Gamma)\) is quasisemisimple, then
so is \((\Gt/\Nt, \Gamma)\).
The converse holds if
\(\Nt\) is central in \(\Gt\).
\end{lem}

\begin{proof}
The first statement follows from
\cite{borel:linear}*{Proposition 11.14(1)}.
This also shows that, if
\((\Bt', \Tt')\) is a Borel--torus pair in \(\Gt/\Nt\) that
is preserved by \(\Gamma\), then
there is some Borel--torus pair \((\Bt, \Tt)\) in \(\Gt\) that
maps onto \((\Bt', \Tt')\), so that
\(\Bt\cdot\Nt\) and \(\Tt\cdot\Nt\) are
preserved by \(\Gamma\).
If \(\Nt\) is central in \(\Gt\), then it is contained in
\(\Bt\) and \(\Tt\), so
the second statement follows.
\end{proof}

Throughout
the rest of
\S\ref{sec:quass-smooth}
and
\S\ref{sec:quass},
we assume that
\((\Gt, \Gamma)\) is quasisemisimple; so, in particular,
\(\Gt\) is quasisplit.
In
\S\S\ref{sec:sl-outer},%
\ref{sec:thm:ka-quass},
we do not impose this assumption directly, although our goal is
Theorem \ref{thm:ka-quass}(\ref{subthm:ka-quass-quass}) that concludes
quasisemisimplicity.

\begin{prop}
\label{prop:quass-rough}\hfill
\begin{enumerate}[label=(\alph*), ref=\alph*]
\item\label{subprop:quass-down}
Let \(\Tt\) be a \(\Gamma\)-stable maximal torus in \(\Gt\)
that is contained in a \(\Gamma\)-stable Borel subgroup of \(\Gt\).
Put \(T = \fix\Tt^\Gamma\), and
let \(S\) be the maximal split torus in \(T\).
Then
\(T\) equals \(\Tt \cap G\) and is a maximal torus in \(G\),
\(S\) is a maximal split torus in \(G\), and
\(\Tt\) equals \(C_\Gt(S)\).
\item\label{subprop:quass-up}
Let \(S\) be a maximal split torus in \(G\).
Then \(C_\Gt(S)\) 
is the unique
maximal torus \(\Tt\) in \(\Gt\) containing \(S\), and
\(\Tt \cap G\) is the unique maximal torus in \(G\) containing \(S\).
We have that \(\Tt\) is \(\Gamma\)-stable and
contained in a \(\Gamma\)-stable Borel subgroup of \(\Gt\).
\item\label{subprop:quass-torus-orbit}
The set of \(\Gamma\)-stable maximal tori in \(\Gt\) that are
contained in a \(\Gamma\)-stable Borel subgroup is
permuted transitively by \(G(k)\).
\end{enumerate}
\end{prop}

\begin{proof}
Since all groups of multiplicative type are smoothable, we have
by Remark \ref{rem:conn-smooth} that
\(T_\ka = (\fix\Tt^\Gamma)_\ka\) equals
\(\fix\Tt_\ka^{\Gamma_\ka}\), which is
a maximal torus in \(\fix\Gt_\ka^{\Gamma_\ka}\)
by \cite{adler-lansky:lifting-1}*{Proposition 3.5(ii)}.
Since
\(G_\ka = (\fix\Gt^\Gamma)_\ka\) is
contained in \(\fix\Gt_\ka^{\Gamma_\ka}\), we have that
\(T\) is a maximal torus in \(G\).

Let \(\delta\) be the cocharacter of \(S\) constructed in
Lemma \ref{lem:cochar-by-torus}, so that
\(C_\Gt(\delta)\) equals \(\Tt\).
Since the first and last terms in the
obvious sequence of containments
\[
\Tt \subseteq
C_\Gt(T) \subseteq
C_\Gt(S) \subseteq
C_\Gt(\delta)
\]
are equal, all the containments are equalities.
Thus \(\Tt \cap G\) equals \(C_\Gt(S) \cap G = C_G(S)\), which is
smooth by
\cite{conrad-gabber-prasad:prg}*{Proposition A.8.10(2)},
connected by
\cite{borel:linear}*{Corollary 11.12}, and
contained in \(\Tt\).
Therefore \(\Tt \cap G\) is a torus in \(G\), hence
contained in the maximal torus \(T\) in \(G\).
The reverse containment being obvious, we have the equality
\(\Tt \cap G = C_G(S) = T\).
In particular, since \(S\) is the maximal split torus in \(T\),
in fact \(S\) is maximal split in \(G\).
This shows (\ref{subprop:quass-down}), and
(\ref{subprop:quass-up}) for one choice of
maximal split torus \(S\) in \(G\).
Since the maximal split tori in \(G\) are \(G(k)\)-conjugate
by \cite{conrad-gabber-prasad:prg}*{Theorem C.2.3},
we have shown
(\ref{subprop:quass-up}) in general, and
(\ref{subprop:quass-torus-orbit}).
\end{proof}

For the remainder of \S\ref{sec:quass-smooth}, fix
a maximal split torus \(S\) in \(G\).
By Proposition \ref{prop:quass-rough}(\ref{subprop:quass-up}),
there are unique maximal tori
\(T\) in \(G\) and
\(\Tt\) in \(\Gt\)
containing \(S\), as well as a \(\Gamma\)-stable Borel subgroup \(\Bt\) of \(\Gt\) containing \(\Tt\).
Let \(\St\) be the maximal split torus in \(\Tt\).

\begin{lem}
\label{lem:quass-to-centralizer}
If \(D\) is a subgroup of \(S\), then
\((C_\Gt(D)\conn, \Gamma)\) is quasisemisimple.
If \(\mathfrak d\) is a subspace of \(\Lie(S)\), then
\((C_\Gt(\mathfrak d)\conn, \Gamma)\) is
quasisemisimple.
\end{lem}

\begin{note}
We do not assume that \(D\) is a torus, or even smooth.
We have that \(C_\Gt(D)\conn\) is reductive by
\cite{conrad-gabber-prasad:prg}*{Proposition A.8.12}, and
\(C_\Gt(\mathfrak d)\conn\) is reductive by
Corollary \ref{cor:toral-Lie-is-Lie-torus}.
\end{note}

\begin{proof}
We have that \(S\) is a split torus in
\(\fix C_\Gt(D)^\Gamma\) (respectively,
\(\fix C_\Gt(\mathfrak d)^\Gamma\)), and
\(C_{C_\Gt(D)\conn}(S)\)
(respectively,
\(C_{C_\Gt(\mathfrak d)\conn}(S)\))
equals \(C_\Gt(S)\), which is a torus by
Proposition \ref{prop:quass-rough}(\ref{subprop:quass-up}).
Then Lemma \ref{lem:quass-from-Levi} gives the result.
\end{proof}

Remark \ref{rem:how-to-restrict}
allows us to apply the results of
Proposition \ref{prop:rd-restriction} and
Lemma \ref{lem:half-pairs-a} in
the setting of connected, reductive groups.

\begin{rem}
\label{rem:how-to-restrict}
We can `restrict' an element of
\(\bX^*(\Tt_\ks)\),
for example, an element of \(\Phi(\Gt_\ks, \Tt_\ks)\), to
\(\St\) by restricting from \(\Tt_\ks\) to \(\St_\ks\), and then
using the fact that
\abmap{\bX^*(\St)}{\bX^*(\St_\ks)}
is an isomorphism.
Similarly, we can `restrict' from \(\Tt_\ks\) or \(T\) to \(S\).

Proposition \ref{prop:quass-rough}(\ref{subprop:quass-up})
gives that \(\Tt\)
is
preserved by \(\Gamma\).
By rigidity of tori
\cite{milne:algebraic-groups}*{Corollary 12.37},
\(\Gamma\conn\) fixes \(\Tt\) pointwise,
so \(\pi_0(\Gamma)\) acts on \(\Tt\); and
the action of \(\Gamma(\ks)\) on
the absolute root datum \(\Psi(\Gt_\ks, \Tt_\ks)\)
factors through
the finite quotient \(\pi_0(\Gamma)(\ks)\).

We have that \(\Phi(\Gt_\ks, T_\ks)\) and \(\Phi(\Gt, S)\)
are the sets of restrictions to \(T_\ks\) and to \(S\)
of elements of \(\Phi(\Gt_\ks, \Tt_\ks)\) and \(\Phi(\Gt, \St)\).
This is just the definition, together with the fact that,
by Proposition \ref{prop:quass-rough}(\ref{subprop:quass-up})
(or Proposition \ref{prop:rd-restriction}(\ref{subprop:rd-root-system})),
no element of \(\Phi(\Gt_\ks, \Tt_\ks)\) has
trivial `restriction' to \(S\)
(so that also no element has trivial restriction to \(T_\ks\)).

Write
\(\Psi(\Gt, \St)\),
\(\Psi(\Gt_\ks, T_\ks)\), and
\(\Psi(\Gt, S)\)
for the ``quotient root data'' of
\(\Psi(\Gt_\ks, \Tt_\ks)\) by
\(\Gal(k)\),
\(\Gamma(\ks)\), and
\(\Gal(k) \ltimes \Gamma(\ks)\), respectively.
The maps
	\begin{itemize}
	\item \abmap{\bX_*(\St)}{\bX_*(\Tt_\ks)^{\Gal(k)}},
	\item \abmap{\bX_*(T_\ks)}{\bX_*(\Tt_\ks)^{\Gamma(\ks)}},
and	\item \abmap{\bX_*(S)}{\bX_*(\Tt_\ks)^{\Gal(k) \ltimes \Gamma(\ks)}}
	\end{itemize}
are all isomorphisms, which we may use to identify
the character lattices of the root data with
\(\bX^*(\St)\),
\(\bX^*(T_\ks)\), and
\(\bX^*(S)\), respectively,
in which case their root systems are identified with
\(\Phi(\Gt, \St)\),
\(\Phi(\Gt_\ks, T_\ks)\), and
\(\Phi(\Gt, S)\),
respectively.
We denote the corresponding duality map
\abmap{\Phi(\Gt, S)}{\bX_*(S)} by
\abmapto a{a^\vee},
and denote its image by \(\Phi^\vee(\Gt, S)\);
and similarly for \(\Phi(\Gt_\ks, T_\ks)\)
(and for \(\Phi(\Gt, \St)\), though that is just
the classical construction of
relative root systems).

Using the above identifications of root systems,
we may also refer to exceptional pairs in \(\Phi(\Gt_\ks, T_\ks)\)
that `extend' roots in \(\Phi(\Gt, S)\), as in Definition~\ref{defn:rd-exceptional}.

Note that restriction from \(\St\) to \(S\) cannot
always be thought of
as in \S\ref{sec:rd-quass},
because \(\St\) need not be preserved by \(\Gamma\),
but
must in general rather be thought of as
`extension' from \(\St\) to \(\Tt_\ks\), followed by
`restriction' from \(\Tt_\ks\) to \(S\).
Thus, for example, it does not always make sense to say that
the fibers of the restriction map
\abmap{\Phi(\Gt, \St)}{\Phi(\Gt, S)}
are \(\Gamma(\ks)\)-orbits; but it does make sense to say, and,
even better, by
Proposition \ref{prop:rd-restriction}(\ref{subprop:rd-fiber}),
is true,
that, for every \(a \in \Phi(\Gt, S)\),
the set of elements of \(\Phi(\Gt_\ks, \Tt_\ks)\) that
`restrict' to \(a\) is a
(\(\Gal(k) \ltimes \Gamma(\ks)\))-orbit, and that
the set of elements of \(\Phi(\Gt, \St)\) that
restrict to \(a\) is parametrized by the
\(\Gal(k)\)-orbits in that
(\(\Gal(k) \ltimes \Gamma(\ks)\))-orbit.
\end{rem}

\begin{rem}
\label{rem:adj-or-sc-T}
Since \(\Gamma\) is smooth, we have that
\(\Gamma(\ks)\) is Zariski dense in \(\Gamma_\ks\), so
\((\Tt^\Gamma)_\ks\) equals \(\Tt_\ks^{\Gamma(\ks)}\), and hence
\(\bX^*((\Tt^\Gamma)_\ks)\) is the co-invariant module
\(\bX^*(\Tt_\ks)_{\Gamma(\ks)}\).

If \(\Gt\) is adjoint (respectively, simply connected), then
\(\Gamma\) permutes
the basis of \(\bX^*(\Tt_\ks)\) given by
the \(\Bt_\ks\)-simple roots
(respectively, the dual to the basis of \(\bX_*(\Tt_\ks)\) consisting of
\(\Bt_\ks\)-simple coroots),
so that the co-invariant module
\(\bX^*((\Tt^\Gamma)_\ks) = \bX^*(\Tt_\ks)_{\Gamma(\ks)}\) is
torsion free.
Thus \(\Tt^\Gamma\) is a torus, hence smooth.
\end{rem}

We will apply
Lemma \ref{lem:root-space} and
Corollary \ref{cor:root-space-facts} only to
the (\(\Gamma \ltimes \Gt\))-module \(\Vt = \Lie(\Gt)\).
However, the slight extra generality in our statements
involves no extra difficulty in the proof.

\begin{lem}
\label{lem:root-space}
Let \(\Vt\) be a \((\Gamma \ltimes \Gt)\)-module such that
\(\Phi(\Vt\otimes_k\ks, \Tt_\ks)\) equals
\(\Phi(\Gt_\ks, \Tt_\ks)\).
Fix \(\alphat \in \Phi(\Gt_\ks, \Tt_\ks)\), and
write \(a\) for its `restriction' to \(S\).
Write \(\pi_\alphat\) for
the \(\Tt_\ks\)-equivariant projection of \(\Vt \otimes_k \ks\) on
its \(\alphat\)-weight space.
The restriction of \(\pi_\alphat\) is
a \(\Tt^\Gamma(k)\)-equivariant, \(k\)-linear isomorphism
\abmap
	{\Vt^\Gamma_a}
	{(\Vt \otimes_k \ks)_\alphat^
		{\stab_{\Gal(k) \ltimes \Gamma(\ks)}
		(\alphat)}}
of \(k\)-vector spaces.
\end{lem}

\begin{proof}
Put \(\Sigma = \Gal(k) \ltimes \Gamma(\ks)\).
It is clear that \(\pi_\alphat\) is
\(\stab_\Sigma(\alphat)\)-equivariant, so maps
\(\Vt_a\) into
\((\Vt \otimes_k \ks)_\alphat^
	{\stab_\Sigma(\alphat)}\).
The map in the other direction that sends
\(\Xt_\alphat\)
to
\(\sum_
	{\sigma \in \Sigma/{\stab_\Sigma(\alphat)}}
	\sigma\Xt_\alphat \in
(\Vt_a \otimes_k \ks)^\Sigma =
\Vt_a^\Gamma\)
is clearly a section, and
Proposition \ref{prop:rd-restriction}(\ref{subprop:rd-fiber}) shows that
it is also a retraction.
\end{proof}

\begin{cor}
\label{cor:root-space-facts}
Preserve the notation and hypotheses of
Lemma \ref{lem:root-space}.
Suppose that \((\Vt \otimes_k \ks)_\alphat\) is one-dimensional.
Let \(\Gal(k) \ltimes \Gamma(\ks)\) act on \(\ks\) through
the projection on \(\Gal(k)\), and write
\(k_\alphat\) for the fixed field in \(\ks\) of
\(\stab_{\Gal(k) \ltimes \Gamma(\ks)}(\alphat)\).
\begin{enumerate}[label=(\alph*), ref=\alph*]
\item\label{subcor:root-space-dim}
\(a\) belongs to \(\Phi(\Vt^\Gamma, S)\) if and only if
the \(k_\alphat\)-vector space
\((\Vt \otimes_k \ks)_\alphat^
	{\stab_{\Gal(k) \ltimes \Gamma(\ks)}(\alphat)}\)
is one-dimensional.
\item\label{subcor:root-space-fix}
Suppose that \(\Gal(k)\) fixes \(\alphat\), and put
\(\Vt_\alphat = (\Vt \otimes_k \ks)_\alphat^{\Gal(k)}\).
Then \(a\) belongs to \(\Phi(\Vt^\Gamma, S)\) if and only if
\(\stab_{\Gamma(\ks)}(\alphat)\) acts trivially on
\(\Vt_\alphat \otimes_k \ks\), in which case
the projection
\abmap{\Vt_a^\Gamma}{\Vt_\alphat}
is an isomorphism.
\end{enumerate}
\end{cor}

\begin{proof}
Put \(\Sigma = \Gal(k) \ltimes \Gamma(\ks)\).
Since
\((\Vt \otimes_k \ks)_\alphat\) is one-dimensional over \(\ks\),
we have that its space of \(\stab_\Sigma(\alphat)\)-fixed points is
at most one-dimensional over \(k_\alphat\).
Thus
(\ref{subcor:root-space-dim}) follows from
Lemma \ref{lem:root-space}.

For (\ref{subcor:root-space-fix}), suppose that
\(\Gal(k)\) fixes \(\alphat\).
Then
\(\stab_\Sigma(\alphat)\) equals
\(\Gal(k) \ltimes \stab_{\Gamma(\ks)}(\alphat)\), so
the dimension over \(k\) of
\((\Vt \otimes_k \ks)_\alphat^{\stab_\Sigma(\alphat)}\) is
the dimension over \(\ks\) of
\((\Vt_\alphat \otimes \ks)^{\stab_{\Gamma(\ks)}(\alphat)}\).
In particular, by (\ref{subcor:root-space-dim}),
we have that \(a\) belongs to \(\Phi(\Gt^\Gamma, S)\) if and only if
\((\Vt_\alphat \otimes_k \ks)^{\stab_{\Gamma(\ks)}(\alphat)}\)
is one-dimensional, i.e., if and only if
\(\stab_{\Gamma(\ks)}(\alphat)\) acts trivially on
(the one-dimensional \(\ks\)-vector space)
\(\Vt_\alphat \otimes_k \ks\); and, when this happens,
Lemma \ref{lem:root-space}
gives that the projection
\abmap{\Vt^\Gamma_a}{\Vt_\alphat}
is an isomorphism.
\end{proof}

\begin{rem}
\label{rem:how-to-irred}
We show how to apply Corollary \ref{cor:rd-irred} in our situation.

If \(\Gt_1\) is an almost-simple component of \(\Gt\), then
\(\Phi(\Gt_{1\,\ks}, \Tt_\ks)\) is
the union of the \(\Gal(k)\)-orbit of
an irreducible component \(\Phit\) of \(\Phi(\Gt_{1\,\ks}, \Tt_\ks)\),
and \(\Phi(\Gt_1, S)\) is
the set of `restrictions' of the elements of \(\Phit\) to \(S\),
hence is an irreducible component of \(\Phi(\Gt, S)\).

Conversely, if
\(\Phi\) is an irreducible component of \(\Phi(\Gt, S)\) and
\(\Phit\) is an element of the corresponding
(\(\Gal(k) \ltimes \Gamma(\ks)\))-orbit of irreducible components of
\(\Phi(\Gt_\ks, \Tt_\ks)\), then
the \(\Gal(k)\)-orbit of \(\Phit\) corresponds to
an irreducible component of
\(\Phi(\Gt, \St)\), hence to
an almost-simple component \(\Gt_1\) of \(\Gt\).
Then \(\Phi(\Gt_1, S)\) is the set of restrictions to \(S\) of
elements of \(\Phi(\Gt_1, \St)\), hence equals \(\Phi\).
\end{rem}

\begin{rem}
\label{rem:gp-nred-facts}
Suppose that \(\Gt\) is split.
Then
\(\Gal(k)\) acts trivially on \(\bX^*(\Tt)\), and
the almost-simple components of \(\Gt\adform\) are
absolutely almost simple and so are
permuted by \(\Gamma(k)\).
Let \(\Gt_1\) be an almost-simple component of \(\Gt\adform\),
and
write \(\Gamma_1'\) for the subgroup of \(\Gamma\) that
acts on \(\Gt_1\) by inner automorphisms.
Suppose that \(\Phi(\Gt_1, T)\) is not reduced.
\begin{enumerate}[label=(\alph*), ref=\alph*]
\item\label{subrem:gp-nred-type}
Since
\(\Phi(\Gt_1, \Tt)\) is reduced,
Proposition \ref{prop:rd-restriction}(\ref{subprop:rd-multipliable})
gives that \(\Gt_1\) is split adjoint of type \(\mathsf A_{2n}\),
i.e., isomorphic to \(\PGL_{2n + 1}\),
for some positive integer \(n\), and that
not every element of \(\stab_\Gamma(\Gt_1)(\ks)\) acts
by inner automorphisms on \(\Gt_{1\,\ks}\).
Since \(\stab_{\Gamma(\ks)}(\Gt_1)/\Gamma_1'(\ks)\) maps into
the automorphism group of the Dynkin diagram of \(\Gt_{1\,\ks}\)
(with respect to \((\Bt_\ks, \Tt_\ks)\)) and so has
cardinality at most \(2\), this implies that
\(\stab_\Gamma(\Gt_1)/\Gamma_1'\) is an \'etale group of order \(2\).
\item\label{subrem:root-commute}
If \(\alphat, \betat \in \Phi(\Gt_1, \Tt)\)
are such that \(\alphat + \betat\) belongs to
\(\Phi(\Gt_1, \Tt)\), then
the fact that
\abmap\Gt{\Gt_1} induces isomorphisms on
appropriate root groups and weight spaces,
together with
explicit computation in \(\mathfrak{pgl}_{2n + 1}\),
shows that
the unique \(\Tt\)-equivariant isomorphisms
\abmap{\uLie(\Gt)_{\tilde r}}{\Ut_{\tilde r}}
for \(\tilde r \in \Phi(\Gt', \Tt)\)
fit together into a commutative diagram
\[\xymatrix{
\Ut_\alphat \times \Ut_\betat \ar[r] & \Ut_{\alphat + \betat} \\
\uLie(\Gt)_\alphat \times \uLie(\Gt)_\betat \ar[r]\ar[u] & \uLie(\Gt)_{\alphat + \betat}, \ar[u]
}\]
where the top arrow is the group commutator and
the bottom arrow is the Lie-algebra commutator;
and that the latter gives a \(k\)-linear isomorphism
\abmap
	{\Lie(\Gt)_\alphat \otimes_k \Lie(\Gt)_\betat}
	{\Lie(\Gt)_{\alphat + \betat}}.
\end{enumerate}
\end{rem}

\begin{defn}
\label{defn:tilde-root}
If \(\mathcal R\) is a subset of \(\bX^*(S)\), then
we write \(\Gt_{\mathcal R}\) for
the derived subgroup of
the identity component of
the centralizer of
\(\bigcap_{a \in \mathcal R} \ker(a)\), so that
\(\Phi(\Gt_{\mathcal R}, S)\) equals
\(\Z\mathcal R \cap \Phi(\Gt, S)\).
We define \(\Gt_{\tilde{\mathcal R}}\) similarly for
a subset \(\tilde{\mathcal R}\) of \(\bX_*(\St)\).

If \(\Z_{\ge 0}\mathcal R \cap \Phi(\Gt, S)\) is
a system of positive roots for
\(\Z\mathcal R \cap \Phi(\Gt, S)\)
(for example, if there is some system
\(\Delta\) of simple roots for \(\Phi(\Gt, S)\) such that
every element of \(\mathcal R\) is
a non-negative integer multiple of
an element of \(\Delta\)), then
Remark \ref{rem:rd-subquotient} gives that
the set of all roots of \(\Tt_\ks\) whose `restriction' to \(S\)
lies in \(\Z_{\ge 0}\mathcal R \cap \Phi(\Gt, S)\) is
a (\(\Gal(k) \ltimes \Gamma(\ks)\))-stable system of positive roots for
\(\Phi(\Gt_{\mathcal R\,\ks}, \Tt_\ks)\).
Write
\(\Bt_{\mathcal R}\) for the corresponding Borel subgroup of
\(\Gt_{\mathcal R}\), and
\(\Ut_{\mathcal R}\) for the unipotent radical of \(\Bt_{\mathcal R}\).

If \(\mathcal R\) is a singleton \(\sset a\), then
we may write \(\Gt_a\), \(\Bt_a\), and \(\Ut_a\) in place of
\(\Gt_{\mathcal R}\), \(\Bt_{\mathcal R}\), and \(\Ut_{\mathcal R}\).
\end{defn}

\begin{rem}
\label{rem:root-characterize}
Preserve the notation and hypotheses of Definition \ref{defn:tilde-root},
including the assumption that
\(\Z_{\ge 0}\mathcal R \cap \Phi(\Gt, S)\) is
a system of positive roots for
\(\Z\mathcal R \cap \Phi(\Gt, S)\).

We have by
Proposition \ref{prop:quass-rough}(\ref{subprop:quass-up})
that \(S\) has no fixed points in \(\Lie(\Ut_{\mathcal R})\), so
\(\Ut_{\mathcal R}\) is an
\(S\)-stable, smooth, connected subgroup of \(\Gt\) whose
Lie algebra is
the sum of the weight spaces for \(S\) on \(\Lie(\Gt)\)
corresponding to weights in \(\Z_{\ge 0}\mathcal R \cap \Phi(\Gt, S)\).
By \cite{conrad-gabber-prasad:prg}*{Proposition 3.3.6},
these properties characterize \(\Ut_{\mathcal R}\) uniquely; in fact,
\(\Ut_{\mathcal R\,\ks}\) is the unique
\(S_\ks\)-stable, smooth, connected subgroup of \(\Gt_\ks\) whose
Lie algebra is
the sum of the weight spaces for \(S_\ks\) on \(\Lie(\Gt_\ks)\)
corresponding to weights in
\(\Z_{\ge 0}\mathcal R_\ks \cap \Phi(\Gt_\ks, S_\ks)\).

In particular,
if \(\Z_{\ge 0}\mathcal R \cap \Phi(\Gt, S)\) is empty, then
\(\Ut_{\mathcal R}\) is trivial.
\end{rem}

{\newcommand\nd{\textsub{nd}}
\begin{lem}
\label{lem:tilde-borel:linear:prop:13.20}
Let \(\Ht\) be a subgroup of \(\Gt\) that is
preserved by \(\Gamma \ltimes \Tt\).
Write \(\Phi(\Ht, S)\nd\) for the set of
non-divisible elements of \(\Phi(\Ht, S)\), i.e.,
elements \(a \in \Phi(\Ht, S)\) such that
\(a/2\) does not belong to \(\Phi(\Ht, S)\).
\begin{enumerate}[label=(\alph*), ref=\alph*]
\item\label{sublem:sub-multipliable}
An element of \(\Phi(\Ht, S)\) is
multipliable in \(\Phi(\Ht, S)\) if and only if it is
multipliable in \(\Phi(\Gt, S)\).
\item\label{sublem:prop:13.20-Lie}
\(\Lie(\Ht)\) equals
\(\Lie(\Ht \cap \Tt) \oplus
	\sum_{a \in \Phi(\Ht, S)\nd} \Lie(\Ut_a)\).
\item\label{sublem:sub-abs-roots}
\(\Phi(\Ht_\ks, \Tt_\ks)\) is the set of elements of
\(\Phi(\Gt_\ks, \Tt_\ks)\) whose
`restriction' to \(S\) lies in
\(\Phi(\Ht, S)\).
\item\label{sublem:prop:13.20-gp}
If \(\Ht\) is smooth and connected, then
it is generated by
\(\Ht \cap \Tt\) and
\(\Ut_a\) as \(a\) ranges over \(\Phi(\Ht, S)\nd\).
\item\label{sublem:prop:13.20-ss}
If \(\Ht\) is semisimple, then
it is generated by
\(\Ut_a\) as \(a\) ranges over \(\Phi(\Ht, S)\nd\).
\item\label{sublem:prop:14.4}
If \(\Ht\) is smooth and
contained in the unipotent radical of
a Borel subgroup \(\Bt\) containing \(\Tt\), then
\(\Ht\) is connected and
directly spanned by
\(\Ut_a\) as \(a\) ranges over \(\Phi(\Ht, S)\nd\),
in any order.
\end{enumerate}
\end{lem}

\begin{note}
The terminology `directly spanned' in
(\ref{sublem:prop:14.4}) is as in \cite{borel:linear}*{\S14.3}.
The Borel subgroup in (\ref{sublem:prop:14.4}) is not
assumed to be preserved by \(\Gamma\).
\end{note}

\begin{proof}
Suppose that \(a\) belongs to \(\Phi(\Ht, S)\), and
let \(\alphat\) be a weight of \(\Tt_\ks\) on
\(\Lie(\Ht)_a \otimes_k \ks\).
Since \(\Lie(\Gt_\ks)_\alphat\) is one dimensional and
intersects \(\Lie(\Ht) \otimes_k \ks\) nontrivially, it is contained in
\(\Lie(\Ht) \otimes_k \ks\).
Since \(\Lie(\Ht) \otimes_k \ks\) is preserved by
\(\Gal(k) \ltimes \Gamma(\ks)\), we have by
Proposition \ref{prop:rd-restriction}(\ref{subprop:rd-fiber}) that
\(\Lie(\Ht) \otimes_k \ks\) contains \(\Lie(\Gt)_a \otimes_k \ks\), hence that
\(\Lie(\Ht)\) contains \(\Lie(\Gt)_a\).

It is clear that, if \(a\) is multipliable in \(\Phi(\Ht, S)\), then
it is multipliable in \(\Phi(\Gt, S)\).
For the converse, suppose that \(a\) is multipliable in \(\Phi(\Gt, S)\).
Let \(\sset{\alphat, \alphat'}\) be
an exceptional pair for
\((\Psi(\Gt_\ks, \Tt_\ks), \Gal(k) \ltimes \Gamma(\ks))\)
`extending' \(a\).
Since \(\Phi(\Gt_\ks, \Tt_\ks)\) is reduced,
we have by
Proposition \ref{prop:rd-restriction}(\ref{subprop:rd-multipliable})
that we are in case
(\ref{case:rd-multipliable(red)}),
hence, by Remark~\ref{rem:gp-nred-facts}(\ref{subrem:root-commute}), that
the commutator map
\abmap
	{\Lie(\Gt_\ks)_\alphat \otimes_\ks
		\Lie(\Gt_\ks)_{\alphat'}}
	{\Lie(\Gt_\ks)_{\alphat + \alphat'}}
is an isomorphism.
Since
\(\Lie(\Gt_\ks)_\alphat\) and \(\Lie(\Gt_\ks)_{\alphat'}\)
are both contained in \(\Lie(\Gt)_a \otimes_k \ks\),
hence in \(\Lie(\Ht) \otimes_k \ks\), so is their commutator
\(\Lie(\Gt_\ks)_{\alphat + \alphat'}\).
That is, \(2a\) belongs to \(\Phi(\Ht, S)\).
This shows (\ref{sublem:sub-multipliable}).

We have shown that \(\Lie(\Ht)\) contains
\(\sum_{a \in \Phi(\Ht, S)\nd} \Lie(\Gt)_a \oplus \Lie(\Gt)_{2a} =
\sum_{a \in \Phi(\Ht, S)\nd} \Lie(\Ut_a)\), and
it is clear that it contains \(\Lie(\Ht \cap \Tt)\).
This shows the containment \(\supseteq\) in
(\ref{sublem:prop:13.20-Lie}).
The reverse containment follows from the equality
\(\Lie(\Ht) \otimes_k \ks =
(\Lie(\Ht \cap \Tt) \otimes_k \ks) \oplus
	\bigoplus_{\alphat \in \Phi(\Ht_\ks, \Tt_\ks)}
		\Lie(\Gt_\ks)_\alphat\)
and the fact that,
if \(\alphat\) is an element of \(\Phi(\Ht_\ks, \Tt_\ks)\) and
we write \(a\) for the `restriction' of \(\alphat\) to \(S\), then
\(\Lie(\Gt_\ks)_\alphat\) is contained in
\(\Lie(\Ut_a) \otimes_k \ks\) and, if \(a\) is divisible, in
\(\Lie(\Ut_{a/2}) \otimes_k \ks\).
This shows
(\ref{sublem:prop:13.20-Lie},%
\ref{sublem:sub-abs-roots}).

In the situation of (\ref{sublem:prop:14.4}),
we have by \cite{borel:linear}*{Proposition 14.4(2a)} that
\(\Ht\) is connected.
Thus we may, and do, assume for the remainder of the proof that
\(\Ht\) is smooth and connected.

Together with (\ref{sublem:prop:13.20-Lie}),
Remark \ref{rem:root-characterize} and
\cite{conrad-gabber-prasad:prg}*{Proposition 3.3.6}
give that \(\Ht\) contains every root subgroup \(\Ut_a\)
corresponding to an element \(a \in \Phi(\Ht, S)\).
Therefore, the subgroup \(\Ht'\) of \(\Gt\) generated by
\(\Ht \cap \Tt\), and
those \(\Ut_a\) with \(a \in \Phi(\Ht, S)\nd\),
is contained in \(\Ht\).
Since \(\Ht \cap \Tt\) is the fixed-point subgroup of \(\Tt\) on
the \(\Tt\)-stable, smooth, connected subgroup \(\Ht\) of \(\Gt\),
we have that it is smooth
\cite{conrad-gabber-prasad:prg}*{Proposition A.8.10(2)}, so
\(\Ht'\) is smooth.
Since
\(\Lie(\Ht')\) contains
\(\Lie(\Ht \cap \Tt) \oplus
	\sum_{a \in \Phi(\Ht, S)\nd} \Lie(\Ut_a) =
\Lie(\Ht)\),
it follows that \(\Ht'\) equals \(\Ht\).
This shows (\ref{sublem:prop:13.20-gp}).

In the situation of (\ref{sublem:prop:14.4}), we have by
\cite{borel:linear}*{Proposition 14.4(2a)}
(which we have already used to show that \(\Ht\) is connected)
that
\(\Ht_\ks\) is directly spanned by
\(\Ut_\alphat\) as \(\alphat\) ranges over
\(\Phi(\Ht_\ks, \Tt_\ks)\), in any order, but also
for every \(a \in \Phi(\Ht, S)\) that
\(\Ut_{a\,\ks}\) is directly spanned by
\(\Ut_\alphat\) as \(\alphat\) ranges over
the `extensions' of \(a\) and \(2a\) in \(\Phi(\Gt_\ks, \Tt_\ks)\),
again in any order.
Grouping the elements of \(\Phi(\Ht_\ks, \Tt_\ks)\)
according to their `restrictions' to \(S\) thus shows that
(\ref{sublem:prop:14.4}) holds
after base change to \(\ks\), hence already holds rationally.

Finally, suppose that \(\Ht\) is semisimple.
Then the subgroup \(\Ht''\) of \(\Ht\) generated by
those \(\Ut_a\) with \(a \in \Phi(\Ht, S)\nd\) has the property that,
for every \(\alphat \in \Phi(\Ht_\ks, \Tt_\ks)\) with
`restriction' to \(S\) equal to \(a\) or \(2a\),
the base-changed group \(\Ht''_\ks\) contains
\(\Ut_{a\,\ks}\), hence its subgroup \(\Ut_\alphat\).
Since \(\Ht_\ks\) is a split, semisimple group, it is
generated by its root subgroups, so
\(\Ht''_\ks\) equals \(\Ht_\ks\), and hence
\(\Ht''\) equals \(\Ht\).
This shows (\ref{sublem:prop:13.20-ss}).
\end{proof}
}

\begin{notation}
\label{notn:pi0-act}
Since \(\Gamma\conn(\ks)\) acts trivially on \(\Tt_\ks\)
(by Remark \ref{rem:torus-quass}(\ref{subrem:quass-to-torus})),
we have that \(\pi_0(\Gamma)(\ks)\) acts on
\(\Phi(\Gt_\ks, \Tt_\ks)\).
If \(\tilde{\mathcal R}\) is a subset of \(\Phi(\Gt, \St)\), then
we write \(\Gamma_{\tilde{\mathcal R}}\) for
the descent to \(k\) of
the subgroup of \(\Gamma_\ks\) generated by
\(\Gamma\conn_\ks\) and
the subgroup of \(\Gamma(\ks)\) preserving
the subset of \(\Phi(\Gt_\ks, \Tt_\ks)\)
consisting of elements whose `restriction' to \(\St\) belongs to
\(\tilde{\mathcal R}\).

If \(\tilde{\mathcal R}\) is a singleton \(\at\), then we
may write \(\Gamma_\at\) in place of \(\Gamma_{\tilde{\mathcal R}}\).
\end{notation}

\begin{rem}
\label{rem:switch-exc-pair}
If \(\sset{\alphat, \alphat'}\) is
an exceptional pair for
\((\Psi(\Gt_\ks, \Tt_\ks), \Gal(k) \ltimes \Gamma(\ks))\), and
\(\at\) and \(\at'\) are the `restrictions' to \(\St\) of
\(\alphat\) and \(\alphat'\), then
Lemma \ref{lem:half-pairs-b}(\ref{sublem:exc-stab})
shows that
\(\Gamma_{\sset{\at, \at'}}\) and
\(\Gamma_{\at + \at'}\)
are both the descent to \(k\) of the subgroup of \(\Gamma_\ks\) generated by
\(\Gamma\conn_\ks\) and
the intersection of the stabilizers in \(\pi_0(\Gamma)(\ks)\) of
the irreducible components of
\(\Phi(\Gt_\ks, \Tt_\ks)\) that
contain some \(\Gal(k)\)-conjugate of
\(\alphat\) or \(\alphat'\), so
they are equal.

If \(\alphat\) and \(\alphat'\) are
fixed by \(\Gal(k)\), so that
\(\Gamma_{\sset{\at, \at'}}\) is
the descent to \(k\) of
\(\Gamma_{\ks\,\sset{\alphat, \alphat'}}\), and
analogously for
\(\Gamma_\at\) and \(\Gamma_{\at'}\), then
Proposition \ref{prop:rd-restriction}(\ref{subprop:rd-fiber})
gives that
\(\Gamma_\at(\ks) = \stab_{\Gamma(\ks)}(\alphat)\) and
\(\Gamma_{\at'}(\ks) = \stab_{\Gamma(\ks)}(\alphat')\)
are the same index-\(2\) subgroup of
\(\stab_{\Gamma(\ks)}\sset{\alphat, \alphat'} =
\Gamma_{\sset{\at, \at'}}(\ks)\).
Thus
\(\Gamma_{\at\,\ks}\) and \(\Gamma_{\at'\,\ks}\) are
the same index-\(2\) subgroup of
\(\Gamma_{\sset{\at, \at'}\,\ks}\), so
\(\Gamma_\at\) and \(\Gamma_{\at'}\) are
the same index-\(2\) (open) subgroup of
\(\Gamma_{\sset{\at, \at'}}\).
\end{rem}

\begin{cor}
\label{cor:root-induced}
Fix an element \(a \in \Phi(\Gt, S)\).
If \(a\) is non-multipliable, then let
\(\alphat = \alphat'\) be a weight of \(\Tt_\ks\) on
\(\Lie(\Gt)_a \otimes_k \ks\).
If \(a\) is multipliable, then let
\(\sset{\alphat, \alphat'}\) be
an exceptional pair for
\((\Psi(\Gt_\ks, \Tt_\ks), \Gal(k) \ltimes \Gamma(\ks))\)
`extending' \(a\).
In either case, write
\(\at\) and \(\at'\) for
the `restrictions' to \(\St\) of
\(\alphat\) and \(\alphat'\).
Then
\(\Gt_{\sset{\at, \at'}}\) is
an almost-simple factor of
\(\Gt_a\).
The corresponding map
\map{\psi_{\sset{\at, \at'}}}
	{\Gt_a}
	{\Ind_
		{\Gamma_{\sset{\at, \at'}}}
	^	\Gamma
		\Gt_{\sset{\at, \at'}\,\adsub}}
from Lemma \ref{lem:ind-simple-ad}
is a central isogeny onto its image, and
restricts to an isomorphism of \(\Ut_a\) on its image, which is
contained in
\(\Ind_{\Gamma_{\sset{\at, \at'}}}^\Gamma \Ut_{\sset{\at, \at'}}\).
If \(\Gamma/\Gamma_{\sset{\at, \at'}}\) is constant, then
\(\psi_{\sset{\at, \at'}}\) is surjective, and
restricts to an isomorphism of \(\Ut_a\) onto
\(\Ind_{\Gamma_{\sset{\at, \at'}}}^\Gamma
	\Ut_{\sset{\at, \at'}}\).
\end{cor}

\begin{proof}
Remark \ref{rem:how-to-irred} and
Proposition \ref{prop:rd-restriction}%
	(\ref{subprop:coroot-nonmultipliable},%
	\ref{subprop:rd-multipliable})
give that \(\Gt_{\sset{\at, \at'}}\) is
an almost-simple factor of \(\Gt_a\), and,
together with
Lemma \ref{lem:ind-simple-ad}(\ref{sublem:Gamma-closure}), that
\(\Gt_a\) is
the smallest \(\Gamma\)-stable subgroup of \(\Gt\)
containing \(\Gt_{\sset{\at, \at'}}\).
Therefore,
Lemma \ref{lem:ind-simple-ad}(\ref{sublem:ind-isogeny})
gives that \(\psi_{\sset{\at, \at'}}\) is
a central isogeny onto its image.
Since \(\Ut_a\) is unipotent, its intersection with
\(\ker(\psi_{\sset{\at, \at'}})\), which is
central in \(\Gt_a\) and so of multiplicative type, is
trivial, so that the restriction of \(\psi_{\sset{\at, \at'}}\) to
\(\Ut_a\) is an isomorphism onto its image.
Thus,
if we write \(\Ut'\) for the image of \(\Ut_a\) in
\(\Gt_{\sset{\at, \at'}\,\adsub}\) under the map of
Lemma \ref{lem:ind-simple-ad}(\ref{sublem:simple-ad}), then,
since all weights of \(S\) on \(\Lie(\Ut_a)\) lie in
\(\Z_{\ge 0}\cdot a\),
it follows that all weights of \(S\) on \(\Lie(\Ut')\)
also lie in \(\Z_{\ge 0}\cdot a\).
However, the set of elements of
\(\Phi(\Gt_{\sset{\at, \at'}}, \St) =
(\Z\at + \Z\at') \cap \Phi(\Gt, \St)\)
that restrict to an element of \(\Z_{\ge 0}\cdot a\) is
\((\Z_{\ge 0}\at + \Z_{\ge 0}\at') \cap \Phi(\Gt, \St) =
\Phi(\Ut_{\sset{\at, \at'}}, \St)\), for each of which
the corresponding root group for \(\St\) in \(\Gt\) is
contained in \(\Ut_{\sset{\at, \at'}}\).
That is, the image of \(\Ut_a\) in
\(\Gt_{\sset{\at, \at'}\,\adsub}\) is
contained in \(\Ut_{\sset{\at, \at'}}\), so
the image of \(\Ut_a\) in
\(\Ind_{\Gamma_{\sset{\at, \at'}}}^\Gamma \Gt_{\sset{\at, \at'}}\) is
contained in
\(\Ind_{\Gamma_{\sset{\at, \at'}}}^\Gamma \Ut_{\sset{\at, \at'}}\).

If \(\Gamma/\Gamma_{\sset{\at, \at'}}\) is constant, then
Lemma \ref{lem:ind-simple-sc}
implies that \(\psi_{\sset{\at, \at'}}\) is surjective.
It follows from
\cite{borel:linear}*{Proposition 11.14(1)} and
the fact that
\(\Ind_{\Gamma_{\sset{\at, \at'}}}^\Gamma \Ut_{\sset{\at, \at'}}\) is
the unipotent radical of a Borel subgroup of
\(\Ind_{\Gamma_{\sset{\at, \at'}}}^\Gamma \Gt_{\sset{\at, \at'}}\) that
it is the image of \(\Ut_a\).
\end{proof}

Although it must be well known, we could not find a reference for
the statement in Lemma \ref{lem:root-linear}
about derived subgroups of root groups even when
\(\Gamma\) acts trivially, so that we are just talking about
ordinary relative-root subgroups.

\begin{lem}
\label{lem:root-linear}
Fix \(a \in \Phi(\Gt, S)\).
The derived subgroup of \(\Ut_a\) is \(\Ut_{2a}\), which is
central in \(\Ut_a\).
The quotient \((\Ut_a/\Ut_{2a})_\ks\) carries a unique
\(S_\ks\)-equivariant linear structure
(i.e., \(S_\ks\)-equivariant isomorphism with
\(\uLie((\Ut_a/\Ut_{2a})_\ks)\),
whose derivative is the identity).
This linear structure is also
\((\Gamma_\ks \ltimes \Tt_\ks)\)-equivariant.
It descends to linear structures on
\(\Ut_a/\Ut_{2a}\) and
\((\Ut_a/\Ut_{2a})^\Gamma\).
\end{lem}

\begin{proof}
Since \(\Phi(\Gt, S)\) is a root system
(by Remark \ref{rem:how-to-restrict}), we have that
\(\Phi(\Gt, S) \cap \Z\cdot a\) is contained in
\(\pm\sset{a, 2a}\).
In particular,
\cite{conrad-gabber-prasad:prg}*
	{Proposition 3.3.5 and Example 3.3.2} give that
\(\Ut_{2a}\) is central in \(\Ut_a\).
We have by
\cite{conrad-gabber-prasad:prg}*{Lemma 3.3.8} that there is
a unique \(S_\ks\)-equivariant linear structure on
\((\Ut_a/\Ut_{2a})_\ks\); and uniqueness shows that
it is fixed by both \(\Gal(k)\) and \(\Gamma(\ks) \ltimes \Tt(\ks)\),
hence by \(\Gamma_\ks \ltimes \Tt_\ks\)
(because \(\Gamma \ltimes \Tt\) is smooth).
It follows that the \(S_\ks\)-equivariant linear structure on
\((\Ut_a/\Ut_{2a})_\ks\)
descends to linear structures on
\(\Ut_a/\Ut_{2a}\) and
\((\Ut_a/\Ut_{2a})^\Gamma\), as claimed.

Since \(\Ut_a/\Ut_{2a}\) is commutative,
\(\Ut_{2a}\) contains
the derived subgroup of \(\Ut_a\).
If \(a\) is not multipliable in \(\Phi(\Gt, S)\), then
\(\Ut_{2a}\) is trivial, so the reverse containment, and
hence equality, is clear.
Thus we may, and do, assume that \(a\) is multipliable.
Let \(\sset{\alphat, \alphat'}\) be
an exceptional pair for
\((\Psi(\Gt_\ks, \Tt_\ks), \Gal(k) \ltimes \Gamma(\ks))\)
`extending' \(a\).
By Remark \ref{rem:gp-nred-facts}(\ref{subrem:root-commute}),
the commutator map
\abmap
	{\Ut_\alphat \times \Ut_{\alphat'}}
	{\Ut_{\alphat + \alphat'}}
is surjective; so, since
\(\Ut_\alphat\) and \(\Ut_{\alphat'}\) are
contained in \(\Ut_{a\,\ks}\), we have that
\(\Ut_{\alphat + \alphat'}\) is contained in
\((\Ut_a)_{\ks\,\dersub} = (\Ut_{a\,\dersub})_\ks\).
It follows that \(\Ut_{a\,\dersub}\), which is
contained in \(\Ut_{2a}\), is not
the trivial group, so that
\(\Phi(\Ut_{a\,\dersub}, S)\) equals \(\sset{2a}\).
Then
Lemma \ref{lem:tilde-borel:linear:prop:13.20}(\ref{sublem:prop:13.20-gp})
gives that
\(\Ut_{a\,\dersub}\) contains, hence equals, \(\Ut_{2a}\).
\end{proof}

\begin{lem}
\label{lem:root-smooth}
Fix \(a \in \Phi(\Gt, S)\).
If \(a\) is not multipliable in \(\Phi(\Gt, S)\), or
\(a\) does not belong to \(\Phi(\Gt^\Gamma, S)\), or
\(p\) is odd, then
\(\Ut_a^\Gamma\) carries a unique
\(T\)-equivariant linear structure.
\end{lem}

\begin{note}
Lemma \ref{lem:root-smooth} can fail if
\(a\) is split for \((\Psi(\Gt_\ks, T_\ks), \Gal(k))\),
even if
\(p\) is odd and
\(\Gamma\) is trivial.
For example, suppose that
\(E/k\) is a quadratic, Galois extension,
write \(\Gt\) for the quasisplit \(\SU_{3, E/k}\)
and \(\Gamma\) for the trivial group,
with (necessarily) its trivial action on \(\Gt\).
Thus, \(G \ldef \fix\Gt^\Gamma\) equals \(\Gt\).
Let \(S = \St\) be a maximal split torus in \(G = \Gt\).
Then there are multipliable elements of
\(\Phi(\Gt, S) = \Phi(G, S)\), and,
for every such element \(a\), we have that
\(\Ut_a^\Gamma\) is the full \(a\)-root subgroup for
\(\St\) in \(\Gt\), and that
\((\Ut_a^\Gamma)\der = \Ut_{a\,\dersub}\) equals
\(\Ut_{2a}\), which is nontrivial, by
Lemma \ref{lem:root-linear}.
In particular, \(\Ut_a\) is not even commutative, so
certainly not a vector group.
\end{note}

\begin{proof}
Lemma \ref{lem:tilde-borel:linear:prop:13.20}(\ref{sublem:prop:14.4})
(applied, here and later, with
\(\ks\) in place of \(k\), hence
\(T_\ks\) in place of \(S\))
gives that
\(\Ut_{a\,\ks}\) is
directly spanned by
\(\Ut_\alpha\) as
\(\alpha\) ranges over the `extensions' of \(a\) in
\(\Phi(\Gt_\ks, T_\ks)\).
If \(a\) does not belong to \(\Phi(\Gt^\Gamma, S)\), then
\(\Ut_a\) equals \(\Ut_{2a}\), and
the result follows from Lemma \ref{lem:root-linear}
(with \(2a\) replacing \(a\)).
Thus we may, and do, assume that
\(a\) belongs to \(\Phi(\Gt^\Gamma, S)\).
Then every weight of \(T_\ks\) on \(\Lie(G)_{a\,\ks}\) is
an `extension' of \(a\) that
belongs to \(\Phi(\Gt_\ks^{\Gamma_\ks}, T_\ks)\), so,
by Proposition \ref{prop:rd-restriction}(\ref{subprop:rd-fiber}),
all such `extensions' do.
If \(a\) is non-multipliable, then
so is every `extension'.

Thus we may, and do, assume, upon replacing
\(a\) by an `extension' \(\alpha \in \Phi(\Gt_\ks^{\Gamma_\ks}, T_\ks)\)
and
\(k\) by \(\ks\) (hence \(S\) by \(T\)), that
\(\alpha\) is non-multipliable or
\(p\) is odd.
Now \cite{conrad-gabber-prasad:prg}*{Lemma 3.3.8}
will allow us to conclude if we can show that
\(\Phi(\Ut_\alpha^\Gamma, T)\) equals \(\sset\alpha\).
If \(\alpha\) is not multipliable, then
this is obvious
(since \(\Phi(\Ut_\alpha, T)\) contains
the positive-integer multiples of \(\alpha\) in \(\Phi(\Gt, T)\),
and \(\alpha\) is the only such).
Thus we may, and do, assume that
\(p\) is odd and
\(\alpha\) is multipliable.

Since \(k\) is separably closed, we have that
\(\Gamma/\Gamma_{\sset{\alphat, \alphat'}}\) is constant.
Therefore, by Corollary \ref{cor:root-induced} (and using its notation),
we have by Lemma \ref{lem:was-cor:trans-induction} that
\(\Ut_\alpha^\Gamma\) is \(T\)-equivariantly isomorphic to
\(\Ut_{\sset{\alphat, \alphat'}}^
	{\Gamma_{\sset{\alphat, \alphat'}}}\).
By Remark \ref{rem:red-to-nred-facts}(\ref{subrem:pre-mult-div}),
there is an element \(\gamma_0 \in \Gamma(k)\) that
swaps \(\alphat\) and \(\alphat'\), hence belongs to
\(\Gamma_{\sset{\alphat, \alphat'}}(k)\).
Let \(X\) be a nonzero element of
\(\Lie(\Gt)_\alpha^\Gamma\).
By Lemma \ref{lem:root-space},
the \(\Tt\)-equivariant projections
\(\Xt_\alphat\) and \(\Xt_{\alphat'}\) of
\(X\) on the indicated weight spaces for \(\Tt\) in \(\Lie(\Gt)\) are
also nonzero.
By Remark \ref{rem:gp-nred-facts}(\ref{subrem:root-commute}),
the commutator
\(\Yt \ldef [\Xt_\alphat, \Xt_{\alphat'}]\) is also nonzero.
On the other hand, since \(X\) is preserved by \(\gamma_0\),
we have that
\(\gamma_0(\Xt_\alphat)\) equals
\(\Xt_{\gamma_0\alphat} = \Xt_{\alphat'}\),
and, similarly,
\(\gamma_0(\Xt_{\alphat'})\) equals
\(\Xt_\alphat\), so
\(\gamma_0(\Yt)\) equals
\([\Xt_{\alphat'}, \Xt_\alphat] = -\Yt\).
Since \(p\) is odd, we have that
\(\Yt\) does not equal \(-\Yt\), so that
\(\gamma_0\) does not act trivially on
\(\Lie(\Gt)_{\alphat + \alphat'}\).
Since \(\gamma_0\) preserves \(\alphat + \alphat'\),
Corollary \ref{cor:root-space-facts}(\ref{subcor:root-space-fix})
gives that the restriction \(2\alpha\) of
\(\alphat + \alphat'\) to
\(T\) does not
belong to \(\Phi(\Gt^\Gamma, T)\).
That is, \(\Phi(\Ut_\alpha^\Gamma, T)\) equals \(\sset\alpha\),
as required.
\end{proof}

Usually, dealing with the case where \(p\) equals \(2\) is harder than
dealing with the case where \(p\) is odd.
Proposition \ref{prop:square-pinned}(\ref{subprop:square-pinned})
shows that, unusually,
dealing with multipliable roots is easier when \(p\) equals \(2\),
in the sense that a root in \(\Phi(\Gt^\Gamma, S)\) that
is multipliable in \(\Phi(\Gt, S)\) remains multipliable in
\(\Phi(\Gt^\Gamma, S)\).
See
Proposition \ref{prop:GS-vs-GtS}(\ref{subprop:G-roots}) and
\cite{adler-lansky-spice:actions3}*{Lemma \ref{GA3-lem:pinning-descent}} for
applications of Proposition \ref{prop:square-pinned}.

\begin{prop}
\label{prop:square-pinned}
Suppose that \(p\) equals \(2\) and \(a \in \Phi(\Gt, S)\)
is multipliable.
Write \((\cdot)^{[2]}\) for the \(2\)-power map on
\(\uLie(\Ut_a)\)
\cite{demazure-gabriel:groupes-algebriques}*
	{Ch.~II, \S7, Proposition 3.4}
\begin{enumerate}[label=(\alph*), ref=\alph*]
\item\label{subprop:square-basic}
\((\cdot)^{[2]}\) is a \((\Gamma \ltimes \Tt)\)-equivariant map
\abmap{\uLie(\Ut_a)}{\uLie(\Ut_{2a})} that
factors uniquely through
\abmap{\uLie(\Ut_a)}{\uLie(\Ut_a/\Ut_{2a})}.
\end{enumerate}
Let \(\sset{\alphat, \alphat'}\) be
an exceptional pair for
\((\Psi(\Gt_\ks, \Tt_\ks), \Gal(k) \ltimes \Gamma(\ks))\)
`extending' \(a\).
\begin{enumerate}[resume*]
\item\label{subprop:root-commute}
With the notation of
Lemma \ref{lem:root-space},
we have that
\(\pi_{\alphat + \alphat'}(\Xt^{[2]})\) equals
\([\pi_\alphat(\Xt), \pi_{\alphat'}(\Xt)]\)
for all functorial points \(\Xt\) of \(\uLie(\Ut_a)\).
\item\label{subprop:square-pinned}
\((\cdot)^{[2]}\) does not take the value \(0\) on
\(\Lie(\Ut_a/\Ut_{2a})^\Gamma \setminus \sset0\).
In particular, if
\(a\) belongs to \(\Phi(\Gt^\Gamma, S)\), then
so does \(2a\).
\end{enumerate}
\end{prop}


\begin{proof}
It is clear that the \(2\)-power map is
compatible with base change, and
functorial
\cite{demazure-gabriel:groupes-algebriques}*
	{Ch.~II, \S7, 1.1}.
Thus
\(((\cdot)^{[2]})_\ks\) is (\(\Gamma(\ks) \ltimes \Tt(\ks)\))-,
hence (\(\Gamma_\ks \ltimes \Tt_\ks\))-,
equivariant,
so \((\cdot)^{[2]}\) is (\(\Gamma \ltimes \Tt\))-equivariant.
In particular, we have by
\cite{demazure-gabriel:groupes-algebriques}*
	{Ch.~II, \S7, D\'efinition 3.3(\(p\)-AL 1)} that
\((\cdot)^{[2]}\) doubles weights of \(S\), hence
carries \(\uLie(\Ut_a)\) into \(\uLie(\Ut_{2a})\).
Since
\(\Ut_{2a}\) is a vector group by Lemma \ref{lem:root-linear},
we have that
\(\uLie(\Ut_{2a})\) is annihilated by \((\cdot)^{[2]}\)
\cite{demazure-gabriel:groupes-algebriques}*
	{Ch.~II, \S7, Exemple 2.2(1)}.
Since \(\Ut_{2a}\) is central in \(\Ut_a\) by
Lemma \ref{lem:root-linear}, and since
\abmap{\uLie(\Ut_a)}{\uLie(\Ut_a/\Ut_{2a})} is
surjective, we have by
\cite{demazure-gabriel:groupes-algebriques}*
	{Ch.~II, \S7, D\'efinition 3.3(\(p\)-AL 3)}
(which simplifies to
\((\Xt + \Yt)^{[2]} =
\Xt^{[2]} + [\Xt, \Yt] + \Yt^{[2]}\) for all
functorial points \(\Xt\) and \(\Yt\) of \(\uLie(\Ut_a)\)
when \(p\) equals \(2\))
that the \([2]\)-power map factors uniquely through
\abmap{\uLie(\Ut_a)}{\uLie(\Ut_a/\Ut_{2a})}.
This shows (\ref{subprop:square-basic}).

We now prove (\ref{subprop:root-commute}).
The \(2\)-power map \(((\cdot)^{[2]})_\ks\) annihilates
\(\uLie(\Ut_\betat)\) for every
\(\betat \in \Phi(\Ut_{a\,\ks}, \Tt_\ks)\)
(indeed, for every
\(\betat \in \Phi(\Gt_\ks, \Tt_\ks)\)).
Thus applying
\cite{demazure-gabriel:groupes-algebriques}*
	{Ch.~II, \S7, D\'efinition 3.3(\(p\)-AL 3)}
iteratively to the computation of the \(2\)-power of an element of
\(\uLie(\Ut_{a\,\ks}) =
\sum_{\betat \in \Phi(\Ut_{a\,\ks}, \Tt_\ks)}
	\uLie(\Ut_\beta)\)
shows that it is a sum of iterated commutators of vectors in
the various \(\uLie(\Ut_\beta)\).
By
Proposition \ref{prop:rd-restriction}%
	(\ref{subprop:rd-multipliable})(\ref{case:rd-multipliable(red)}),
the only nontrivial (i.e., with more than one term)
expression of
\(\alphat + \alphat'\) as
a sum of elements of \(\Phi(\Ut_{a\,\ks}, \Tt_\ks)\) is
the obvious one, so
the equality in (\ref{subprop:root-commute}) follows.

Now
(\ref{subprop:root-commute}),
Lemma \ref{lem:root-space},
and
Remark \ref{rem:gp-nred-facts}(\ref{subrem:root-commute}) give that
\(\pi_{\alphat + \alphat'}(X^{[2]})\),
hence \(X^{[2]}\) itself,
is nonzero for every \(X \in \Lie(\Ut_a/\Ut_{2a})^\Gamma\).
This shows (\ref{subprop:square-pinned}).
\end{proof}

Corollary \ref{cor:no-mult-fixed} says that
an inert, multipliable root in \(\Phi(\Gt^\Gamma, S)\)
disappears upon smoothing, i.e., does not belong to \(\Phi(G, S)\).
However, we can say more than this about the structure of
\(\Ut_a^\Gamma\), and do so in
Proposition \ref{prop:no-mult-fixed}.
Specifically, we view
Proposition \ref{prop:no-mult-fixed}(\ref{subprop:no-mult-fixed}) as
a computation of the connecting map in the usual
``non-commutative'' generalization of the result
\cite{demazure-gabriel:groupes-algebriques}*
	{Ch.~II, \S3, Proposition 1.3}
on Hochschild cohomology, in the sense of
\cite{demazure-gabriel:groupes-algebriques}*{Ch.~II, \S3, 1.1}.
Note that
Proposition \ref{prop:no-mult-fixed}(\ref{subprop:factor-Frob-right})
is stronger than
Proposition \ref{prop:square-pinned}(\ref{subprop:square-pinned}),
but we use the latter in the proof of the former.

\begin{prop}
\label{prop:no-mult-fixed}
Preserve the notation and hypotheses of
Proposition \ref{prop:square-pinned}, and
suppose additionally that
\(a\) is inert for \((\Psi(\Gt_\ks, T_\ks), \Gal(k))\).
Write \((\cdot)^{(2)}\) for the Frobenius twist
and \(\operatorname{Frob}_{(\cdot)}\)
for the Frobenius natural transformation
\abmap{(\cdot)}{(\cdot)^{(2)}}
\cite{demazure-gabriel:groupes-algebriques}*
	{Ch.~II, \S7, 1.1},
and use the linear structures of
Lemma \ref{lem:root-linear} to regard \((\cdot)^{[2]}\) as
a \(\Gamma\)-equivariant map \abmap{\Ut_a/\Ut_{2a}}{\Ut_{2a}}.
\begin{enumerate}[label=(\alph*), ref=\alph*]
\item\label{subprop:factor-Frob-left}
\(\Frob_{(\Ut_a/\Ut_{2a})^\Gamma}\) factors through
\map{(\cdot)^{[2]}}{(\Ut_a/\Ut_{2a})^\Gamma}{\Ut_{2a}^\Gamma}.
\item\label{subprop:factor-Frob-right}
If \(a\) belongs to \(\Phi(\Gt^\Gamma, S)\), then
\map{(\cdot)^{[2]}}{(\Ut_a/\Ut_{2a})^\Gamma}{\Ut_{2a}^\Gamma}
is an infinitesimal isogeny, so that
(\ref{subprop:factor-Frob-left}) yields a unique arrow
\abmap{\Ut_{2a}^\Gamma}{((\Ut_a/\Ut_{2a})^\Gamma)^{(2)}}, and
the diagram
\[\xymatrix{
\Ut_{2a}^\Gamma \ar[r]\ar@(ur,ul)[rr]^{\operatorname{Frob}_{\Ut_{2a}^\Gamma}} & ((\Ut_a/\Ut_{2a})^\Gamma)^{(2)} \ar[r]^-{(\cdot)^{[2]}} & (\Ut_{2a}^\Gamma)^{(2)}
}\]
commutes.
\item\label{subprop:no-mult-fixed}
The sequence
\[\xymatrix{
\Ut_a^\Gamma \ar[r] & (\Ut_a/\Ut_{2a})^\Gamma \ar[r]^-{(\cdot)^{[2]}} & \Ut_{2a}^\Gamma
}\]
is exact.
\end{enumerate}
\end{prop}

\begin{proof}
If \(a\) does not belong to \(\Phi(\Gt^\Gamma, S)\), then
(\ref{subprop:factor-Frob-right}) is vacuously true,
and \((\Ut_a/\Ut_{2a})^\Gamma\) is trivial, so
the rest of the result is obvious.
Thus we may, and do, assume that
\(a\) belongs to \(\Phi(\Gt^\Gamma, S)\).
Then every weight of \(T_\ks\) on \((\Lie(\Gt)_a^\Gamma)_\ks\)
is an `extension' of \(a\) in \(\Phi(\Gt_\ks, T_\ks)\) that
belongs to \(\Phi(\Gt_\ks^{\Gamma_\ks}, T_\ks)\);
so Proposition \ref{prop:rd-restriction}(\ref{subprop:rd-fiber})
gives that every `extension' of \(a\) in \(\Phi(\Gt_\ks, T_\ks)\)
belongs to \(\Phi(\Gt_\ks^{\Gamma_\ks}, T_\ks)\).

By Definition \ref{defn:inert}, since
\(a\) is inert for \((\Psi(\Gt_\ks, T_\ks), \Gal(k))\),
every `extension' of \(a\) in \(\Phi(\Gt_\ks, T_\ks)\) is
multipliable, and
Lemma \ref{lem:half-pairs-b}(\ref{sublem:coexc-extensions})
gives that the `extensions' of \(2a\) are
precisely the characters \(2\alpha\) as
\(\alpha\) ranges over the `extensions' of \(a\).
Proposition \ref{prop:rd-restriction}%
	(\ref{subprop:rd-multipliable})(\ref{case:rd-multipliable(nred)})
and
Lemmas
\ref{lem:tilde-borel:linear:prop:13.20}(\ref{sublem:prop:14.4}) and
\ref{lem:root-linear}
give that the multiplication maps
\abmap{\prod \Ut_\alpha}{\Ut_{a\,\ks}},
\abmap{\prod \Ut_{2\alpha}}{\Ut_{2a\,\ks}}, and
\abmap{\prod \Ut_\alpha/\Ut_{2\alpha}}{(\Ut_a/\Ut_{2a})_\ks},
where all products range over
the `extensions' \(\alpha\) of \(a\) in
\(\Phi(\Gt_\ks, T_\ks)\),
are group isomorphisms.
Thus we may, and do, assume, upon
replacing \(k\) by \(\ks\), \(S\) by \(T_\ks\), and
\(a\) by an `extension' \(\alpha \in \Phi(\Gt_\ks, T_\ks)\), that
\(k\) is separably closed.

Choose an exceptional pair \(\sset{\alphat, \alphat'}\) for
\((\Psi(\Gt, \Tt), \Gamma(k))\)
extending \(\alpha\).
Put \(\Gt_1 = \Gt_{\sset{\alphat, \alphat'}}\),
\(\Ut_1 = \Ut_{\sset{\alphat, \alphat'}}\), and
\(\Ut_2 = \Ut_{\alphat + \alphat'}\).

Remark \ref{rem:switch-exc-pair} gives that
\(\Gamma_{\sset{\alphat, \alphat'}}\) equals
\(\Gamma_{\alphat + \alphat'}\),
\(\Gamma_\alphat\) equals \(\Gamma_{\alphat'}\),
and, \(\Gamma_1: = \Gamma_\alphat = \Gamma_{\alphat'}\)
is an index-\(2\) subgroup of
\(\Gamma_2:=\Gamma_{\sset{\alphat, \alphat'}} = \Gamma_{\alphat + \alphat'}\).
Since \(\alpha\) belongs to \(\Phi(\Gt^\Gamma, T)\) and
\(2\alpha\) belongs to \(\Phi(\Gt, T)\),
Proposition \ref{prop:square-pinned}(\ref{subprop:square-pinned})
gives that
\(2\alpha\) belongs to \(\Phi(\Gt^\Gamma, T)\).
Thus,
Corollary \ref{cor:root-space-facts}(\ref{subcor:root-space-fix})
gives that
\(\Gamma_2\)
acts trivially on
\(\Lie(\Gt)_{\alphat + \alphat'} = \Lie(\Ut_2)\), and
\(\Gamma_1\) acts trivially on
\(\Lie(\Gt_1)_\alpha =
\Lie(\Gt)_\alphat + \Lie(\Gt)_{\alphat'} =
\Lie(\Ut_\alphat) + \Lie(\Ut_{\alphat'})\), hence on
\(\Lie(\Ut_1/\Ut_2)\).
Since \(\Ut_\alphat\), \(\Ut_{\alphat'}\), \(\Ut_1/\Ut_2\), and \(\Ut_2\)
all carry \(\Gamma_2\)-equivariant linear structures
(by Lemma \ref{lem:root-linear}), also
\(\Gamma_1\) acts trivially on
\(\Ut_\alphat\), \(\Ut_{\alphat'}\), and \(\Ut_1/\Ut_2\); and
\(\Gamma_2\) acts trivially on \(\Ut_2\).
In particular, \(\Ut_2^{\Gamma_2}\) equals \(\Ut_2\), but
we will still sometimes
include the superscript \(\Gamma_2\) for emphasis.

Notice that, if we replace
\((\Gt, \Gamma)\) by
\((\Gt_1, \Gamma_2)\), then
the common restriction to
\(\Tt \cap \fix\Gt_1^{\Gamma_2}\)
of \(\alphat\) and \(\alphat'\) is
multipliable, so we may apply
Proposition \ref{prop:square-pinned}(\ref{subprop:square-pinned})
(and other results about multipliable restricted roots) to it.
Alternatively, we can observe that
\(\Ut_1/\Ut_2\) is
a subgroup of
\(\Ut_\alpha/\Ut_{2\alpha}\).

Choose a nonzero element
\(\Xt_0 \in \Lie(\Ut_1/\Ut_2)^{\Gamma_2}\).
By Proposition \ref{prop:square-pinned}(\ref{subprop:square-pinned}),
we have that
\(\Xt_0^{[2]} \in \Lie(\Ut_1)\)
is nonzero.
For each \(k\)-algebra \(A\) with
structure map \map{i_A}k A, we have the additive map
\map{\varphi_A}
	{\Lie(\Ut_2)^{\Gamma_2} \otimes_k A}
	{\Lie(\Ut_1/\Ut_2)^{\Gamma_2} \otimes_k \lsub f A}
defined as follows:
for every
\(Y \in \Lie(\Ut_2)^{\Gamma_2}\) and
\(a \in A\),
there is a unique scalar \(c \in k\) such that
\(Y\) equals \(c\Xt_0^{[2]}\), and we
put \(\varphi_A(Y \otimes a) = \Xt_0 \otimes i_A(c)a\).
(Here we have used the notation of
\cite{demazure-gabriel:groupes-algebriques}*
	{Ch.~II, \S7, 1.1},
so that \map f k k is the Frobenius automorphism
\abmapto x{x^2} and
\(\lsub f A\) denotes the restriction of scalars of
\(A\) along \(f\).)
The map \(\varphi_A\) is independent of the choice of \(\Xt_0\).
Since \(\alphat_A\) equals \(\alphat'_A\) on
\(\Tt^{\Gamma_2}(A)\), we have that
\(\varphi_A\) is \(\Tt^{\Gamma_2}(A)\)-equivariant.
If \(B\) is an \(A\)-algebra, then
\((\varphi_A)_B\) equals \(\varphi_B\).

The linear structures from Lemma \ref{lem:root-linear} on
\((\Ut_1/\Ut_2)^{\Gamma_2}\)
and
\(\Ut_2^{\Gamma_2}\)
provide, for each \(k\)-algebra \(A\), isomorphisms
\[
(\Ut_1/\Ut_2)^{\Gamma_2}(A) \cong \uLie(\Ut_1/\Ut_2)^{\Gamma_2}(A)
\quad\text{and}\quad
\Ut_2^{\Gamma_2}(\lsub f A) \cong \uLie(\Ut_2)^{\Gamma_2}(\lsub f A),
\]
which allow us to transform \(\varphi_A\) into
a \(\Tt^{\Gamma_2}(A)\)-equivariant homomorphism
\[\abmap
	{\Ut_2^{\Gamma_2}(A)}
	{(\Ut_1/\Ut_2)^{\Gamma_2}(\lsub f A) =
		((\Ut_1/\Ut_2)^{\Gamma_2})^{(2)}(A)}.
\]
This family of homomorphisms is precisely
a \(\Tt^{\Gamma_2}\)-equivariant homomorphism
\abmap
	{\Ut_2^{\Gamma_2}}
	{((\Ut_1/\Ut_2)^{\Gamma_2})^{(2)}}.
By inspection, the diagram
\[\xymatrix{
(\Ut_1/\Ut_2)^{\Gamma_2} \ar[r]\ar@(ur,ul)[rr]^{\operatorname{Frob}} & \Ut_2^{\Gamma_2} \ar[r]\ar@(ur,ul)[rr]^{\operatorname{Frob}} & ((\Ut_1/\Ut_2)^{\Gamma_2})^{(2)} \ar[r] & (\Ut_2^{\Gamma_2})^{(2)}
}\]
commutes.
In particular,
(\ref{subprop:factor-Frob-left}) holds for
\((\Gt_1, \Gamma_2)\).
Further, the isomorphism
\map\Frob
	{(\Ut_1/\Ut_2)^{\Gamma_2}(\ka)}
	{((\Ut_1/\Ut_2)^{\Gamma_2})^{(2)}(\ka)}
factors through
\map{(\cdot)^{[2]}}
	{(\Ut_1/\Ut_2)^{\Gamma_2}(\ka)}
	{\Ut_2^{\Gamma_2}(\ka)},
which is therefore injective; and the isomorphism
\map\Frob
	{\Ut_2^{\Gamma_2}(\ka)}
	{(\Ut_2^{\Gamma_2})^{(2)}(\ka)}
factors through
\map{((\cdot)^{[2]})^{(2)}}
	{((\Ut_1/\Ut_2)^{\Gamma_2})^{(2)}(\ka)}
	{(\Ut_2^{\Gamma_2})^{(2)}(\ka)},
which is therefore surjective, so that
\map{(\cdot)^{[2]}}
	{(\Ut_1/\Ut_2)^{\Gamma_2}(\ka)}
	{\Ut_2^{\Gamma_2}(\ka)}
is also surjective.
Since \(\Ut_2^{\Gamma_2}\) is smooth, it follows that
\((\cdot)^{[2]}\) is an infinitesimal isogeny.
In particular, (\ref{subprop:factor-Frob-right}) also holds for
\((\Gt_1, \Gamma_2)\).
Finally,
Proposition \ref{prop:rd-restriction}%
	(\ref{subprop:rd-multipliable})(\ref{case:rd-multipliable(red)})
and
Lemma
\ref{lem:tilde-borel:linear:prop:13.20}(\ref{sublem:prop:14.4})
give that the multiplication map
\abmap
	{\Ut_\alphat \times \Ut_{\alphat'}}
	{\Ut_{\sset{\alphat, \alphat'}}/\Ut_{\alphat + \alphat'} = \Ut_1/\Ut_2}
is an isomorphism of schemes (not of group schemes).
Thus, since \(\Gamma_2\) acts trivially on \(\Ut_2\), we have that
a functorial point of \((\Ut_1/\Ut_2)^{\Gamma_2}\)
lifts to a functorial point of \(\Ut_1^{\Gamma_2}\) if and only if
its unique lift in
\(\Ut_\alphat\cdot\Ut_{\alphat'}\) is fixed by \(\Gamma_2\).
Since \(\sset{\alphat, \alphat'}\) is an exceptional pair for
\((\Psi(\Gt_\ks, \Tt_\ks), \Gamma(k))\),
Proposition \ref{prop:rd-restriction}(\ref{subprop:rd-fiber})
gives that there is some
\(\gamma \in \Gamma_2(k)\) such that
\(\gamma\alphat\) equals \(\alphat'\).
Since \(\gamma\) does not belong to \(\Gamma_1\), which is
an index-\(2\) subgroup of \(\Gamma_2\), we have that
\(\Gamma_2\) equals \(\Gamma_1 \sqcup \gamma\Gamma_1\).
Since \(\Gamma_1\) acts trivially on
\(\Ut_\alphat\) and \(\Ut_{\alphat'}\),
we have that an element of
\(\Ut_\alphat\cdot\Ut_{\alphat'}\) is fixed by \(\Gamma_2\)
if and only if
it is fixed by \(\gamma\)
if and only if
the multiplicands commute.
By Proposition \ref{prop:square-pinned}(\ref{subprop:square-pinned}),
this is equivalent to
its lying in the kernel of \((\cdot)^{[2]}\).
This shows (\ref{subprop:no-mult-fixed}) for
\((\Gt_1, \Gamma_2)\).

We have proven the entire result for \((\Gt_1, \Gamma_2)\).
To finish, we need to use
Corollary \ref{cor:root-induced} to realize
\(\Ut_\alpha\) and \(\Ut_{2\alpha}\) as
\(\Ind_{\Gamma_1}^\Gamma \Ut_{\sset{\alphat, \alphat'}}\) and
\(\Ind_{\Gamma_1}^\Gamma \Ut_{\alphat + \alphat'}\);
Corollary \ref{cor:ksep-ind-scheme} to see that
exactness is preserved by induction, and so to identify
\(\Ut_\alpha/\Ut_{2\alpha}\) with an induced group;
Lemma \ref{lem:was-cor:trans-induction} to identify
\(\Gamma_1\)-fixed points with
\(\Gamma\)-fixed points; and
Lemma \ref{lem:ind-Frob}
to handle Frobenius twists.
\end{proof}

\begin{cor}
\label{cor:no-mult-fixed}
Preserve the notation and hypotheses of
Proposition \ref{prop:no-mult-fixed}.
Then \((\Ut_a^\Gamma)\smooth\) equals \(\Ut_{2a}^\Gamma\).
\end{cor}

\begin{proof}
Lemma \ref{lem:root-smooth} gives that
\((\Ut_a^\Gamma)\smooth\) contains \(\Ut_{2a}^\Gamma\).
Proposition \ref{prop:no-mult-fixed}%
	(\ref{subprop:factor-Frob-left},%
	\ref{subprop:no-mult-fixed})
gives that
\[\xymatrix{
\Ut_a^\Gamma(\ks) \ar[r] &
(\Ut_a/\Ut_{2a})^\Gamma(\ks) \ar[r]^-\Frob &
((\Ut_a/\Ut_{2a})^\Gamma)^{(2)}(\ks)
}\]
is trivial, hence, since \(\Frob\) is injective on \(\ks\)-points, that
\abmap{\Ut_a^\Gamma(\ks)}{(\Ut_a/\Ut_{2a})^\Gamma(\ks)} is trivial.
Since \((\Ut_a^\Gamma)_{\smoothsub\,\ks}\) is
the Zariski closure of \(\Ut_a^\Gamma(\ks)\), we have that
\((\Ut_a^\Gamma)_{\smoothsub\,\ks}\) is contained in
\((\Ut_{2a}^\Gamma)_\ks\), hence that
\((\Ut_a^\Gamma)\smooth\) is contained in \(\Ut_{2a}^\Gamma\).
\end{proof}

\begin{cor}
\label{cor:root-smoothable}
For every \(a \in \Phi(\Gt, S)\), we have that
\(\Ut_a^\Gamma\) is smoothable and connected.
\end{cor}

\begin{proof}
Lemma \ref{lem:tilde-borel:linear:prop:13.20}(\ref{sublem:prop:14.4})
shows that \(\Ut_{a\,\ks}\) is directly spanned by
subgroups \(\Ut_\alpha\) as \(\alpha\) ranges over
the `extensions' of \(a\) in \(\Phi(\Gt_\ks, T_\ks)\),
so we may, and do, assume,
upon replacing \(k\) by \(\ks\), hence \(S\) by \(T_\ks\), and
\(a\) by an `extension' \(\alpha\), that
\(k\) is separably closed.

If \(\alpha\) is not multipliable in \(\Phi(\Gt, T)\) or
\(p\) is odd, then
Lemma \ref{lem:root-smooth}
shows that \(\Ut_\alpha^\Gamma\) is smooth and connected.
Otherwise,
Corollary \ref{cor:no-mult-fixed} gives that
\((\Ut_\alpha^\Gamma)\smooth\) equals \(\Ut_{2\alpha}^\Gamma\) and
\(((\Ut_\alpha^\Gamma)_\ka)\smooth =
(\Ut_{\alpha_\ka}^{\Gamma_\ka})\smooth\) equals
\(\Ut_{2\alpha_\ka}^{\Gamma_\ka} = (\Ut_{2\alpha}^\Gamma)_\ka\), so that
\(\Ut_\alpha^\Gamma\) is smoothable.
Since the maximal reduced subscheme
\(((\Ut_\alpha^\Gamma)_\ka)\smooth = (\Ut_{2\alpha}^\Gamma)_\ka\)
of \((\Ut_\alpha^\Gamma)_\ka\) is connected
(by Lemma \ref{lem:root-smooth}), also
\((\Ut_\alpha^\Gamma)_\ka\), hence \(\Ut_\alpha^\Gamma\), is
connected.
\end{proof}

{\newcommand\theadhocthm{\ref{thm:quass}(\ref{subthm:quass-smoothable})}
\begin{adhocthm}
\((\Gt^\Gamma)\conn\) is smoothable.
\end{adhocthm}}

\begin{proof}
Let \(\Bt\) be a Borel subgroup of \(\Gt\) that
contains \(S\) and
is preserved by \(\Gamma\).
Lemma \ref{lem:tilde-borel:linear:prop:13.20}(\ref{sublem:prop:14.4})
gives that
the groups \(\Ut_a\) as \(a\) ranges over \(\Phi(\Bt, S)\)
directly span the unipotent radical of \(\Bt\), and
the groups \(\Ut_{-a}\) as \(a\) ranges over \(\Phi(\Bt, S)\)
directly span the unipotent radical of the Borel subgroup
\(\Bt^-\) opposite to \(\Bt\) with respect to \(\Tt\).
Therefore,
the multiplication map
\abmap
	{\prod_{a \in \Phi(\Bt, S)} \Ut_a \times
	\Tt \times
	\prod_{a \in \Phi(\Bt, S)} \Ut_{-a}}
	{\Gt}
is a \(\Gamma\)-equivariant isomorphism of schemes onto
the open subscheme \(\Bt\cdot\Bt^-\) of \(\Gt\), so
the multiplication map
\abmap
	{\prod_{a \in \Phi(\Bt, S)} \Ut_a^\Gamma \times
	\Tt^\Gamma \times
	\prod_{a \in \Phi(\Bt, S)} \Ut_{-a}^\Gamma}
	{\Gt^\Gamma}
is an isomorphism of schemes onto
an open subscheme of \(\Gt^\Gamma\).
Therefore,
the multiplication map
\abmap
	{\prod_{a \in \Phi(\Bt, S)} (\Ut_a^\Gamma)\conn \times
	(\Tt^\Gamma)\conn \times
	\prod_{a \in \Phi(\Bt, S)} (\Ut_{-a}^\Gamma)\conn}
	{(\Gt^\Gamma)\conn}
is an isomorphism onto
an open subscheme \(V\) of \((\Gt^\Gamma)\conn\).
Since subgroups of tori are always smoothable
\cite{conrad-gabber-prasad:prg}*{Corollary A.8.2},
we have by Corollary \ref{cor:root-smoothable} that
there is a smooth subscheme \(V'\) of \(V\) such that
\(V'_\ka\) is the maximal reduced subscheme of \(V_\ka\).
In particular, since
\(\fix\Gt^\Gamma\) contains \(V'\), we have that
\((\fix\Gt^\Gamma)_\ka\) is
a closed subscheme of
the irreducible scheme \(\fix\Gt_\ka^{\Gamma_\ka}\) that
contains a nonempty open subset \(V'_\ka\) of \(\fix\Gt_\ka^{\Gamma_\ka}\),
hence is all of \(\fix\Gt_\ka^{\Gamma_\ka} = ((\Gt^{\Gamma})_\ka)\smooth\conn\).
That is, \((\Gt^\Gamma)\conn\) is smoothable by Remark~\ref{rem:conn-smooth}.
\end{proof}

{\newcommand\theadhocthm{\ref{thm:quass}(\ref{subthm:quass-reductive})}
\begin{adhocthm}
$G$ is reductive.
\end{adhocthm}}

\begin{proof}
Note that \((\Gt_\ka, \Gamma_\ka)\) is quasisemisimple.

Theorem \ref{thm:quass}(\ref{subthm:quass-smoothable})
and
Remark \ref{rem:conn-smooth}
give that
\(G_\ka = (\fix\Gt^\Gamma)_\ka\) equals
\(\fix\Gt_\ka^{\Gamma_\ka}\), so we may, and do, assume,
upon replacing \(k\) by \(\ka\), that
\(k\) is algebraically closed.

Let \((\Bt, \Tt)\) be a Borel--torus pair in \(\Gt\) that is
preserved by \(\Gamma\).
We have by
Remark \ref{rem:torus-quass}(\ref{subrem:quass-to-torus}) that
the action of \(\Gamma\conn\) on \(\Gt\) factors through
\abmap{\Tt/Z(\Gt)}{\uAut(\Gt)} to give a map
\abmap\Gamma{\Tt/Z(\Gt)}.
Then the image of \(\Gamma\conn\) in \(\Tt/Z(\Gt)\) is
a smooth, connected subgroup of a torus, hence is itself a torus;
so
\cite{borel:linear}*{\S13.17, Corollary 2(a)} gives that
\(\Gt^{\Gamma\conn}\) is reductive, and
\cite{borel:linear}*{Proposition 11.15} gives that
\(\Bt^{\Gamma\conn}\) is a Borel subgroup of \(\Gt^{\Gamma\conn}\).
In particular, \((\Bt^{\Gamma\conn}, \Tt)\) is
a Borel--torus pair in \(\Gt^{\Gamma\conn}\) that is
preserved by \(\pi_0(\Gamma)(k)\).
Thus, applying \cite{adler-lansky:lifting1}*{Proposition 3.5(i,ii)} to
the action of the abstract, finite group \(\pi_0(\Gamma)(k)\) on
\(\Gt^{\Gamma\conn}\) gives that
\(G = \fix(\Gt^{\Gamma\conn})^{\pi_0(\Gamma)(k)}\) is reductive.
\end{proof}

\numberwithin{equation}{section}
\section{Fixed points for quasisemisimple actions}
\label{sec:quass}

We continue to work with
the field \(k\) of characteristic exponent \(p\), and
reductive datum \((\Gt, \Gamma)\) over \(k\), from
\S\ref{sec:quass-smooth}, and again put
\(G = \fix\Gt^\Gamma\).

Throughout this section, we assume that
\((\Gt, \Gamma)\) is quasisemisimple.

Proposition \ref{prop:quass-facts}
is very close to \cite{adler-lansky:lifting1}*{Proposition 3.5}
and \cite{digne-michel:non-connected}*{Th\'eor\`eme 1.8},
but stated in a way that is more convenient for our purposes.
In particular, it takes into account questions about fields of definition.
It also allows us to translate Proposition \ref{prop:rd-restriction}
to the language of reductive groups.

\begin{prop}
\label{prop:quass-facts}\hfill
\begin{enumerate}[label=(\alph*), ref=\alph*]
\item\label{subprop:quass-split-descends}
\(G\) is quasisplit, and
split if \(\Gt\) is split.
\item\label{subprop:fixed-Borus}
If \((\Bt, \Tt)\) is a \(\Gamma\)-stable Borel--torus pair in \(\Gt\),
then
\((\Bt^\Gamma)\conn\) and \((\Tt^\Gamma)\conn\) are smoothable,
and
\((B, T) \ldef (\fix\Bt^\Gamma, \fix\Tt^\Gamma)\)
equals \((\Bt \cap G, \Tt \cap G)\)
and is a Borel--torus pair in \(G\).
The map \mapto{\overline\pi}{(\Bt, \Tt)}{(B, T)}
is a surjection from the set of \(\Gamma\)-stable Borel--torus pairs in \(\Gt\)
onto the set of Borel--torus pairs in \(G\).
\end{enumerate}
\end{prop}


\begin{proof}
Let \((\Bt, \Tt)\) be a \(\Gamma\)-stable Borel--torus pair in \(\Gt\), and
put \((B, T) =  (\fix\Bt^\Gamma, \fix\Tt^\Gamma)\).
As in the proof of Theorem \ref{thm:quass}(\ref{subthm:quass-smoothable}),
we have by
\cite{conrad-gabber-prasad:prg}*{Corollary A.8.2}
that \((\Tt^\Gamma)\conn\) is smoothable; so
we have by Remark \ref{rem:conn-smooth} that
\(T_\ka = (\fix\Tt^\Gamma)_\ka\) equals
\(\fix\Tt_\ka^{\Gamma_\ka}\).
Proposition \ref{prop:quass-rough}(\ref{subprop:quass-down})
shows that \(\Tt \cap G\) equals \(T\) and
\(C_G(S) = C_\Gt(S) \cap G\) equals \(\Tt \cap G = T\),
where \(S\) is the maximal split torus in \(G\), so that
\(T\) is a maximal torus in \(G\).

Let \(\delta\) be the cocharacter of \(T\) constructed in
Lemma \ref{lem:cochar-by-torus},
so that \(\Bt\) equals \(P_\Gt(\delta)\).
Then \(\Bt \cap G\) equals
\(P_\Gt(\delta) \cap G = P_G(\delta)\)
\cite{conrad-gabber-prasad:prg}*{p.~49}.
Since \(G\) is reductive
(by Theorem \ref{thm:quass}(\ref{subthm:quass-reductive})),
we have by
\cite{springer:lag}*{Proposition 8.4.5} that
\(\Bt \cap G\) is
a parabolic subgroup of \(G\); but it is also
solvable (because it is a subgroup of \(\Bt\)), hence is
a Borel subgroup of \(G\).
In particular, \(\Bt \cap G\) is
a smooth, connected subgroup of
\(\Bt \cap \Gt^\Gamma = \Bt^\Gamma\), and so of
\(\fix\Bt^\Gamma\).
Since the reverse containment is obvious,
 \(\Bt \cap G\) equals \(\fix\Bt^\Gamma\).
Since Remark \ref{rem:conn-smooth} gives that
\(G_\ka = (\fix\Gt^\Gamma)_\ka\)
equals
\(\fix\Gt_\ka^{\Gamma_\ka}\),
the same argument that showed that
\(\Bt \cap G\) equals \(\fix\Bt^\Gamma\) shows that
\((\fix\Bt^\Gamma)_\ka = (\Bt \cap G)_\ka = \Bt_\ka \cap G_\ka\)
equals
\(\fix\Bt_\ka^{\Gamma_\ka}\); so
another application of
Remark \ref{rem:conn-smooth} gives that
\((\Bt^\Gamma)\conn\) is smoothable.

Since \((B, T)\) is a Borel--torus pair in \(G\),
in particular \(G\) is quasisplit.
Further, if \(\Gt\) is split, then so is \(\Tt\)
(because it is a maximal torus in a Borel subgroup of \(\Gt\)), so
the maximal torus \(T\) in \(G\) is split, so
\(G\) is split.
This shows (\ref{subprop:quass-split-descends}) and
part of (\ref{subprop:fixed-Borus}).
Note that
\(G(k)\) acts on
the set of \(\Gamma\)-stable Borel--torus pairs in \(\Gt\),
\(\overline\pi\) is \(G(k)\)-equivariant, and
\(G(k)\) acts transitively on
the set of Borel--torus pairs in \(G\)
(the rational conjugacy of Borel subgroups is
\cite{conrad-gabber-prasad:prg}*{Theorem C.2.5}, and then
the rational conjugacy, in that Borel subgroup, of
maximal tori in a Borel subgroup is
\cite{conrad-gabber-prasad:prg}*{Theorem C.2.3}).
This shows that \(\overline\pi\) is surjective, and so completes
the proof of (\ref{subprop:fixed-Borus}).
\end{proof}

\begin{cor}
\label{cor:inherit-quass}
If \(\Gamma'\) is a smooth, normal subgroup of \(\Gamma\), then
\((\fix\Gt^{\Gamma'}, \Gamma)\) is quasisemisimple.
\end{cor}


\begin{lem}
\label{lem:H1-inject}
If \((\Bt_0, \Tt_0)\) is a Borel--torus pair in \(\Gt\) that
is preserved by \(\Gamma\), then
\abmap
	{\smashset{\gt \in \Gt(k)}{\gt\Tt_0 \in (\Gt/\Tt_0)^\Gamma(k)}}
	{(\Gt/\Bt_0)^\Gamma(k)}
is surjective.
If the map
\abmap{N_\Gt(\Tt_0)^\Gamma(k)}{W(\Gt, \Tt_0)^\Gamma(k)}
is surjective, then even
\abmap{\Gt^\Gamma(k)}{(\Gt/\Bt_0)^\Gamma(k)}
is surjective.
\end{lem}

\begin{proof}
Let \(\Bt_1\) be the Borel subgroup opposite to \(\Bt_0\)
with respect to \(\Tt_0\), and \(\Ut_1\) its unipotent radical.
Then \(\Bt_{1\,\ks}\) is the Borel subgroup opposite to
\(\Bt_{0\,\ks}\) with respect to \(\Tt_{0\,\ks}\).
This characterizes \(\Bt_{1\,\ks}\) uniquely, so
\(\Bt_{1\,\ks}\), and hence
its unipotent radical \(\Ut_{1\,\ks}\), is
preserved by \(\Gamma(\ks)\).
Since \(\Gamma\) is smooth, we have that
\(\Gamma(\ks)\) is Zariski dense in \(\Gamma_\ks\), so
\(\Ut_{1\,\ks}\) is preserved by \(\Gamma_\ks\), and hence
\(\Ut_1\) is preserved by \(\Gamma\).
Similarly, the unipotent radical \(\Ut_0\) of \(\Bt_0\) itself is
preserved by \(\Gamma\).

By \cite{springer:lag}*{Corollary 15.1.4},
there exists an element \(\gt \in \Gt(k)\) whose image in
\((\Gt/\Bt_0)(k)\)
lies in
\((\Gt/\Bt_0)^\Gamma(k)\).
The double coset \(\Ut_{0\,\ks}\gt\Bt_{0\,\ks}\)
is preserved by \(\Gal(k) \ltimes \Gamma(\ks)\), and
equals \(\Ut_{0\,\ks}w\Bt_{0\,\ks}\)
for a unique element \(w\) of
\(W(\Gt, \Tt_0)(\ks)\).
Uniqueness implies that
\(w\) belongs to
\(W(\Gt, \Tt_0)(\ks)^{\Gal(k) \ltimes \Gamma(\ks)} =
W(\Gt, \Tt_0)_\ks^{\Gamma_\ks}(\ks)^{\Gal(k)} =
W(\Gt, \Tt_0)^\Gamma(k)\).
Note that the maximal split torus \(\St_0\) in \(\Tt_0\) is
maximal split in \(\Gt\)
(because \(\Tt_0\) is contained in
a Borel subgroup of \(\Gt\)).
Lemma \ref{lem:which-Weyl} gives that
\(w\) belongs to \(W(\Gt, \St_0)(k)\), and
\cite{conrad-gabber-prasad:prg}*{Proposition C.2.10}
gives that \(w\) has a representative \(n\) in
\(N_\Gt(\St_0)(k)\), which equals
\(N_\Gt(\Tt_0)(k)\) because
\(\Tt_0\) equals \(C_\Gt(\St_0)\).
If \abmap{N_\Gt(\Tt_0)^\Gamma(k)}{W(\Gt, \Tt_0)^\Gamma(k)}
is surjective, then, of course, we can choose
\(n \in N_\Gt(\Tt_0)^\Gamma(k)\).

Put \(\Ut_{0\,w} = \Ut_0 \cap \Int(w)\Ut_1\).
Then \(\Ut_{0\,w}\) is preserved by \(\Gamma\), and
the restriction to \(\Ut_{0\,w}n\) of the quotient map
\abmap{\Ut_0 w\Bt_0}{\Ut_0 w\Bt_0/\Bt_0}
is an isomorphism of schemes
\cite{milne:algebraic-groups}*{Theorem 21.80(b)}.
If \(n\) belongs to \(N_\Gt(\Tt_0)^\Gamma(k)\), then
the isomorphism is \(\Gamma\)-equivariant.
Otherwise, it only becomes \(\Gamma\)-equivariant after
factoring through the projection to \(\Gt/\Tt_0\).
Applying the inverse of this isomorphism to \(\gt\Bt_0\) yields
an element of \((\Gt/\Tt_0)^\Gamma(k)\) in general, and even
an element of \(\Gt^\Gamma(k)\) if
\(n\) belongs to \(N_\Gt(\Tt_0)^\Gamma(k)\).
\end{proof}

\begin{cor}
\label{cor:H1-inject}
If \((\Bt_0, \Tt_0)\) is a Borel--torus pair in \(\Gt\) that
is preserved by \(\Gamma\), and
\(\Ut_0\) is the unipotent radical of \(\Bt_0\), then
\abmap
	{\Gt^\Gamma(k)}
	{(\Gt/\Ut_0)^\Gamma(k)}
is surjective.
\end{cor}

\begin{proof}
Since the multiplication map
\abmap{\Tt_0 \ltimes \Ut_0}{\Bt_0} is an isomorphism,
it follows from Lemma \ref{lem:H1-inject} that,
given a coset in \((\Gt/\Ut_0)^\Gamma(k)\),
we may choose a representative
\(\gt \in \Gt(k)\) such that
\(\gt\Tt_0\) belongs to \((\Gt/\Tt_0)^\Gamma(k)\).
Then, for every \(\gamma \in \Gamma(\ks)\), we have that
\(\gt\inv\gamma(\gt)\) belongs to
\(\Ut_0(\ks) \cap \Tt_0(\ks)\), which is trivial.
That is, as an element of \(\Gt(\ks)\), we have that
\(\gt\) is fixed by \(\Gamma(\ks)\), hence,
since \(\Gamma\) is smooth, by \(\Gamma_\ks\).
Thus \(\gt\) belongs to
\(\Gt_\ks^{\Gamma_\ks}(\ks) = \Gt^\Gamma(\ks)\).
Since \(\gt\) already belongs to \(\Gt(k)\), it
belongs to \(\Gt^\Gamma(k)\).
\end{proof}

The statement about the existence of
\(\Gamma\)-stable Levi components of
\(\Gamma\)-stable parabolics in
Proposition \ref{prop:parabolic-Borus}(\ref{subprop:quass-fixed-Levi})
is
closely related to complete reducibility, in the sense of
\cite{serre:cr}*{\S3.2.1}.
In particular, it says that every
quasisemisimple action is completely reducible.

An easy variant of
Lemma \ref{lem:cochar-by-torus} shows that,
in the notation of Proposition \ref{prop:parabolic-Borus},
\(\Pt \cap G\) is a parabolic subgroup of \(G\), and
a bit more work shows that
\(\Ut^\Gamma(k)\) is the group of \(k\)-rational points of
the unipotent radical of \(\Pt \cap G\); but
we do not need this.

\begin{prop}
\label{prop:parabolic-Borus}
Let \(\Pt\) be a parabolic subgroup of \(\Gt\)
that is preserved by \(\Gamma\).
\begin{enumerate}[label=(\alph*), ref=\alph*]
\item\label{subprop:parabolic-Borus}
There is a Borel--torus pair in \(\Pt\) that is preserved by \(\Gamma\).
\item\label{subprop:quass-fixed-Levi}
For all Levi components \(\Mt\) of \(\Pt\) that
are preserved by \(\Gamma\),
we have that
\((\Mt, \Gamma)\) is quasisemisimple.
Such Levi components exist, and they are all
\(\Ut^\Gamma(k)\)-conjugate, where
\(\Ut\) is the unipotent radical of \(\Pt\).
\end{enumerate}
\end{prop}

\begin{proof}
As observed in the proof of Corollary \ref{cor:H1-inject},
since \(\Gamma\) is smooth,
a point of \(\Gt(\ks)\) that is fixed by \(\Gamma(\ks)\)
belongs to
\(\Gt^\Gamma(\ks) = (\Gt^\Gamma)\smooth(\ks)\).
Similarly,
a closed subscheme \(\Xt\) of \(\Gt\) such that
\(\Xt_\ks\) is preserved by \(\Gamma(\ks)\)
actually has the property that
\(\Xt_\ks\) is preserved by \(\Gamma_\ks\), and so
\(\Xt\) is preserved by \(\Gamma\).

Let \((\Bt_0, \Tt_0)\) be a Borel--torus pair in \(\Gt\)
that is preserved by \(\Gamma\).
There is
a minimal \(\Gamma\)-stable, parabolic subgroup \(\Bt\) of \(\Gt\)
that is contained in \(\Pt\)
(since \(\Pt\) is Noetherian).
If we write \(\Ut\) for the unipotent radical of \(\Bt\), then
\(\Ut_\ks\) is the unipotent radical of \(\Bt_\ks\), hence
preserved by \(\Gamma(\ks)\), so
\(\Ut\) is preserved by \(\Gamma\).
Thus \((\Bt \cap \Bt_0)\Ut\) is
a parabolic subgroup of \(\Gt\)
\cite{borel-tits:reductive-groups}*{Proposition 4.4(b)}
that is contained in \(\Bt\) and preserved by \(\Gamma\),
hence equals \(\Bt\).
If we write \(\Ut_0\) for the unipotent radical of \(\Bt_0\), then
\((\Bt \cap \Ut_0)\Ut\) is a normal, unipotent subgroup of
\((\Bt \cap \Bt_0)\Ut = \Bt\), and
\(\Bt/((\Bt \cap \Ut_0)\Ut) =
((\Bt \cap \Bt_0)\Ut)/((\Bt \cap \Ut_0)\Ut) \cong
(\Bt \cap \Bt_0)/(\Bt \cap \Ut_0)\)
embeds into
\(\Bt_0/\Ut_0 \cong \Tt_0\), hence is
of multiplicative type.
Thus \(\Bt_\ks\) is trigonalizable
\cite{milne:algebraic-groups}*{Theorem 16.6},
hence solvable
\cite{milne:algebraic-groups}*{Theorem 16.21},
so that \(\Bt\) is a Borel subgroup of \(\Gt\).

There is a unique element \(\gt\Bt_{0\,\ks} \in (\Gt/\Bt_0)(\ks)\)
such that
\(\Int(\gt)\Bt_{0\,\ks}\) equals \(\Bt_\ks\).
By uniqueness, it belongs to
\((\Gt/\Bt_0)(\ks)^{\Gal(k) \ltimes \Gamma(\ks)} =
(\Gt/\Bt_0)^\Gamma(k)\).
By Lemma \ref{lem:H1-inject}, we may adjust the choice of
representative of the coset to
an element \(\gt\) of \(\Gt(k)\) such that
\(\gt\Tt_0\) is \(\Gamma\)-fixed.
Then
\(\Tt \ldef \Int(\gt)\Tt_0\)
is a maximal torus in \(\Gt\) that is preserved by \(\Gamma\) and
contained in
\(\Int(\gt)\Bt_0 = \Bt\).
This shows (\ref{subprop:parabolic-Borus}).


The Levi component \(\Mt\) of \(\Pt\) that
contains \(\Tt\) has the property that
\(\Mt_\ks\) is the (unique) Levi component of \(\Pt_\ks\) that
contains \(\Tt_\ks\), hence is preserved by \(\Gamma(\ks)\); so
\(\Mt\) is preserved by \(\Gamma\).
This shows, in particular, that
such Levi components of \(\Pt\)
exist.

Since \((\Bt \cap \Mt, \Tt)\) is a Borel--torus pair in \(\Mt\)
that is preserved by \(\Gamma\),
we have shown that \((\Mt, \Gamma)\) is quasisemisimple,
but this only handles the particular choice of
Levi component arising as above.
If \(\Mt_1\) is another such Levi component of \(\Pt\), then
\(\Mt_\ks\) and \(\Mt_{1\,\ks}\) are
Levi components of \(\Pt_\ks\) that are
preserved by \(\Gamma_\ks\).
By
\cite{borel:linear}*{Proposition 11.23(ii)},
there is a unique \(\ks\)-rational point \(u\) in
the unipotent radical of \(\Pt_\ks\), which equals \(\Ut_\ks\),
such that
\(\Int(u)\Mt_\ks\) equals \(\Mt_{1\,\ks}\).
Since \(\Mt_\ks\) and \(\Mt_{1\,\ks}\) are
preserved by \(\Gal(k)\),
we have that \(u\) is fixed by \(\Gal(k)\), i.e.,
belongs to \(\Ut(k)\).
Since \(\Mt_\ks\) and \(\Mt_{1\,\ks}\) are
preserved by \(\Gamma(\ks)\),
so is \(u\)
(viewed as a point of \(\Ut(\ks)\))
so that \(u\in\Ut^\Gamma(\ks)\cap\Ut(k) = \Ut^\Gamma(k)\).
Thus the quasisemisimplicity of \((\Mt_1, \Gamma)\) is witnessed by
the Borel--torus pair \(\Int(u)(\Bt \cap \Mt, \Tt)\).
\end{proof}

{\newcommand\theadhocthm{\ref{thm:quass}(\ref{subthm:quass-spherical-bldg})}
\begin{adhocthm}
The functorial map from the spherical building \(\SS(G)\) of \(G\)
to the spherical building \(\SS(\Gt)\) of \(\Gt\)
identifies \(\SS(G)\) with
\(\SS(\Gt) \cap \SS(\Gt_\ka)^{\Gamma(\ka)}\).
\end{adhocthm}}

\begin{proof}
Since
\(\Gamma(\ka)\) acts trivially on \(\SS(G_\ka)\), we have by
functoriality that the image of the composition
\(\xymatrix@1{
\SS(G) \ar[r] & \SS(G_\ka) \ar[r] & \SS(\Gt_\ka)
}\)
lies in \(\SS(\Gt_\ka)^{\Gamma(\ka)}\).
Since the diagram
\[\xymatrix{
\SS(G_\ka) \ar[r] & \SS(\Gt_\ka) \\
\SS(G) \ar[r]\ar[u] & \SS(\Gt) \ar[u]
}\]
commutes, the image of
\(\SS(G)\) in \(\SS(\Gt_\ka)\) actually lies in
\(\SS(\Gt) \cap \SS(\Gt_\ka)^{\Gamma(\ka)}\).

Conversely, suppose that \(b_+\) belongs to
\(\SS(\Gt) \cap \SS(\Gt_\ka)^{\Gamma(\ka)}\).
Our argument is similar to that of
Lemma \ref{lem:spherical-descent}.
With
\(\Pt^+\) the parabolic subgroup \(P_\Gt(b_+)\) of \(\Gt\),
we have that
\(\gamma\Pt^+_\ka = \gamma P_\Gt(b_+)_\ka\)
equals
\(P_{\Gt_\ka}(\gamma b_{{+}\,\ka}) =
P_{\Gt_\ka}(b_{{+}\,\ka}) = \Pt^+_\ka\)
for all \(\gamma \in \Gamma(\ka)\).
In particular, \(\Pt^+_\ks\) is preserved by \(\Gamma(\ks)\), hence
by \(\Gamma_\ks\).
Proposition \ref{prop:parabolic-Borus}(\ref{subprop:quass-fixed-Levi})
gives that there is a Levi component \(\Mt\) of \(\Pt^+_\ks\) that is
preserved by \(\Gamma_\ks\).
Let \(\Pt^-\) be the parabolic subgroup of \(\Gt_\ks\) that is
opposite to \(\Pt^+_\ks\), and satisfies
\(\Pt^+_\ks \cap \Pt^- = \Mt\); and then let
\(b_-\) be the point of \(\SS(\Gt_\ks)\) that is
opposite to \(b_+\) and
satisfies \(P_{\Gt_\ks}(b_-) = \Pt^-\).
This condition determines \(b_-\) uniquely, so that it is
preserved by \(\Gamma(\ks)\), and hence \(\Gamma_\ks\).
Since \(\fix\Gt_\ks^{\Gamma(\ks)}\) equals
\(\fix\Gt_\ks^{\Gamma_\ks} = (\fix\Gt^\Gamma)_\ks = G_\ks\),
Lemma \ref{lem:spherical-cr} gives that
\(b_{{+}\,\ks}\) belongs to
\(\SS(\fix\Gt_\ks^{\Gamma(\ks)}) = \SS(G_\ks)\).
Then two applications of Lemma \ref{lem:spherical-descent} give
first that
\(b_{{+}\,\ks}\) is fixed by \(\Gal(k)\)
(by regarding it as an element of \(\SS(\Gt)\)), and
then that
\(b_+\) belongs to \(\SS(G)\).
\end{proof}

%

As in \S\ref{sec:quass-smooth},
for the remainder of \S\ref{sec:quass}, fix
a maximal split torus \(S\) in \(G\), and let
\(T\) and \(\Tt\) be the maximal split tori in \(G\) and \(\Gt\)
containing \(S\),
and \(\St\) the maximal split torus in \(\Tt\).

See \cite{adler-lansky-spice:actions3}*{Corollary \ref{GA3-cor:pinned-roots}}
for a sharper version of
Proposition \ref{prop:GS-vs-GtS} when
\((\Gt, \Gamma)\) is pinned.

\begin{prop}
\label{prop:GS-vs-GtS}\hfill
\begin{enumerate}[label=(\alph*), ref=\alph*]
\item\label{subprop:g-root-space}
For every \(a \in \Phi(G, S)\), we have that
\(\Lie(G)_a\) equals \(\Lie(\Gt)^\Gamma_a\).
\item\label{subprop:G-roots}
\(\Phi(\Gt^\Gamma, S)\)
is a sub-root system of \(\Phi(\Gt, S)\) that
contains \(\Phi(G, S)\).
If \(p\) is odd, then
\(\Phi(G, S)\) equals \(\Phi(\Gt^\Gamma, S)\).
If \(p\) equals \(2\), then
\(\Phi(G, S)\) is the set of
roots in \(\Phi(\Gt^\Gamma, S)\) that
are either non-multipliable (in \(\Phi(\Gt^\Gamma, S)\)) or
split for \((\Psi(\Gt_\ks, T_\ks), \Gal(k))\).
\end{enumerate}
\end{prop}

\begin{note}
Recall that
\(\Phi(\Gt^\Gamma, S)\)
means
\(\Phi(\Lie(\Gt^\Gamma), S) =
\smashset[\big]{a \in \Phi(\Gt, S)}{\Lie(\Gt^\Gamma)_a \ne \sset0}\).
Since \(\Lie(\Gt^\Gamma)_a\) equals \(\Lie(\Gt)^\Gamma_a\),
we may also describe \(\Phi(\Gt^\Gamma, S)\) as
the set of weights of \(S\) on \(\Lie(\Gt)\) such that
the corresponding weight space admits
nonzero \(\Gamma\)-fixed vectors.

Proposition \ref{prop:GS-vs-GtS}(\ref{subprop:g-root-space})
can be viewed
as saying that, if smoothing does not totally eliminate the \(a\)-weight space, then it leaves it unchanged.
\end{note}

\begin{proof}
%
For (\ref{subprop:g-root-space}), note that the set
\(\Phi(\Lie(G)_a \otimes_k \ks, T_\ks)\)
of weights of \(T_\ks\) on \(\Lie(G)_a \otimes_k \ks\) that
`extend' \(a\) is preserved by \(\Gal(k)\); so, by
Proposition \ref{prop:rd-restriction}(\ref{subprop:rd-fiber}),
we have that \(\Phi(\Lie(G)_a \otimes_k \ks, T_\ks)\) contains
all `extensions' of \(a\) to \(T_\ks\) in
\(\Phi(\Gt_\ks, T_\ks)\).
That is,
\(\Phi(\Lie(G)_a \otimes_k \ks, T_\ks)\) contains
\(\Phi(\Lie(\Gt)_a \otimes_k \ks, T_\ks)\).
Since the reverse containment is obvious, we have equality.
Since
\(\Lie(G)_a \otimes_k \ks\) equals
\(\bigoplus \Lie(G_\ks)_\alpha\) and
\(\Lie(\Gt)_a^\Gamma \otimes_k \ks\) equals
\(\bigoplus \Lie(\Gt_\ks)^{\Gamma_\ks}_\alpha\),
the sums over all
\(\alpha\) in
\(\Phi(\Lie(G)_a \otimes_k \ks, T_\ks) =
\Phi(\Lie(\Gt)_a \otimes_k \ks, T_\ks)\),
for the purposes of proving (\ref{subprop:g-root-space}), we may, and do, assume
upon replacing \(k\) by \(\ks\), hence \(S\) by \(T\), and
\(a\) by an `extension' \(\alpha\), that
\(k\) is separably closed.
Then Corollary \ref{cor:root-space-facts}(\ref{subcor:root-space-dim})
gives that \(\Lie(\Gt)_\alpha^\Gamma\) is one-dimensional.
Since \(\Lie(G)_a\) is certainly one-dimensional and contained in \(\Lie(\Gt)_a^\Gamma\),
we have shown
(\ref{subprop:g-root-space}).

We now turn to (\ref{subprop:G-roots}).
This has two claims: that
\(\Phi(\Gt^\Gamma, S)\) is a root system in
the subspace of \(V(S)\) that it spans, and that
\(\Phi(G, S)\) is a certan subset of \(\Phi(\Gt^\Gamma, S)\).
We prove the latter claim first.

Since \(C_\Gt(S)\) equals \(\Tt\) by
Proposition \ref{prop:quass-rough}(\ref{subprop:quass-up}),
hence also equals \(C_\Gt(T)\),
we have that
\(\Phi(\Gt^\Gamma, S)\), respectively
\(\Phi(G, S)\), is the set of `restrictions' of elements of
\(\Phi((\Gt^\Gamma)_\ks, T_\ks)\), respectively
\(\Phi(G_\ks, T_\ks)\).
We now apply Lemma \ref{lem:root-smooth}
in the case where \(k\) is \(k\sep\) and \(S\) is \(T_\ks\).
If \(p\) is odd, we have that
\(\Phi(G_\ks, T_\ks)\) equals
\(\Phi((\Gt^\Gamma)_\ks, T_\ks)\), hence that
\(\Phi(G, S)\) equals
\(\Phi(\Gt^\Gamma, S)\).
If \(p\) equals \(2\), then
we have that
\(\Phi(G_\ks, T_\ks)\) contains at least the roots in
\(\Phi((\Gt^\Gamma)_\ks, T_\ks)\) that
are not multipliable in \(\Phi(\Gt_\ks, T_\ks)\);
i.e.,
by Proposition \ref{prop:square-pinned}(\ref{subprop:square-pinned}),
the non-multipliable elements of
\(\Phi((\Gt^\Gamma)_\ks, T_\ks)\).
Since \(\Phi(G_\ks, T_\ks)\) is reduced, it is
precisely the set of such roots.
The set of `restrictions' of such roots
to $S$ certainly contains all roots in
\(\Phi(\Gt^\Gamma, S)\) that are not
multipliable in \(\Phi(\Gt, S)\).
On the other hand, a root \(a\) in \(\Phi(\Gt^\Gamma, S)\) that
\emph{is} multipliable in \(\Phi(\Gt, S)\) is
the restriction of a root in \(\Phi((\Gt^\Gamma)_\ks, T_\ks)\)
that is \emph{not} multipliable in \(\Phi(\Gt_\ks, T_\ks)\)
if and only if \(a\) is split. This proves the second part of (\ref{subprop:G-roots}).

To show that \(\Phi(\Gt^\Gamma, S)\) is a root system,
we must show that, for every
\(a \in \Phi(\Gt^\Gamma, S)\), there is
a cocharacter \(a^\vee\) of \(S\) such that
\(\pair{a^\vee}a\) equals \(2\) and
the reflection corresponding to \((a, a^\vee)\)
preserves \(\Phi(\Gt^\Gamma, S)\).
Since
the reflections corresponding to
\((2a, \frac1 2 a^\vee)\) and
\((a, a^\vee)\) are the same,
we may, and do, assume that
\(a\) belongs to \(\Phi(G, S)\)
(by replacing \(a\) by \(2a\), if \(p\) equals \(2\) and
\(a\) is multipliable).
Then we may take \(a^\vee\) to be the
coroot in \(\Phi^\vee(G, S)\) corresponding to \(a\).
Since every element of \(W(G, S)(k)\) has
a representative in \(N_G(S)(k)\), we have that
there is some \(n \in N_G(S)(k)\) whose action on \(\bX^*(S)\) is
the reflection corresponding to \((a, a^\vee)\)
\cite{conrad-gabber-prasad:prg}*{Theorem C.2.15}.
In particular, since \(n\) preserves \(\Gt^\Gamma\),
the reflection preserves \(\Phi(\Gt^\Gamma, S)\).
This completes the proof of (\ref{subprop:G-roots}), and hence
of the result.
\end{proof}

\begin{cor}
\(\Z\Phi^\vee(\Gt^\Gamma, S)\) equals
\(\Z\Phi^\vee(G, S)\), and
the natural map
\abmap{W(G, S)}{W(\Phi(\Gt^\Gamma, S))}
is an isomorphism.
\end{cor}

%

Corollary \ref{cor:G-smooth} is our main tool for
establishing smoothness or near-smoothness,
in the sense of
Theorem \ref{thm:quass}(\ref{subthm:quass-smooth}), of
fixed-point groups.

\begin{cor}
\label{cor:G-smooth}
Suppose that \(p\) is odd or \(\Phi(\Gt^\Gamma, S)\) is reduced.
If \(\Tt^\Gamma\) is smooth, then \(\Gt^\Gamma\) is smooth.
\end{cor}

\begin{proof}
We have that \(\Lie(\Gt^\Gamma)\) is the sum of the \(0\)-weight space
\(\Lie(\Gt^\Gamma)^S = \Lie(\Tt^\Gamma)\) for \(S\) and
the nonzero-weight spaces for \(S\).
Since \(\Phi(\Gt^\Gamma, S)\) equals \(\Phi(G, S)\)
by Proposition \ref{prop:GS-vs-GtS}(\ref{subprop:G-roots}),
it follows from Proposition \ref{prop:GS-vs-GtS}(\ref{subprop:g-root-space})
that each weight space in \(\Lie(\Gt^\Gamma)\) for a nonzero weight is
contained in \(\Lie(G)\).
Thus, since
\(\Lie(\Tt^\Gamma) = \Lie((\Tt^\Gamma)\conn)\) equals
\(\Lie(\fix\Tt^\Gamma) = \Lie(T)\),
we have that \(\Lie(\Gt^\Gamma)\) is contained in, hence equals,
\(\Lie(G) = \Lie(\fix\Gt^\Gamma)\);
so \cite{milne:algebraic-groups}*{Proposition 10.15}
gives that \(G\) equals \((\Gt^\Gamma)\conn\).
It follows that
\((\Gt^\Gamma)\conn\), and hence \(\Gt^\Gamma\), is smooth.
\end{proof}

{\newcommand\theadhocthm{\ref{thm:quass}(\ref{subthm:quass-smooth})}
\begin{adhocthm}
\((\Gt^\Gamma)\conn\) equals
\((Z(\Gt)^\Gamma)\conn\cdot\fix\Gt^\Gamma\) unless
\(p\) equals \(2\) and \((\Gt_\ks, \Gamma_\ks)\) is exceptional.
\end{adhocthm}}

\begin{proof}
We may, and do, assume, upon replacing \(k\) by \(\ks\), that
\(k\) is separably closed.
Suppose that \(p\) does not equal \(2\), or
\((\Gt, \Gamma)\) is not exceptional.

Since
\abmap{\Lie(\Gt)^\Gamma_\alpha}{\Lie(\Gt\adform)^\Gamma_\alpha} and,
by Corollary \ref{cor:fixed-surjective}, also
\abmap{\Lie(G)_\alpha}{\Lie(\fix\Gt\adform^\Gamma)_\alpha}
is an isomorphism for every nonzero \(\alpha \in \bX^*(T)\),
we have that
\abmap
	{\Phi(G, T)}
	{\Phi(\fix\Gt\adform^\Gamma, T/(Z(\Gt) \cap T))}
and
\abmap
	{\Phi(\Gt^\Gamma, T)}
	{\Phi(\Gt\adform^\Gamma, T/(Z(\Gt) \cap T))}
are bijections.
Thus,
\(p\) does not equal \(2\), or
\((\Gt\adform, \Gamma)\) is not exceptional.

By Corollary \ref{cor:lift-smooth},
we may, and do, thus assume, upon replacing \(\Gt\) by \(\Gt\adform\), that
\(\Gt\) is adjoint, at which point the conclusion becomes that
\((\Gt^\Gamma)\conn\) is smooth.

Suppose first that \(p\) does not equal \(2\) or
\(\Phi(\Gt^\Gamma, T)\) is reduced.
Remark \ref{rem:adj-or-sc-T} gives that \(\Tt^\Gamma\) is smooth, so
Corollary \ref{cor:G-smooth} gives that
\(\Gt^\Gamma\) is smooth.

Thus we may, and do, suppose that \(p\) equals \(2\) and
\(\Phi(\Gt^\Gamma, T)\) is not reduced.
Let \(\alpha\) be a multipliable element of \(\Phi(\Gt^\Gamma, T)\).
Then Proposition \ref{prop:GS-vs-GtS}(\ref{subprop:G-roots})
gives that \(2\alpha\) belongs to \(\Phi(G, T)\).
That is, \((\Gt, \Gamma)\) is exceptional, which is
a contradiction.
\end{proof}

It is easy, regardless of (positive) characteristic, for
passage to fixed points to create non-smoothness, but
this non-smoothness should be thought of as
coming from the failure of smoothness for
an action on a torus.
Remark \ref{rem:adj-or-sc-T} thus suggests that
it should be easier for
\(\Gt\adform^\Gamma\) than for
\(\Gt^\Gamma\)
to be smooth.
Lemma \ref{lem:smooth-to-iso} formalizes this idea
for use in the proof of
Lemma \ref{lem:inductive-step-Z|Lie}.

\begin{lem}
\label{lem:smooth-to-iso}
If \(\Gt^\Gamma\) is smooth and \(\Nt\) is a
\(\Gamma\)-stable, normal subgroup of \(\Gt\) such that
\((\Tt/\Nt \cap \Tt)^\Gamma\) is smooth, then
\((\Gt/\Nt)^\Gamma\) is smooth.
\end{lem}

\begin{proof}
We may, and do, assume, upon replacing \(k\) by \(\ks\), that
\(k\) is separably closed.

Since \(\Gt^\Gamma\) is smooth, we have that
\((\Gt^\Gamma)\conn\) equals \(G\), so that
\((\Gt^\Gamma)\conn\) is reductive and
\(T\) is a maximal torus in \((\Gt^\Gamma)\conn\).
These conditions together imply that
\(\Phi(\Gt^\Gamma, T)\) is reduced
\cite{borel:linear}*{Theorem 14.8}.

We now reason by contradiction.
Suppose that \((\Gt/\Nt)^\Gamma\) is not smooth.
Corollary \ref{cor:G-smooth} gives that
\(p\) equals \(2\) and \(\Phi((\Gt/\Nt)^\Gamma, T)\) is not reduced.
Let \(\alpha\) be a multipliable element of
\(\Phi((\Gt/\Nt)^\Gamma, T)\), and let
\(\Gt_1'\) be an almost-simple component of \(\Gt/\Nt\) such that
\(\alpha\) belongs to \(\Phi(\Gt_1', T)\).
Remark \ref{rem:how-to-irred} gives that
\(\Phi(\Gt_1', T)\) is an irreducible component of
\(\Phi(\Gt/\Nt, T)\), and so also contains
\(2\alpha\); so
Proposition \ref{prop:square-pinned}(\ref{subprop:square-pinned})
gives that
\(\Phi((\Gt_1')^\Gamma, T)\) contains \(2\alpha\).
Remark \ref{rem:gp-nred-facts}(\ref{subrem:gp-nred-type})
gives that
there is a positive integer \(n\) such that \((\Gt_1')\adform\),
and hence \(\Gt_1'\), is
of type \(\mathsf A_{2n}\), and that
there is an element of \(\Gamma(k)\) that preserves \(\Gt_1'\) but
does not act on it by an inner automorphism.
Write \(\Gt_1\) for an almost-simple component of \(\Gt\)
whose image in \(\Gt/\Nt\) is \(\Gt_1'\).
Then \(\ker(\abmap{\Gt_1}{\Gt_1'})\) is a subquotient of
\(\ker(\abmap{(\Gt_1')\scform}{(\Gt'_1)\adform}) = \mu_{2n + 1}\).
Since \(p\) equals \(2\), we have that \(\mu_{2n + 1}\), and hence
\(\ker(\abmap{\Gt_1}{\Gt_1'})\), is \'etale.
It follows that the (obviously) \(\Gamma\)-equivariant morphism
\abmap{\Lie(\Gt_1)}{\Lie(\Gt_1')} is an isomorphism, so
\(\alpha, 2\alpha \in \Phi((\Gt_1')^\Gamma, T)\) also belong to
\(\Phi(\Gt_1^\Gamma, T)\), hence to
\(\Phi(\Gt^\Gamma, T)\).
This is a contradiction of the fact that
\(\Phi(\Gt^\Gamma, T)\) is reduced.
\end{proof}

\numberwithin{equation}{section}
\section{Quasisemisimple outer involutions of special linear groups}
\label{sec:sl-outer}

In this section,
we give an explicit description of an important example that is
already implicit in
the proof of \cite{steinberg:endomorphisms}*{Theorem 8.2}, specifically
\cite{steinberg:endomorphisms}*{pp.~53--54, (2\textquotesingle b)}.
Our explicit understanding is necessary for the proof of
Theorem \ref{thm:quass}(\ref{subthm:quass-smoothable}).
We will handle another specific example in
\S\ref{subsec:D_4-outer}.

We continue to work with
the field \(k\) of
characteristic exponent \(p\) from
\S\ref{sec:quass}.

Let
\(X\) be a nonzero, finite-dimensional \(k\)-vector space,
and put
\(\Gt = \GL(X)\).
Let
\(E/k\) be a field extension, and
\(\gamma\) an involution of \(\Gt\) such that
\(\gamma\) acts by inversion on \(Z(\Gt)\), and
\(\gamma_E\) acts quasisemisimply on \(\Gt_E\).
We do \emph{not} yet assume that
\(\gamma\) acts quasisemisimply on \(\Gt\), although
see Theorem \ref{thm:ka-quass}(\ref{subthm:ka-quass-quass}) for
conditions under which we can conclude this.
For notational convenience, we put \(n = \dim(X) - 1\).
Remember that
\(\Gt\der^\gamma\) means \((\Gt\der)^\gamma\), not
\((\Gt^\gamma)\der\), when they differ; and
\(\Gt_E^{\gamma_E}\) means
\((\Gt_E)^{\gamma_E}\).

We are most interested in the cases where
\(p\) equals \(2\) and \(n\) is even, but we do not require this.

If \(b\) is a bilinear form on \(X\),
then we denote by \(q_b\) the quadratic form
\abmapto x{b(x, x)} on \(X\).
If \(p\) equals \(2\), then
\(q_b\) is a linear map
\abmap X{\lsub f k}, where
\map f k k is the Frobenius automorphism
\abmapto x{x^2} and
\(\lsub f k\) is the restriction of scalars of
\(k\) along \(f\), as in
\cite{demazure-gabriel:groupes-algebriques}*
	{Ch.~II, \S7, 1.1}.

Lemma \ref{lem:sp-fixer} will help us deal with
the obstruction to smoothability of
the fixed-point group in
Proposition \ref{prop:sl-even}.
The statement involves a lot of notation.
It may be informally, but perhaps more clearly, summarized as follows:
the subgroup of a symplectic group fixing a given subspace of
the defining representation is
smooth and connected, and
its maximal pseudo-reductive quotient is
a symplectic group, hence reductive.

{
\newcommand\pminus{^{\prime\,{-}}}
\newcommand\pmstar{^{\prime\,{-}\,{*}}}
\newcommand\pperp{^{\prime\,\perp}}
\newcommand\pplus{^{\prime\,{+}}}
\newcommand\ppstar{^{\prime\,{+}\,{*}}}
\newcommand\pstar{^{\prime\,{*}}}
\newcommand\ppminus{^{\prime\,{\pm}}}
\newcommand\ppperp{^{\prime\prime\,\perp}}
\newcommand\uSkew{\underline\Skew}
\begin{lem}
\label{lem:sp-fixer}
Suppose that \(b\) is a nondegenerate, alternating form on \(X\).
We denote \(b\)-orthogonal spaces by \((\cdot)^\perp\).
Let \(X''\) be a subspace of \(X\), and put
\(X' = X'' + X\ppperp\).
Write
\(G''\) for
the subgroup of \(\Sp(X, b)\) that fixes \(X''\) pointwise,
\(U'\) and \(U''\) for the subgroups
of \(G''\) that fix
\(X'\) and \(X'/X\pperp\) pointwise,
and
\(b'\) for the alternating form on \(X'/X\pperp\)
induced by \(b\).
The form \(b\) puts
\(X\pperp\) and \(X/X'\) in duality, so that there is
a duality involution on the space
\(\Hom(X/X', X\pperp)\) of
\(k\)-linear homomorphisms.
Write \(\Skew(X/X', X\pperp)\) for the space of
skew homomorphisms
(i.e., homomorphisms negated by the duality involution).
\begin{enumerate}[label=(\alph*), ref=\alph*]
\item\label{sublem:(sp-fixer)red}
The subspaces
\(X''/X\pperp\) and \(X\ppperp/X\pperp\) of
\(X'/X\pperp\) are complementary and
nondegenerate for \(b'\).
\item\label{sublem:Ru(sp-fixer)}
The group \(U'\) is
the vector group associated to
\(\Skew(X/X', X\pperp)\).
The group \(U''\) is an extension by \(U'\) of
the vector group associated to
\(\Hom(X'/X'', X\pperp)\).
\item\label{sublem:sp-fixer}
The natural map
\abmap{G''}{\Sp(X\ppperp/X\pperp, b')}
is a quotient map with kernel
\(U''\).
\end{enumerate}
\end{lem}

\begin{proof}
(\ref{sublem:(sp-fixer)red}) is clear, and implies that
\(
\abmap{\Fix_{\Sp(X'/X\pperp, b')}(X''/X\pperp)}{\Sp(X\ppperp/X\pperp, b')}
\)
is an isomorphism.


Since \(b'\) puts
\(X'/X''\) and
\(X\ppperp/X\pperp\) in duality,
the kernel of the natural map
\(
\abmap{G''}{\Sp(X\ppperp/X\pperp, b')}
\)
is
\(\Fix_{\Sp(X, b)}(X'', X'/X\pperp) = U''\).
Thus (\ref{sublem:sp-fixer}) will follow once we show that
\(
\abmap{G''/U''}{\Fix_{\Sp(X'/X\pperp, b')}(X''/X\pperp)}
\)
is surjective.

Choose a complement \(Y''\) to \(X\pperp\) in \(X''\), and
enlarge it to a complement \(Y'\) to \(X\pperp\) in \(X'\).
Since \(b\) is nondegenerate on \(Y'\), we have that
\(X\) equals \(Y' \oplus Y\pperp\), and hence that
the natural map
\abmap{Y\pperp/X\pperp}{X/X'} is an isomorphism.
Since \(b\) puts \(X/X'\) and \(X\pperp\) in duality, it also puts
\(Y\pperp/X\pperp\) and \(X\pperp\) in duality.
In particular, \(X\pperp\) is a maximal
totally isotropic subspace of \(Y\pperp\).
Let \(Y\) be a complementary (necessarily totally isotropic) subspace.
In addition to the duality involution on
\(\Hom(X/X', X\pperp) \cong \Hom(Y, X\pperp)\) mentioned
in the statement, since \(Y'\) is self-dual, we have a duality isomorphism \((\cdot)^*\) of
\(\Hom(Y, Y')\) with \(\Hom(Y', X\pperp)\).
Choose a polarization of \(Y'\), i.e., a pair
\((Y\pplus, Y\pminus)\) of complementary totally isotropic subspaces.
These furnish maps
\abmap{Y'}{Y\ppminus}, hence
\map{(\cdot)^\pm}{\Hom(Y, Y')}{\Hom(Y, Y\ppminus)}.
Write \(\uSkew(Y, X\pperp)\) and \(\uHom(Y, Y')\) for
the vector groups associated to
\(\Skew(Y, X\pperp)\) and \(\Hom(Y, Y')\).
Write \(P'\) for the parabolic subgroup of \(\Sp(X, b)\) associated to
the self-dual flag
\(0 \subseteq X\pperp \subseteq X' \subseteq X\).
Note that \(U'\) is a normal subgroup of \(P'\).
We have an isomorphism of schemes, \emph{not} of group schemes,
from
\(\Sp(Y', b') \times
	\uHom(Y, Y') \times
	\uSkew(Y, X\pperp)\) onto
the subgroup of \(P'\) that fixes \(X\pperp\) pointwise, given by
\[
\abmapto
	{(g, \xi', \xi\pperp)}
	{\begin{pmatrix} 1 \\ & g \\ && 1 \end{pmatrix}
	\begin{pmatrix} 1 & -\xi\ppstar \\ & 1 & \xi\pplus \\ && 1 \end{pmatrix}
	\begin{pmatrix} 1 & -\xi\pmstar \\ & 1 & \xi\pminus \\ && 1 \end{pmatrix}
	\begin{pmatrix} 1 && \xi\pperp \\ & 1 \\ && 1 \end{pmatrix}}
,
\]
where we use block-matrix notation organized as
\[\bordermatrix{
& X\pperp & Y' & Y \cr
X\pperp & * & * & * \cr
Y' & * & * & * \cr
Y & * & * & *
}.\]
Concretely, the embedding sends
\((g, \xi', \xi\pperp) \in
\Sp(Y', b') \times \uHom(Y, Y') \times \uSkew(Y, X\pperp)\) to
the symplectomorphism
\[
\abmapto
	{x\pperp + y' + y}
	{(x\pperp - \xi\pstar(y') + (\xi\pperp - \xi\ppstar\xi\pminus)(y)) +
		g(y' + \xi'(y)) +
		y}
\]
for all \(x\pperp \in X\pperp\), \(y' \in Y'\), and \(y \in Y\).
Although this map is not a morphism of group schemes, we have that
\begin{itemize}
\item the restriction of our map to
\(\uSkew(Y, X\pperp)\) is
an isomorphism of group schemes onto \(U'\), which
shows part of (\ref{sublem:Ru(sp-fixer)});
\item the composition
\(\xymatrix@1{
\Sp(Y', b) \times \uHom(Y, Y') \times \uSkew(Y, X\pperp) \ar[r] & P' \ar[r] & P'/U'
}\)
factors uniquely through projection on the first two factors
to give an isomorphism of group schemes from
\(\Sp(Y', b') \ltimes \uHom(Y, Y')\) onto
a closed subgroup of \(P'/U'\), and the isomorphism is
independent of the choice of polarization of \(Y'\).
\end{itemize}
Now \(U''\) is the inflation to \(P'\) of
the image in \(P'/U'\)
of
the vector subgroup of \(\uHom(Y, Y')\) corresponding to
\(\sett
	{\xi' \in \Hom(Y, Y')}
	{\(\xi\pstar\) is trivial on \(Y''\)}\), which
shows (\ref{sublem:Ru(sp-fixer)}); and
\(\Fix_{\Sp(X'/X\pperp, b')}(X''/X\pperp) \cong
	\Fix_{\Sp(Y', b')}(Y'')\)
maps isomorphically onto \(G''/U''\), giving
a section of the natural map
\abmap{G''/U''}{\Fix_{\Sp(X'/X\pperp, b')}(X''/X\pperp)}, which
is therefore surjective.
This shows (\ref{sublem:sp-fixer}), and completes the proof.
\end{proof}
}

\begin{notation}
\label{notn:sl-bilinear}
Since
\begin{itemize}
\item \(\gamma\) acts by inversion on \(Z(\Gt)\), and
\item \(\gamma\) conjugates the defining representation
of \(\Gt\der\) to its dual
(either because
\(n\) is greater than \(1\) and
\(\gamma\) restricts to a nontrivial outer automorphism of \(\Gt\der\), or because
\(n\) is at most \(1\),
\(\gamma\) restricts to an inner automorphism of \(\Gt\der\), and
the defining representation of \(\Gt\der\) is self-dual),
\end{itemize}
the \(\gamma\)-conjugate of \(X\) is isomorphic to the dual representation \(X^*\)%
.

Write \(\gamma_X\) for a map \abmap X{X^*}
that intertwines the \(\gamma\)-twisted action of \(\Gt\) on \(X\)
with the natural action of \(\Gt\) on \(X^*\),
and
\(b\) for the associated bilinear form
\abmapto{(x_1, x_2)}{\pair{\gamma_X(x_1)}{x_2}}.
We will always use the notation \(\perp\) for
orthogonal spaces with respect to \(b\).
That is, if \(X'\) is a subspace of \(X\), then
\(X^{\prime\,\perp}\) means
\(\sett{x \in X}{\(b(x, x') = 0\) for all \(x' \in X'\)}\).
\end{notation}

Since \(X\) and \(X^*\) are irreducible representations of \(\Gt\),
the map \(\gamma_X\),
and hence the bilinear form \(b\),
in Notation \ref{notn:sl-bilinear}
are uniquely determined up to
multiplication by a nonzero scalar.

We have that \(\Gt^\gamma\) is the full isometry group
\(\Isom(X, b)\) of \(b\), and
\(\Gt\der^\gamma\) is the group of determinant-\(1\) isometries.
In particular, \(\Gt^\gamma\) is contained in the orthogonal group
\(\Or(X, q_b)\) of \(q_b\).
Since \(\gamma\)
acts by inversion on \(Z(\Gt)\),
we have that \(\det \circ \gamma\) equals \(-{\det}\) as
characters of \(\Gt\), so
\(\det(\Gt^\gamma)\) is contained in \(\mu_2\).

\begin{lem}
\label{lem:b-symmetry}
The pairing \(b\) is symmetric or anti-symmetric.
If \(n\) is even, then \(b\) is symmetric.
\end{lem}

\begin{proof}
Since \(\gamma\) is an involution,
we have that the dual map \map{\gamma_X^*}X{X^*}
to \(\gamma_X\) also intertwines the \(\gamma\)-conjugate of \(X\)
with \(X^*\) as representations of \(\Gt\),
so \(\gamma_X^*\) equals \(c\gamma_X\) for some constant \(c\).
That is,
\(b(x_1, x_2) = \pair{\gamma_X(x_1)}{x_2}\)
equals
\(\pair{\gamma_X^*(x_2)}{x_1} = c\pair{\gamma_X(x_1)}{x_2}
= c b(x_2, x_1)\).

Since \(\gamma_X^{**}\) equals \(\gamma_X\),
we have that \(c^2\) equals \(1\), so that \(b\) is symmetric or
anti-symmetric.
Moreover, \(\det(\gamma_X\inv\gamma_X^*) = 1\) equals \(c^{n + 1}\),
so that if \(n\) is even, then \(c = c^{n + 1}(c^2)^{-n/2}\) equals \(1\)
and hence \(b\) is symmetric.
\end{proof}

Fix a Borel--torus pair \((\Bt, \Tt)\)
in \(\Gt_E\) that is preserved by \(\gamma_E\).

\begin{notation}
\label{notn:sl-data}
Let \(\ell_0, \dotsc, \ell_n\) be the weight spaces for
\(\Tt\) in \(X \otimes_k E\), and
\(e_0, \dotsc, e_n\) the corresponding weights,
numbered so that
\(\Phi(\Bt, \Tt)\) equals \(\set{e_i - e_j}{i < j}\)
and
\(\gamma(e_i)\) equals \(-e_{n - i}\)
for all \(0 \le i \le n\).
Let \((e_0^\vee, \dotsc, e_n^\vee)\) be the ordered basis of
\(\bX_*(\Tt)\) dual to \((e_0, \dotsc, e_n)\),
so that \(\bX_*(\Tt \cap \Gt\der)\) is the $\Z$-span of
\(\set{e_i^\vee - e_j^\vee}{0 \le i, j \le n}\).
\end{notation}

\begin{rem}
\label{rem:sl-nondeg}
For every \(0 \le i \le n\), we have that
\((\gamma_{X})_E \, (\ell_i)\) is the (\(-e_{n - i}\))-weight
space for \(\Tt\) in \((X \otimes_k E)^*\),
hence can be
identified with \(\ell_{n - i}^*\) by restriction.
In particular, we have for every \(0 \le i, j \le n\)
that \(\ell_i\) is \(b_E\)-orthogonal to \(\ell_j\)
unless \(i + j\) equals \(n\).
\end{rem}

\begin{lem}
\label{lem:q-linear}
Suppose that \(p\) equals \(2\).
\begin{enumerate}[label=(\alph*), ref=\alph*]
\item\label{sublem:X'}
\(\ker(q_{b_E})\) equals
\(\bigoplus_{\substack{i = 0 \\ 2i \ne n}}^n \ell_i\).
The restriction \(b_E'\) of \(b_E\) to \(\ker(q_{b_E})\)
is a nondegenerate, alternating form.
\item\label{sublem:X''}
If \(n\) is odd, then \(\ker(q_{b_E})^\perp\) equals \(\sset0\).
If \(n\) is even, then \(\ker(q_{b_E})^\perp\) equals \(\ell_{n/2}\).
In either case, the restriction \(b_E''\) of \(b_E\) to
\(\ker(q_{b_E})^\perp\) is
nondegenerate.
\end{enumerate}
\end{lem}

\begin{proof}
We may, and do, replace \(k\) by \(E\).

For each \(0 \le i \le n\), write \abmapto x{x_i}
for the \(\Tt\)-equivariant projection \abmap X{\ell_i}.
For convenience, write \abmapto x{x_{n/2}} for the \(0\) map if
\(n\) is odd.
Remark \ref{rem:sl-nondeg} and symmetry
(which follows from Lemma \ref{lem:b-symmetry} since
anti-symmetry is the same as symmetry when
\(p\) equals \(2\))
give that
\[
q_b(x) = b(x, x)
\quad\text{equals}\quad
b(x_{n/2}, x_{n/2}) + \sum_{i = 0}^{\lceil n/2\rceil - 1}
	2b(x_i, x_{n - i})
= b(x_{n/2}, x_{n/2})
\]
for all \(x \in X\).
This shows that \(\ker(q_b)\)
equals \(\bigoplus_{\substack{i = 0 \\ 2i \ne n}}^n \ell_i\), and
that \(b'\) is alternating.
Now another application of Remark \ref{rem:sl-nondeg} shows that
\(b'\) is nondegenerate, giving (\ref{sublem:X'}); and that
\(\ker(q_b)^\perp\) equals
\(\sset0\) if \(n\) is odd, and equals
\(\ell_{n/2}\) if \(n\) is even,
hence that
\(b''\) is nondegenerate,
giving (\ref{sublem:X''}).
\end{proof}

\begin{cor}
\label{cor:q-linear}
Suppose that \(p\) equals \(2\).
Then \(\ker(q_{b_E}) \otimes_E L\) equals \(\ker(q_{b_L})\) for
all field extensions \(L/E\).
\end{cor}

\begin{lem}
\label{lem:sl-odd}
Suppose that \(p\) or \(n\) is odd.
If \(b\) is anti-symmetric, then
\(\Gt^\gamma\) and \(\Gt\der^\gamma\) both equal
\(\Sp(X, b)\), which is smooth and connected.
If \(p\) is odd and \(b\) is symmetric, then
\(\Gt^\gamma\) equals \(\Or(X, q_b)\), which is smooth, and
\(\Gt\der^\gamma\) equals \(\SO(X, q_b)\), which is
the identity component of \(\Gt^\gamma\).
\end{lem}

\begin{proof}
If \(p\) is odd, then every anti-symmetric form is alternating;
so, if \(b\) is anti-symmetric, then it is alternating.
If \(p\) equals \(2\) and \(n\) is odd, then
Lemma \ref{lem:q-linear} gives that
\(b_E\), and hence \(b\), is alternating.
Thus the group \(\Gt^\gamma\) of isometries of \(b\) is
\(\Sp(X, b)\), which is smooth, connected, and
contained in \(\Gt\der\).

An isometry of \(b\) is always also an isometry of \(q_b\).
If \(p\) is odd and \(b\) is symmetric, then an
isometry of \(q_b\) is also an
isometry of \(b\), so
the group \(\Gt^\gamma\) of isometries of \(b\) is
\(\Or(X, q_b)\), which is
smooth \cite{conrad:red-gp-sch}*{Theorem C.1.5}; and
\(\Gt\der^\gamma\) equals
\(\Gt^\gamma \cap \Gt\der = \Or(X, q_b) \cap \SL(X)\),
which equals \(\SO(X, q_b)\) and
is the identity component of \(\Or(X, q_b)\) by
\cite{conrad:red-gp-sch}*{Theorem C.2.11 and Corollary C.3.2}.
\end{proof}

\begin{prop}
\label{prop:sl-even}
Suppose that \(p\) equals \(2\).
\begin{enumerate}[label=(\alph*), ref=\alph*]
\item\label{subprop:gl-even}
\((\Gt^\gamma)\smooth\) equals
\((\Gt\der^\gamma)\smooth\), and
is the subgroup of \(\Gt^\gamma\) fixing \(\ker(q_b)^\perp\) pointwise.
\item\label{subprop:sl-even}
Write \(b_E'\) for the restriction of \(b_E\) to
\(\ker(q_{b_E})\).
Extension trivially across \(\ker(q_{b_E})^\perp\) furnishes
an isomorphism onto \(((\Gt^\gamma)\smooth)_E\) from
the subgroup of
\(\Sp(\ker(q_{b_E}), b_E')\) that fixes
\(\ker(q_{b_E}) \cap (\ker(q_b)^\perp \otimes_k E)\) pointwise.
\end{enumerate}
\end{prop}

\begin{note}
Lemma \ref{lem:q-linear} gives that
\(b_E'\) is a
nondegenerate, alternating form on \(\ker(q_{b_E})\), so that
it makes sense to speak of \(\Sp(\ker(q_{b_E}), b_E')\); and that
\(\ker(q_{b_E})\) and \(\ker(q_{b_E})^\perp\) are
complementary subspaces of \(X \otimes_k E\), so that
the extension map in (\ref{subprop:sl-even}) is
well defined.
\end{note}

\begin{proof}
We may, and do, assume, upon replacing
\(k\) and \(E\) by their separable closures, that they are
separably closed.
Put \(X' = \ker(q_b)\) and \(X'' = \ker(q_b)^\perp\).

We make a few observations about the extension map of
(\ref{subprop:sl-even}).
Remember that
\(\Gt_E^{\gamma_E} = (\Gt^\gamma)_E\) equals
\(\Isom(X, b)_E = \Isom(X\otimes_k E, b_E)\).
Since
\(X' \otimes_k E = \ker(q_b) \otimes_k E\) is
contained in \(\ker(q_{b_E})\),
we have that
\(X'' \otimes_k E =
\ker(q_b)^\perp \otimes_k E =
(\ker(q_b) \otimes_k E)^\perp\)
contains
\(\ker(q_{b_E})^\perp\).
We observe two consequences.
First,
\(X'' \otimes_k E\) is the direct sum of
\(\ker(q_{b_E}) \cap (X'' \otimes_k E)\) and
\(\ker(q_{b_E})^\perp\), so that the extension map
\abmap
	{\Fix_{\Sp(\ker(q_{b_E}), b_E')}
		(\ker(q_{b_E}) \cap (X'' \otimes_k E))}
	\Gt\
has image in
\(\Isom(X\otimes_k E, b_E) \cap \Fix_{\Gt_E}(X'' \otimes_k E) =
\Fix_{\Gt_E^{\gamma_E}}(X'' \otimes_k E) =
(\Fix_{\Gt^\gamma}(X''))_E\).
Second, there is a restriction map
\abmap
	{(\Fix_{\Gt^\gamma}(X''))_E}
	{\Fix_{\GL(\ker(q_{b_E}))}
		(\ker(q_{b_E}) \cap (X'' \otimes_k E))},
and its image lies in
\(\Isom(\ker(q_{b_E}), b_E') = \Sp(\ker(q_{b_E}), b_E')\).
These maps are mutually inverse, so that the extension map of
(\ref{subprop:sl-even}) is
an isomorphism of
\(\Fix_{\Sp(\ker(q_{b_E}), b_E')}
	(\ker(q_{b_E}) \cap (X'' \otimes_k E))\)
onto
\((\Fix_{\Gt^\gamma}(X''))_E\).
Thus Lemma \ref{lem:sp-fixer}(\ref{sublem:Ru(sp-fixer)},\ref{sublem:sp-fixer})
shows that
\((\Fix_{\Gt^\gamma}(X''))_E\) is smooth, hence that
\(\Fix_{\Gt^\gamma}(X'')\) is smooth, and so contained in
\((\Gt^\gamma)\smooth\).
It remains to show that \((\Gt^\gamma)\smooth\)
fixes \(X''\) pointwise, hence equals \((\Fix_{\Gt^\gamma}(X''))_E\)
(completing the proof of (\ref{subprop:sl-even}) and part of
(\ref{subprop:gl-even})), and that
\((\Gt^\gamma)\smooth\) equals \((\Gt\der^\gamma)\smooth\)
(proving the other part of (\ref{subprop:sl-even})).

Suppose first that \(E\) equals \(k\),
hence that \(\gamma\) acts quasisemisimply on \(\Gt\).
Lemma \ref{lem:q-linear} gives that
\(b'\) is a nondegenerate, alternating form on \(X'\), so
\(X' \cap X''\) is trivial.
Fix \(g \in \Gt^\gamma(k)\).
We have that \(\det(g)\) belongs to \(\mu_2(k) = \sset1\),
i.e., \(g\) belongs to \(\Gt\der(k)\).
Since \(g\) fixes \(b\), it also fixes \(q_b\), and hence
preserves \(\ker(q_b) = X'\).
Thus \(g\) also preserves
the \(b\)-orthogonal space \(X''\) of \(X'\); and
the restriction of \(g\) to \(X'\) preserves \(b'\), hence
belongs to \(\Sp(X', b')(k)\).
By Lemma \ref{lem:q-linear}(\ref{sublem:X''}),
if \(n\) is odd, then \(X''\) is trivial, so we are done;
whereas, if \(n\) is even, then \(X''\) is \(1\)-dimensional,
so that \(g\) acts on it by a scalar.
In the latter case,
since \(g\) has determinant \(1\),
and the restriction of \(g\) to \(X'\) also has determinant \(1\)
(because it belongs to \(\Sp(X', b')(k)\)),
the scalar by which \(g\) acts on \(X''\) is also \(1\), so that
\(g\) fixes \(X''\) pointwise.
Thus \((\Gt^\gamma)\smooth\), which is
the Zariski closure of \(\Gt^\gamma(k)\)
(because \(k\) is separably closed),
is contained in \(\Gt\der\), and fixes \(X''\).
As observed, this proves the result,
under the assumption that \(E\) equals \(k\).

Now drop the assumption that \(E\) equals \(k\)
(but keep the assumption that \(k\) is separably closed).
By the special case of (\ref{subprop:gl-even}) that
we have already handled, we have
for every
\(g \in \Gt^\gamma(k) \subseteq
	\Gt^\gamma(E) = (\Gt^\gamma)\smooth(E)\) that
\(g\) belongs to
\((\Gt_E)\der(E) = \Gt\der(E)\), hence to
\(\Gt\der(k)\); and that
\(g\) fixes \(\ker(q_{b_E})^\perp\) pointwise, so that
the fixed-point subspace \(\ker(g - 1)\) of \(g\) on \(X\) satisfies
the containment
\(\ker(q_{b_E})^\perp \subseteq \ker(g - 1) \otimes_k E\),
from which we deduce successively the containments
\begin{align*}
\ker(g - 1)^\perp \otimes_k E =
(\ker(g - 1) \otimes_k E)^\perp & \subseteq \ker(q_{b_E}),
\intertext{then}
\ker(g - 1)^\perp & \subseteq \ker(q_{b_E}) \cap X = \ker(q_b),
\intertext{and finally}
X'' = \ker(q_b)^\perp & \subseteq \ker(g - 1).
\end{align*}
That is,
\(\Gt^\gamma(k)\) is contained in \(\Gt\der(k)\), and
every element of \(\Gt^\gamma(k)\) fixes \(X''\) pointwise; so
the Zariski closure \((\Gt^\gamma)\smooth\) of \(\Gt^\gamma(k)\)
is contained in \(\Gt\der\), hence
equals \((\Gt\der^\gamma)\smooth\), and
fixes \(X''\) pointwise.
This completes the proof of the result.
\end{proof}

\begin{cor}
\label{cor:sl-red}
Suppose that \(p\) equals \(2\).
The form \(b'\) on \(\ker(q_b)/(\ker(q_b) \cap \ker(q_b)^\perp)\)
induced by \(b\)
is nondegenerate and alternating, and the
natural map
\[
\abmap
	{(\Gt\der^\gamma)\smooth = (\Gt^\gamma)\smooth}
	{\Sp(\ker(q_b)/(\ker(q_b) \cap \ker(q_b)^\perp), b')}
\]
is a quotient whose kernel is an extension of
the vector group associated to
\[
\Hom((\ker(q_b) + \ker(q_b)^\perp)/{\ker(q_b)^\perp},
	\ker(q_b) \cap \ker(q_b)^\perp)
\]
by the vector group associated to the skew-symmetric elements of
\[
\Hom(X/(\ker(q_b) + \ker(q_b)^\perp),
	\ker(q_b) \cap \ker(q_b)^\perp),
\]
i.e., those negated
by the duality involution coming from \(b\).
\end{cor}

\begin{note}
That the map in the statement exists follows from
Proposition \ref{prop:sl-even}(\ref{subprop:gl-even}).

Since \(p\) equals \(2\),
requiring that a homomorphism be
skew-symmetric (i.e., negated by duality)
is the same as
requiring that it be
symmetric (i.e., fixed by duality).
\end{note}

\begin{proof}
Since \(E\) is faithfully flat over \(k\),
the statement may be checked after base change to \(E\).
Since
\(X \otimes_k E\) is the direct sum of
\(\ker(q_{b_E})\) and \(\ker(q_{b_E})^\perp\),
also
\(\ker(q_b)^\perp \otimes_k E\) is the direct sum of
\(\ker(q_{b_E}) \cap (\ker(q_b)^\perp \otimes_k E)\) and
\(\ker(q_{b_E})^\perp\), so
the natural embeddings
\begin{multline*}
\ker(q_{b_E})/
	\bigl(\ker(q_b) \otimes_k E + (\ker(q_{b_E}) \cap (\ker(q_b)^\perp \otimes_k E))\bigr)
\maparrow{} \\
\bigl(X/(\ker(q_b) + \ker(q_b)^\perp)\bigr) \otimes_k E
\end{multline*}
and
\begin{multline*}
\bigl(\ker(q_b) \otimes_k E + (\ker(q_{b_E}) \cap (\ker(q_b)^\perp \otimes_k E))\bigr)
	/(\ker(q_{b_E}) \cap (\ker(q_b)^\perp \otimes_k E))
\maparrow{} \\
\bigl((\ker(q_b) + \ker(q_b)^\perp)/{\ker(q_b)^\perp}\bigr) \otimes_k E
\end{multline*}
are isomorphisms.
Thus, by Proposition \ref{prop:sl-even}(\ref{subprop:sl-even}),
the claim follows from
Lemma \ref{lem:sp-fixer}
with \(\ker(q_{b_E})\)
playing the role of \(X\);
 \(\ker(q_{b_E}) \cap (\ker(q_b)^\perp \otimes_k E)\) that of \(X''\);
and
\(\ker(q_b) \otimes_k E + (\ker(q_{b_E}) \cap (\ker(q_b)^\perp \otimes_k E)\)
that of  \(X'\).
\end{proof}

\begin{cor}
\label{cor:sl-smoothable}
Suppose that \(p\) equals \(2\).
The following statements are equivalent.
\begin{enumerate}[label=(\alph*), ref=\alph*]
\item\label{case:gl-smoothable}
\(\Gt^\gamma\) is smoothable.
\item\label{case:sl-smoothable}
\(\Gt\der^\gamma\) is smoothable.
\item\label{case:sl-imperfect}
\(\ker(q_b) \otimes_k E\) equals \(\ker(q_{b_E})\).
\end{enumerate}
\end{cor}

\begin{proof}
If either of \(\Gt^\gamma\) or \(\Gt\der^\gamma\) is smoothable,
so that
\(((\Gt^\gamma)\smooth)_\ka\) equals \((\Gt_\ka^{\gamma_\ka})\smooth\)
or
\(((\Gt\der^\gamma)\smooth)_\ka\) equals \(((\Gt\der)_\ka^{\gamma_\ka})\smooth\),
then, since Proposition \ref{prop:sl-even} gives that
\((\Gt^\gamma)\smooth\) equals \((\Gt\der^\gamma)\smooth\) and
\((\Gt_\ka^{\gamma_\ka})\smooth\) equals \(((\Gt\der)_\ka^{\gamma_\ka})\smooth\),
we have that both \(\Gt^\gamma\) and \(\Gt\der^\gamma\) are smoothable.
Thus (\ref{case:gl-smoothable}) and (\ref{case:sl-smoothable})
are equivalent.

Since
\(\gamma_{E\alg}\) acts quasisemisimply on \(\Gt_{E\alg}\),
Lemma \ref{lem:q-linear} gives that
\(\ker(q_{b_E}) \cap \ker(q_{b_E})^\perp\) and
\(\ker(q_{b_{E\alg}}) \cap \ker(q_{b_{E\alg}})^\perp\) are
both trivial, and
Corollary \ref{cor:q-linear} gives that
\(\ker(q_{b_E}) \otimes_E E\alg\) equals
\(\ker(q_{b_{E\alg}})\), so
Proposition \ref{prop:sl-even}(\ref{subprop:sl-even})
gives that
\(((\Gt_E^{\gamma_E})\smooth)_{E\alg}\) and
\((\Gt_{E\alg}^{\gamma_{E\alg}})\smooth\) are
both the image in \(\Gt_{E\alg}\) of
\(\Sp(\ker(q_{b_{E\alg}}), b_{E\alg}')\).
In particular, they are equal.

Since smoothing commutes with base change from
an algebraically closed field, we have that
\(((\Gt_\ka^{\gamma_\ka})\smooth)_{E\alg}\) equals
\((\Gt_{E\alg}^{\gamma_{E\alg}})\smooth\).
By definition, (\ref{case:gl-smoothable}) means that
\(((\Gt^\gamma)\smooth)_\ka\) equals \((\Gt_\ka^{\gamma_\ka})\smooth\),
which is equivalent to the equality of
\(((\Gt^\gamma)\smooth)_{E\alg} = (((\Gt^\gamma)\smooth)_\ka)_{E\alg}\)
and
\(((\Gt_\ka^{\gamma_\ka})\smooth)_{E\alg} =
(\Gt_{E\alg}^{\gamma_{E\alg}})\smooth\).
Since \(((\Gt^\gamma)\smooth)_{E\alg}\) is obviously
the base change to \(E\alg\) of
\(((\Gt^\gamma)\smooth)_E\), and
we have observed that
\((\Gt_{E\alg}^{\gamma_{E\alg}})\smooth\) is
the base change to \(E\alg\) of
\((\Gt_E^{\gamma_E})\smooth\), we have that
(\ref{case:gl-smoothable}) is equivalent to the equality of
\((\Gt^\gamma)\smooth)_E\)
and
\((\Gt_E^{\gamma_E})\smooth\).
By Proposition \ref{prop:sl-even}(\ref{subprop:gl-even}), this
is equivalent to the equality of
\((\ker(q_b) \otimes_k E)^\perp = \ker(q_b)^\perp \otimes_k E\)
with
\(\ker(q_{b_E})^\perp\), which is equivalent to
statement (\ref{case:sl-imperfect}).
\end{proof}

{\newcommand\Span{\operatorname{Span}}
Examples \ref{ex:sl-not-reductive} and
\ref{ex:sl-accidentally-reductive}
show different reasons why \(\Gt^\gamma\) is not always smoothable.
In Example \ref{ex:sl-not-reductive}, it is because
\(\fix\Gt^\gamma\) is not reductive.

\begin{example}[Alex Bauman and Sean Cotner]
\label{ex:sl-not-reductive}
Suppose that \(p\) equals \(2\) and \(k\) is imperfect.
Let \(t\) be an element of \(k\mult \setminus (k\mult)^2\), and let
\(\gamma\) be the automorphism
\abmapto
	\gt
	{\Int\begin{smallpmatrix} t \\ & 1 \\ && t \end{smallpmatrix}\gt^{-\mathsf T}}
of \(\Gt:=\GL_3\).
Since \(\gamma\) is an involution,
Lemma \ref{lem:A_{2n}-quass} below
guarantees that it acts quasisemisimply on
\(\SL_{3, \ka}\), hence on \(\GL_{3, \ka}\).
Concretely, if we put \(E = k(\sqrt t)\), then
\(\gamma_E\) preserves the opposite Borel subgroups of
\(\GL_{3, E}\)
corresponding to the flags
\begin{gather*}
0 \subseteq
\Span_E \sset{\begin{pmatrix} 1 \\ 0 \\ 1 \end{pmatrix}} \subseteq
\Span_E \sset{\begin{pmatrix} 1 \\ 0 \\ 1 \end{pmatrix}, \begin{pmatrix} \sqrt t \\ 1 \\ \sqrt t \end{pmatrix}} \subseteq
E^3
\intertext{and}
0 \subseteq
\Span_E \sset{\begin{pmatrix} \sqrt t \\ 1 \\ 0 \end{pmatrix}} \subseteq
\Span_E \sset{\begin{pmatrix} \sqrt t \\ 1 \\ 0 \end{pmatrix}, \begin{pmatrix} \sqrt t \\ 1 \\ \sqrt t \end{pmatrix}} \subseteq
E^3,
\end{gather*}
hence their common maximal torus.
The first Borel subgroup descends to
the Borel subgroup \(\Bt\) of \(\GL_3\) corresponding to the flag
\[
0 \subseteq
\Span_k \sset{\begin{pmatrix} 1 \\ 0 \\ 1 \end{pmatrix}} \subseteq
\Span_k \sset{\begin{pmatrix} 1 \\ 0 \\ 1 \end{pmatrix}, \begin{pmatrix} 0 \\ 1 \\ 0 \end{pmatrix}} \subseteq
k^3,
\]
but the second is not defined over \(k\).

We have that
\(q_b(x)\) equals \(t\inv x_0^2 + x_1^2 + t\inv x_2^2\) for all
\(x = (x_0, x_1, x_2) \in k^3\), so that
\(\ker(q_b)\) equals
\(\Span \sset{\begin{smallpmatrix} 1 \\ 0 \\ 1 \end{smallpmatrix}}\),
but
\(\ker(q_{b_E})\) equals
\(\Span \sset{\begin{smallpmatrix} 1 \\ 0 \\ 1 \end{smallpmatrix}, \begin{smallpmatrix} \sqrt t \\ 1 \\ 0 \end{smallpmatrix}}\).
Using Proposition \ref{prop:sl-even}, we have that
\(((\Gt\der)_E^{\gamma_E})\smooth\) is
the extension trivially across
\(\Span_E \sset{\begin{smallpmatrix} \sqrt t \\ 1 \\ \sqrt t \end{smallpmatrix}}\)
of the symplectic group on
\(\Span_E \sset{\begin{smallpmatrix} 1 \\ 0 \\ 1 \end{smallpmatrix}, \begin{smallpmatrix} \sqrt t \\ 1 \\ 0 \end{smallpmatrix}}\); but
\((\Gt\der^\gamma)\smooth\) is
the subgroup of \(\Isom(k^3, b)\) that fixes
\(\begin{smallpmatrix} 1 \\ 0 \\ 1 \end{smallpmatrix}\),
which is the additive group
\(\sset{\begin{smallpmatrix}
c + 1 & 0 & c \\
0 & 1 & 0 \\
c & 0 & c + 1
\end{smallpmatrix}}\).
Since \((\Gt\der^\gamma)\smooth\) is not reductive,
we have
(by Theorem \ref{thm:quass}(\ref{subthm:quass-reductive}))
that \(\gamma\) does not act quasisemisimply on \(\Gt\der\).
In particular, there is no Borel subgroup of \(\Gt\) that is
opposite to \(\Bt\) and
preserved by \(\gamma\).
\end{example}

In Example \ref{ex:sl-accidentally-reductive},
the fixed-point group \(\fix\Gt\der^\gamma\) is reductive, but
``too small''.
This explains the need for the largeness condition that
a certain centralizer be of multiplicative type in
Theorem \ref{thm:ka-quass}(\ref{subthm:ka-quass-quass})%
	(\ref{case:ka-quass-reductive}).

\begin{example}
\label{ex:sl-accidentally-reductive}
A slight modification of Example \ref{ex:sl-not-reductive}
shows that \(\Gt^\gamma\) can fail to be smoothable even if
\((\Gt\der^\gamma)\smooth\) is reductive.
Namely, continue to suppose that
\(p\) equals \(2\) and
\(k\) is imperfect, but now choose
\(t_0, t_2 \in k\) such that
\(\sset{1, \sqrt{t_0}, \sqrt{t_2}}\) is linearly independent over \(k^2\), and
consider the automorphism
\abmapto
	\gt
	{\Int\begin{smallpmatrix} t_0 \\ & 1 \\ && t_2 \end{smallpmatrix}\gt^{-\mathsf T}}
of \(\GL_3\).
Put \(E = k(\sqrt{t_0}, \sqrt{t_2})\).
This time, \(\gamma_E\)
preserves the opposite Borel subgroups corresponding to the flags
\begin{gather*}
0 \subseteq
\Span_E \sset{\begin{pmatrix} \sqrt{t_0} \\ 1 \\ 0 \end{pmatrix}} \subseteq
\Span_E \sset{\begin{pmatrix} \sqrt{t_0} \\ 1 \\ 0 \end{pmatrix}, \begin{pmatrix} \sqrt{t_0} \\ 1 \\ \sqrt{t_2} \end{pmatrix}} \subseteq
E^3
\intertext{and}
0 \subseteq
\Span_E \sset{\begin{pmatrix} 0 \\ 1 \\ \sqrt{t_2} \end{pmatrix}} \subseteq
\Span_E \sset{\begin{pmatrix} 0 \\ 1 \\ \sqrt{t_2} \end{pmatrix}, \begin{pmatrix} \sqrt{t_0} \\ 1 \\ \sqrt{t_2} \end{pmatrix}} \subseteq
E^3,
\end{gather*}
hence their common maximal torus.

Now
\(q_b(x)\) equals \(t_0\inv x_0^2 + x_1^2 + t_2\inv x_2^2\) for all
\(x = (x_0, x_1, x_2) \in k^3\), so
\(\ker(q_b)\) is trivial, but
\(\ker(q_{b_E})\) equals
\(\Span_E \sset{\begin{smallpmatrix} \sqrt{t_0} \\ 1 \\ 0 \end{smallpmatrix}, \begin{smallpmatrix} 0 \\ 1 \\ \sqrt{t_2} \end{smallpmatrix}}\).
We have that
\(((\Gt\der)_E^{\gamma_E})\smooth\) is the extension trivially across
\(\Span_E \sset{\begin{smallpmatrix} \sqrt{t_0} \\ 1 \\ \sqrt{t_2} \end{smallpmatrix}}\)
of the symplectic group on
\(\ker(q_{b_E})\); but
\((\Gt\der^\gamma)\smooth\) is trivial.
In particular, \((\Gt\der^\gamma)\smooth\) is reductive.
Nonetheless,
\(((\Gt\der^\gamma)\smooth)_\ka\) does not equal
\(((\Gt\der)_\ka^{\gamma_\ka})\smooth\), so
\(\Gt\der^\gamma = (\Gt\der^\gamma)\conn\) is not smoothable.
Thus, Theorem \ref{thm:quass}(\ref{subthm:quass-smoothable}) gives that
\(\gamma\) does not act quasisemisimply on \(\Gt\der\).
\end{example}
}

\begin{cor}[to Lemma \ref{lem:sl-odd} and Proposition \ref{prop:sl-even}]
\label{cor:sl-when-exceptional}
The following statements are equivalent.
\begin{enumerate}[label=(\alph*), ref=\alph*]
\item\label{case:sl-exceptional}
\(\gamma_E\) is an exceptional automorphism of \(\Gt_{\dersub\,E}\).
\item\label{case:p-and-n-even}
\(p\) equals \(2\) and \(n\) is even.
\item\label{case:sl-smooth}
\(\Gt\der^\gamma\) is not smooth.
\end{enumerate}
\end{cor}

\begin{proof}
Since \(\Gt\der^\gamma\) is smooth if and only if
\((\Gt\der)_E^{\gamma_E} = (\Gt_E)\der^{\gamma_E}\) is,
we may, and do, replace \(k\) by \(E\).
The equivalence of
(\ref{case:p-and-n-even}) with
(\ref{case:sl-smooth}) is
Lemma \ref{lem:sl-odd} and Proposition \ref{prop:sl-even}.

Put \(T = \fix(\Tt \cap \Gt\der)^\gamma\).
By Lemma \ref{lem:rd-Dynkin},
the quotient root system \(\Phi(\Gt\der, T)\) is of type
\(\mathsf C_{(n + 1)/2}\) if \(n\) is odd, and of type
\(\mathsf{BC}_{n/2}\) if \(n\) is even.
In particular, if \(n\) is odd, then
\(\Phi(\Gt, T)\) is reduced, so that
\(\gamma\) is not exceptional.
If \(n\) is even, then another application of
Lemma \ref{lem:sl-odd} and Proposition \ref{prop:sl-even}
gives that
\(\Phi(\fix\Gt\der^\gamma, T)\) is the set
\(\mathsf C_{n/2}\), respectively \(\mathsf B_{n/2}\), of
non-multipliable, respectively non-divisible, roots in
\(\mathsf{BC}_{n/2}\), according as
\(p\) equals or does not equal \(2\).
Thus, if \(p\) does not equal \(2\), then (\ref{case:sl-exceptional})
does not hold.
Conversely, suppose that \(p\) equals \(2\).
Using Notation \ref{notn:sl-data}, we have that
\(\gamma(e_0 - e_{n/2})\) equals \(e_{n/2} - e_n\),
so that they have
the same restriction \(a\) to \(T\), and
\(2a\) belongs to \(\mathsf C_{n/2} = \Phi(G, T)\); but
the image of
\map{\gamma + 1}{\Lie(\Gt)_{e_0 - e_{n/2}}}{\Lie(\Gt)_a}
is nonzero and pointwise fixed by \(\gamma\), so
\(a\) belongs to \(\Phi(\Gt^\gamma, T)\).
Thus, if \(p\) equals \(2\), then
(\ref{case:sl-exceptional}) holds.
\end{proof}

Proposition \ref{prop:sl-smoothable} is a special case of
Theorem \ref{thm:quass}(\ref{subthm:quass-smoothable})
that is needed in the proof of the latter.
It isolates the obstruction to upgrading
``smoothable'' to ``smooth'' in that result.

\begin{prop}
\label{prop:sl-smoothable}
Let \(\Gamma\) be a smooth \(k\)-group acting on \(\Gt\der\),
and suppose that
\((\Gt_{\dersub\,\ka}, \Gamma_\ka)\) is quasisemisimple.
\begin{enumerate}[label=(\alph*), ref=\alph*]
\item\label{subprop:sl-usually-smooth}
\((\Gt\der^\Gamma)\conn\) is smooth and reductive unless
\(p\) equals \(2\) and
\((\Gt_{\dersub\,\ka}, \Gamma_\ka)\) is exceptional.
\item\label{subprop:sl-how-reductive}
\(\fix\Gt\der^\Gamma\) is an extension of
a reductive group by
a split unipotent group.
\item\label{subprop:sl-when-reductive}
If \(\fix\Gt\der^\Gamma\) is reductive and
\(C_{\Gt\der}(\fix\Gt\der^\Gamma)\) is of multiplicative type, then
\((\Gt\der^\Gamma)\conn\) is smoothable.
\end{enumerate}
\end{prop}

\begin{proof}
We may, and do, assume, upon replacing \(k\) by \(\ks\), that
\(k\) is separably closed.

Note that, if
\((\Gt\der^\Gamma)\conn\) is smooth and reductive, then
the result is satisfied.

Let \(\Gamma'\) be the subgroup of \(\Gamma\) that
acts on \(\Gt\der\) by inner automorphisms.
In particular, it acts trivially on
\(Z(\Gt\der) = Z(\Gt) \cap \Gt\der\), so its action
may be extended trivially across \(Z(\Gt)\) to
\(Z(\Gt)\cdot\Gt\der = \Gt\).
Put \(\Mt = (\Gt^{\Gamma'})\conn\).

We have that \(\Gamma'_\ka\) acts by
inner automorphisms of
\(\Gt_\ka\) that preserve \((\Bt, \Tt)\), so that
the action factors through
\abmap{\Tt/Z(\Gt_\ka)}{\uInn(\Gt_\ka)}
to give a map
\abmap{\Gamma'_\ka}{\Tt/Z(\Gt_\ka)}.
Lemma \ref{lem:gl-Levi} gives that
\(\Mt_\ka \ldef (\Gt_\ka^{\Gamma'_\ka})\conn\) is
a Levi subgroup of \(\Gt_\ka\), hence reductive.
In particular, \(\Mt_\ka\), hence \(\Mt\), is smooth and reductive; so,
if \(\Gamma'\) is all of \(\Gamma\), then we are done.

Thus, we may, and do, assume that
\(\Gamma'\) is not all of \(\Gamma\).
Since \(\Gamma(k)\) is Zariski dense in \(\Gamma\),
there is some
\(\gamma \in \Gamma(k)\) that acts on \(\Gt\der\) by
an outer automorphism, hence by inversion on
\(Z(\Gt\der) = Z(\Gt) \cap \Gt\der\).
We may thus extend \(\gamma\) to an automorphism of
\(Z(\Gt)\cdot\Gt\der = \Gt\) that acts
by inversion on \(Z(\Gt)\).

Since
\((\Gamma/\Gamma')_\ka = \Gamma_\ka/\Gamma'_\ka\)
embeds into
\(\uAut(\Gt)/\uInn(\Gt) = \uOut(\Gt)\), which
is trivial if \(n\) is at most \(1\) and
has order \(2\) otherwise, we have that
the image of \(\gamma\) generates \(\Gamma/\Gamma'\).
In particular, we may extend
the action of all of \(\Gamma\) to \(\Gt\) so that
\(\Gamma'\) acts trivially on \(Z(\Gt)\) and
\(\gamma\) acts by inversion on \(Z(\Gt)\).
Further,
\(\Gt^\Gamma = \Mt^{\Gamma/\Gamma'}\)
equals
\(\Mt^\gamma\).
Since \(\gamma\) acts by inversion on \(Z(\Gt)\), we have that
\(\fix\Gt^\gamma\) is contained in \(\Gt\der\), hence
equals
\(\fix\Gt\der^\gamma\); and similarly that
\(\fix\Gt_\ka^{\gamma_\ka}\) equals \(\fix(\Gt\der)_\ka^{\gamma_\ka}\).

Since \(\Mt_\ka\) is a Levi subgroup of \(\Gt_\ka = \GL(X \otimes_k \ka)\),
it is the product of general linear groups
corresponding to the weight spaces in \(X \otimes_k \ka\) for the
maximal central torus in \(\Mt_\ka\).
These factors are permuted by \((\Gamma/\Gamma')(\ka)\), and
at most one of them is preserved by \(\gamma_\ka\).
(Concretely, if we form
the Dynkin diagram of \(\Gt_\ka\) with respect to \((\Bt, \Tt)\), then
there is a \(\Gamma(\ka)\)-equivariant bijection between
factors of \(\Mt_\ka\) and connected components of the associated
subdiagram of the Dynkin diagram.)

Since \(\Mt_\ka\) is reductive, so is \(\Mt\).
If \(\At\) is the maximal central torus in \(\Mt\), then
\cite{conrad-gabber-prasad:prg}*{Lemma C.4.4} gives that
\(\At_\ka\) is the maximal central torus in \(\Mt_\ka\).
Since \(\Mt_\ka\) is a Levi subgroup of \(\Gt_\ka\), we have that
\(\Mt_\ka\) equals \(C_{\Gt_\ka}(\At_\ka) = C_\Gt(\At)_\ka\),
so that \(\Mt\) equals \(C_\Gt(\At)\) and hence is
a Levi subgroup of \(\Gt\).
Again, it is the product of general linear groups
corresponding to the weight spaces in \(X\) for \(\At\).
Since the weight spaces in \(X \otimes_k \ka\) for \(\At_\ka\) are just
the base changes to \(\ka\) of
the weight spaces in \(X\) for \(\At\),
it follows that at most one of them is preserved by \(\gamma\).

If no weight space is preserved by \(\gamma\), then
\(\Gt^\Gamma = \Mt^\gamma\) is
a product of general linear groups,
one for each \(\gamma\)-orbit of weight spaces.
In particular,
\((\Gt^\Gamma)\conn = \Gt^\Gamma\) is smooth and reductive.
Thus
\((\Gt_\ka^{\Gamma_\ka})\conn = ((\Gt^\Gamma)\conn)_\ka\) is
also smooth, and so equals
\(\fix\Gt_\ka^{\Gamma_\ka}\), which
we have already observed is
contained in \(\Gt_{\dersub\,\ka}\); so
\((\Gt^\Gamma)\conn\) is contained in \(\Gt\der\), hence
equals \((\Gt\der^\Gamma)\conn\), which is therefore also
smooth and reductive.
Again, in this case, we are done.

Thus we may, and do, assume that
some weight space
\(Y\) is preserved by \(\gamma\).
Note that \(\gamma\) is an involution of \(\GL(Y)\) that
acts by inversion on \(Z(\GL(Y))\).
Since \(\SL(X)^\Gamma\) is
a direct product of a smooth group
(a product of general linear groups) with
\(\SL(Y)^\gamma\),
we may, and do, replace
\(X\) by \(Y\), and \(\Gamma\) by \(\sgen\gamma\).
Since now \(\gamma\) is an involution of \(\Gt\) that
acts by inversion on \(Z(\Gt)\), we may apply
the results of this section.

If \(p\) or \(n\) is odd, then
Lemma \ref{lem:sl-odd} gives
that \(\Gt\der^\gamma\) is connected,
hence equals \((\Gt\der^\gamma)\conn\), and
is smooth and reductive,
so we are done.
Thus we may, and do, finally suppose that
\(p\) equals \(2\) and \(n\) is even.
Then Corollary \ref{cor:sl-when-exceptional} gives that
\((\Gt_{\dersub\,\ka}, \sgen{\gamma_\ka})\) is exceptional, so that
(\ref{subprop:sl-usually-smooth}) is vacuously true; and
(\ref{subprop:sl-how-reductive}) in this case follows from
Corollary \ref{cor:sl-red}.

Finally, suppose that \(\fix\Gt\der^\gamma\) is reductive and
\(C_{\Gt\der}(\fix\Gt\der^\gamma)\) is of multiplicative type.
We now apply Corollary \ref{cor:sl-red} several more times.
First, we observe that reductivity implies that
\[
\Hom((\ker(q_b) + \ker(q_b)^\perp)/{\ker(q_b)^\perp},
	\ker(q_b) \cap \ker(q_b)^\perp)
\]
is trivial.
Second, if \(\ker(q_b) \cap \ker(q_b)^\perp\) is nontrivial, then
its \(b\)-orthogonal space
\(\ker(q_b) + \ker(q_b)^\perp\) is a proper subspace of \(X\); so
\[
\Hom(X/(\ker(q_b) + \ker(q_b)^\perp),
	\ker(q_b) \cap \ker(q_b)^\perp)
\]
is nontrivial, and hence
has a nonzero vector negated by the duality involution
(since \(p\) equals \(2\)).
This contradicts reductivity by Corollary \ref{cor:sl-red},
so in fact
\(\ker(q_b) \cap \ker(q_b)^\perp\) is trivial.
Third and finally,
another application of Corollary \ref{cor:sl-red} gives that
the restriction map
\abmap{(\Gt\der^\gamma)\smooth}{\Sp(\ker(q_b), b')}
is an isomorphism.
Thus the factor \(\SL(\ker(q_b)^\perp)\) of
the subgroup \(\SL(\ker(q_b)) \times \SL(\ker(q_b)^\perp)\)
of \(\Gt\der\)
centralizes \(\fix\Gt\der^\gamma\).
Since \(C_{\Gt\der}(\fix\Gt\der^\gamma)\), and hence
\(\SL(\ker(q_b)^\perp)\), is of multiplicative type, we have that
\(\ker(q_b)^\perp\) is at most one dimensional, so that
\(\ker(q_b)\) is at least \(n\) dimensional.
Therefore the subspace \(\ker(q_b) \otimes_k E\) of \(\ker(q_{b_E})\)
is at least \(n\) dimensional; but
\(\ker(q_{b_E})\) is \(n\) dimensional, by
Lemma \ref{lem:q-linear}(\ref{sublem:X'}), so they are equal.
Then Corollary \ref{cor:sl-smoothable} gives that
\(\Gt\der^\gamma\) is smoothable.
This shows (\ref{subprop:sl-when-reductive}).
\end{proof}

\numberwithin{equation}{section}
\section{Proof of Theorem \ref{thm:ka-quass}}
\label{sec:thm:ka-quass}

As in Notation \ref{notn:main}, we
let \(k\) be a field,
\(\Gt\) a connected, reductive \(k\)-group, and
\(\Gamma\) a smooth \(k\)-group acting on \(\Gt\), and
put \(G = \fix\Gt^\Gamma\).
We do not require the particular choice
\(\Gt = \GL_{n + 1}\) of \S\ref{sec:sl-outer}.

{\newcommand\theadhocthm{\ref{thm:ka-quass}}
\begin{adhocthm}
Suppose that \((\Gt_\ka, \Gamma_\ka)\) is quasisemisimple.
	\begin{enumerate}[label=(\arabic*), ref=\arabic*]
	\item
	\(G\) is an extension of
a reductive group by
a split unipotent group.
	\item
	The following statements are equivalent.
		\begin{enumerate}[label=(\alph*), ref=\alph*]
		\item
		\((\Gt_\ks, \Gamma_\ks)\) is quasisemisimple.
		\item
		\((\Gt^\Gamma)\conn\) is smoothable.
		\item
		\(G\) is reductive, and
\(C_\Gt(G)\) is of multiplicative type.
		\item
		There is a torus \(T\) in \(G\) such that
\(T_\ka\) is a maximal torus in \(\fix\Gt_\ka^{\Gamma_\ka}\).
		\item
		There are a
\(\Gamma_\ks\)-stable maximal torus \(\Tt\) in \(\Gt_\ks\), and a
\(\Gamma_\ka\)-stable Borel subgroup of \(\Gt_\ka\) containing
\(\Tt_\ka\).
		\end{enumerate}
	\end{enumerate}
\end{adhocthm}}

\begin{proof}[Proof of
	Theorem \ref{thm:ka-quass}(\ref{subthm:ka-quass-how-reductive}) and
	Theorem \ref{thm:ka-quass}(\ref{subthm:ka-quass-quass})%
		(\ref{case:ka-quass-reductive} \(\implies\)
			\ref{case:ka-quass-smoothable})]
\newcommand\red{\textsup{red}}
Suppose that we have proven the result for \(\Gt\adform\).
Write \(R\textsub u(\fix\Gt\adform^\Gamma)\) for
the unipotent radical of \(\fix\Gt\adform^\Gamma\),
which is split; and
\((\fix\Gt\adform^\Gamma)\red\) for the quotient of
\(\fix\Gt\adform^\Gamma\) by its unipotent radical,
which is reductive.
Then we have by Corollary \ref{cor:fixed-surjective} that
the natural map
\abmap{\fix\Gt^\Gamma}{\fix\Gt\adform^\Gamma} is a quotient,
obviously with kernel \(Z(\Gt) \cap \fix\Gt^\Gamma\).
The pre-image \(S\) of \(R\textsub u(\fix\Gt\adform^\Gamma)\) is
an extension of \(R\textsub u(\fix\Gt\adform^\Gamma)\) by
\(Z(\Gt) \cap \fix\Gt^\Gamma\).
Since \(R\textsub u(\fix\Gt\adform^\Gamma\) is split, we have by
\cite{SGA-3.3}*{Expos\'e XVII, Th\'eor\`eme 6.1.1(A)(ii)}
that
\(S\) is a trivial extension, i.e., is isomorphic to
\((Z(\Gt) \cap \fix\Gt^\Gamma) \times R\textsub u(\fix\Gt\adform^\Gamma)\).
In particular, we may view \(R\textsub u(\fix\Gt\adform^\Gamma)\) as
a subgroup of \(G\).
Then \(G/R\textsub u(\fix\Gt\adform^\Gamma)\) is an extension
\((\fix\Gt\adform^\Gamma)\red\) by \(Z(\Gt) \cap \fix\Gt^\Gamma\), hence
is reductive.
Finally, Corollary \ref{cor:smooth-surjective} shows that
(\ref{subthm:ka-quass-quass})(\ref{case:ka-quass-smoothable})
is unchanged if we replace \(\Gt\) by \(\Gt\adform\).

Thus we may, and do, assume, upon
replacing \(\Gt\) by \(\Gt\adform\), that
\(\Gt\) is adjoint.
Since a unipotent group is split if and only if
it becomes so after separable base change, and since
formation of the unipotent radical commutes with
separable base change, we may, and do, assume,
upon replacing \(k\) by \(\ks\), that
\(k\) is separably closed.
Then, by Remark \ref{rem:fixed-simple}, we may, and do, assume,
upon replacing \(\Gt\) by
an almost-simple component \(\Gt_1\) and
\(\Gamma\) by \(\stab_\Gamma(\Gt_1)\), that
\(\Gt\) is almost simple
(hence simple, because it is adjoint).

Unless \(p\) equals \(2\) and
\((\Gt_\ka, \Gamma_\ka)\) is exceptional,
Theorem \ref{thm:quass}%
	(\ref{subthm:quass-smooth},%
	\ref{subthm:quass-reductive})
gives that
\(((\Gt^\Gamma)\conn)_\ka = (\Gt_\ka^{\Gamma_\ka})\conn\) is
smooth and reductive, so that
\((\Gt^\Gamma)\conn\) is also smooth and reductive.
In particular,
\((\Gt^\Gamma)\conn\) equals \(\fix\Gt^\Gamma = G\), which
is therefore itself reductive; and
(\ref{subthm:ka-quass-quass})(\ref{case:ka-quass-smoothable}) holds
(hence is certainly implied by
(\ref{subthm:ka-quass-quass})(\ref{case:ka-quass-reductive})).

Thus we may, and do, assume for the remainder of the proof that
\(p\) equals \(2\) and
\((\Gt_\ka, \Gamma_\ka)\) is exceptional.
Remark \ref{rem:gp-nred-facts}(\ref{subrem:gp-nred-type})
gives that \(\Gt\) is of type \(\mathsf A_{2n}\) for
some positive integer \(n\).
We shall use twice the consequence of
Corollary \ref{cor:fixed-surjective} that
\(G = \fix\Gt^\Gamma\) is (isomorphic to)
the quotient of \(\fix\Gt\scform^\Gamma\) by
\(Z(\Gt\scform) \cap \fix\Gt\scform^\Gamma\).

We begin by proving
(\ref{subthm:ka-quass-how-reductive}).
Proposition \ref{prop:sl-smoothable}(\ref{subprop:sl-how-reductive})
gives that \(\fix\Gt\scform^\Gamma\) is
an extension of a reductive group
\((\fix\Gt\scform^\Gamma)\red\)
by a split unipotent group \(R\textsub u(\fix\Gt\scform^\Gamma)\).
Since the multiplicative-type group
\(Z(\Gt\scform) \cap \fix\Gt\scform^\Gamma\) necessarily intersects
the (split) unipotent group \(R\textsub u(\fix\Gt\scform)^\Gamma\)
trivially, we have that
\(G\) is an extension by
\(R\textsub u(\fix\Gt\scform^\Gamma)\) of
a group that is a quotient of \((\fix\Gt\scform^\Gamma)\red\), and so
reductive.
This shows (\ref{subthm:ka-quass-how-reductive}).

Next we show that
(\ref{subthm:ka-quass-quass})(\ref{case:ka-quass-reductive}) implies
(\ref{subthm:ka-quass-quass})(\ref{case:ka-quass-smoothable}),
beginning by assuming that
(\ref{subthm:ka-quass-quass})(\ref{case:ka-quass-reductive}) holds,
i.e., that \(G\) is reductive and
\(C_\Gt(G)\) is of multiplicative type.
Then, first,
\(R\textsub u(\fix\Gt\scform^\Gamma)\) is trivial, so
\(\fix\Gt\scform^\Gamma\) equals \((\fix\Gt\scform^\Gamma)\red\), and hence is
reductive.
Second, the restriction to
\(C_{\Gt\scform}(\fix\Gt\scform^\Gamma)\) of
the natural quotient map \abmap{\Gt\scform}\Gt\
has image in the multiplicative-type group \(C_\Gt(G)\).
Thus
\(C_{\Gt\scform}(\fix\Gt\scform^\Gamma)\) is
a central extension of a multiplicative-type group by
the multiplicative-type group \(Z(\Gt\scform)\), hence is
itself multiplicative
\cite{milne:algebraic-groups}*{Corollary 12.22}.
Then
Proposition \ref{prop:sl-smoothable}(\ref{subprop:sl-when-reductive})
gives that \((\Gt\scform^\Gamma)\conn\), hence,
by Corollary \ref{cor:smooth-surjective}, also
\((\Gt^\Gamma)\conn\), is smoothable.
That is, (\ref{subthm:ka-quass-quass})(\ref{case:ka-quass-smoothable})
holds.
\end{proof}

\begin{proof}[Proof of
	Theorem \ref{thm:ka-quass}(\ref{subthm:ka-quass-quass})%
		(\ref{case:ka-quass-quass} \(\iff\)
		\ref{case:ka-quass-smoothable} \(\iff\)
		\ref{case:ka-quass-torus} \(\iff\)
		\ref{case:ka-quass-Borel} \(\implies\)
		\ref{case:ka-quass-reductive})]
We may, and do, assume, upon replacing \(k\) by \(\ks\), that
\(k\) is separably closed.

First assume (\ref{case:ka-quass-quass}).
Then (\ref{case:ka-quass-Borel}) is obvious;
Theorem \ref{thm:quass}(\ref{subthm:quass-smoothable}) gives
(\ref{case:ka-quass-smoothable}); and
Theorem \ref{thm:quass}(\ref{subthm:quass-reductive}) and
Proposition \ref{prop:quass-rough}(\ref{subprop:quass-up}) give
(\ref{case:ka-quass-reductive}).

Remark \ref{rem:conn-smooth} shows that
(\ref{case:ka-quass-smoothable}) implies that
\(G_\ka = (\fix\Gt^\Gamma)_\ka\) equals
\(\fix\Gt_\ka^{\Gamma_\ka}\).
Then
\cite{conrad-gabber-prasad:prg}*{Lemma C.4.4} gives
(\ref{case:ka-quass-torus}).

Assuming (\ref{case:ka-quass-torus}), we have by
Proposition \ref{prop:quass-rough}(\ref{subprop:quass-up}) that
\(C_\Gt(T)_\ka = C_{\Gt_\ka}(T_\ka)\) is
a maximal torus in \(\Gt_\ka\), so
\(C_\Gt(T)\) is a maximal torus in \(\Gt\).
Then
Lemma \ref{lem:quass-from-Levi} gives that
\((\Gt_\ks, \Gamma_\ks)\) is quasisemisimple, which is
(\ref{case:ka-quass-quass}).

Finally, assume (\ref{case:ka-quass-Borel}).
Proposition \ref{prop:quass-rough}(\ref{subprop:quass-down}) gives that
\(\fix\Tt_\ka^{\Gamma_\ka}\) is
a maximal torus in \(\fix\Gt_\ka^{\Gamma_\ka}\).
Since all subgroups of tori are smoothable, we have by
Remark \ref{rem:conn-smooth} that
\((\fix\Tt^\Gamma)_\ka\) equals \(\fix\Tt_\ka^{\Gamma_\ka}\), hence is
a maximal torus in \(\fix\Gt_\ka^{\Gamma_\ka}\),
giving (\ref{case:ka-quass-torus}).
\end{proof}

\begin{cor}
\label{cor:ss-is-quass}
Suppose that \(k\) is separably closed.
If \(\gamma\) is a semisimple automorphism of \(\Gt\),
in the sense of \cite{steinberg:endomorphisms}*{\S7, p.~51},
then \(\gamma\) is a quasisemisimple automorphism of \(\Gt\).
\end{cor}

\begin{proof}
We have by \cite{steinberg:endomorphisms}*{Theorem 7.5} that
\(\gamma_\ka\) is a quasisemisimple automorphism of \(\Gt_\ka\).
Since \(\Gt^\gamma\) is smooth, by
\cite{conrad-gabber-prasad:prg}*{Proposition A.8.10(2)},
the result follows from
Theorem \ref{thm:ka-quass}(\ref{subthm:ka-quass-quass}).
\end{proof}

{\newcommand\red{\textsup{red}}
Theorem \ref{thm:quass} is nearly subsumed by
Theorem \ref{thm:ka-quass}, except that
the latter has nothing to say about
spherical buildings.
Conjecture \ref{conj:ka-quass-spherical-bldg} describes
an analogue of
Theorem \ref{thm:quass}(\ref{subthm:quass-spherical-bldg}) in
the setting of
Theorem \ref{thm:ka-quass}.

It is not hard to prove
the existence of the set \(\SS(G)\red\) and
the map \(i\) as in Conjecture \ref{conj:ka-quass-spherical-bldg},
but we do not do it here.
Once existence is proven,
uniqueness is obvious, and it is clear that
\(i(\SS(G)\red)\) lies in
\(\SS(\Gt) \cap \SS(\Gt_\ks)^{\Gamma(\ks)}\).
Determining whether \(i\) is surjective will be
the subject of future work.

\begin{conj}
\label{conj:ka-quass-spherical-bldg}
Suppose that \((\Gt_\ka, \Gamma_\ka)\) is quasisemisimple.
Write \(G\red\) for the maximal
pseudo-reductive quotient of \(G\), which is
reductive by
Theorem \ref{thm:ka-quass}(\ref{subthm:ka-quass-how-reductive}).
There are a unique
subset \(\SS(G)\red\) of \(\SS(G\red)\) and
map \map i{\SS(G)\red}{\SS(\Gt)} with
the following properties.
For every (split) torus \(S\) in \(G\red\),
the subset
\(\SS(S) \cap \SS(G)\red\) of
\(\SS(G\red)\) contains precisely the rays through
elements \(\lambda \in \bX_*(S) \otimes_\Z \R \setminus \sset0\) such that
\(P_G(\lambda)\) contains the unipotent radical of \(G\); and
the diagram
\[\xymatrix{
\SS(S) \ar[r] & \SS(\Gt) \\
\SS(S) \cap \SS(G)\red \ar@{^(->}[u]\ar[r] & \SS(G)\red \ar[u]_i
}\]
commutes.
Then \(i\) is a bijection from
\(\SS(G)\red\) onto \(\SS(\Gt) \cap \SS(\Gt_\ks)^{\Gamma(\ks)}\).
\end{conj}
}

Corollary \ref{cor:quass-imperfect}
is quite close to
\cite{lemaire:twisted-characters}*{Th\'eor\`eme 4.6}.
A special case of this latter result is proven by a different method in
\cite{adler-lansky:lifting1}*{Lemma A.1}.
Note that it provides a practical way to verify
Theorem \ref{thm:ka-quass}(\ref{subthm:ka-quass-quass})(\ref{case:ka-quass-quass}), hence
the equivalent conditions of
Theorem \ref{thm:ka-quass}(\ref{subthm:ka-quass-quass}).

\begin{cor}
\label{cor:quass-imperfect}
If \(G\) contains a split torus \(S\) such that
\(S_\ka\) is a maximal torus in
\(\fix\Gt_\ka^{\Gamma_\ka}\), then
\((\Gt, \Gamma)\) is quasisemisimple.
\end{cor}

\begin{proof}
Proposition \ref{prop:quass-rough}(\ref{subprop:quass-up})
gives that
\(C_\Gt(S)_\ka = C_{\Gt_\ka}(S_\ka)\) is
a maximal torus in \(\Gt_\ka\), so
\(C_\Gt(S)\) is a maximal torus in \(\Gt\).
The result follows from Lemma \ref{lem:quass-from-Levi}.
\end{proof}

\numberwithin{equation}{subsection}
\section{Proof of Theorem \ref{thm:loc-quass}}
\label{sec:thm:loc-quass}

In this section,
\(k\) is any field, and
\(\Gt\) is a connected, reductive \(k\)-group.
We let \(p\) be \(1\) or a prime number, and assume that
\(k\) has characteristic exponent \(p\) or \(1\).

We allow the possibility of characteristic exponent \(1\) to
handle valued fields of mixed characteristic \(p\) in
\cite{adler-lansky-spice:actions3}; but, for
our applications in this paper,
we are most interested in
the case where \(p\) is a prime, and
\(k\) has characteristic exponent \(p\).

Beginning with \S\ref{subsec:thm:loc-quass}, we will
impose the full hypotheses of
Theorem \ref{thm:loc-quass}, and
assume that
\(k\) has characteristic exponent \(p\); but
we do not do so yet.

\subsection{Unipotent, or topologically unipotent, automorphisms}
\label{subsec:quass-unip}

In this subsection, we are mostly interested in
order-\(p\) automorphisms that are
assumed to be quasisemisimple.
We begin, however, with a family of automorphisms for which
quasisemisimplicity is automatic.

Specifically,
Lemma \ref{lem:A_{2n}-quass} shows that
involutions on groups of type \(\mathsf A_{2n}\) in
characteristic \(2\),
which, for many purposes, are the \emph{hardest} case to handle
(see, for example, Proposition \ref{prop:no-mult-fixed}),
are actually \emph{easier} to handle in one respect:
that they are all quasisemisimple.
This should be surprising; see
Remark \ref{rem:auto-quass} on its rarity.

\begin{lem}
\label{lem:A_{2n}-quass}
Suppose that \(k\) is algebraically closed.
If
\(\Gt\) is an almost-simple group of type \(\mathsf A_{2n}\)
for some positive integer \(n\),
and
\(\gamma\) is an outer involution of \(\Gt\), then
\(\gamma\) is quasisemisimple.
\end{lem}

\begin{proof}
By \cite{steinberg:endomorphisms}*{Theorem 7.5},
the result holds if
\(p\) is not \(2\) or
\(k\) has characteristic exponent \(1\).
Thus we may, and do, assume that
\(p\) equals \(2\) and
\(k\) has characteristic exponent \(p = 2\).
By \cite{steinberg:endomorphisms}*{Theorem 7.2}, there is
a Borel subgroup \(\Bt\) of \(\Gt\) that is
preserved by \(\gamma\).
Let \(\Tt\) be a maximal torus in \(\Bt\).
By, for example, \cite{springer:corvallis}*{Proposition 2.13},
there is a quasisemisimple involution \(\gamma_0\) of
\(\Gt\) that preserves \((\Bt, \Tt)\) and
has the same image in the outer-automorphism group as \(\gamma\).
We claim that \(\gamma\) is conjugate to \(\gamma_0\) in
\(\Aut(\Gt)\) (indeed, in \(\Gt\adform(k)\)).
Since this claim is unaffected if we replace \(\gamma\) by
a \(\Gt\adform(k)\)-conjugate, we do so freely.

Write \(\Bt\adform\) and \(\Tt\adform\) for
the images of \(\Bt\) and \(\Tt\) in \(\Gt\adform\).
Since \(\gamma\gamma_0\inv\) is inner and preserves \(\Bt\),
it belongs to \(\Bt\adform(k)\).
Write \(\bt_0 = \gamma\gamma_0\inv\).
Since \(\bt_0\gamma_0 = \gamma\) is an involution,
we have that
\(\bt_0\cdot\gamma_0(\bt_0)\) is trivial.
Let \(\Ut\) be the unipotent radical of \(\Bt\adform\) and
write \(\bt_0 = \tt_0\ut_0\), with
\(\tt_0 \in \Tt\adform(k)\) and \(\ut_0 \in \Ut(k)\).
By \cite{digne-michel:quass}*{Lemma 1.2(iii)},
we may write \(\tt_0\) as \(\tt_+\tt_1\gamma_0(\tt_1)\inv\), where
\(\tt_1\) belongs to \(\Tt\adform(k)\) and
\(\tt_+\) is a \(k\)-rational point of
the maximal subtorus of \(\Tt\adform\) on which
\(\gamma_0\) acts trivially.
Then
\(\tt_+^2 = \tt_0\cdot\gamma_0(\tt_0)\)
equals
\(\bt_0\gamma_0(\bt_0)\cdot \Int(\gamma_0(\bt_0))\inv \ut_0\inv\cdot \gamma_0(\ut_0)\inv
 = \ut_1\gamma_0(\ut_0)\inv\).
 Since \(\Int(\gamma_0(\bt_1))\inv\ut_0\inv\)
belongs to \(\Ut(k)\), \(\tt_+^2\) belongs to
\(\Tt(k) \cap \Ut(k)\), hence is trivial.
Since \(p\) equals \(2\), this implies that
\(\tt_+\) is trivial.
Then
\(\tt_1\inv\cdot\gamma\cdot\tt_1 =
\tt_1\inv\cdot\bt_0\cdot\gamma_0\cdot\tt_1\) equals
\(\ut_1\cdot\gamma_0\), where
\(\ut_1 \ldef \Int(\gamma_0(\tt_1))\inv\ut_0\in\Ut(k)\).
We may, and do, replace \(\gamma\) by
\(\tt_1\inv\cdot\gamma\cdot\tt_1\).


For each positive integer \(h\),
write \(\Ut_{\ge h}\) for the subgroup of \(\Ut\)
generated by root subgroups corresponding to
roots in \(\Phi(\Gt, \Tt)\) of height at least \(h\).
Thus, each \(\Ut_{\ge h}\) is preserved by \(\gamma_0\).

We now proceed by induction on \(h\).
Fix a positive integer \(h\), and
suppose that we have arranged, after
replacing \(\gamma\) by
a \(\Gt\adform(k)\)-conjugate if needed, that
there is an element \(\ut_h\) of \(\Ut_{\ge h}(k)\)
such that
\(\gamma\) equals \(\ut_h\cdot\gamma_0\).
(The above element \(u_1\) satisfies this condition when \(h=1\).)
We prove the existence of an analogous element \(\ut_{h+1}\in\Ut_{\ge h+1}(k)\).
Note that, since
\(\gamma^2\) is trivial, so is
\(\ut_h\cdot\gamma_0(\ut_h)\);
i.e., \(\ut_h\) is inverted by \(\gamma_0\).

\newcommand\uu{{\mathfrak u}}
\newcommand\uut{{\widetilde\uu}}
We now make a number of computations backed up by
the Chevalley commutation relations
\cite{adler:thesis}*{Proposition 1.2.3}.
We have that \(\Ut_{\ge h + 1}\) is normal in \(\Ut_{\ge h}\).
The unique \(\Tt\)-equivariant linear structures on
the various root groups for \(\Tt\) in \(\Gt\)
\cite{conrad-gabber-prasad:prg}*{Lemma 2.3.8}
piece together to a
\(\Tt\)-equivariant linear structure on
\(\Ut_{\ge h}/\Ut_{\ge h + 1}\).
Let us denote this structure by \(\exp_h\).
Uniqueness of the structures on the individual root groups
implies that \(\exp_h\) is \(\gamma_0\)-equivariant.

There are linearly disjoint,
sub-\(\Tt\)-representations \(\uut^\pm_h\) of
\(\Lie(\Ut_{\ge h})\) such that
\(\gamma_0(\uut^+_h)\) equals
\(\uut^-_h\) and
\(\Lie(\Ut_{\ge h})/
	(\uut^+_h + \uut^-_h + \Lie(\Ut_{\ge h + 1}))\)
is trivial or one dimensional, according as
\(h\) is even or odd.
The subspaces \(\uut^\pm_h\) are not uniquely determined, but
we only need their existence.
If \(h\) is even, then there is
a unique root \(\betat_h\) of height \(h\) that is
pre-divisible, in the sense of
Remark \ref{rem:red-to-nred-facts}, and it is
the weight of \(\Tt\) on
\(\Lie(\Ut_{\ge h})/
	(\uut^+_h + \uut^-_h + \Lie(\Ut_{\ge h + 1}))\).
(Specifically, in the Bourbaki numbering
\cite{bourbaki:lie-gp+lie-alg_4-6}*{Chapter VI, Plate I},
except that we write \(\alphat\) in place of just \(\alpha\),
we have that
\(\betat_h\) equals
\(\alphat_{n - h/2 + 1} + \dotsb + \alphat_{n + h/2}\).)
For convenience, we put \(\betat_h = 0\) if \(h\) is odd.

Choose \(\Xt^\pm_h \in \uut^\pm_h\) such that
\(\exp_h\inv(\ut_h) - (\Xt^+_h + \Xt^-_h)\)
belongs to the \(\betat_h\)-weight space for \(\Tt\) in
\(\Lie(\Ut_{\ge h}/\Ut_{\ge h + 1})\)
(hence is trivial if \(h\) is odd).
Since \(\ut_h\) is inverted by \(\gamma_0\),
\(\betat_h\) is fixed by \(\gamma_0\), and
\(\exp_h\) is \(\gamma_0\)-equivariant, we have that
\(\Xt^-_h\) equals \(-\gamma_0(\Xt^+_h)\).
(Of course the minus sign has no effect, but
we include it to be suggestive.)
Write \(\vt^+_h\) for any element of \(\Ut_{\ge h}(k)\)
such that
\(\Xt^+_h\) belongs to the coset \(\exp_h\inv(\vt^+_h)\).
Then
\(\exp_h\inv\bigl((\vt^+_h)\inv\cdot\ut_h\cdot\gamma_0(\vt^+_h))\bigr)\)
belongs to the \(\betat_h\)-weight space in
\(\Lie(\Ut_{\ge h}/\Ut_{\ge h + 1})\).
Thus we may, and do, assume, upon
replacing \(\gamma\) by \((\vt^+_h)\inv\gamma\vt^+_h\), that
\(\exp_h\inv(\ut_h)\) belongs to
the \(\betat_h\)-weight space.

In particular, if \(h\) is odd, then
\(\exp_h\inv(\ut_h)\) is trivial, so
we may put \(\ut_{h + 1} = \ut_h\).
Thus we may, and do, assume that \(h\) is even.
Let \(\Xt^0_h\) be the vector in
\(\exp_h\inv(\ut_h)\) that belongs to
the \(\betat_h\)-weight space.
Note that, since \(\ut_h\) is inverted by \(\gamma_0\),
it follows that
\(\Xt^0_h\) is negated, and hence fixed, by \(\gamma_0\).
(Actually this does not need any special condition on \(\ut_h\),
since it is not hard to show that \(\gamma_0\) acts trivially on
\(\Lie(\Gt)_{\betat_h}\).)
By Remark \ref{rem:red-to-nred-facts}(\ref{subrem:pre-mult-div}),
the restriction of \(\betat_h\) to \(\fix\Tt^\gamma\) is
divisible, hence may be written as
\(2a\) for some root \(a \in \Phi(\Gt, \fix\Tt^\Gamma)\).
By Proposition \ref{prop:square-pinned}(\ref{subprop:square-pinned})
(applied to \(C_\Gt(\ker(a))\conn\)),
there is a unique element
\(\Xt^0_{h/2}\) in \(\Lie(\Gt)_a^{\gamma_0}\) such that
\((\Xt^0_{h/2})^{[2]}\) equals \(\Xt^0_h\).
Concretely, by
Proposition \ref{prop:square-pinned}(\ref{subprop:root-commute}),
if we let \(\smashsset{\alphat_{h/2}, \alphat_{h/2}'}\)
be an exceptional pair for
\((\Psi(\Gt, \Tt), \sgen{\gamma_0}(k))\) extending \(a\),
then \(\Xt^0_{h/2}\) equals \(\Xt_{h/2} + \Xt_{h/2}'\)
for some
\(\Xt_{h/2} \in \Lie(\Gt)_{\alphat_{h/2}}\) and
\(\Xt_{h/2}' \in \Lie(\Gt)_{\alphat_{h/2}'}\), and
\(\Xt^0_h\) equals \([\Xt_{h/2}, \Xt_{h/2}']\).
Since \(\Xt^0_h\) is fixed by \(\gamma_0\), so is
\(\Xt^0_{h/2}\), so
\(\Xt_{h/2}'\) equals \(\gamma_0(\Xt_{h/2})\).
Now write \(\vt_{h/2}\) for the element of the coset
\(\exp_{h/2}(\Xt_{h/2})\) that lies in the
\(\alphat_{h/2}\)-root group, and
put \(\vt^0_{h/2} = \vt_{h/2}\gamma_0(\vt_{h/2})\).
Then \(\ut_h\) and
\([\vt_{h/2}, \gamma_0(\vt_{h/2})]\) have
the same image in \(\Ut_{\ge h}/\Ut_{\ge h + 1}\),
so
\[(\vt^0_{h/2})\inv\cdot\gamma\cdot\vt^0_{h/2} =
\gamma_0(\vt_{h/2})\inv\vt_{h/2}\inv
	\cdot\ut_h\cdot\gamma_0\cdot
	\vt_{h/2}\gamma_0(\vt_{h/2})\]
equals
\(\ut_{h + 1}\gamma_0\), where
\(\ut_{h + 1} \ldef
\gamma_0(\vt_{h/2})\inv\vt_{h/2}\inv\cdot\ut_h\cdot
	\gamma_0(\vt_{h/2})\vt_{h/2}\)
belongs to \(\Ut_{\ge h + 1}(k)\).

Since \(\Ut_{\ge h}\) is trivial for all
sufficiently large positive integers \(n\), eventually
our process of successive replacements will have replaced
\(\gamma\) by \(\gamma_0\), which is quasisemisimple
by assumption.
\end{proof}

\begin{rem}
\label{rem:auto-quass}
The ``automatic quasisemisimplicity'' property of
Lemma \ref{lem:A_{2n}-quass} is
specific to \(\mathsf A_{2n}\),
in the following sense.
For every other connected Dynkin diagram that admits
an automorphism of order \(p\),
there is at least one node fixed by all automorphisms
of the Dynkin diagram;
and, if \(\Gt\) is a quasisplit
reductive group of that type over \(k\),
then there is a quasisemisimple automorphism
\(\gamma_0\) of \(\Gt\) that
acts trivially on the corresponding root subgroup.
If \(u\) is an element of that root subgroup,
then \(\gamma_0\Int(u)\) is not quasisemisimple.
(This can be shown by combining
Proposition \ref{prop:parabolic-Borus}%
	(\ref{subprop:quass-fixed-Levi})
with a computation as in the proof of
Lemma \ref{lem:A_{2n}-quass}.)
See \cite{chernousov-elduque-knus-tignol:d4}*{\S11} for
the case of \(\mathsf D_4\).
\end{rem}

For the remainder of
\S\ref{subsec:quass-unip},
let \(\gamma\) be a
quasisemisimple, outer automorphism of \(\Gt\) such that
\(\gamma^p\) is trivial.
(Remember that we call the trivial automorphism outer, so
our results here include the possibility that
\(\gamma\) itself is trivial, though
of course they have little content in that case.)
Remember that
\(k\) has characteristic exponent \(1\) or \(p\).
If \(k\) has characteristic exponent \(p\), then
`outer' is redundant; in this case,
every quasisemisimple automorphism \(\gamma\) of \(\Gt\)
satisfying \(\gamma^p = 1\) is already outer.

Recall that notation like
\(\sgen\gamma\) stands for
the \emph{algebraic} group, not the \emph{abstract} group,
generated by \(\gamma\).
Since \(\gamma\) has finite order, there is little distinction;
we have that \(\sgen\gamma\) is the constant group such that
\(\sgen\gamma(k)\) is the abstract group generated by \(\gamma\).
However, for an element not known to be of finite order,
such as the element \(s\) of Lemma \ref{lem:ss-fixed},
it is possible that \(\sgen s(k)\) is strictly bigger than
the abstract group generated by \(s\).

Lemma \ref{lem:ss-fixed} does not assume that
\(k\) has characteristic exponent \(p\).
If \(k\) \emph{does} have characteristic exponent \(p\), then
the semisimplicity of \(s\) already implies that
the hypothesis of Lemma \ref{lem:ss-fixed} is satisfied.
If \(k\) is a valued field of mixed characteristic \(p\) and
\(s\) has finite order
(which already implies that it is semisimple), then
\(\pi_0(\sgen s)(k)\) is the abstract group generated by \(s\), so
the hypothesis of Lemma \ref{lem:ss-fixed} is
equivalent to \(s\) being topologically semisimple, in
the sense that its order is relatively prime to
the residue characteristic \(p\).

\begin{lem}
\label{lem:ss-fixed}
Suppose that \(s\) is a
semisimple, \(\gamma\)-fixed element of \(\Gt(k)\) such that
the order of \(\pi_0(\sgen s)(k)\) is relatively prime to \(p\).
Put \(Z = ((Z(\Gt)\conn)^\gamma)\smooth\).
Then there is a maximal torus \(T'\) in
\(\fix(\Gt/Z)^\gamma\) such that
the image of \(s\) belongs to \(T'(k)\).
\end{lem}

\begin{proof}
Suppose first that \(k\) is algebraically closed.

Since passing to the maximal subgroup scheme does not
affect the group of \(k\)-points, we have that
\(Z(k) =
((Z(\Gt)\conn)^\gamma)\smooth(k)\)
equals
\((Z(\Gt)\conn)^\gamma(k) =
Z(\Gt)\conn(k)^\gamma =
(Z(\Gt)\conn)\smooth(k)\).
Recall that we do not always have that
the maximal smooth subgroup of
a connected subgroup is
connected;
but, over a perfect field like \(k\), we do have that
\((Z(\Gt)\conn)\smooth\) is
the maximal reduced subscheme of
the connected subscheme
\(Z(\Gt)\conn\), hence is
itself connected.
(Actually, we do not even need that \(k\) is perfect here,
since \(Z(\Gt)\) is of multiplicative type.)
Thus
\((Z(\Gt)\conn)\smooth\) is both smooth and connected, hence
contained in
\((Z(\Gt)\smooth)\conn\), which we have agreed to denote by
\(Z(\Gt)\smooth\conn\).
The reverse containment is automatic
(for any group scheme over any field), so we have equality.
Thus,
\(Z(k) = (Z(\Gt)\conn)\smooth(k)^\gamma\) equals
\(Z(\Gt)\smooth\conn(k)^\gamma\).

Because \(k\) is algebraically closed, the sequence
\[\xymatrix{
Z(\Gt)\smooth\conn(k) \times \Gt\scform(k) \ar[r] & \Gt(k) \ar[r] & 1
}\]
is exact.
Thus, arguing as in the proof of
\cite{steinberg:endomorphisms}*{Lemma 9.2}, we may apply
\cite{steinberg:endomorphisms}*{\S4.5} to get an exact sequence
\[\xymatrix{
Z(\Gt)\smooth\conn(k)^\gamma \times \Gt\scform(k)^\gamma \ar[r] & \Gt(k)^\gamma \ar[r] & H^1(\sgen\gamma(k), \Zt(k)),
}\]
where we have put
\(\Zt = \ker(\abmap{Z(\Gt)\smooth\conn \times \Gt\scform}\Gt)\).
It is a general fact
\cite{serre:galois}*{Ch.~I, Proposition 2.4.9}
that \(H^1(\sgen\gamma(k), \Zt(k))\) is
annihilated by the order of \(\sgen\gamma(k)\), hence,
in particular, by \(p\).
(In this case, we can see this fact concretely by observing that
\(H^1(\sgen\gamma(k), \Zt(k))\) may be realized as
a subgroup of the co-invariant quotient \(\Zt(k)_\gamma\) by
evaluating at \(\gamma\);
that \(1 + \dotsb + \gamma^{p - 1}\) equals
the \(p\)-power map on \(\Zt(k)_\gamma\); and that
it annihilates the subgroup \(H^1(\sgen\gamma(k), \Zt(k))\)
by the cocycle condition.)
Thus \(s^p\) lifts to
\[
Z(\Gt)\smooth\conn(k)^\gamma \times \Gt\scform(k)^\gamma =
(Z(\Gt)\smooth\conn)^\gamma(k) \times \Gt\scform^\gamma(k) =
Z(k) \times (\Gt\scform^\gamma)\smooth(k).
\]
Since \((\Gt\scform^\gamma)\smooth\) is connected
\cite{steinberg:endomorphisms}*{Theorem 8.2},
so is its image in \((\Gt^\gamma)\smooth\), so
we have that \(s^p\) lifts to
\(Z(k)\cdot\fix\Gt^\gamma(k)\).

Now \(\sgen s/\sgen{s^p}\) is \'etale,
so \(\sgen s\) and \(\sgen{s^p}\) have
the same identity component.
Therefore, we may regard \(\pi_0(\sgen{s^p})\) as
a subgroup of \(\pi_0(\sgen s)\).
We have that the index of \(\pi_0(\sgen{s^p})(k)\) in
\(\pi_0(\sgen s)(k)\) divides \(p\), hence
equals \(1\).
It follows that
\(\sgen{s^p}(k)\) equals \(\sgen s(k)\), and,
in particular, contains \(s\); so
\(s\) belongs to \(Z\cdot\fix\Gt^\gamma\); so
the image of \(s\) in \(\Gt/Z\) belongs to the image there of
\(\fix\Gt^\gamma\), hence to
\(\fix(\Gt/Z)^\gamma\).
By \cite{borel:linear}*{Corollary 18.12},
we have that \(s\) belongs to a maximal torus in
\(\fix(\Gt/Z)^\gamma\).

Now drop the assumption that \(k\) is algebraically closed.
Since groups of multiplicative type, such as
\((Z(\Gt)\conn)^\gamma\), are
smoothable, we have that
\(Z_\ka = (((Z(\Gt)\conn)^\gamma)\smooth)_\ka\) equals
\(((Z(\Gt_\ka)\conn)^{\gamma_\ka})\smooth\).
Since Lemma \ref{lem:quass-by-iso} gives that
\(\gamma\) acts quasisemisimply on \(\Gt/Z\),
it follows from Theorem \ref{thm:quass}(\ref{subthm:quass-smoothable}) that
\(((\Gt/Z)^\gamma)\conn\) is smoothable, so
Remark \ref{rem:conn-smooth} gives that
\((\fix(\Gt/Z)^\gamma)_\ka\) equals
\(\fix(\Gt_\ka/Z_\ka)^{\gamma_\ka}\).
Thus the special case that we have already proven shows that
the image of \(s_\ka\) in \((\Gt/Z)(\ka)\) belongs to
a maximal torus in \((\fix(\Gt/Z)^\gamma)_\ka\).
We note two consequences.
First, we see that \(s_\ka\) is a \(\ka\)-rational point of
some maximal torus in
\(C_{(\fix(\Gt/Z)^\gamma)_\ka}(s_\ka)\), hence,
by their conjugacy
\cite{conrad-gabber-prasad:prg}*{Theorem C.2.3}, of
all such maximal tori.
Second, we see that
\(C_{\fix(\Gt/Z)^\gamma}(s)_\ka =
C_{(\fix(\Gt/Z)^\gamma)_\ka}(s_\ka)\) has the same rank as
\((\fix(\Gt/Z)^\gamma)_\ka\),
hence that
\(C_{\fix(\Gt/Z)^\gamma}(s)\) has the same absolute rank as
\(\fix(\Gt/Z)^\gamma\).
That is,
\(C_{\fix(\Gt/Z)^\gamma}(s)\) contains a maximal torus \(T'\) in
\(\fix(\Gt/Z)^\gamma\).
Since \(T'_\ka\) is still maximal in
\(C_{\fix(\Gt/Z)^\gamma}(s)_\ka\), we have that
\(s_\ka\) belongs to \(T'(\ka)\), so
\(s\) belongs to \(T'(k)\).
\end{proof}

\begin{lem}
\label{lem:inductive-step-Z|Lie}
Suppose that
\(k\) has characteristic exponent \(p\) and
\(\Gt^\gamma\) is smooth.
Let
\(\Gamma'\) be a \(\gamma\)-stable subgroup of \(\uAut(\Gt)\) such that
\((\Gt, \Gamma')\) is quasisemisimple.
Fix \(s' \in Z(\fix\Gt^{\Gamma'})^\gamma(k)\) and put
\(\Ht = C_\Gt(s')\conn\), or
let \(\mathfrak s'\) be
a subspace of \(\Lie(Z(\fix\Gt^{\Gamma'})^\gamma\) and put
\(\Ht = C_\Gt(\mathfrak s')\conn\).
Then
\(\Ht\) is reductive,
\(\gamma\) gives a quasisemisimple automorphism of \(\Ht\), and
\(\Ht^\gamma\) is smooth.
\end{lem}

\begin{proof}
Reductivity follows from
\cite{conrad-gabber-prasad:prg}*{Proposition A.8.12}
or
Corollary \ref{cor:toral-Lie-is-Lie-torus}.

Put
\(\mathfrak s' = \Lie(Z(\fix\Gt\adform^{\Gamma'})^\gamma\).
Let us say that we are
in case (I) if we have put
\(\Ht = C_\Gt(s')\conn\), and
in case (II) if we have put
\(\Ht = C_\Gt(\mathfrak s')\conn\).
We now argue in parallel.

Let \(T_H\) be a maximal torus in \(H\).
Proposition \ref{prop:quass-rough}(\ref{subprop:quass-up})
gives that
\(C_\Gt(T_H)\) is a torus in \(\Gt\), so
\(C_\Ht(T_H)\) is a torus in \(\Ht\).
Thus, Lemma \ref{lem:quass-from-Levi} gives that
\((\Ht, \Gamma')\) is quasisemisimple.

Lemma \ref{lem:quass-by-iso} gives that
\(\gamma\) acts quasisemisimply on \(\Gt\adform\), and
Remark \ref{rem:adj-or-sc-T} and Lemma \ref{lem:smooth-to-iso}
together give that
\(\Gt\adform^\gamma\) is smooth.
We have by
Proposition \ref{prop:quass-facts}%
	(\ref{subprop:quass-split-descends}) that
\(G \ldef (\Gt^\gamma)\conn\) and
\(G' \ldef (\Gt\adform^\gamma)\conn\) are reductive, and
by Corollary \ref{cor:fixed-surjective} that
\abmap G{G'} is a central quotient.
In case (I), we have by
\cite{conrad-gabber-prasad:prg}*{Proposition A.8.10(2)} that
\(C_{\Gt^\gamma}(s') = C_\Gt(s')^\gamma\) is smooth.
Since \(C_\Gt(s')/\Ht\) is \'etale, so is
\(C_\Gt(s')^\gamma/\Ht^\gamma\), so
\(\Ht^\gamma\) is smooth.
Further, since
the characteristic exponent of \(k\) is \(p\), we have by
Lemma \ref{lem:ss-fixed} that
there is a maximal torus \(T_{G'}\) in \(G'\) such that
\(s'\) belongs to \(T_{G'}(k)\).
The corresponding torus \(T_G\) in \(G\) is contained in
\(C_\Gt(s')\conn = \Ht\).
In case (II), since
\(\mathfrak s'\) is a commuting algebra of semisimple elements that
is contained in
\(\Lie(\Gt\adform)^\gamma = \Lie(G')\), we have by
Lemma \ref{lem:toral-Lie-is-Lie-torus} that
there is a maximal torus \(T_{G'}\) in \(G'\) such that
\(\mathfrak s'\) is contained in \(\Lie(T_{G'})\), and by
Corollary \ref{cor:toral-Lie-is-Lie-torus} that
\(C_G(\mathfrak s')\conn\) is smooth.
The torus \(T_G\) in \(G\) corresponding to \(T_{G'}\) is contained in
\(C_\Gt(\mathfrak s')\conn = \Ht\).
Since \(\Gt^\gamma/G\) is \'etale, so is
\(C_{\Gt^\gamma}(\mathfrak s')/C_G(\mathfrak s')\), hence also
\(C_{\Gt^\gamma}(\mathfrak s')/C_G(\mathfrak s')\conn\).
Since \(C_\Gt(\mathfrak s')/\Ht\) is \'etale, so is
\(C_\Gt(\mathfrak s')^\gamma/\Ht^\gamma\).
Since \(C_{\Gt^\gamma}(\mathfrak s')\) equals
\(C_\Gt(\mathfrak s')^\gamma\), and
\(C_G(\mathfrak s')\conn\) is smooth, so is
\(\Ht^\gamma\).

We have now shown, in both cases, that
\(\Ht^\gamma\) is smooth, and that there is a
maximal torus \(T_G\) in \(G\) that is contained in \(\Ht\).
We have by
Proposition \ref{prop:quass-rough}(\ref{subprop:quass-up})
that \(C_\Gt(T_G)\), hence also \(C_\Ht(T_G)\), is
a torus; and then by Lemma \ref{lem:quass-from-Levi} that
\(\gamma\) is a quasisemisimple automorphism of \(\Ht\).
\end{proof}

The following cohomological remark will
be useful in the proofs of
Propositions \ref{prop:E_6-outer} and \ref{prop:D_4-outer}.

\begin{rem}
\label{rem:rel-pos-coboundary}
Suppose that
\((\Gt, \Gamma)\) is a
quasisemisimple reductive datum, and
\((\Bt, \Tt)\) is a Borel--torus pair in \(\Gt\) such that
\(\Tt\), but not necessarily \(\Bt\), is preserved by \(\Gamma\).
There is a unique function \abmap{\Gamma(k)}{W(\Gt, \Tt)(k)},
which is easily verified to be a coboundary, sending
\(\gamma \in \Gamma(k)\) to
the unique element \(w(\gamma) \in W(\Gt, \Tt)(k)\) such that
\(\gamma\Bt\) equals \(\Int(w(\gamma))\inv\Bt\).
Let \((\Bt_0, \Tt_0)\) be
a Borel--torus pair in \(\Gt\) that is preserved by \(\Gamma\).
Then there is a unique element \(w_0 \in W(\Gt, \Tt)(k)\) such that
\(\Bt\) equals \(\Int(w_0)\inv\Bt_0)\), and it follows that
\(w(\gamma)\) equals \(w_0\inv\gamma(w_0)\) for all
\(\gamma \in \Gamma(k)\).
That is, the cohomology class of \abmapto\gamma{w(\gamma)} is
trivial.
\end{rem}

Proposition \ref{prop:inductive-step-p-prep} is
a statement about the action of \(\gamma\) on
a certain fixed-point group.
The hardest case there is when
\(\Gt\) is of type \(\mathsf E_6\), and
certain other conditions are satisfied.
We isolate this case as Proposition \ref{prop:E_6-outer}.

\begin{prop}
\label{prop:E_6-outer}
Suppose that
\(\Gt\) is split and adjoint of type \(\mathsf E_6\)
and
\(\gamma\) preserves a
semisimple subgroup \(\Ht\) of \(\Gt\) of type
\(3\mathsf A_2\).
Then \(Z(\Ht)\) is \'etale of order \(3\), and
\(\gamma\) acts trivially on \(Z(\Ht)\).
\end{prop}

\begin{proof}
We may, and do, assume, upon
replacing \(k\) by \(\ka\), that
\(k\) is algebraically closed.

Let \(\Tt_\bullet\) be a split maximal torus in \(\Ht\).
Since the ranks of \(\Ht\) and \(\Gt\) are equal,
also \(\Tt_\bullet\) is a split maximal torus in \(\Gt\).
By Remark \ref{rem:gp-BdS-facts},
for some Borel subgroup \(\Bt_\bullet\) of \(\Gt\) containing \(\Tt_\bullet\),
and using the Bourbaki numbering
\[\xymatrix@-.5em{
\alphat_{\bullet\,1} \ar@{-}[r] & \alphat_{\bullet\,3} \ar@{-}[r] & \alphat_{\bullet\,4} \ar@{-}[d]\ar@{-}[r] & \alphat_{\bullet\,5} \ar@{-}[r] & \alphat_{\bullet\,6} \\
&& \alphat_{\bullet\,2}
}\]
of \(\Delta(\Bt_\bullet, \Tt_\bullet)\),
we have that the \(\Delta(\Bt_\bullet, \Tt_\bullet)\)-highest root
\(\alphat_{\bullet\,0}\) equals
\(\alphat_{\bullet\,1} + 2\alphat_{\bullet\,2} + 2\alphat_{\bullet\,3} +
	3\alphat_{\bullet\,4} + 2\alphat_{\bullet\,5} + \alphat_{\bullet\,6}\),
and
the extended Dynkin diagram looks like
\begin{equation}
\label{eq:E6t-Bourbaki}
\tag{$*$}
\xymatrix@-.5em{
\alphat_{\bullet\,1} \ar@{-}[r] & \alphat_{\bullet\,3} \ar@{-}[r] & \alphat_{\bullet\,4} \ar@{-}[d]\ar@{-}[r] & \alphat_{\bullet\,5} \ar@{-}[r] & \alphat_{\bullet\,6} \\
&& \alphat_{\bullet\,2} \ar@{-}[d] \\
&& -\alphat_{\bullet\,0}
}
\end{equation}
\cite{bourbaki:lie-gp+lie-alg_4-6}*{Chapter VI, Plate V.IV}.
The Borel--de Siebenthal subgroups
(Definition \ref{defn:gp-BdS})
associated to
\(\alphat_{\bullet\,2}\), \(\alphat_{\bullet\,3}\), and \(\alphat_{\bullet\,5}\)
are all of type \(\mathsf A_5 + \mathsf A_1\),
and the one associated to
\(\alphat_{\bullet\,4}\) is
of type \(3\mathsf A_2\).
Since \(\mathsf A_1 + \mathsf A_5\) does not contain
a subsystem of type \(3\mathsf A_2\), and
\(3\mathsf A_2\) does not contain
a \emph{proper} subsystem of type \(3\mathsf A_2\),
Remark \ref{rem:gp-BdS-facts}(\ref{subrem:maximal-ss}) gives that
\(\Ht\) is the Borel--de Siebenthal subgroup
corresponding to \(\alphat_{\bullet\,4}\).
Write \(\Bt\smashsub{\Ht\,\bullet}\) for the
Borel subgroup of \(\Ht\) containing \(\Tt\) that
corresponds to the Borel--de Siebenthal basis
(Definition \ref{defn:rd-BdS}).

Since
\(\Gt\) is adjoint and
the coefficient of \(\alphat_{\bullet\,4}\) in \(\alphat_{\bullet\,0}\) is \(3\),
Remark \ref{rem:gp-BdS-facts}(\ref{subrem:BdS-Z}) gives that
\(\alphat_{\bullet\,4}\) is
an isomorphism of \(Z(\Ht)\) onto \(\mu_3\).
Explicit computation shows that each of the cocharacters
\(-2\alphat_{\bullet\,1}^\vee - \alphat_{\bullet\,3}^\vee\),
\(-2\alphat_{\bullet\,6}^\vee - \alphat_{\bullet\,5}^\vee\), and
\(2\alphat_{\bullet\,0}^\vee - \alphat_{\bullet\,2}^\vee\)
pairs to \(1\) with \(\alphat_{\bullet\,4}\), and
maps \(\mu_3\) onto
the center of the almost-simple component
containing the image of the cocharacter.
In particular, each almost-simple component of \(\Ht\)
has center \(Z(\Ht)\).

If \(\gamma\) is trivial, then the remainder of the
result is obvious.
Thus we may, and do, assume that \(\gamma\) is nontrivial,
hence has order \(p\).
Since the outer-automorphism group of \(\mathsf E_6\) has order \(2\),
we have that \(p\) equals \(2\).

If there is some
almost-simple component of \(\Ht\) that is
preserved by \(\gamma\) and on which
the action of \(\gamma\) is inner, then
the action of \(\gamma\) on
the center of that component, and hence on
\(Z(\Ht)\), is
trivial, as desired.
Thus we may, and do, assume that
there is no such component.
Then Lemma \ref{lem:A_{2n}-quass} gives that
the action of \(\gamma\) on every
almost-simple component of \(\Ht\) that it
preserves is quasisemisimple.

Since
\(\Ht\) has three almost-simple components and
\(\gamma\) has order \(2\), it must
preserve at least one almost-simple component
(and possibly all three).
Let \(\Ht_2\) be an almost-simple component of \(\Ht\) that is
preserved by \(\gamma\).
By
\cite{bourbaki:lie-gp+lie-alg_4-6}*{Chapter VI, Plate V.XII},
we may, and do, assume, upon replacing
\(\Bt\smashsub{\Ht\,\bullet}\) by
its conjugate by
a suitable element of \(N_\Gt(\Ht, \Tt_\bullet)(k)\), that
\(\alphat_{\bullet\,2}\) belongs to \(\Phi(\Ht_2, \Tt_\bullet)\).
Let \((\Bt_2, \Tt_2)\) be
a Borel--torus pair in \(\Ht_2\) that is
preserved by \(\gamma\).

For \(i \in \sset{3, 5}\), write
\(\Ht_i\) for the almost-simple component of \(\Ht\) such that
\(\alphat_i\) belongs to \(\Phi(\Ht_i, \Tt_\bullet)\).
Either
\(\gamma\) preserves both \(\Ht_3\) and \(\Ht_5\), or
it swaps them.
If \(\gamma\) preserves both \(\Ht_3\) and \(\Ht_5\), then,
for each \(i \in \sset{3, 5}\),
let \((\Bt_i, \Tt_i)\) be
a Borel--torus pair in \(\Ht_i\) that is
preserved by \(\gamma\).
Otherwise, let \((\Bt_3, \Tt_3)\) be any
Borel--torus pair in \(\Ht_3\), and
put \((\Bt_5, \Tt_5) = \gamma(\Bt_3, \Tt_3)\).

These three Borel--torus pairs determine
a new Borel--torus pair
\((\Bt\smashsub\Ht, \Tt)\) in \(\Ht\) that,
by construction, is preserved by \(\gamma\).
Since
\((\Bt\smashsub{\Ht\,\bullet}, \Tt_\bullet)\) and
\((\Bt\smashsub\Ht, \Tt)\) are
both Borel--torus pairs in \(\Ht\),
there is some \(\htilde\in\Ht(k)\)
such that
\(\Int(\htilde)(\Bt\smashsub{\Ht\,\bullet}, \Tt_\bullet)\)
equals \((\Bt\smashsub\Ht, \Tt)\).
Put
\(\Bt = \Int(h)\Bt\) and
\(\alphat_i = \alphat_{\bullet\,i} \circ \Int(\htilde)\inv\)
for all \(i \in \sset{0, \dotsc, 6}\).
Note that
\(\alphat_4\) equals
\(-\frac1 3(-\alphat_0 + \alphat_1 + 2\alphat_2 +
	2\alphat_3 + 2\alphat_5 + \alphat_6)\), since
the analogous formula holds for \(\alphat_{\bullet\,4}\).

Since \(\gamma\) preserves
\(\Delta(\Bt\smashsub\Ht, \Tt) \cap \Phi(\Ht_2, \Tt) =
	\sset{-\alphat_0, \alphat_2}\), but
\(\gamma\) is not inner on \(\Ht_2\), we have that
\(\gamma\) swaps \(-\alphat_0\) and \(\alphat_2\).

Note that \(\Aut(\Phi(\Gt, \Tt))/W(\Phi(\Gt, \Tt))\) is
generated by the image of the unique
diagram automorphism \(\gamma_0\) with respect to
\(\Delta(\Bt, \Tt)\), which
swaps \(\alphat_1\) and \(\alphat_6\), and
fixes \(\alphat_2\) and \(\alphat_4\), and
swaps \(\alphat_3\) and \(\alphat_5\), hence
has determinant \(1\) as
an automorphism of \(\bX^*(\Tt)\).

Suppose first that \(\gamma\) preserves \(\Ht_3\) and \(\Ht_5\).
As with \(\Ht_2\), we conclude from
the fact that
the automorphisms of \(\Ht_3\) and \(\Ht_5\) induced by \(\gamma\) are
outer that
\(\gamma\)
swaps \(\alphat_1\) and \(\alphat_3\), and
swaps \(\alphat_5\) and \(\alphat_6\).
Thus \(\gamma\) sends
\(\alphat_4\) to
\(-\frac1 3(\alphat_2 + \alphat_3 - 2\alphat_0 + 2\alphat_1 + 2\alphat_6 + \alphat_5) =
\alphat_2 + \alphat_3 + 2\alphat_4 + \alphat_5\).
Combining this with the rest of our information about
\(\gamma\) shows that it has determinant \(-1\)
as an automorphism of \(\bX^*(\Tt)\).
Therefore, the unique element of \(W(\Gt, \Tt)(k)\) that
conjugates \(\gamma\Bt\) to \(\Bt\) also has determinant \(-1\)
as an automorphism of \(\bX^*(\Tt)\).
This is a contradiction of Remark \ref{rem:rel-pos-coboundary}.
(Our computations so far do not tell us whether or not
the automorphism induced by \(\gamma\) lies in the Weyl group.
In fact it does not.
If we put
\(\wt =
\st_2\st_4\st_3\st_1\st_5\st_4\st_3\st_6\st_5\st_4\st_2\), where
\(\st_i\) denotes reflection in \(\alphat_i\)
for every \(i \in \sset{1, \dotsc, 6}\), then
\(\gamma\) equals \(\wt\gamma_0\), so
the relative position of \(\Bt\) and \(\gamma(\Bt)\) is \(\wt\).)

Thus \(\gamma\) must swap \(\Ht_3\) and \(\Ht_5\), hence swap
\(\Delta(\Bt_3, \Tt) = \sset{\alphat_1, \alphat_3}\) and
\(\Delta(\Bt_5, \Tt) = \sset{\alphat_5, \alphat_6}\).
If \(\gamma\)
swaps \(\alphat_1\) and \(\alphat_6\), and
swaps \(\alphat_3\) and \(\alphat_5\), then
it carries
\(-2\alphat_1^\vee - \alphat_3^\vee\) to
\(-2\alphat_6^\vee - \alphat_5^\vee\).
Since these cocharacters carry \(\mu_3\) into
the centers of their respective almost-simple components, hence into
\(Z(\Ht)\), and since
they have equal pairings with \(\alphat_4\),
they restrict to the same isomorphism
\abmap{\mu_3}{Z(\Ht)}, so it follows that
\(\gamma\) acted trivially on \(Z(\Ht)\).
This contradicts the fact that
\(\gamma\) does \emph{not} act trivially on
\(Z(\Ht_2) = Z(\Ht)\).
Thus, actually \(\gamma\)
swaps \(\alphat_1\) and \(\alphat_5\), hence
swaps \(\alphat_3\) and \(\alphat_6\), hence
sends \(\alphat_4\) to
\(\alphat_2 + \alphat_3 + 2\alphat_4 + \alphat_5\).
That is, the automorphism of \(\bX^*(\Tt)\) induced by \(\gamma\) is
the automorphism of the previous paragraph, followed by
the diagram automorphism \(\gamma_0\).
In particular, the unique element of \(W(\Gt, \Tt)(k)\) that
conjugates \(\Bt\) to \(\gamma\Bt\) again
has determinant \(-1\) as an automorphism of \(\bX^*(\Tt)\),
so we have once more reached a contradiction of
Remark \ref{rem:rel-pos-coboundary}.
(Concretely,
with notation for reflections as in the previous paragraph,
the automorphism of \(\bX^*(\Tt)\) induced by \(\gamma\) equals
\(\st_2\st_4\st_5\st_6\st_3\st_4\st_5\st_1\st_3\st_4\st_2\).)
\end{proof}

The conclusion of
Proposition \ref{prop:inductive-step-p-prep} is
surprising, at least to us.
It boils down to a case-by-case check,
which is feasible because of the short list of possibilities in
Remark \ref{rem:inspection}.

\begin{prop}
\label{prop:inductive-step-p-prep}
Suppose that
\(\Gt^\gamma\) is smooth.
Let \(\Gamma'\) be a
nontrivial,
\(\gamma\)-stable
subgroup of
a maximal torus in \(\Gt\adform\).
Then
\(Z(C_{\Gt\adform}(\Gamma')\conn)^\gamma\) is nontrivial
or
\(\Lie(Z(C_{\Gt\adform}(\Gamma')\conn))\) is nontrivial.
\end{prop}

\begin{proof}
We may, and do, assume, upon
replacing \(k\) by \(\ks\), that
\(k\) is separably closed.

Suppose that
\(Z(C_{\Gt\adform}(\Gamma')\conn)^\gamma\) and
\(\Lie(Z(C_{\Gt\adform}(\Gamma')\conn))\) are trivial.
Thus,
\(Z(C_{\Gt\adform}(\Gamma')\conn)\) is \'etale.

Let \(\Tt\) be a maximal torus in \(\Gt\) that
contains \(\Gamma'\).
Note that \(C_{\Gt\adform}(C_\Gt(\Gamma')\conn)\) is contained in
\(C_{\Gt\adform}(\Tt) = \Tt/Z(\Gt)\), hence in
\(C_{\Gt\adform}(\Gamma')\conn\), hence in
\(Z(C_{\Gt\adform}(\Gamma')\conn)\).
Since the reverse containment is obvious,
we have that
\(C_{\Gt\adform}(C_\Gt(\Gamma')\conn)\) equals
\(Z(C_{\Gt\adform}(\Gamma')\conn)\).
We will use this later.
For now,
we have that \(\Phi(C_\Gt(\Gamma')\conn, \Tt)\) is
an integrally closed subsystem of \(\Phi(\Gt, \Tt)\), and
\(\Z\Phi(\Gt, \Tt)/\Z\Phi(C_\Gt(\Gamma')\conn, \Tt)\) is
the character group of the \'etale, multiplicative-type group
\(Z(C_{\Gt\adform}(\Gamma')\conn)\), hence is
finite, but has no nontrivial torsion of
order the characteristic exponent of \(k\).
By Remark \ref{rem:gp-BdS-facts}(\ref{subrem:maximal-ss}),
there are a Borel subgroup \(\Bt\) of \(\Gt\) containing \(\Tt\),
an almost-simple component \(\Gt_1\) of \(\Gt\), and
an element \(\alphat \in \Delta(\Bt \cap \Gt_1, \Tt)\) such that
the coefficient \(\ell\) of \(\alphat\) in the
\(\Delta(\Bt, \Tt)\)-highest root \(\alphat_0\) of \(\Phi(\Gt, \Tt)\)
is prime, and
\(C_\Gt(\Gamma')\conn\) is contained in
the Borel--de Siebenthal subgroup corresponding to
\((\Bt, \Tt, \alphat)\)
(Definition \ref{defn:gp-BdS}).
Note that \(\ell\) is different from
the characteristic exponent of \(k\).
Let \(\tilde\varpi^\vee\) be the fundamental coweight of
\(\Gt\adform\) corresponding to \(\alphat\).
Then
\(\tilde\varpi^\vee(\mu_\ell)\),
which is the center of the relevant Borel--de Siebenthal subgroup
of \(\Gt\adform\)
by Remark \ref{rem:gp-BdS-facts}(\ref{subrem:BdS-Z}),
is contained in
\(C_{\Gt\adform}(C_\Gt(\Gamma')\conn)\), which, we recall,
equals
\(Z(C_{\Gt\adform}(\Gamma')\conn)\).

If \(\gamma\) does not preserve \(\Gt_1\), then
the various conjugates
\(\gamma^i\tilde\varpi^\vee\) with \(0 \le i < p\) are
nontrivial cocharacters of different
almost-simple components of \(\Gt\adform\), so
their sum
\((1 + \gamma + \dotsb + \gamma^{p - 1})\tilde\varpi^\vee\)
is a \(\gamma\)-fixed, faithful cocharacter of \(\Gt\adform\) that
carries \(\mu_\ell\) into
\(Z(C_{\Gt\adform}(\Gamma')\conn)\).
This is a contradiction.
Thus \(\gamma\) preserves \(\Gt_1\).

By
Lemmas \ref{lem:quass-by-component} and \ref{lem:quass-by-iso},
the action of \(\gamma\) on \(\Gt_{1\,\adsub}\)
remains quasisemisimple.
Since the image of \(\Gamma'\) in \(\Gt_{1\,\adsub}\) contains
\(\tilde\varpi^\vee(\mu_\ell)\), it is nontrivial.
Since replacing
\(\Bt\), \(\Tt\), and \(\Gamma'\) by
their images in \(\Gt_{1\,\adsub}\), then
\(\Gt\) by \(\Gt_{1\,\adsub}\), replaces
\(Z(C_{\Gt\adform}(\Gamma')\conn)\) by a subgroup,
we may, and do, make this replacement, and so assume that
\(\Gt\) is almost simple of adjoint type.
In particular, we now need not distinguish between
\(\Gt\) and \(\Gt\adform\).
Now put \(\Ht = C_\Gt(\tilde\varpi^\vee(\mu_\ell))\conn\), so that
\(Z(\Ht)\) equals
\(\tilde\varpi^\vee(\mu_\ell)\).

Since \(\gamma\) has no nontrivial fixed points on
\(Z(C_{\Gt\adform}(\Gamma')\conn)\),
hence on \(\Gamma'\),
in particular it acts nontrivially on \(\Gt\).
That is, \(p\) is prime, and
\(\gamma\) has order \(p\).

Suppose first that \(p\) equals \(2\).
Since \(\gamma^2\) acts trivially, but
\(\gamma\) acts without fixed points, on
\(Z(C_\Gt(\Gamma')\conn)\),
in fact \(\gamma\) acts on \(Z(C_\Gt(\Gamma')\conn)\) by
inversion,
and so preserves \emph{every} subgroup of it.
In particular, \(\gamma\) preserves
\(\tilde\varpi^\vee(\mu_\ell)\).
If \(\ell\) equals \(p\)
(which therefore does not equal the characteristic exponent of \(k\)),
then, since the restriction of \(\gamma\) to
\(\tilde\varpi^\vee(\mu_\ell(k))\) is
an automorphism of a cyclic group of order \(p\) whose
\(p\)th power is trivial,
actually \(\gamma\) itself acts trivially on
the nontrivial group
\(\tilde\varpi^\vee(\mu_\ell(k))\).
This is a contradiction.

Thus we may, and do, suppose that
\(p\) does not equal \(2\), or
\(\ell\) does not equal \(p\).
Since \(\Gt\) is (absolutely) almost simple and
admits an outer automorphism \(\gamma\) of order \(p\),
by Remark \ref{rem:inspection},
there are now only two possibilities.

First,
if \(p\) equals \(2\), then
\(\Gt\) is of type \(\mathsf E_6\),
\(\Ht\) is of type \(3\mathsf A_2\), and
Proposition \ref{prop:E_6-outer} shows that
\(Z(C_\Gt(\Gamma'))^\gamma\) is nontrivial.

Second,
if \(p\) equals \(3\), then
\(\ell\) equals \(2\) and
\(\Gt\) is of type \(\mathsf D_4\).
In the Bourbaki numbering
\cite{bourbaki:lie-gp+lie-alg_4-6}*{Chapter VI, Plate IV}
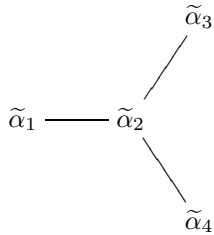
\begin{equation}
\label{eq:D4-Bourbaki}
\xymatrix@-1.5em{
&&& \alphat_3 \\
\\
\alphat_1  \ar@{-}[rr] && \alphat_2 \ar@{-}[uur]\ar@{-}[ddr] \\
\\
&&& \alphat_4
}
\end{equation}
of \(\Delta(\Bt, \Tt)\) (except with \(\alphat\) in place of \(\alpha\)),
we have that
\(\alphat\) equals \(\alphat_2\), and
\(\Phi(\Ht, \Tt)\) is
the orthogonal direct sum of the
integrally closed root subsystems of \(\Phi(\Gt, \Tt)\) spanned by
\(\sset{\alphat_1}\),
\(\sset{\alphat_3}\),
\(\sset{\alphat_4}\), and
\(\sset{-\alphat_0}\), each of which
has type \(\mathsf A_1\).
In particular, since
\(C_\Gt(\Gamma')\conn\)
is semisimple (as \(Z(C_{\Gt}(\Gamma')\conn)\) is finite) and equals
\(C_\Gt(Z(C_\Gt(\Gamma')\conn))\conn =
	C_\Ht(Z(C_\Gt(\Gamma')\conn))\conn\),
Lemma \ref{lem:gl-Levi} gives that
\(C_\Gt(\Gamma')\conn\) equals
\(\Ht\); thus
\(Z(C_\Gt(\Gamma')\conn)\) equals
\(Z(\Ht) = \tilde\varpi^\vee(\mu_\ell)\),
hence
is preserved by \(\gamma\).
Since \(\gamma^3\) is trivial but
\(\tilde\varpi^\vee(\mu_\ell)\) has order \(\ell = 2\), we have that
\(\gamma\) acts trivially on
\(\tilde\varpi^\vee(\mu_\ell) = Z(C_\Gt(\Gamma')\conn)\).
This is a contradiction.
\end{proof}

\begin{prop}
\label{prop:inductive-step-p}
Suppose that \(k\) is separably closed, that
\(\Gamma'\) is a \(\gamma\)-stable subgroup of
a torus in \(\Gt\adform\), and that
\(\Gt^\gamma\) is smooth.
Then
\(\gamma\) is a quasisemisimple automorphism of
\(C_\Gt(\Gamma')\conn\), and
\((C_\Gt(\Gamma')\conn)^\gamma\) is smooth.
\end{prop}

\begin{proof}
If \(k\) has characteristic exponent \(1\), then
every \(k\)-group is smooth by Cartier's theorem
\cite{milne:algebraic-groups}*{Theorem 3.23}; and
\(\gamma\), being of finite order, is semisimple, hence induces
a quasisemisimple automorphism of any group that
it preserves, such as
\(C_\Gt(\Gamma')\conn\)
(Corollary \ref{cor:ss-is-quass}).
Thus we may, and do, assume that
\(p\) is prime, and
\(k\) has characteristic exponent \(p\).
In particular, \(\gamma\) is unipotent.

We reason by induction on \(\dim(\Gt)\).
If the dimension is \(0\), then \(\Gt\), and hence
the result, is trivial.
Thus we may, and do, assume that
\(\dim(\Gt)\) is positive, and
that we have proven the result for all smaller-dimensional groups.

The result is obvious if \(\Gamma'\) is trivial, so
we assume that it is not.

Since \(C_{\Gt\adform}(\Gamma')\conn\) is reductive by
\cite{conrad-gabber-prasad:prg}*{Proposition A.8.12},
we have that \(Z(C_{\Gt\adform}(\Gamma')\conn)\) is
contained in every maximal torus in
\(C_{\Gt\adform}(\Gamma')\conn\).
Remark \ref{rem:adj-or-sc-T} and
Lemma \ref{lem:smooth-to-iso} give that
\(\Gt\adform^\gamma\) is smooth.

By, and with the notation of,
Proposition \ref{prop:inductive-step-p-prep},
we have that
\(Z(C_\Gt(\Gamma')\conn)^\gamma\) or
\(\Lie(Z(C_\Gt(\Gamma')\conn))\) is nontrivial.

If \(Z(C_\Gt(\Gamma')\conn)^\gamma\) is nontrivial, then
choose a nontrivial \(k\)-point \(s'\), and
put \(\Ht = C_\Gt(s')\conn\).
Since \(s'\) is nontrivial, we have that
\(\Ht\) is a proper subgroup of \(\Gt\).
If \(\Lie(Z(C_\Gt(\Gamma')\conn)\) is nontrivial, then,
since \(\gamma\) is unipotent, so is
\(\Lie(Z(C_\Gt(\Gamma')\conn)^\gamma\).
Put \(\Ht = C_\Gt(\mathfrak s')\conn\).
Since \(\mathfrak s'\) is nontrivial, we have that
\(\Ht\) is a proper subgroup of \(\Gt\).

In either case,
\(\Ht\) is a proper subgroup of \(\Gt\) that
contains \(C_\Gt(\Gamma')\conn\).
We have by Lemma \ref{lem:inductive-step-Z|Lie} that
\(\Ht\) is a connected, reductive subgroup of \(\Gt\), that
\(\gamma\) acts quasisemisimply on \(\Ht\), and that
\(\Ht^\gamma\) is smooth.
Let us temporarily write \(\Tt'_1\) for
a maximal torus in \(\Gt\adform\) that
contains \(\Gamma'\), and
\(\Tt_1\) for the corresponding
maximal torus in \(\Gt\).
Then \(\Tt_1\) is centralized by \(\Tt'_1\), hence by \(\Gamma'\), so
\(\Tt_1\) is contained in \(C_\Gt(\Gamma')\conn\), hence in \(\Ht\).
In particular, \(\Gamma'\) is contained in \(\Ht/Z(\Gt)\); and
the image of \(\Tt_1\) in \(\Ht\adform\) is
a maximal torus there that contains
the image of \(\Gamma'\).
That is, all the hypotheses remain valid if
we replace \(\Gt\) by \(\Ht\), and
\(\Gamma'\) by its image in \(\Ht\adform\).

Since \(\Ht\) is a proper subgroup of \(\Gt\),
we may apply the result inductively to \(\Ht\).
Since
\(C_\Ht(\Gamma')\conn\) equals \(C_\Gt(\Gamma')\conn\),
the conclusion for \(\Ht\) is the same as that for \(\Gt\).
\end{proof}

Proposition \ref{prop:inner-solvable} relies on
surprisingly deep facts about simple groups whose orders avoid
certain small primes.
We thank the user who pointed out that the Suzuki groups are
the only non-commutative, simple groups of
order relatively prime to \(3\) in \cite{MSE4959294},
and the user who provided the excellent history and
literature survey in \cite{MSE4959778}.

We shall also need to rely on the following fact.

\begin{rem}
\label{rem:rank-4-facts}
Suppose that $k$ is separably closed,
$\Gt$ is almost simple of rank at most $4$,
and $\Gt$ does not have type
\(\mathsf F_4\).
Then there is a homomorphism
\abmap\Gt{\PGL_8} with central kernel,
and the group of diagram automorphisms of $\Gt$ has
order at most \(6\).
\end{rem}

\begin{prop}
\label{prop:inner-solvable}
Suppose that
\(\Gt\) is absolutely almost simple, and
\(\gamma\) is nontrivial.
Let \(\mathcal H\) be a nontrivial, finite subgroup of \(\Gt(k)\) of
order relatively prime to \(p\).
Then we have the following:
\begin{enumerate}[label=(\alph*), ref=\alph*]
\item\label{subprop:inner-solvable}
\(\mathcal H\) is solvable.
\item\label{subprop:inner-toral}
\(\mathcal H\)
admits a nontrivial, \(\gamma\)-stable subgroup that is
contained in a torus in \(\Gt\).
\end{enumerate}
\end{prop}

\begin{proof}
Let \(\mathcal H\) be such a subgroup of \(\Gt(k)\).
Since \(\gamma\) is nontrivial and
\(\gamma^p\) is trivial, we have that \(p\)
is not \(1\), hence is prime.
We may, and do, replace \(k\) by \(\ka\).
Note that every element of \(\mathcal H\) is semisimple.

If \(p\) equals \(2\), then
(\ref{subprop:inner-solvable}) follows from
the Feit--Thompson odd-order theorem (!)\
\cite{feit-thompson:solvable}.
For (\ref{subprop:inner-toral}),
we may, and do, assume, upon replacing \(\mathcal H\) by
the last nontrivial term in its derived series, that
\(\mathcal H\) is nontrivial and commutative.
Of course, every cyclic subgroup of \(\mathcal H\) is
contained in a maximal torus in \(\Gt\)
\cite{borel:linear}*{Corollary 11.12}.
If \(\gamma\) has a nontrivial fixed point on
\(\mathcal H\), then
the subgroup generated by that fixed point is \(\gamma\)-stable.
If not, then
	 \(\bar s\gamma(\bar s)\) is \(\gamma\)-fixed, so
	 trivial, for all
	 \(\bar s \in \mathcal H\); hence
	 \(\gamma\) acts by inversion on \(\mathcal H\),
so every cyclic subgroup of \(\mathcal H\) is \(\gamma\)-stable.
This shows (\ref{subprop:inner-toral}).

We have handled the case \(p = 2\).
Since \(\Gt\) is (absolutely) almost simple and
admits an outer automorphism \(\gamma\) of order \(p\),
by Remark \ref{rem:inspection},
the only other possibility is that
\(p\) equals \(3\) and \(\Gt\) is of type \(\mathsf D_4\).
We suppose for the remainder of the proof that
we are in this case.

We first show how (\ref{subprop:inner-toral})
follows from (\ref{subprop:inner-solvable}).
Thus, we assume for now that
\(\mathcal H\) is solvable.
Again, we may assume that \(\mathcal H\) is commutative and that
it contains no nontrivial \(\gamma\)-fixed points.
Fix a nontrivial element \(\bar s \in \mathcal H\) of
prime order \(\ell\).
Since
\(\gamma^2(\bar s)\gamma(\bar s)\bar s\)
is preserved by \(\gamma\), it is trivial; so
\(\gamma^2(\bar s)\) equals
\(\bar s\inv\gamma(\bar s)\inv\), whence
the group generated by
\(\bar s\) and \(\gamma(\bar s)\) is
\(\gamma\)-stable.
We claim that
\(\gamma(\bar s)\), which certainly lies in \(C_\Gt(\bar s)(k)\)
(because \(\mathcal H\) is commutative), actually lies in
\(C_\Gt(\bar s)\conn(k)\).
This will show that
\(\gamma(\bar s)\) belongs to a torus in \(C_\Gt(\bar s)\conn\),
which necessarily also contains \(\bar s\)
\cite{borel:linear}*{Corollary 11.12}, hence that
the group generated by \(\bar s\) and \(\gamma(\bar s)\) is
contained in a torus in \(\Gt\).

Suppose first that \(\ell\) is odd.
We have by
\cite{digne-michel:non-connected}*{Proposition 1.27} that
\(\pi_0(C_\Gt(\bar s))(k)\) has order dividing \(\ell\), and by
\cite{steinberg:endomorphisms}*{Theorem 9.1(a)} that
\(\pi_0(C_\Gt(\bar s))(k)\) is isomorphic to a subquotient of
\(\ker(\abmap{\Gt\scform}\Gt)(k)\).
Since \(\Gt\) is of type \(\mathsf D_4\) and
\(\ker(\abmap{\Gt\scform}\Gt)\) is central in
\(\Gt\scform\),
the group of \(k\)-rational points of the kernel, hence
its subquotient \(\pi_0(C_\Gt(\bar s))(k)\), has
\(2\)-power order.
That is, \(\pi_0(C_\Gt(\bar s))(k)\) is trivial, so
\(C_\Gt(\bar s)\) is connected.

Thus we may, and do, assume that
\(\bar s\) has order \(\ell = 2\).
We now turn to the more subtle matter of
showing that \(\gamma(\bar s)\) still
belongs to \(C_\Gt(\bar s)\conn(k)\).

By two applications of \cite{steinberg:endomorphisms}*{Lemma 9.16},
there is a unique automorphism
\(\gamma\scform\) of \(\Gt\scform\) whose action is intertwined,
\textit{via} the quotient \abmap{\Gt\scform}{\Gt\der},
with the action of \(\gamma\) on \(\Gt\der\); and
\(\gamma\scform^3\) is trivial
(because it is intertwined with \(\gamma^3 = 1\)).
In particular, \(\gamma\scform\) is an
order-\(3\), outer automorphism of \(\Gt\scform\).
With the Bourbaki numbering
\cite{bourbaki:lie-gp+lie-alg_4-6}*{Chapter VI, Plate IV}
(see \eqref{eq:D4-Bourbaki})
of a system of simple roots of \(\Phi(\Gt, \Tt)\)
(except with \(\alphat\) instead of \(\alpha\)),
we have that
\(\gamma\scform\) acts as the \(3\)-cycle
\((\alphat_1\ \alphat_3\ \alphat_4)\) on the Dynkin diagram, and
\(Z(\Gt\scform)\) is the constant Klein 4-group
whose nontrivial elements are
\((\alphat_1^\vee + \alphat_3^\vee)(-1)\),
\((\alphat_1^\vee + \alphat_4^\vee)(-1)\), and
\((\alphat_3^\vee + \alphat_4^\vee)(-1)\).
In particular, the action of \(\gamma\scform\) on
\(Z(\Gt\scform)(k)\) has
no nontrivial fixed points.

Let \(s\) be a lift of \(\bar s\) to \(\Gt\scform(k)\).
Remember that
\(\gamma^2(\bar s)\) equals
\(\bar s\inv\gamma(\bar s)\inv = \bar s\gamma(\bar s)\), so the lifts
\(\gamma\scform^2(s)\) of \(\gamma^2(\bar s)\) and
\(s\gamma\scform(s)\) of \(\bar s\gamma(\bar s)\) are
translates of one another by
\(\ker(\abmap{\Gt\scform}\Gt) \subseteq Z(\Gt\scform)\).
That is, there is some \(z \in Z(\Gt\scform)(k)\) such that
\(\gamma\scform^2(s)\) equals \(s\gamma(s)z\).
Similarly, since the image
\([\bar s, \gamma(\bar s)]\) of
\([s, \gamma\scform(s)]\) in \(\mathcal H\) is trivial, we have that
\([s, \gamma\scform(s)]\) lies in \(Z(\Gt\scform)(k)\).
Further, we have that
\begin{multline*}
\gamma\scform[s, \gamma\scform(s)]
\quad\text{equals}\quad \\
[\gamma\scform(s), \gamma\scform^2(s)] =
[\gamma\scform(s), s\gamma\scform(s)z] =
[\gamma\scform(s), s]\cdot\Int(s)[\gamma\scform(s), \gamma\scform(s)z] =
[\gamma\scform(s), s] = [s, \gamma\scform(s)]\inv.
\end{multline*}
It follows that \(\gamma\scform^2\), hence also
\(\gamma\scform = \gamma\scform^4\), fixes \([s, \gamma\scform(s)]\).
Since \(\gamma\scform\) has no nontrivial fixed points on
\(Z(\Gt\scform)(k)\), it follows that
\([s, \gamma\scform(s)]\) is trivial, i.e., that
\(\gamma\scform(s)\) belongs to \(C_{\Gt\scform}(s)(k)\); and
so, by \cite{steinberg:endomorphisms}*{Lemma 9.2(a)}, that
its image
\(\gamma(\bar s)\) belongs to \(C_\Gt(\bar s)\conn(k)\).
This completes the proof of
(\ref{subprop:inner-toral}), assuming
(\ref{subprop:inner-solvable}).

Recall that we have reduced to the case that
\(p\) equals \(3\) and \(\Gt\) is of type \(\mathsf D_4\).
To prove (\ref{subprop:inner-solvable}),
we drop the assumption that \(\mathcal H\) is solvable;
indeed, we assume for the sake of contradiction that it is not.

We may, and do, assume, upon replacing \(\Gt\) by
a central quotient of a connected, reductive subgroup of \(\Gt\), that
$\Gt$ has no
nontrivial such subquotient
whose
group of \(k\)-rational points contains
a finite, non-solvable subgroup of
order relatively prime to \(p = 3\).
(Although this replacement preserves almost simplicity, it
may destroy the existence of a
nontrivial, order-\(p\), quasisemisimple automorphism of \(\Gt\);
but this is no issue, since we used the existence of
such an automorphism only to force \(p = 3\).)
Since \(\mathcal H\) is contained in
the extension of its image in \(\Gt\adform(k)\) by
the commutative group \(Z(\Gt)(k)\),
the image of \(\mathcal H\) in \(\Gt\adform(k)\) is still
a finite, non-solvable subgroup of
order relatively prime to \(p = 3\); so
our minimality assumption forces \(\Gt\) to be adjoint.

Now there is a positive integer \(n\) such that
the Suzuki group \(\mathcal H'' \ldef \lsup2\mathsf B_2(2^{2n + 1})\)
(of order \(2^{2(2n + 1)}(2^{2(2n + 1)} + 1)(2^{2n + 1} - 1)\))
is a composition factor of \(\mathcal H\)
\citelist{
	\cite{aschbacher:s3-free-2-fusion}*{p.~29, Corollary 4}
	\cite{toborg-waldecker:3'=>Sz}*{Theorem 3.1}
}.
Upon replacing \(\mathcal H\) by
an appropriate subgroup, we may, and do, assume that it is
an extension of \(\mathcal H''\) by a solvable group \(\mathcal H'\).

Suppose that \(\mathcal H'\) is nontrivial.
Then the last nontrivial term in the derived series of \(\mathcal H''\) is
a finite, commutative subgroup \(\mathcal S\) of \(\Gt(k)\) of
order relatively prime to \(p\) that
is normalized by \(\mathcal H\), i.e., for which
\(\mathcal H\) is contained in \(N_\Gt(\mathcal S)(k)\).
Since the constant group associated to
\(\mathcal S\) is commutative, and linearly reductive by
Remark \ref{rem:all-ss=>lr}, it is
of multiplicative type by
\cite{milne:algebraic-groups}*{Proposition 12.54}.
The rigidity of such groups
\cite{milne:algebraic-groups}*{Corollary 12.37} gives that
\(C_\Gt(\mathcal S)\conn\), which is
reductive \cite{conrad-gabber-prasad:prg}*{Proposition A.8.12} and
obviously connected,
is the identity component of \(N_\Gt(\mathcal S)\).
Since \(N_\Gt(\mathcal S)(k)/C_\Gt(\mathcal S)\conn(k)\) may
be identified with a group of diagram automorphisms of
\(C_\Gt(\mathcal S)\conn\), it has order at most \(6\)
from Remark \ref{rem:rank-4-facts}.
Since the order of \(\mathcal H''\) is at least
\(8^2(8^2 + 1)(8 - 1)\), which is greater than \(6\), we have that
\(\mathcal H''\) is a composition factor of
\(\mathcal H \cap C_\Gt(\mathcal S)\conn(k)\).
That is, \(\mathcal H \cap C_\Gt(\mathcal S)\conn(k)\) is
a finite, non-solvable subgroup of \(C_\Gt(\mathcal S)\conn(k)\)
consisting of semisimple elements.
By minimality of \(\Gt\), we have that
\(\mathcal S\) is central in \(\Gt\), hence, because
\(\Gt\) is adjoint, is trivial.
This is a contradiction.

That is, \(\mathcal H'\) is trivial, so
the Suzuki group \(\mathcal H''\) is a subgroup of
\(\Gt\).
We may, and do, replace \(\mathcal H\) by \(\mathcal H''\).
From Remark \ref{rem:rank-4-facts},
there is a homomorphism
\map\rho\Gt{\PGL_8}
with central kernel.
Since \(\mathcal H\) is now simple and non-commutative,
the intersection of \(\mathcal H\) with
\(\ker(\rho)(k)\) is trivial, so
\(\rho\) restricts to an embedding of
\(\mathcal H\) into \(\PGL_8(k)\).
Write \(\mathcal H^+\) for the pre-image of
\(\mathcal H\) in \(\SL_8(k)\).
Then
the characteristic exponent of \(k\)
(which is \(1\) or \(3\))
is relatively prime to the order of \(\mathcal H^+\), and
\(\mathcal H^+\) has a nontrivial representation on \(k^8\).
By Maschke's theorem \cite{isaacs:char-fg}*{Theorem 1.9},
all representations of \(\mathcal H^+\) on \(k\)-vector spaces are
completely reducible, so
the existence of
a nontrivial representation of \(\mathcal H^+\) on \(k^8\) implies
the existence of
an \emph{irreducible} representation of \(\mathcal H^+\) on
a \(k\)-vector space of dimension at most \(8\).
By \cite{isaacs:char-fg}*{Theorem 15.13},
there is a
nontrivial, irreducible, complex representation of
\(\mathcal H^+\) of dimension at most \(8\),
hence a
nontrivial, irreducible, complex \emph{projective} representation of
\(\mathcal H\) of dimension at most \(8\).

If \(n\) equals \(1\), that is, if
\(\mathcal H\) is the Suzuki group \(\lsup2\mathsf B_2(8)\), then
\cite{atlas-fg}*{p.~28} shows that the smallest dimension
of a nontrivial, irreducible, complex projective representation of
\(\mathcal H\) is \(14\),
which is a contradiction.

If \(n\) is greater than \(1\), then
\(\mathcal H\) has trivial Schur multiplier
\cite{alperin-gorenstein:schur}*{p.~515, Theorem 1}, so
\(\mathcal H\) has a
nontrivial, irreducible, complex (linear) representation of
dimension at most \(8\).
By \cite{suzuki:my-groups}*{\S11, p.~127, Theorem 5},
the smallest dimension of a
nontrivial, irreducible, complex representation
of \(\mathcal H\) is
\(2^n(2^{2n + 1} - 1)\), which is greater than \(8\).
Again, this is a contradiction.
This completes the proof of
(\ref{subprop:inner-solvable}).
\end{proof}

\subsection{Purely outer, non-cyclic actions on $\mathsf D_4$}
\label{subsec:D_4-outer}

We are building to the proof of
Theorem \ref{thm:loc-quass}, which concerns the action of
a group of automorphisms each of which individually
preserves a Borel--torus pair, but where the pair might
depend on the automorphism.
Although this is a weakening of the hypothesis of
Theorem \ref{thm:quass}, where the automorphisms are
required to preserve a \emph{common} Borel--torus pair,
we will show in \S\ref{subsec:thm:loc-quass} that
fine control of the automorphism groups of
absolutely almost-simple groups allows us nearly to
reduce to that case, or the prime-to-\(p\) case handled by
\cite{prasad-yu:actions}.
Most of the difficulty is caused by the presence of
an outer-automorphism group whose order is divisible by \(p\),
which was handled in \S\ref{subsec:quass-unip}, and by
a non-cyclic outer-automorphism group.
In this subsection, we handle the latter case.

Remember that \(p\) is \(1\) or a prime number, and that
\(k\) has characteristic exponent \(p\) or \(1\).
In this section, we consider only \(p = 2\).

\begin{prop}
\label{prop:D_4-outer}
Suppose that
\(k\) is separably closed,
\(p\) equals \(2\),
\((\Gt, \Gamma)\) is a reductive datum with
\(\Gt\) the adjoint group of type \(\mathsf D_4\), and
the map \(\abmap{\Gamma}{\uOut(\Gt)}\) is an isomorphism.
Let \(\sigma\) and \(\tau\) be elements of \(\Gamma(k)\cong\Out(\Gt)\)
of
respective orders \(2\) and \(3\).
Suppose that \(\sigma\) and \(\tau\) act quasisemisimply on \(\Gt\).
Then \(\sigma\) acts quasisemisimply on \((\Gt^\tau)\conn\), and
\(((\Gt^\tau)\conn)^\sigma\) is smooth.
\end{prop}


\begin{proof}
Let \((\Bt, \Tt)\) be a Borel--torus pair in \(\Gt\) that
is preserved by \(\tau\).
Since the characteristic exponent of \(k\) is not \(3\),
we have that \(\tau\) is semisimple,
so Remark \ref{rem:all-ss=>lr} and
\cite{conrad-gabber-prasad:prg}*{Proposition A.8.10(2)} give that
\(\Tt^\tau\) and \(\Gt^\tau\) are smooth.
Theorem \ref{thm:quass}(\ref{subthm:quass-reductive})
gives that \(H \ldef (\Gt^\tau)\conn\) is reductive, while
Proposition \ref{prop:quass-facts}(\ref{subprop:quass-split-descends})
and
Proposition \ref{prop:quass-rough}(\ref{subprop:quass-down},\ref{subprop:quass-up})
give that
\(T \ldef (\Tt^\tau)\conn\) is
a maximal torus in \(H\), and
\(C_\Gt(T)\) equals \(\Tt\).

If \(k\) has characteristic exponent \(1\), then
\(\sigma\) is semisimple, so the result follows from
Corollary \ref{cor:ss-is-quass}.
Thus we may, and do, assume that
\(k\) has characteristic exponent \(2\).

We have that \(\Delta(\Bt, \Tt)\) is
the union of two \(\tau\)-orbits of sizes \(1\) and \(3\).
Specifically, if we number \(\Delta(\Bt, \Tt)\) as
in \eqref{eq:D4-Bourbaki},
then \(\tau\) fixes \(\alphat_2\), and
admits \(\sset{\alphat_1, \alphat_3, \alphat_4}\) as
an orbit.
Upon replacing \(\tau\) by \(\tau\inv\) if necessary, which
does not affect the conclusion, we may, and do, assume that
\(\tau\alphat_1\) equals \(\alphat_3\), and
\(\tau\alphat_3\) equals \(\alphat_4\).
If \(\Xt_{\alphat_1}\) is a nonzero element of
\(\Lie(\Gt)_\alphat\), then
the \(\tau\)-orbit of \(\Xt_{\alphat_1}\) is
obviously preserved by \(\tau\), and
contains exactly one nonzero root vector for each
root in \(\sset{\alphat_1, \alphat_3, \alphat_4}\).
Thus, combining it with any nonzero element of
the \(\alphat_2\)-root space gives
a pinning \(\tildepin\), in the sense of
\citelist{
	\cite{SGA-3.3}*{Expos\'e XXIII, D\'efinition 1.1}
	\cite{kaletha-prasad:bt-theory}*{Definition 2.9.1}
}.
This pinning is preserved by \(\tau\) if and only if
\(\tau\) is pinned, in the sense of
\cite{SGA-3.3}*{Expos\'e XXIII, D\'efinition 1.3}.
It determines a section of the natural map from
\(\Aut(\Gt)\) to
the group of diagram automorphisms of
the based root datum \(\Psi(\Gt, \Bt, \Tt)\), hence a
retraction of \(\Aut(\Gt)\) onto
the subgroup of automorphisms that
preserve \(\Bt\), \(\Tt\), and \(\tildepin\).
Write \(\sigma_0\) and \(\tau_0\) for
the images of \(\sigma\) and \(\tau\) under
this retraction.
Thus, \(\sigma_0\) and \(\tau_0\) are the pinned automorphisms corresponding
to \(\sigma\) and \(\tau\).

Write \(\alpha_1 = \alpha_3 = \alpha_4\) for
the common restriction to \(T\) of
\(\alphat_1\), \(\alphat_3\), and \(\alphat_4\), and
\(\alpha_2\) for the restriction of \(\alphat_2\).
We have that
\(\Phi((\Gt^{\tau_0})\conn, T)\) equals \(\Phi(\Gt, T)\) and is of
type \(\mathsf G_2\), and
\(\Delta((\Bt^{\tau_0})\conn, T)\)
and
\(\Delta(\Bt, T)\)
both equal
\(\sset{\alpha_1, \alpha_2}\).

Suppose first that \(\tau\) is pinned, so that
\(\tau\) equals \(\tau_0\) and
\(H\) equals \((\Gt^{\tau_0})\conn\).
We have that
\(\sigma_0\) preserves \((\Bt,\Tt)\),
and \(\sigma\) equals \(\sigma_0\circ\Int(u)\) for some \(u\in\Gt(k)\).
Since both \(\sigma\) and \(\sigma_0\) normalize \(\sgen\tau\),
they preserve \(H\), so
\(\Int(u)\) must as well.

Since \(H\) is of type \(\mathsf G_2\),
it has no nontrivial outer automorphisms, so
the restriction of \(\sigma_0\) to \(H\)
is a toral inner automorphism, hence is trivial since its order is \(2\).
Thus the common restriction of
\(\sigma\) and \(\Int(u)\) to \(H\) is inner, hence
coincides with \(\Int(u')\) for some
\(u'\in H(k)\)
(because \(H\) is adjoint).
That is, \(u u^{\prime\,{-1}}\) centralizes \(H\).
Since \(H\) is adjoint and
\(\Int(u')^2\) is the restriction of \(\sigma^2 = 1\) to \(H\),
we have that \(u^{\prime\,2}\) is trivial, so \(u'\) is unipotent.
Examining the orbits of \(\tau\) on \(\Psi(\Gt,\Tt)\)
shows that each root space for
\(\Tt\) in \(\Lie(\Gt)\) is
the image under projection from a root space for
\(T\) in \(\Lie(H)\).
Consequently, the group \( C_\Gt(H)\), which is certainly contained in
\(C_\Gt(T) = \Tt\), also
acts trivially on each root space for \(\Tt\) in \(\Lie(\Gt)\),
hence is central in \(\Gt\).
In particular, \(u u^{\prime\,{-1}}\) is central in \(\Gt\).
That is, \(\Int(u)\) equals \(\Int(u')\) on all of \(\Gt\), so
we may, and do, assume, upon replacing \(u\) by \(u'\), that
\(u\) is a unipotent element of \(H(k)\).
Let \(B_H\) be a Borel subgroup of \(H\) with \(u\in B_H(k)\),
and let \(\Ut\) be the \(\tau\)-stable unipotent radical of a
\(\tau\)-stable Borel subgroup of
\(\Gt\) containing \(B_H\)
(whose existence is given by
Proposition \ref{prop:quass-facts}(\ref{subprop:fixed-Borus})).

Since \(k\) has characteristic exponent \(2\)
and \(\sigma\) is a quasisemisimple involution of \(\Gt\),
we have that \(\sigma\) is pinned.
(We are \emph{not} (yet) claiming that \(\sigma\) preserves
(\(\Bt\), \(\Tt\)), or \(\tildepin\), only that \(\sigma\) preserves
some pinning with respect to
a Borel--torus pair that it preserves.
The argument for this is the same as we used to show that
\(\tau\) was ``almost pinned'',
namely, collecting
pairs of root vectors for every pair of roots swapped by \(\sigma\),
but now noting that the scalar by which \(\sigma\) can
act on the root space corresponding to
a \(\sigma\)-fixed root is
an element of \(\mu_2(k) = \sset1\).)
Since all pinnings of \(\Gt\) are conjugate in \(\Gt(k)\)
(because \(\Gt\) is adjoint),
there exists \(g\in\Gt(k)\) such that
\(\Int(g)\sigma\Int(g)\inv = \Int(g)\sigma_0\Int(u)\Int(g)\inv\) equals \(\sigma_0\).
Letting \(\bar u\) and \(\bar g\) be the respective images
of \(u\) and \(g\) in \(\Gt\adform(k)\),
we have \(\bar u = \sigma_0(\bar g){\bar g}\inv\).
Let \(\Ut'\) be the image of \(\Ut\) in \(\Gt\adform\), and
let \(z\) be the class of the cocycle
\(\theta\mapsto\theta(\bar g)\bar g\inv\) in \(H^1(\langle\sigma_0\rangle,\Ut'(k))\).

Corollary~\ref{cor:H1-inject} implies that the map
\abmap{\Gt\adform^{\sigma_0}(k)}{(\Gt\adform/\Ut')^{\sigma_0}(k)} is surjective.
It follows from the cohomology exact sequence associated to the
inclusion \abmap{\Ut'(k)}{\Gt\adform(k)} that the kernel of the map
\abmap{H^1(\langle\sigma_0\rangle,\Ut'(k))}{H^1(\langle\sigma_0\rangle,\Gt\adform(k))}
is trivial. Since the image of \(z\) in \(H^1(\langle\sigma_0\rangle,\Gt\adform(k))\) is trivial,
\(z\) must therefore be as well. Thus \(\bar u\) equals \(\sigma_0(\bar v)\bar v\inv\) for some \(\bar v\in \bar\Ut(k)\).
Since \(u\) is fixed by \(\tau\), so is
\(\bar u = \sigma_0(\bar v)\bar v\inv\).
A straightforward explicit computation involving the root groups in \(\Ut'\)
shows that this is possible only if \(\sigma_0(\bar v)\) equals \(\bar v\), i.e., \(\bar u = \sigma_0(\bar v)\bar v\inv\) is trivial.
It follows that \(\sigma = \sigma_0 \circ \Int(u)\)
equals \(\sigma_0\), hence acts trivially on \(H\).
This is certainly a quasisemisimple action, and
the fixed-point group is all of \(H\), hence smooth.

Now suppose instead that \(\tau\) is not pinned.
The fundamental coweight
\(\varpi^\vee_2\) associated to
\(\alpha_2 \in \Delta((\Bt^{\tau_0})\conn, T)\)
(which is the same as the fundamental coweight associated to
\(\alphat_2 \in \Delta(\Bt, \Tt)\))
equals
\(\alphat_1^\vee + 2\alphat_2^\vee + \alphat_3^\vee + \alphat_4^\vee\).
Put \(t = \varpi^\vee_2(\zeta)\), where
\(\zeta\) is the scalar by which \(\tau\) acts on
the \(\alphat_2\)-root space.
Then \(\tau\) equals \(\tau_0 \circ \Int(t)\), and
\(H\) equals \(C_{\Gt^{\tau_0}}(t)\conn\).
(There is content to the latter assertion, as
the analogous statement is not true for
an arbitrary product of a pinned automorphism and
a semisimple, inner automorphism with which it commutes.)
In fact, \(Z(H)\) is a torus, and
its cocharacter lattice is spanned by \(\varpi^\vee_2\).
Since \(\sigma\) has order \(2\) and preserves \(Z(H)\), it
sends \(\varpi^\vee_2\) to \(\pm\varpi^\vee_2\), so
fixes \(d\varpi^\vee_2\).
By Lemma \ref{lem:inductive-step-Z|Lie}, we have that
\(\sigma\) acts quasisemisimply on
\(\Ht \ldef C_\Gt(d\varpi^\vee_2)\conn\).
Since \(\varpi^\vee_2\) is a central cocharacter of \(H\), we have that
\(H\) is contained in \(\Ht\), hence equals
\((\Ht^\tau)\conn\).

We have that \(\Phi(\Ht, \Tt)\) is
the orthogonal direct sum of the
integrally closed root subsystems of \(\Phi(\Gt, \Tt)\) spanned by
\(\sset{\alphat_1}\),
\(\sset{\alphat_3}\),
\(\sset{\alphat_4}\), and
\(\sset{-\alphat_0}\), each of which
has type \(\mathsf A_1\).
Both \(\sigma\) and \(\tau\) preserve \(\Ht\),
and \(\sigma\) permutes the almost-simple components of \(\Ht\adform\) that
intersect \((\Ht\adform^\tau)\conn\) nontrivially.
There are three of these components, with root systems generated by
\(\sset{\alphat_1}\),
\(\sset{\alphat_3}\) and
\(\sset{\alphat_4}\).
Since \(\sigma\) has order \(2\), it must
preserve at least one of these almost-simple components.
Let \(\Ht_1\) be such an almost-simple component of \(\Ht\adform\) that is
preserved by \(\sigma\).
(This labelling is arbitrary; but,
if desired for notational consistency, then
we could replace \(\sigma\) by
its conjugate by \(\tau\) or \(\tau\inv\), which
does not affect the conclusion, to
ensure that \(\alphat_1\) belongs to \(\Phi(\Ht_1, \Tt)\).)
Since \(\sigma\) acts quasisemisimply on \(\Ht\),
Lemma \ref{lem:quass-by-iso} and Lemma \ref{lem:quass-by-component}
give that
it also acts quasisemisimply on \(\Ht_1\).
Since
the characteristic exponent of \(k\) is \(2\),
\(\sigma^2\) is trivial, and
\(\Ht_1\) has type \(\mathsf A_1\),
we have that \(\sigma\) has a
unipotent inner action on \(\Ht_1\),
hence is trivial on \(\Ht_1\).
Since the canonical projection \abmap\Ht{\Ht_1}
restricts to
a \(\sigma\)-equivariant, central isogeny
from \(H = (\Ht^\tau)\conn\) onto \(\Ht_1\),
Lemma \ref{lem:quass-by-iso} gives that
\(\sigma\) acts quasisemisimply on \(H\).

Since \(\sigma\) acts trivially on a
central quotient of \(H\), it acts trivially on
\(H\der\).
Remember that \(Z(H)\) is the image of \(\varpi_2^\vee\), and that
\(\sigma\varpi_2^\vee\) equals \(\pm\varpi_2^\vee\).
If \(\sigma\varpi_2^\vee\) equals \(\varpi_2^\vee\), then
\(\sigma\) also acts trivially on \(Z(H)\).
That is, \(H^\sigma\) is all of \(H\).
If \(\sigma\varpi_2^\vee\) equals \(-\varpi_2^\vee\), then
\(H^\sigma\) is generated by
\(H\der\) and
\(Z(H)^\sigma = \varpi_2^\vee(\mu_2)\).
Since
\(\varpi_2^\vee - (\alpha_1^\vee + \alpha_3^\vee + \alpha_4^\vee)\)
equals \(2\alpha_2^\vee\), hence is trivial on \(\mu_2\),
we have that
\(\varpi_2^\vee(\mu_2)\)
equals
\((\alpha_1^\vee + \alpha_3^\vee + \alpha_4^\vee)(\mu_2)\),
which is contained in \(H\der\).
Thus, \(H^\sigma\) equals \(H\der\), which is again smooth.
\end{proof}

\subsection{Completion of the proof of Theorem \ref{thm:loc-quass}}
\label{subsec:thm:loc-quass}

In this subsection,
we use the notation and hypotheses of Theorem \ref{thm:loc-quass}.
That is, we now consider, not just
our connected, reductive \(k\)-group \(\Gt\), but
a reductive datum \((\Gt, \Gamma)\) over \(k\).
We assume that every \(\gamma \in \Gamma(\ka)\)
acts quasisemisimply on \(\Gt_\ka\), with
smooth fixed-point group on \(\Gt_{\adsub\,\ka}\).
As usual, we put \(G = \fix\Gt^\Gamma\).

So far in \S\ref{sec:thm:loc-quass}, we have assumed that
\(p\) is \(1\) or a prime number, and that
\(k\) has characteristic exponent \(1\) or \(p\).
We now require that \(k\) actually have
characteristic exponent \(p\)
(but we continue to allow the possibility that
\(p\) equals \(1\)).

\begin{rem}
\label{rem:loc-quass-p=1}
Suppose that \(p\) equals \(1\), and that
\(\gamma \in \Gamma(\ka)\) is unipotent.
Let \(\Gt_1\) be an almost-simple component of \(\Gt\).
There is some positive integer \(n\) such that
\(\gamma^n\) preserves \(\Gt_1\), and
some positive multiple \(N\) of \(n\) such that
\(\gamma^N\) is an inner automorphism of \(\Gt_1\).
Since \(\gamma^N\) is an inner automorphism of \(\Gt_1\), and
acts quasisemisimply on \(\Gt_1\) by Lemma \ref{lem:quass-by-component},
we have that \(\gamma^N\) is a
inner, quasisemisimple, unipotent automorphism of \(\Gt_1\),
hence trivial
(Remark \ref{rem:torus-quass}(\ref{subrem:quass-to-torus})).
Thus \(\gamma^n\) acts as a
finite-order, unipotent automorphism of \(\Gt_1\), hence is trivial.
\end{rem}

\begin{prop}
\label{prop:inductive-step}
Suppose that \(k\) is algebraically closed,
and that
\(\Gamma'\) is a smooth, normal subgroup of \(\Gamma\) such that
\(\Gamma'(k)\) contains only semisimple elements.
For every element \(\gamma \in \Gamma(k)\), we have that
\(\Gt\adform^{\smashsgen{\gamma, \Gamma'}}\) is smooth and
\(\gamma\) acts quasisemisimply on
\(C_\Gt(\Gamma')\conn\).
\end{prop}

\begin{note}
The group \(C_\Gt(\Gamma')\conn\) is reductive by
Remark \ref{rem:all-ss=>lr} and
\cite{conrad-gabber-prasad:prg}*{Proposition A.8.12}.
\end{note}

\begin{proof}
We reason by induction on
\(\dim(\Gt) + \dim(Z(\Gt)) + \smashcard{Z(\Gt\der)}\).
Here, \(\smashcard{Z(\Gt\der)}\) is the cardinality of
the finite group scheme \(Z(\Gt\der)\)
(the dimension of its ring of regular functions), not just of
its group of \(k\)-rational points.
Thus, for example, \(\smashcard{\mu_p}\) equals \(p\), not \(1\)
(unless \(p\) equals \(1\)).

If the sum is \(1\), then
\(\Gt\), and the result, are trivial.

Suppose first that we have proven the result under
the additional hypothesis that
\(\Gamma/\Gamma'\) is generated by
a unipotent element, and
the conclusion replaced by the claim that
\(\Gt\adform^\Gamma\) is smooth and
\(\Gamma\) or, equivalently, any generator of \((\Gamma/\Gamma')(k)\)
acts quasisemisimply on \(C_\Gt(\Gamma')\conn\).
We
fix \(\gamma \in \Gamma(k)\), and write
\(\gamma\semi\) and \(\gamma\unip\) for
its semisimple and unipotent parts.
Then, with
\(\Gamma\) replaced by \(\sgen{\gamma, \Gamma'}\) and
\(\Gamma'\) replaced by \(\sgen{\gamma\semi, \Gamma'}\),
our additional hypothesis is satisfied; so
\(\Gt\adform^{\sgen{\gamma\unip, \gamma\semi, \Gamma'}}\),
which equals
\(\Gt\adform^{\sgen\gamma, \Gamma'}\),
is smooth, and
\(\sgen{\gamma, \Gamma'}\) acts quasisemisimply on
\(C_\Gt(\Gamma', \gamma\semi)\conn\).
Then \cite{digne-michel:non-connected}*{Lemme 1.14} gives that
\(\gamma\) acts quasisemisimply on
\(C_\Gt(\Gamma')\conn\).

Thus we may, and do, assume that
\(\Gamma/\Gamma'\) is generated by
a unipotent element.
If \(\gamma\) is any element of \(\Gamma(k)\) whose
image in \((\Gamma/\Gamma')(k)\) generates \(\Gamma/\Gamma'\), then
the image of its unipotent part also
generates \(\Gamma/\Gamma'\).
Thus we may, and do, assume, upon replacing
\(\gamma\) by its unipotent part, that
\(\gamma\) is unipotent.
Since the result is trivial if
\(\gamma\) or \(\Gamma'\) acts trivially on \(\Gt\),
we assume that neither does.

The bulk of the proof consists of reducing to the situation of
Proposition \ref{prop:inductive-step-p}.
This takes some work.

First, we show that we may assume that
\(\Gt\) is adjoint.
If it is not, then
\(\dim(Z(\Gt))\) or
\(\smashcard{Z(\Gt\der)}\) is strictly greater than \(1\).
This means that
\(\dim(\Gt\adform) + \dim(Z(\Gt\adform)) + \smashcard{Z(\Gt\adform)} =
\dim(\Gt\adform) + 1\)
is strictly less than
\(\dim(\Gt) + \smashcard{Z(\Gt)}\), so
we already have the result for \(\Gt\adform\) by
the inductive hypothesis.
Then Lemma \ref{lem:quass-by-iso} gives
the result for \(\Gt\).
Thus we may, and do, assume that \(\Gt\) is adjoint.

Next, we show that we may assume that
\(\Gt\) is almost simple (as well as adjoint).
By Lemma \ref{lem:quass-by-component},
the following assertions are equivalent:
\begin{enumerate}[label={(\roman*\textsubscript{qs})}, ref={\roman*\textsubscript{qs}}]
\item\label{case:prop:inductive-step:quass-Gt}
the action of \(\Gamma\)
on
\((\Gt^{\Gamma'})\conn\) is quasisemisimple; and
\item\label{case:prop:inductive-step:quass-Ht1}
for every almost-simple component \(\Ht_1\) of
\((\Gt^{\Gamma'})\conn\),
the action of \(\stab_\Gamma(\Ht_1)\) on
\(\Ht_1\) is quasisemisimple.
\end{enumerate}
Corollary \ref{cor:ind-simple} and
Lemma \ref{lem:was-cor:trans-induction} show that
\(\Gt^{\Gamma'}\) is isomorphic to
\(\prod \Gt_1^{\stab_{\Gamma'}(\Gt_1)}\), the product taken over one
almost-simple component \(\Gt_1\) of \(\Gt\) from each
\(\Gamma'(k)\)-orbit of such components.
Therefore, another application of
Lemma \ref{lem:quass-by-component} gives that
(\ref{case:prop:inductive-step:quass-Ht1}) is equivalent to
the following statement:
\begin{enumerate}[resume*]
\item\label{case:prop:inductive-step:quass-Gt1}
for every almost-simple component \(\Gt_1\) of \(\Gt\),
the action of \(\stab_\Gamma(\Gt_1)\) on
\((\Gt_1^{\stab_{\Gamma'}(\Gt_1)})\conn\) is quasisemisimple.
\end{enumerate}
Finally, another application of
Corollary \ref{cor:ind-simple} and
Lemma \ref{lem:was-cor:trans-induction} shows that
\(\Gt^\Gamma\) is isomorphic to
\(\prod \Gt_1^{\stab_\Gamma(\Gt_1)}\), the product taken over one
almost-simple component \(\Gt_1\) of \(\Gt\) from each
\(\Gamma(k)\)-orbit of such components.
Thus,
remembering that \(\Gt\) and hence
each of its simple components is
adjoint,
the following assertions are also equivalent:
\begin{enumerate}[label={(\roman*\textsubscript{sm})}, ref={\roman*\textsubscript{sm}}]
\item\label{case:prop:inductive-step:smooth-Gt}
\(\Gt\adform^\Gamma\) is smooth; and
\item\label{case:prop:inductive-step:smooth-Gt1}
for every almost-simple component \(\Gt_1\) of \(\Gt\),
the fixed-point group \(\Gt_{1\,\adsub}^{\stab_\Gamma(\Gt_1)}\)
is smooth.
\end{enumerate}
Thus, if \(\Gt\) is not almost simple, then we know inductively that
(\ref{case:prop:inductive-step:quass-Gt1}) and
(\ref{case:prop:inductive-step:smooth-Gt1})
hold, so that
(\ref{case:prop:inductive-step:quass-Gt}) and
(\ref{case:prop:inductive-step:smooth-Gt}) hold; and
these two together give the modified result for \(\Gt\).
Thus we may, and do, assume that
\(\Gt\) is almost simple.
In particular,
since \(\gamma\) acts nontrivially on \(\Gt\),
Remark \ref{rem:loc-quass-p=1} gives that
\(p\) does not equal \(1\).

Now, since \(\gamma\) is unipotent, it has finite,
\(p\)-power order.
We have by Remark \ref{rem:inspection} that
the image of \(\gamma\) in
the outer-automorphism group of \(\Gt\) has
order dividing \(p\).
In particular, since an
inner, quasisemisimple, unipotent automorphism of \(\Gt\)
is trivial
(Remark \ref{rem:torus-quass}(\ref{subrem:quass-to-torus})),
we have that
\(\sgen\gamma\) is constant of order \(p\), and that
the natural map from \(\sgen\gamma\) to
the outer-automorphism group of \(\Gt\) is
an embedding.
This will now allow us to apply
Proposition \ref{prop:inductive-step-p} in all cases but one,
which we handle separately.

If \(\Gamma' \cap \Gt\) is trivial, then
the images of
\(\sgen\gamma\) and \(\Gamma'\) in
the outer-automorphism group of \(\Gt\) are
nontrivial subgroups of order \(p\) and relatively prime to \(p\),
respectively.
Another appeal to Remark \ref{rem:inspection} gives that
\(p\) equals \(2\),
\(\Gt\) is of type \(\mathsf D_4\), and
\(\sgen\gamma \ltimes \Gamma'\) maps isomorphically onto
\(\uOut(\Gt)\).
Then the result follows from Proposition \ref{prop:D_4-outer}.
Thus we may, and do, assume that
\(\Gamma' \cap \Gt\) is nontrivial.

We have that \(\Gamma^{\prime\,\connsup}\) is a torus
\cite{milne:algebraic-groups}*{Corollary 17.25}.
If it does not act trivially on \(\Gt\), then
Proposition \ref{prop:inductive-step-p} allows us to
conclude by applying the inductive hypothesis to
the action of \(\Gamma/\Gamma^{\prime\,\connsup}\) on
\(C_\Gt(\Gamma^{\prime\,\connsup})\conn\).
Thus we may, and do, assume that \(\Gamma^{\prime\,\connsup}\)
acts trivially, and so, upon
replacing \(\Gamma'\) by \(\pi_0(\Gamma')\), that
\(\Gamma'\) is \'etale.

Suppose first that
\(\Gamma'\) is commutative, and contained in \(\Gt'\).
Since \(\Gt\) is adjoint, we may regard it as
the identity component of \(\uAut(\Gt)\).
If \(\gamma\) acts trivially on \(\Gt\), then
we are done.
Otherwise, by
Proposition \ref{prop:inner-solvable}(\ref{subprop:inner-toral}),
there is a nontrivial, \(\gamma\)-stable
subgroup \(\Gamma''\) of \(\Gamma'\) that is
contained in a torus in \(\Gt\).
Since \(\Gamma'\) is commutative, and
\(\Gamma/\Gamma'\) is generated by the image of \(\gamma\), we have that
\(\Gamma''\) is normal in \(\Gamma\).
Proposition \ref{prop:inductive-step-p} gives that
\(\gamma\) is a quasisemisimple automorphism of
\(C_\Gt(\Gamma'')\conn\), and
\(C_\Gt(\Gamma'')^\gamma\) is smooth.
Then we may apply our inductive hypothesis to
the action of \(\Gamma/\Gamma''\) on
\(C_\Gt(\Gamma'')\conn\) to obtain the desired result.

Now drop the assumption that
\(\Gamma'\) is commutative and contained in \(\Gt\), but
keep the assumption that \(\Gamma' \cap \Gt\) is nontrivial.
Since \(\gamma\) is also nontrivial,
Proposition \ref{prop:inner-solvable}(\ref{subprop:inner-solvable})
shows that \(\Gamma' \cap \Gt\) is solvable.
Now let \(\Gamma''\) be the last term in the derived series of
\(\Gamma' \cap \Gt\), so that
\(\Gamma''\) is a commutative, normal, \(\gamma\)-stable subgroup of \(\Gamma'\),
hence a normal subgroup of \(\Gamma\).
The special case that we have already handled shows that
\(\gamma\) is a quasisemisimple automorphism of
\(C_\Gt(\Gamma'')\conn\), and
\(C_\Gt(\Gamma'')^\gamma\) is smooth.
Then we conclude by applying the inductive hypothesis to
the action of \(\Gamma/\Gamma''\) on \(C_\Gt(\Gamma'')\conn\).
\end{proof}

{\newcommand\theadhocthm{\ref{thm:loc-quass}}
\begin{adhocthm}
Suppose, for every \(\gamma \in \Gamma(\ka)\), that
\(\gamma\) acts quasisemisimply on \(\Gt_\ka\) and
\((\Gt\adform)_\ka^\gamma\) is smooth.
	\begin{enumerate}[label=(\arabic*), ref=\arabic*]
	\item
	\((\Gt^\Gamma)\conn\) equals
\((Z(\Gt)^\Gamma)\conn\cdot(\Gt^\Gamma)\smooth\conn\).
	\item
	$G$ is reductive.
	\item
	The functorial map from the spherical building \(\SS(G)\) of \(G\)
	to the spherical building \(\SS(\Gt)\) of \(\Gt\)
	identifies \(\SS(G)\) with
	\(\SS(\Gt) \cap \SS(\Gt_\ka)^{\Gamma(\ka)}\).
	\end{enumerate}
\end{adhocthm}}

\begin{rem}
\label{rem:conn-act-torus}
In the context of
Theorem \ref{thm:loc-quass},
let \(\Gamma'\) be the subgroup of \(\Gamma\) that
acts on \(\Gt\) by inner automorphisms,
so that there is a map
\abmap{\Gamma'}{\Gt\adform}.
The image in \(\Gt\adform\) of every element of \(\Gamma'(\ka)\)
is a quasisemisimple, inner automorphism, hence semisimple.
In particular,
the image of \(\Gamma'\) is linearly reductive
by Remark \ref{rem:all-ss=>lr}.

Remark \ref{rem:torus-quass}(\ref{subrem:act-inner}) gives that
\(\Gamma'\) contains \(\Gamma\conn\).
The image of \(\Gamma\conn\) in \(\Gt\adform\) is
a smooth, connected group all of whose
\(\ka\)-rational points are semisimple, so
\cite{milne:algebraic-groups}*{Corollary 17.25}
gives that the image is a torus.
(This generalizes part of
Remark \ref{rem:torus-quass}(\ref{subrem:quass-to-torus}).)
\end{rem}

\begin{proof}[Proof of Theorem \ref{thm:loc-quass}]
By Lemma \ref{lem:spherical-descent}, we may, and do, assume,
upon replacing \(k\) by \(\ks\), that
\(k\) is separably closed.

Remark \ref{rem:conn-act-torus} shows that
the image of \(\Gamma\conn\) in \(\uAut(\Gt)\)
is a torus in \(\Gt\adform\) that is
preserved by \(\Gamma\), and so, by
Remark \ref{rem:torus-quass}(\ref{subrem:torus-to-quass}),
that \((\Gt, \Gamma\conn)\) is quasisemisimple.
We have by
\cite{borel:linear}*{Corollary 11.12} that
\(\Gt^{\Gamma\conn}\) is connected and by
\cite{conrad-gabber-prasad:prg}*{Proposition A.8.12}
that \(\Gt^{\Gamma\conn}\)
is reductive---in particular, smooth.
Theorem \ref{thm:quass}(\ref{subthm:quass-spherical-bldg})
shows that \(\SS(\Gt^{\Gamma\conn})\) equals
\(\SS(\Gt) \cap \SS(\Gt_\ka)^{\Gamma\conn(\ka)}\).
Proposition \ref{prop:inductive-step-p} shows,
for every \(\gamma \in \Gamma(\ka)\), that
\(\gamma\) acts quasisemisimply on \((\Gt^{\Gamma\conn})_\ka\)
and
\((\Gt^{\Gamma\conn})_\ka^\gamma\) is smooth.
Thus we may, and do, assume, upon replacing
\(\Gt\) by \((\Gt^{\Gamma\conn})\conn\) and
\(\Gamma\) by its image in \(\Aut((\Gt^{\Gamma\conn})\conn)\),
that
\(\Gamma\) is \'etale, hence constant.

We now proceed by induction on
\(\dim(\Gt) + \card\Gamma\).
If the sum is \(1\), then
\(\Gt\), and hence the result, is trivial.

Suppose that we have proven the result for \(\Gt\adform\).
In particular, \((\Gt\adform^\Gamma)\conn\) is smooth, so
Corollary \ref{cor:lift-smooth} gives
(\ref{subthm:loc-quass-smoothable}).
Corollary \ref{cor:fixed-surjective} gives that
\(\fix\Gt^\Gamma\) is a smooth, connected group that is
an extension of the reductive group \(\fix\Gt\adform^\Gamma\) by
a group of multiplicative type, whence (\ref{subthm:loc-quass-reductive}).
To handle the reduction of (\ref{subthm:loc-quass-spherical-bldg}) to
the adjoint case, we argue once more as in the proofs of
Lemma \ref{lem:spherical-descent} and
Theorem \ref{thm:quass}(\ref{subthm:quass-spherical-bldg}).
Let \(b_+\) be a point of
\(\SS(\Gt) \cap \SS(\Gt_\ka)^{\Gamma(\ka)}\), and
\(b'_+\) its image in \(\SS(\Gt\adform)\).
(Functoriality of the formation of spherical buildings is discussed
only with respect to embeddings in
\cite{curtis-lehrer-tits:spherical}*{\S4}, or
isogenies in \cite{curtis-lehrer-tits:spherical}*{\S4, Remark (iv)},
but this is only needed if we insist that
the resulting map of spherical buildings be
an injection.
An arbitrary homomorphism of reductive groups still
gives a perfectly good map of
the corresponding spherical buildings in
the obvious fashion.)
Then \(b'_+\) belongs to
\(\SS(\Gt\adform) \cap \SS((\Gt\adform)_\ka)^{\Gamma(\ka)}\),
hence, by assumption, is the image in
\(\SS(\Gt\adform)\) of
an element of \(\SS(\fix\Gt\adform^\Gamma)\),
which we will also denote by \(b'_+\).
Let \(b'_-\) be any point of
the spherical building \(\SS(\fix\Gt\adform^\Gamma)\)
opposite to \(b'_+\).
The pullback \(P_\Gt(b'_-)\) to \(\Gt\) of
the corresponding parabolic subgroup \(P_{\Gt\adform}(b'_-)\) of
\(\Gt\adform\) is
a \(\Gamma\)-stable parabolic subgroup of \(\Gt\) that
is opposite to \(P_\Gt(b_+)\).
It follows from
\cite{curtis-lehrer-tits:spherical}*{\S3} that there is
a unique point \(b_- \in \SS(\Gt)\) that is
opposite to \(b_+\) and satisfies
\(P_\Gt(b_-) = P_\Gt(b'_-)\).
By uniqueness,
\(b_-\) is also fixed by \(\Gamma(k)\), so
Lemma \ref{lem:spherical-cr} gives that
\(b_+\) belongs to \(\SS(G)\).

That is, we may, and do, assume that \(\Gt\) is adjoint.
Now (\ref{subthm:loc-quass-smoothable}) is the statement that
\((\Gt^\Gamma)\conn\) is smooth, not just smoothable.
Since this statement is unaffected by arbitrary base change,
we may, and do, assume, upon
replacing \(k\) by \(\ka\), that
\(k\) is algebraically closed.

By
Corollary \ref{cor:ind-simple} and
Lemma \ref{lem:was-cor:trans-induction},
we may, and do, assume, upon
replacing
\(\Gt\) by an (absolutely) almost-simple component \(\Gt_1\) and
\(\Gamma\) by \(\stab_\Gamma(\Gt_1)\), that
\(\Gt\) is almost simple.
Suppose that there is a nontrivial
normal subgroup \(\Gamma'\) of \(\Gamma\) such that
\(\Gamma'(k)\) contains only semisimple elements.
Then \cite{prasad-yu:actions}*{Theorem 2.1 and Proposition 3.4}
gives that
\((\Gt^{\Gamma'})\conn\) is (smooth and) reductive, and that
\(\SS((\Gt^{\Gamma'})\conn)\) equals \(\SS(\Gt)^{\Gamma'(k)}\).
(In particular, note that this implies that (\ref{subthm:loc-quass-spherical-bldg})
holds for \((\Gt,\Gamma')\), as \(k  =\ka\).)
Proposition \ref{prop:inductive-step} gives that every
\(\gamma \in \Gamma(k)\) acts
quasisemisimply, with smooth fixed-point group, on
\((\Gt^{\Gamma'})\conn\).
Thus we may, and do, conclude by
applying the inductive hypothesis to
the action of \(\Gamma/\Gamma'\) on
\((\Gt^{\Gamma'})\conn\).

We thus may, and do, assume that
there is no such normal subgroup of \(\Gamma\).
Since \(\Gamma \cap \uInn(\Gt)\) is
a normal subgroup of \(\Gamma\) that, by
Remark \ref{rem:torus-quass}(\ref{subrem:quass-to-torus}),
has only semisimple \(k\)-rational points,
it is trivial; that is, the action of
\(\Gamma\) on \(\Gt\) is purely outer.
By Remark \ref{rem:inspection},
we have that \(\Gamma\) is cyclic, or
\(p\) equals \(3\), \(\Gt\) is of type \(\mathsf D_4\), and
\abmap\Gamma{\uOut(\Gt)} is an isomorphism.
(It cannot happen that \(p\) equals \(2\),
\(\Gt\) is of type \(\mathsf D_4\), and
\abmap\Gamma{\uOut(\Gt)} is an isomorphism, since then
\(\Gamma\) would have a
normal subgroup of order relatively prime to \(p\).)
Let \(\gamma\) be a generator of \(\Gamma\)
(if \(\Gamma\) is cyclic), or
a generator of the normal, order-\(3\) subgroup of \(\Gamma\)
(in the \(\mathsf D_4\) case).
Then \(\Gt^\gamma\) is smooth by assumption, so
\((\Gt^\gamma)\conn\) equals \(\fix\Gt^\gamma\), hence
is reductive by
Theorem \ref{thm:quass}(\ref{subthm:quass-reductive}).
Theorem \ref{thm:quass}(\ref{subthm:quass-spherical-bldg})
gives that
\(\SS((\Gt^\gamma)\conn)\) equals
\(\SS(\Gt)^\gamma\).
Since \(\sgen\gamma\) is normal in \(\Gamma\), the result follows by
applying the inductive hypothesis to
the action of \(\Gamma/\sgen\gamma\) on
\((\Gt^\gamma)\conn\).
\end{proof}

\begin{example}
\label{ex:not-quass}
Theorem \ref{thm:loc-quass}%
	(\ref{subthm:loc-quass-smoothable},%
	\ref{subthm:loc-quass-reductive})
can fail if we remove the assumption
that every element of $\Gamma(\ka)$ preserves a Borel--torus pair
in \(\Gt_\ka\).
Suppose that \(p\) does not equal \(1\), and put
$\Gt = \SL_{2, k}$.
If $\Gamma$ is the constant \(k\)-group generated by
$\Int\begin{smallpmatrix}1&1\\0&1\end{smallpmatrix}$,
then $G$ is not reductive. (A similar counterexample also holds
in arbitrary characteristic if one replaces the above constant group by the
group of upper triangular unipotent matrices in \(\Gt\).)
In fact, the behavior of the fixed-point group can be even worse.
If \(p\) equals \(2\),
\(k\) is not perfect,
and
\(\Gamma\) is instead generated by
\(\Int\begin{smallpmatrix} 0 & 1 \\ t & 0 \end{smallpmatrix}\),
where \(t\) is a non-square in \(k\),
then
\(\Gt^\Gamma = \set
	{\begin{smallpmatrix} a & b \\ b t & a \end{smallpmatrix}}
	{a^2 - b^2 t = 1}\)
is reduced and connected, but not geometrically reduced.
We have that
\((\Gt^\Gamma)\smooth\) is trivial,
and
\(((\Gt^\Gamma)_\ka)\smooth\) equals
\(\set
	{\begin{smallpmatrix} a & b \\ b t & a \end{smallpmatrix}}
	{a - b\sqrt t = 1}\),
which does not descend to a subgroup of \(\Gt\).
\end{example}

\begin{example}
\label{ex:vust:cones:prop:5:cor}
Theorem \ref{thm:loc-quass}(\ref{subthm:loc-quass-reductive})
can fail without the hypothesis about
smoothness of fixed points.

Consider the involution
\mapto\gamma
	\gt
	{\Int\begin{smallpmatrix}
		0 & 0 & 1 \\
		0 & 1 & 0 \\
		1 & 0 & 0
	\end{smallpmatrix}\gt^{-\mathsf T}}
of \(\GL_3\), which
we will also regard as an automorphism of
\(\Gt \ldef \PGL_3\).
Note that Proposition \ref{prop:sl-even}(\ref{subprop:gl-even}) implies
that $\Gt^\gamma$ is not smooth.
The map
\mapto\lambdat
	t
	{\begin{smallpmatrix}
	t & 1 - t & 1 - t \\
	0 & 1 & 1 - t \\
	0 & 0 & t
	\end{smallpmatrix}}
is a cocharacter of \(\GL_3\), and
\(\gamma \circ \lambdat\) equals \(-\lambdat\).
The general linear group of
the weight-\(1\) space for \(\lambdat\) in
the defining representation \(k^3\) of \(\GL_3\) maps
isomorphically onto \(C_\Gt(\lambdat)\), and
the ordered basis \(((1, 0, 0), (0, 1, -1))\) of
the weight-\(1\) space provides
an isomorphism with \(\GL_2\) that
identifies \(\lambdat\) with
an isomorphism from \(\GL_1\) onto \(Z(\GL_2)\).
Explicit computation shows that the involution of \(\GL_2\)
induced by \(\gamma\) is
\abmapto
	\gt
	{\det(\gt)\inv\Int\begin{smallpmatrix}
		-1 & 1 \\
		0 & 1
	\end{smallpmatrix}\gt}.

So far we have been agnostic about the characteristic.
Now, to fit this example into the general framework of
the rest of the section, suppose that \(p\) equals \(2\).
Then \(\gamma\) acts on \(\SL_2\) as conjugation by
a regular unipotent element and,
in particular, preserves no maximal (or even nontrivial) torus,
i.e., \(\im\lambda\) centralizes no maximal torus in \(\fix\Gt^\gamma\).
Moreover, note that since \(\fix\GL_2^\gamma\) is contained in \(\SL_2\),
it is the unipotent group
\(\fix\SL_2^\gamma = C_{\SL_2}{\begin{smallpmatrix}
	-1 & 1 \\
	0 & 1
\end{smallpmatrix}}\smooth\);
in particular, \(\fix C_\Gt(\lambdat)^\gamma\) is not reductive.
\end{example}

\appendix

\numberwithin{equation}{section}
\section{Induction of schemes with group action}
\label{app:ind}

In this section we quickly recall some definitions and results concerning sites and sheaves. In a way,
this is overkill for the purposes of this paper,
where ultimately we are only interested in
sheaves on one particular site. However, in order to clarify the roles that the various concepts play in
our results, it is useful to separate out this material from the main
development. For more details on the contents of this section, see \cite{vistoli:descent} or
\cite{milne:etale-cohom}. We will ignore
set-theoretic issues; they can be dealt with in a number of different ways, each of which would be
a distraction to the main aim of this appendix.

\begin{defn}\cite{vistoli:descent}*{Definition 2.24}
    A \textit{site} is a category $\mathcal{C}$ equipped with a collection of sets of morphisms
$\{\abmap{U_i}U\}_{i \in I}$, called \textit{covers}, subject to the following conditions.
    \begin{itemize}
        \item If $U$ is an object of $\mathcal{C}$, then $\{\map{\mathrm{id}}U U\}$ is a cover.
        \item If $\{\abmap{U_i}U\}_{i \in I}$ is a cover and $\abmap V U$ is a morphism in $\mathcal{C}$, then every
fiber product $U_i \times_U V$ exists and $\{\abmap{U_i \times_U V}V\}_{i \in I}$ is a
cover.
        \item If $\{\abmap{U_i}U\}_{i \in I}$ is a cover and for each $i \in I$ we are given a cover
$\{\abmap{V_{ij}}{U_i}\}_{j \in J_i}$, then $\{\abmap{V_{ij}}U\}_{\substack{i \in I \\ j \in J_i}}$ is a cover.
    \end{itemize}
    The collection of covers is called a \textit{topology} on $\mathcal{C}$.
(Often
(and originally in \cite{SGA-4.1}*{Expos\'e II, D\'efinition 1.3}),
this is called a \textit{pretopology}.)
\end{defn}

\begin{example}
    If $\mathcal{C}$ is any category, then it can be given the \textit{discrete topology}, whose only
covers are of the form $\{\map{\mathrm{id}}U U\}$
as \(U\) ranges over the objects of \(\mathcal C\).
\end{example}

\begin{example}\label{example:topologies-on-affsch}
    If $k$ is a ring and $\mathcal{C}$ is the category $\AffSch_k$ of affine $k$-schemes, then there are
several topologies on $\mathcal{C}$ which are commonly in use, among which
are the Zariski, \'etale, and fppf topologies. In each of these topologies, the coverings are jointly
surjective collections $\{\map{j_i}{U_i}U\}_{i \in I}$ of morphisms. In the Zariski
topology, each $j_i$ is an open embedding; in the \'etale topology, each $j_i$ is \'etale; in the fppf
topology, each $j_i$ is flat and locally of finite presentation. These sites are called the
(big) Zariski site, the (big) \'etale site, and the (big) fppf site, respectively.

    One can also define the \textit{small} Zariski, \'etale,
    and fppf sites of a ring $k$ as follows: let
$\mathcal{C}_\text{Zar}$,
$\mathcal{C}_\text{\'et}$, and $\mathcal{C}_\text{fppf}$ denote the full
subcategories of $\AffSch_k$ consisting of those affine $k$-schemes $X$ which are disjoint unions of
Zariski open subschemes of $\Spec k$
(respectively, \'etale over $\Spec k$;
respectively, fppf over $\Spec k$). We give these categories the Zariski, \'etale, and fppf topologies,
respectively.
\end{example}

\begin{example}
    If $k$ is a field, then the small Zariski site of $k$ has as objects $\Spec \prod_{i=1}^n k$ for every
integer $n \geq 0$. The small \'etale site of $k$ has as objects all schemes $\Spec
\prod_{i=1}^n k_i$, where each $k_i$ is a finite separable extension of $k$. The (small or big) fppf site
of $k$ is much larger: it includes \emph{all} $k$-algebras.
\end{example}

\begin{defn}[\cite{vistoli:descent}*{Definition 2.37}]
\label{defn:sheaf}
    If $\mathcal{C}$ and $\mathcal{D}$ are categories, then a \textit{$\mathcal{D}$-valued presheaf} on
$\mathcal{C}$ is a contravariant functor $\abmap{\mathcal{C}}{\mathcal{D}}$. If
$\mathcal{C}$
is a site, then a \textit{$\mathcal{D}$-valued sheaf} is a $\mathcal{D}$-valued presheaf $\mathcal{F}$ on
$\mathcal{C}$
satisfying the following \textit{sheaf condition}: if $\{\abmap{U_i}U\}_{i \in I}$ is a cover in $\mathcal{C}$, then the products $\prod_{i \in I} \mathcal{F}(U_i)$ and $\prod_{i, j \in I} \mathcal{F}(U_i
\times_U U_j)$ exist in $\mathcal{D}$ and the diagram
    \[
    \mathcal{F}(U) \to \prod_{i \in I} \mathcal{F}(U_i) \rightrightarrows \prod_{i, j \in I} \mathcal{F}(U_i \times_U U_j)
    \]
    is an equalizer in $\mathcal{D}$.

    If $X$ is an object of $\mathcal{C}$ and $a \in \mathcal{F}(X)$, then we will call $a$ a \textit{local
section} of $\mathcal{F}$. If $\abmap Y X$ is a morphism in $\mathcal{C}$, then we will often write $a_Y$
to denote the image of $a$ under the corresponding map $\abmap{\mathcal{F}(X)}{\mathcal{F}(Y)}$. We note
that this notation is abusive because $a_Y$ depends not just on $Y$, but on the \emph{map} $\abmap Y X$.
However, in practice there will only be one map under consideration at a time, so this should not lead
to substantial confusion.

    There is an evident notion of morphism for presheaves, and we let
$\PSh_{\mathcal{C}}(\mathcal{D})$
(respectively, $\Sh_{\mathcal{C}}(\mathcal{D})$) denote the category of
$\mathcal{D}$-valued
presheaves on $\mathcal{C}$ (respectively, its full subcategory of sheaves). We will use the simplifying
notation $\PSh_{\mathcal{C}} \ldef
\PSh_{\mathcal{C}}(\Sets)$
and $\Sh_{\mathcal{C}} \ldef \Sh_{\mathcal{C}}(\Sets)$.
By default, a \textit{sheaf} is assumed to
be set-valued (i.e., \(\mathcal D\) is \(\Sets\)),
while a \textit{group sheaf} is valued in the category of groups.
\end{defn}

One should have in mind that $\Sh_{\mathcal{C}}(\mathcal{D})$ has all good categorical properties
enjoyed by $\mathcal{D}$.
For example, if $\mathcal{D}$ admits products, then so does $\Sh_{\mathcal{C}}(\mathcal{D})$: given
two sheaves $\mathcal{F}$ and $\mathcal{G}$, one defines the product
$\mathcal{F}
\times \mathcal{G}$ to be the sheaf sending $X$ to $\mathcal{F}(X) \times \mathcal{G}(X)$; see
\cite{milne:etale-cohom}*{II, Lemma 2.12}.
Similar
constructions work for all finite limits and colimits.

If $\map F{\mathcal{D}}{\mathcal{E}}$ is a limit-preserving functor, then there is an induced functor
$\abmap{\Sh_{\mathcal{C}}(\mathcal{D})}{\Sh_{\mathcal{C}}(\mathcal{E})}$ (which we will also denote by
$F$) given by sending a sheaf $\mathcal{F}$ to the sheaf $\abmapto X{F(\mathcal{F}(X))}$. If $F$ is a
forgetful functor, then we will often omit explicit mention of this functor.

\begin{example}[\cite{milne:etale-cohom}*{II, Example 2.18(a)}]
    If $\mathcal{C}$ is a category which is considered as a site with the discrete topology and
$\mathcal{D}$
is any category, then $\Sh_{\mathcal{C}}(\mathcal{D})$
equals
$\PSh_{\mathcal{C}}(\mathcal{D})$.
If $\mathcal{C}$ is the category with a unique object and morphism, then
$\Sh_{\mathcal{C}}(\mathcal{D})$
is naturally equivalent to $\mathcal{D}$.
\end{example}

\begin{example}
    Let $k$ be a ring. A set-valued sheaf on the big Zariski site (respectively, the big \'etale site;
respectively, the big fppf site) is called a Zariski sheaf (respectively, \'etale sheaf;
respectively, fppf sheaf). A sheaf on the small Zariski site of $k$ is the same as a sheaf on the
topological space $\lvert\Spec k\rvert$, and this serves as the motivation for the general definition of sheaves on a site.
On the
other hand, a sheaf on the big Zariski site of $k$ contains \emph{much} more information; it takes
values on (the spectrum of) every $k$-algebra.

    If $X$ is a $k$-scheme, then there is a set-valued functor $h_X$ on $\AffSch_k$ defined by
$h_X(\Spec A) =
\Mor_k(\Spec A, X)$. By \cite{vistoli:descent}*{2.55}, the functor $h_X$ is an fppf sheaf
(and therefore also a Zariski sheaf and an \'etale sheaf). We call the functor $\abmapto X{h_X}$ the \textit{Yoneda embedding}.
We will (abusively) use the letter $X$ to refer also to the image of $X$ under the Yoneda embedding.
\end{example}

Just as in the topological case, if $\mathcal{C}$ is a site and $\mathcal{D}$ is either the category of
sets, the category of groups, or the category of abelian groups, then the inclusion
$\abmap{\Sh_{\mathcal{C}}(\mathcal{D})}{\PSh_{\mathcal{C}}(\mathcal{D})}$ admits a left adjoint $\abmapto{\mathcal{F}}{\mathcal{F}^\text{sh}}$,
called the \textit{sheafification} functor.
See \cite{vistoli:descent}*{Theorem 2.64} for the case
$\mathcal{D} = \Sets$.

\begin{example}
    If $\mathcal{C}$ is a site and $S$ is a set, then we define the \textit{constant sheaf}
$\underline S$
to be the sheafification of the presheaf $\abmapto X S$ on
$\mathcal{C}$.
(To see the effect of sheafification in our setting, note that,
if \(\mathcal C\) is \(\AffSch_k\) for some ring \(k\), then
\(\underline S(\Spec(k \oplus k))\) is
\(S \times S\), not just \(S\).)
If $S$ is a group, then $\underline S$ is a group sheaf on $\mathcal{C}$.
\end{example}

In practice, one is interested in \emph{sheaves} on sites, and \emph{not} in sites themselves. In our
case, we are largely interested in sheaves on the (big) fppf site of a field, and in fact
our interest lies mainly in group sheaves. One benefit of working with fppf group sheaves is that they
allow us to work with group schemes ``as if'' they were ordinary groups (see below). In order to
do this, we must first set up some formalism which is best understood in our abstract setting. Note
that a group sheaf on a site is the same as a group object in
the category of sheaves.

\begin{defn}[\cite{milne:etale-cohom}*{II, Theorem 2.15}]
    Let $\mathcal{C}$ be a site, and let $\mathcal{G}$ and $\mathcal{H}$ be group sheaves on
$\mathcal{C}$.
If $\map f{\mathcal{G}}{\mathcal{H}}$ is a homomorphism of group
sheaves (i.e., a morphism of sheaves
such that for every object $X \in \mathcal{C}$, the map $\map{f(X)}{\mathcal{G}(X)}{\mathcal{H}(X)}$
is a group homomorphism), then we define the \textit{kernel} $\ker f$ of $f$ by
$(\ker f)(X) = \ker(f(X))$ for every object $X \in \mathcal{C}$; with this definition, one can check that
$\ker
f$ is a group sheaf. If $f(\mathcal{G}(X)) \subseteq \mathcal{H}(X)$ is a normal
subgroup for every object $X$, then we define the cokernel $\coker f$ of $f$ to be the
\textit{sheafification} of the group presheaf $\abmapto X{\coker(f(X))}$. We say that $f$ is a
\textit{monomorphism} if $\ker f$ is the trivial sheaf $\abmapto X{\sset1}$ (i.e., if $f(X)$ is injective for all
$X$), and we say that $f$ is an \textit{epimorphism} if $\coker f$ is the trivial sheaf
(i.e., for each $X$ and each $h \in \mathcal{H}(X)$, there is a cover $\{\abmap{U_i}X\}$ and elements $g_i
\in \mathcal{G}(U_i)$ such that $f(g_i) = h_{U_i}$ for all $i$). The \textit{image}
$\im
f$ of $f$ is the cokernel of the map $\abmap{\ker f}{\mathcal{G}}$.
There is a natural monomorphism \abmap{\im f}{\mathcal H}, induced by \(f\), which we use to regard \(\im f\) as a subsheaf of \(\mathcal H\).
Using this identification, the image of \(f\) is the sheafification of \abmapto X{f(\mathcal G(X))}.
A sequence $\mathcal{F} \xrightarrow{\varphi} \mathcal{G} \xrightarrow{\psi} \mathcal{H}$ is \textit{exact} if $\ker \psi = \im \varphi$.
\end{defn}

\begin{example}
    If $k$ is a field and $n > 1$ is an integer, then the $n$th power map $\map{[n]}{\GL_1}{\GL_1}$ is
an epimorphism of fppf group sheaves, but it is usually \emph{not} an epimorphism
in the category of functors: there are many $k$-algebras $A$ which admit elements $a \in A$ which
are not $n$th roots. On the other hand, for any such $a \in A$, the algebra $B =
A[x]/(x^n
- a)$ is $A$-flat and admits an $n$th root of $a$.
\end{example}

Before specializing to the case of fppf group sheaves, we would like to attempt to motivate why one
would be interested in the fppf site over the Zariski or \'etale sites when studying
group schemes. In essence, if one is interested only in smooth $k$-group schemes and smooth
homomorphisms between them, then the \'etale site is entirely sufficient: for instance, if
$\abmap G Q$ is a smooth surjective homomorphism of smooth $k$-group schemes with kernel $N$,
then the sequence
\[
1 \to h_N \to h_G \to h_Q \to 1
\]
is a short exact sequence of \'etale group sheaves (but \emph{not} of group presheaves!). However, even
this simple statement fails when smoothness is relaxed: for example, let $k$ be an
imperfect field of characteristic exponent $p > 1$. The homomorphism $\map{[p]}{\GL_1}{\GL_1}$ is surjective with kernel $\mu_p$, but the corresponding sequence
\[
1 \to h_{\mu_p} \to h_{\GL_1} \xrightarrow{[p]} h_{\GL_1} \to 1
\]
of \'etale group sheaves is \emph{not} exact: if $x \in k$
is not a $p$th power, then there is no separable extension $k'/k$ such that $x$ is
a $p$th power in $k'$. Moreover, if one simply extends the small \'etale site to include \emph{all} finite
extensions of $k$, one still encounters the strange issue that $[p]$
is both a monomorphism and an epimorphism, but not an isomorphism;
this is a symptom of the fact that $\mu_p$ is a nontrivial group scheme which has
no nontrivial field-valued points. Thus, in order to work
with group schemes ``as if they are groups'', particularly in positive characteristic, one must work with
rings which are not fields. The fppf site is the right setting in which to do this.

In particular, the Yoneda embedding gives a fully faithful embedding from the category
of $k$-group schemes to the category of fppf group sheaves
on $\AffSch_k$, and it is a theorem of Grothendieck
(see
\cite{SGA-3.1}*{Expos\'e VI\textsubscript{B}, Th\'eor\`eme 3.2}
and
Theorem~\ref{thm:fppf-descent})
that, if $k$ is a field, then $\Sh_{\AffSch_k}(\Grp)$
is closed under quotients and extensions in the sheaf category.
Thus it is reasonable to transport the notion of exact sequence from the category of
fppf group sheaves to the category of finite type flat $k$-group schemes.

\begin{defn}
\label{defn:mor-sheaf}
Let $\mathcal{C}$ be a site and
\(\mathcal D\) a category.
If $X$ is an object of $\mathcal{C}$, then we define a category $\mathcal{C}_{/X}$
with objects being morphisms $\abmap Y X$ in $\mathcal{C}$, and morphisms
$\abmap{(\abmap Y X)}{(\abmap Z X)}$ being morphisms $\abmap Y Z$ which are compatible with the maps to $X$. One can give
a natural topology on $\mathcal{C}_{/X}$. If $\mathcal{F}$ is any sheaf on $\mathcal{C}$, then there is
a natural restricted sheaf $\res\mathcal{F}|_X$ defined by
\[
\res\mathcal{F}|_X(\abmap Y X) \coloneqq \mathcal{F}(Y),
\]
and with the natural restriction maps.

For sheaves $\mathcal{F}, \mathcal{G} \in \Sh_{\mathcal{C}}(\mathcal{D})$, there is a set of
morphisms $\Mor_{\Sh_{\mathcal{C}}(\mathcal{D})}(\mathcal{F}, \mathcal{G})$. We define
an object $\uMor(\mathcal{F}, \mathcal{G}) \in \Sh_{\mathcal{C}}(\Sets)$ by defining, for every
object $X$ of $\mathcal{C}$,
\[
\uMor(\mathcal{F}, \mathcal{G})(X) =
\Mor_{\Sh_{\mathcal{C}_{/X}}(\mathcal{D})}(\res\mathcal{F}|_X,
\res\mathcal{G}|_X).
\]
It is straightforward to check that $\uMor(\mathcal{F}, \mathcal{G})$ is a sheaf on
$\mathcal{C}$.
Notice that $\uMor(\mathcal{F}, \mathcal{F})$ is a sheaf of monoids.

One special case of interest to us is when
\(\mathcal D\) is the category of groups.
In that case, we will often write
\(\uHom\) instead of \(\uMor\).

Another case that is of interest is when
\(\mathcal D\) is the category of sets, but
the sheaf \(\mathcal G\) of sets arises by
forgetting the group structure on a sheaf of groups.
Then the resulting morphism
\(\mathcal G \times \mathcal G \to \mathcal G\) makes
each morphism set
\(\underline\Mor_{\Sh_{\mathcal C_{/X}}(\Sets)}(\mathcal F, \mathcal G)\)
into a group, so
\(\underline\Mor(\mathcal F, \mathcal G)\) is
a group sheaf.
\end{defn}

We briefly introduce two more constructions which will be useful later.

\begin{defn}
    Let $\mathcal{C}$ be a site, and let $\mathcal{F}$ be a sheaf of sets on $\mathcal{C}$. Define $\SubSh(\mathcal{F})(X)$ to be the set of subsheaves of $\res\mathcal{F}|_X$ on $\mathcal{C}_{/X}$. Note that this defines a presheaf $\SubSh(\mathcal{F})$ on $\mathcal{C}$.
\end{defn}

\begin{lem}\label{lemma:subsh-is-sheaf}
    If $\mathcal{C}$ is a site and $\mathcal{F}$ is a sheaf of sets on $\mathcal{C}$, then
$\SubSh(\mathcal{F})$ is a sheaf of sets on $\mathcal{C}$.
\end{lem}

\begin{proof}
    Let $\{\abmap{U_i}X\}$ be a cover in $\mathcal{C}$. We must show that the sequence
    \[
    \SubSh(\mathcal{F})(X) \to \prod_i \SubSh(\mathcal{F})(U_i) \rightrightarrows \prod_{i, j} \SubSh(\mathcal{F})(U_i \times_X U_j)
    \]
    is an equalizer sequence. Note first that the left arrow is injective: let $\mathcal{G}$ and
$\mathcal{G}'$ be two subsheaves of $\res\mathcal{F}|_X$ such that
$\res\mathcal{G}|_{U_i} = \res\mathcal{G}'|_{U_i}$ for all $i$. If $\abmap Y X$ is any morphism in $\mathcal{C}$,
then there is an equalizer sequence
    \[
    \mathcal{G}(Y) \to \prod_i \mathcal{G}(U_i \times_X Y) \rightrightarrows \prod_{i, j} \mathcal{G}(U_i \times_X U_j \times_X Y)
    \]
     and similarly for $\mathcal{G}'$.
(Note that we have abused notation by writing, e.g., $\mathcal{G}(Y)$ instead of the more proper $\mathcal{G}(\abmap Y X)$.)
Since all but the leftmost term in the above sequence are the same for $\mathcal{G}$ and
$\mathcal{G}'$, it follows that $\mathcal{G}(Y) = \mathcal{G}'(Y)$, as desired.

     Now let $\mathcal{G}_i$ be a subsheaf of $\res\mathcal{F}|_{U_i}$ on $\mathcal{C}_{/U_i}$ for each
$i$, and suppose that $\res\mathcal{G}_i|_{U_i \times_X U_j} = \res\mathcal{G}_j|_{U_i \times_X U_j}$ for all
$i$, $j$. For each morphism $\abmap Y X$ in $\mathcal{C}$, we get maps $\abmap{Y \times_X U_i}{U_i}$, and we
define $\mathcal{G}(Y)$ such that the sequence
\[
    \mathcal{G}(Y) \to \prod_i \mathcal{G}_i(U_i \times_X Y) \rightrightarrows \prod_{i, j} \mathcal{G}_{ij}(U_i \times_X U_j \times_X Y)
\]
is exact, where $\mathcal{G}_{ij} = \res\mathcal{G}_i|_{U_{ij}} = \res\mathcal{G}_j|_{U_{ij}}$. The fact that
$\mathcal{G}$ is a sheaf follows from a version of the nine lemma which we leave to the reader.
\end{proof}

\begin{lem}\label{lemma:incidence-sheaf}
    Let $\mathcal{F}$ be a sheaf on a site $\mathcal{C}$. Let $\mathcal{G}$ be the presheaf defined by
sending an object $X$ to the set of pairs $(\mathcal{G}_1, \mathcal{G}_2)$ in $\SubSh(\res\mathcal{F}|_X)^2$
satisfying $\mathcal{G}_1 \subseteq \mathcal{G}_2$.
Then $\mathcal{G}$ is a sheaf.
\end{lem}

\begin{proof}
    This follows directly from Lemma~\ref{lemma:subsh-is-sheaf} and the fact that, if $\mathcal{G}_1$
and $\mathcal{G}_2$ are two subsheaves of $\res\mathcal{G}|_X$, then one can check that
$\mathcal{G}_1 \subset \mathcal{G}_2$ after passing to a cover of $X$.
\end{proof}

\begin{defn}
    Let $\mathcal{C}$ be a site. If $\mathcal{F}$ is a sheaf of sets on $\mathcal{C}$ and $\Gamma$ is a
sheaf of groups on $\mathcal{C}$ with multiplication morphism $m_{\Gamma}$, then an
\textit{action} of $\Gamma$ on $\mathcal{F}$ is a map of sheaves $\map\alpha{\Gamma \times \mathcal{F}}{\mathcal{F}}$ such that the following diagram commutes:
\[
\begin{tikzcd}
\Gamma \times \Gamma \times \mathcal{F} \arrow[r, "m_{\Gamma} \times \mathrm{id}_{\mathcal{F}}"]
\arrow[d, "\mathrm{id}_{\Gamma} \times \alpha"]
    &\Gamma \times \mathcal{F} \arrow[d, "\alpha"] \\
\Gamma \times \mathcal{F} \arrow[r, "\alpha"]
    &\mathcal{F}
\end{tikzcd}
\]
Equivalently, a group action is a homomorphism $\abmap\Gamma{\uMor(\mathcal{F}, \mathcal{F})}$.
We call the pair \((\mathcal F, \alpha)\) a
\(\Gamma\)-sheaf, and will often suppress \(\alpha\) from the notation.

If $(\mathcal{F}_1,\alpha_1)$ and $(\mathcal{F}_2,\alpha_2)$ are
$\Gamma$-sheaves,
then we say that a morphism $\map f{\mathcal{F}_1}{\mathcal{F}_2}$ is
$\Gamma$-equivariant if $f \circ \alpha_1 = \alpha_2 \circ f$.
We denote the set of $\Gamma$-equivariant morphisms by $\Mor_{\Gamma}(\mathcal{F}_1,
\mathcal{F}_2)$. As before, there is also a sheaf of sets
$\uMor_{\Gamma}(\mathcal{F}_1,
\mathcal{F}_2)$.
The collection of \(\Gamma\)-sheaves (along with \(\Gamma\)-equivariant morphisms) assembles into a category
$\gpon{\Gamma}\Sh_{\mathcal{C}}$.

If
$\mathcal{F}$ is a $\Gamma$-sheaf, then we define a sheaf
of sets $\mathcal{F}^{\Gamma}$ by sending an object $X$ of $\mathcal{C}$ to the set
$\mathcal{F}(X)^{\Gamma}$
consisting of those elements $a$ of $\mathcal{F}(X)$ such that, for every morphism $\abmap Y X$ and
every $\gamma \in \Gamma(Y)$, we have $\gamma \cdot a_Y =
a_Y$. The map $\abmapto{\mathcal{F}}{\mathcal{F}^{\Gamma}}$ induces a functor
$\map{(\cdot)^{\Gamma}}
{\gpon{\Gamma}\Sh_{\mathcal{C}}}{\Sh_{\mathcal{C}}}$.
\end{defn}

Most familiar facts about group actions on sets extend straightforwardly to facts about actions of group
sheaves on sheaves of sets. We will give complete details only a couple of times;
as the reader will hopefully find, proving these extensions only requires more bookkeeping than the
usual facts, and it requires few new ideas.

\begin{lem}
	Let $\Gamma$ be a group sheaf on a site $\mathcal{C}$. The functor $(\cdot)^{\Gamma}$ is right
adjoint to the functor $\abmap{\Sh_{\mathcal{C}}}{\gpon{\Gamma}\Sh_{\mathcal{C}}}$
which gives a sheaf of sets the trivial $\Gamma$-action.
\end{lem}

\begin{proof}
	There is an evident inclusion $\Mor_{\Sh_{\mathcal{C}}}(\mathcal{F}, \mathcal{G}^{\Gamma})
\subseteq \Mor_{\gpon{\Gamma}\Sh_{\mathcal{C}}}(\mathcal{F}, \mathcal{G})$ for
$\mathcal{F}
\in \Sh_{\mathcal{C}}$ and $\mathcal{G} \in \gpon{\Gamma}\Sh_{\mathcal{C}}$
(where we have denoted the sheaf \(\mathcal F\) with
the trivial \(\Gamma\)-action again by
\(\mathcal F\)),
and we need only
check that it is an equality. Indeed, if the morphism $\map\phi{\mathcal{F}}{\mathcal{G}}$ is $\Gamma$-equivariant,
$X$ is an object of $\mathcal{C}$, and $f$ belongs to
$\mathcal{F}(X)$, then, for any $\abmap Y X$ and any $\gamma \in \Gamma(Y)$, we have
	\[
	\phi(f)_Y = \phi(f_Y) = \phi(\gamma \cdot f_Y) = \gamma \cdot \phi(f_Y) = \gamma \cdot \phi(f)_Y,
	\]
	so by definition $\phi(f)$ belongs to $\mathcal{G}^{\Gamma}(X)$.
\end{proof}

Let $\map f{\Gamma'}\Gamma$ be a homomorphism of group sheaves on a site $\mathcal{C}$,
and let $\mathcal{F}$ be a sheaf of sets on $\mathcal{C}$. Define an action of
$\Gamma'$
on $\uMor(\Gamma, \mathcal{F})$ as follows: if $X$ is an object, $\gamma' \in \Gamma'(X)$
and $\gamma \in \Gamma(X)$ are local sections, and $\map\phi
{\res\Gamma|_X}{\res\mathcal{F}|_X}$ is a morphism, then we set $(\gamma' \cdot \phi)(\gamma) =
\phi(f(\gamma')^{-1} \gamma)$.

\begin{rem}
\label{rem:act-on-power}
If $\map f{\Gamma_1}\Gamma$ is a homomorphism,
then we may equip \(\Gamma\) with the structure of a \(\Gamma_1\)-sheaf
via right multiplication
(that is, for every object \(X\) of \(\mathcal C\) and every
\(\gamma_1 \in \Gamma_1(X)\) and \(\gamma \in \Gamma(X)\),
the result of acting by \(\gamma_1\) on \(\gamma\) is
\(\gamma f(\gamma_1)\inv\)).
If $\mathcal{F}$ is a $\Gamma_1$-sheaf, then the
$\Gamma$-action
on $\uMor(\Gamma, \mathcal{F})$
deduced from the identity map \abmap\Gamma\Gamma\
preserves the
subsheaf $\uMor_{\Gamma_1}(\Gamma, \mathcal{F})$. Thus
$\uMor_{\Gamma_1}(\Gamma, (\cdot))$ is a functor from $\Gamma_1$-sheaves to
$\Gamma$-sheaves.
When we take into account the group-sheaf structure on
morphism sets defined in
Definition \ref{defn:mor-sheaf}, we have that
\(\underline\Mor_{\Gamma_1}(\Gamma, (\cdot))\)
restricts to a functor from
group \(\Gamma_1\)-sheaves to
group \(\Gamma\)-sheaves
\end{rem}

\begin{lem}
\label{lem:act-on-power}
	Let $\Gamma' \subseteq \Gamma$ and $\Delta$ be group sheaves on a site $\mathcal{C}$, and suppose that
$\Gamma$
acts on $\Delta$ through group automorphisms. If $\mathcal{F}$ is a (\(\Gamma \ltimes \Delta\))-sheaf
on $\mathcal{C}$, then the restriction morphism
$\map\rho
{\uMor_{\Gamma' \ltimes \Delta}(\Gamma \ltimes \Delta, \mathcal{F})}
{\uMor_{\Gamma'}(\Gamma, \mathcal{F})}$
is a \(\Gamma\)-equivariant isomorphism of sheaves.
\end{lem}

\begin{proof}
It is clear that \(\rho\) is \(\Gamma\)-equivariant.
To show that \(\rho\) is an isomorphism,
	we must define an inverse morphism $\eta$. To do so, if $X$ is an
object of $\mathcal{C}$ and $\map\psi{\res\Gamma|_X}{\res\mathcal{F}|_X}$ is a $\Gamma'$-equivariant morphism,
define $\map{\eta(\psi)}{\res(\Gamma \ltimes \Delta)|_X}{\res\mathcal{F}|_X}$ as follows: if $\abmap Y X$ is a
morphism in $\mathcal{C}$ and $(\gamma, \delta)$ belongs to $(\Gamma \ltimes \Delta)(Y)$,
then we set $\eta(\psi)(\gamma, \delta) = \gamma\delta\gamma^{-1} \cdot \psi(\gamma)$. The fact that $\eta(\psi)$ is
$\Gamma' \ltimes \Delta$-equivariant
is baked into the definition, and the fact that $\eta \circ \rho$ and
$\rho \circ \eta$ are the respective identity morphisms is a direct calculation from the definitions.
\end{proof}

For the remainder of the appendix, let $\Gamma_1 \subseteq \Gamma$ be an inclusion of group
sheaves on some fixed site $\mathcal{C}$.
As in Remark \ref{rem:act-on-power},
we let $\Gamma_1$ act on $\Gamma$ by
(inverted) right translation.
From now on, we will begin to be less strict about choosing objects in proofs, relying more heavily on
the terminology of local sections. We do this with the aim that it will make the following proofs less
heavy on notation, without sacrificing too much clarity.

\begin{lem}
\label{lem:ind-first-adjoint}
	The functor $\map{\uMor_{\Gamma_1}(\Gamma, (\cdot))}
	{\gpon{\Gamma_1}\Sh_{\mathcal{C}}}{\gpon{\Gamma}\Sh_{\mathcal{C}}}$
	from Remark~\ref{rem:act-on-power} is right adjoint to the forgetful functor.
\end{lem}

\begin{proof}
	Let $\mathcal{F}$ and $\mathcal{G}$ be sheaves on $\mathcal{C}$, equipped with actions of \(\Gamma\) and \(\Gamma_1\), respectively.
We define $\map\eta
{\Mor_{\Gamma_1}(\mathcal{F}, \mathcal{G})}
{\Mor_{\Gamma}(\mathcal{F}, \uMor_{\Gamma_1}(\Gamma, \mathcal{G}))}$
by $\eta(\phi)(x)(\gamma) \ldef
\phi(\gamma^{-1} x)$ for local sections $x$ of $\mathcal{F}$ and $\gamma$ of $\Gamma$. Note first
that $\eta(\phi)(x)$ is $\Gamma_1$-equivariant for every $x$: indeed, if \(\gamma_1\) is a local section of \(\Gamma_1\), then
we have
	\[
	\eta(\phi)(x)(\gamma \gamma_1^{-1}) = \phi(\gamma_1 \gamma^{-1} x) = \gamma_1 \cdot
\phi(\gamma^{-1}x) = \gamma_1 \cdot \eta(\phi)(x).
	\]
	Next, $\eta(\phi)$ is $\Gamma$-equivariant: indeed, if $\gamma_0$ is a local section of $\Gamma$, then we have
	\[
	\eta(\phi)(\gamma_0 \cdot x)(\gamma) = \phi(\gamma^{-1} \gamma_0 x) =
\eta(\phi)(x)(\gamma_0^{-1}
\gamma) = (\gamma_0 \cdot \eta(\phi)(x))(\gamma).
	\]
	We define now an inverse map $\map\rho
{\Mor_{\Gamma}(\mathcal{F}, \uMor_{\Gamma_1}(\Gamma, \mathcal{G}))}
{\Mor_{\Gamma_1}(\mathcal{F}, \mathcal{G})}$
via $\rho(\psi)(x) \ldef \psi(x)(1_\Gamma)$. To see that this is well defined, we compute
	\[
	\rho(\psi)(\gamma_1 \cdot x) = \psi(\gamma_1 \cdot x)(1_\Gamma) = (\gamma_1 \cdot \psi(x))(1_\Gamma) =
\psi(x)(\gamma_1^{-1})
= \gamma_1 \cdot \psi(x)(1_\Gamma) = \gamma_1 \cdot \rho(\psi)(x).
	\]
	It is straightforward to check that $\eta$ and $\rho$ are mutually inverse, as desired.
\end{proof}

\begin{lem}\label{lem:was-cor:trans-induction}
	There is a natural isomorphism between the functors
$\uMor_{\Gamma_1}(\Gamma, (\cdot))^{\Gamma}$ and $(\cdot)^{\Gamma_1}$
from $\gpon{\Gamma_1}\Sh_{\mathcal{C}}$ to $\Sh_{\mathcal{C}}$,
given by evaluation at $1_{\Gamma}$.
\end{lem}

\begin{proof}
	First, this is actually well-defined: let $\mathcal{F}$ be a sheaf with $\Gamma_1$-action and let
$\map\phi\Gamma{\mathcal{F}}$ be a $\Gamma_1$-equivariant morphism which is fixed by the
$\Gamma$-action. If $\map i{\Gamma_1}\Gamma$ is the inclusion, then for all local sections
$\gamma$ of $\Gamma$ and $\gamma_1$ of $\Gamma_1$ we have
	\[
	\gamma_1 \cdot \phi(1_{\Gamma}) = \phi(\gamma_1) = (i(\gamma_1)^{-1} \cdot \phi)(1_{\Gamma}) = \phi(1_{\Gamma}).
	\]
	We now define the inverse natural transformation
$\map\eta{(\cdot)^{\Gamma_1}}{\uMor_{\Gamma_1}(\Gamma, (\cdot))^{\Gamma}}$ as
follows: let $\mathcal{F}$ be a sheaf with $\Gamma_1$-action and let $a$ be a local section of
$\mathcal{F}$ which is fixed by the $\Gamma_1$-action. We define
$\map{\eta(a)}\Gamma{\mathcal{F}}$ via $\eta(a)(\gamma) = a$ for all $\gamma$, and note that $\eta$ is clearly an inverse.
\end{proof}

\begin{lem}\label{lem:induction-split-case}
	Suppose that the quotient $\Gamma/\Gamma_1$ is isomorphic to the constant sheaf
$\underline{S}$ for some set $S$ and there is a sheaf-theoretic section
$\map{\sigma}{\underline{S}}{\Gamma}$. There is a unique natural isomorphism
$\map{\epsilon_\sigma}
{\uMor_{\Gamma_1}(\Gamma, (\cdot))}{\prod_S (\cdot)}$
whose composition with the projection on the factor
corresponding to $s \in S$ is given by evaluation at $\sigma(s)$.
\end{lem}

\begin{proof}
	There is clearly a natural transformation as claimed. In the reverse direction, if $\mathcal{F}$ is a
$\Gamma_1$-sheaf and $(x_s)_{s \in S}$ is a local section of $\prod_S
\mathcal{F}$, then we define $\eta_\sigma((x_s)_{s \in S})(\sigma(s)\gamma_1^{-1}) = \gamma_1 \cdot
x_s$, and check that $\epsilon_\sigma$ and $\eta_\sigma$ are mutually
inverse.
\end{proof}

We finally specialize to the category of finite type affine group schemes
over a ring (eventually, a field).
Thus, from now on, we fix a ring $k$ and consider only the site $\mathcal C = \AffSch_k$ equipped with the fppf topology.
Our work in the sequel will require
some basic elements of descent theory, summarized in
Theorem \ref{thm:fppf-descent}.
We will assume further after the theorem that
\(k\) is a field, but we do not do so yet.

\begin{thm}\label{thm:fppf-descent}
	Let $k$ be a ring, and let $k'$ be a faithfully flat $k$-algebra. Let $\mathcal{F}$ be an fppf sheaf
over $k$, and suppose that the restriction $\mathcal{F}_{k'}$ to the category of affine $k'$-schemes is
isomorphic to $h_{X'}$ for some affine $k'$-scheme $X'$. Then there is an affine $k$-scheme $X$ such that $\mathcal{F} \cong h_X$. If $X'$ is of finite type (respectively, smooth), then the same is true of $X$.
\end{thm}

\begin{proof}
	Let $\map{p_1, p_2}{\Spec(k' \otimes_k k')}{\Spec(k')}$ be the two projection maps. Note
that there is a natural isomorphism $p_1^* X' \cong p_2^* X'$ of affine (\(k' \otimes_k k'\))-schemes
between the pullbacks of $X'$ along $p_1$ and $p_2$,
coming from the fact that the pullbacks
$p_i^* \mathcal{F}_{k'}$ both equal $\mathcal{F}_{k' \otimes_k k'}$ for $i = 1, 2$. This isomorphism is compatible
in the natural way with the three projection maps
$\map{p_{ij}}{\Spec(k' \otimes_k k' \otimes_k k')}{\Spec(k' \otimes_k k')}$, i.e., it forms
a \textit{descent datum} in the sense of
\cite{bosch-lutkebohmert-raynaud:neron}*{Section 6.1}.
By \cite{bosch-lutkebohmert-raynaud:neron}*{6.1, Theorem 6},
it follows that there is an affine
$k$-scheme $X$ such that $\mathcal{F} \cong h_X$. The final claims follow from
\citelist{
	\cite{grothendieck:EGA-IV.2}*{Proposition 2.7.1}
	\cite{grothendieck:EGA-IV.4}*{Corollaire 17.7.3}
}.
\end{proof}

We now assume that \(k\) is a field
(in addition to the standing assumption from before
Theorem \ref{thm:fppf-descent} that
\(\mathcal C\) is the site
\(\AffSch_k\)).
It is now convenient for us to view
affine \(k\)-schemes as special sorts of fppf sheaves;
that is, we will leave the Yoneda embedding implicit,
so that, for example, we may ask whether
a sheaf \(\mathcal F\) \emph{is} a scheme, meaning that
it is of the form \(h_X\) for some scheme \(X\).
In particular, we will now usually use the letter \(X\)
and related notation,
rather than \(\mathcal F\), for sheaves
when we expect most of the applications to be to schemes.

With this implicit identification in mind, we require that
the fppf group sheaf $\Gamma$ is actually
a smooth finite type $k$-group scheme, and
the subgroup sheaf $\Gamma_1$ is
an open $k$-subgroup scheme.
A \textit{\(\Gamma_1\)-scheme} is
a \(\Gamma_1\)-sheaf that is a scheme, and
similarly for \(\Gamma\)-schemes.

\begin{rem}
\label{rem:induction-split-case}
Let \(\Xt_1\) be an fppf \(\Gamma_1\)-sheaf.
With the notation and hypotheses of
Lemma \ref{lem:induction-split-case},
we may equip \(\Xt \ldef \prod_S \Xt_1\) with
an action of \(\Gamma(k)\) as follows.
For each \(s \in S\), write \(\pi_s\) for the
corresponding projection \abmap\Xt{\Xt_1}.
We equip each \(\gamma \in \Gamma(k)\) with
the unique action on \(\Xt\) so that
\(\pi_s \circ \gamma\) equals
\((\gamma s)\inv\sigma(\gamma s) \circ \pi_{\sigma(\gamma s)}\)
for every \(s \in S\).
Then the morphism
\(\epsilon_\sigma\) of
Lemma \ref{lem:induction-split-case} is
\(\Gamma(k)\)-equivariant.
\end{rem}

\begin{rem}
\label{rem:ind-bc+Res}
Let \(A\) be a \(k\)-algebra.

For every fppf \(\Gamma_1\)-sheaf \(\Xt_1\),
the fppf \(\Gamma_A\)-sheaves
\(\uMor_{\Gamma_{1\,A}}(\Gamma_A, \Xt_{1\,A})\) and
\(\uMor_{\Gamma_1}(\Gamma, \Xt_1)_A\)
are equal (not just naturally isomorphic!).

The functors
\(\WRes_{A/k}\uMor_{\Gamma_{1\,A}}(\Gamma_A, (\cdot))\)
and
\(\uMor_{\Gamma_1}(\Gamma, \WRes_{A/k}(\cdot))\)
from fppf \(\Gamma_{1\,A}\)-sheaves to fppf \(\Gamma\)-sheaves are
naturally isomorphic, because they are both left adjoint to
the ``forgetful base-change'' functor that sends
an fppf \(\Gamma\)-sheaf \(\Xt\) to
\(\Xt_A\), regarded as an fppf \(\Gamma_{1\,A}\)-sheaf.
\end{rem}

\begin{lem}\label{lem:local-sections-separable}
	There exists a finite separable extension $k'/k$ such that
$(\Gamma/\Gamma_1)_{k'}$ is constant over $\Spec k'$ and the map
$\abmap{\Gamma_{k'}}{(\Gamma/\Gamma_1)_{k'}}$ admits a section.
\end{lem}

\begin{proof}
	Since $\Gamma$ is smooth, $\Gamma(\ks)$ is dense in $\Gamma_{\ks}$. Thus there is a finite
separable extension $k'$ of $k$ such that each component of $\Gamma_{k'}$ contains a $k'$-point.
The conclusions of the lemma hold for this choice of $k'$.
\end{proof}

\begin{cor}
\label{cor:ksep-ind-scheme}
The functor
$\uMor_{\Gamma_1}(\Gamma, (\cdot))$ on fppf group \(\Gamma_1\)-sheaves over $k$ is exact.
\end{cor}

\begin{proof}
	Exactness of a sequence of group sheaves can be
checked after passage to a finite separable extension of
$k$, so by Lemma~\ref{lem:local-sections-separable} we may and do
assume that $\Gamma/\Gamma_1 \cong \underline{S}$ is
constant and the map $\abmap\Gamma{\Gamma/\Gamma_1}$ admits a
sheaf-theoretic section. In particular, if
$1 \to \mathcal{F} \to \mathcal{G} \to \mathcal{H} \to 1$ is an
exact sequence of fppf group sheaves over $k$
equipped with $\Gamma_1$-actions, then Lemma~\ref{lem:induction-split-case}
shows that the sequence obtained by
applying $\uMor_{\Gamma_1}(\Gamma, (\cdot))$ can
be identified with the product sequence
	\[
	1 \to \prod_{s \in S} \mathcal{F} \to \prod_{s \in S} \mathcal{G} \to \prod_{s \in S} \mathcal{H} \to 1,
	\]
	which is exact because finite products are exact in the category of sheaves.
\end{proof}

Write \(k_2\) for the dual numbers
\(k[\epsilon]/(\epsilon^2)\).
The evaluation map that sends \(\epsilon\) to \(0\) provides
a ring homomorphism
\abmap{k_2}k, which we use to regard \(k\) as a \(k_2\)-algebra.
For any fppf \(k\)-sheaf \(X\),
the identity map on \(X = (X_{k_2})_k\) provides a morphism
\abmap{X_{k_2}}{\WRes_{k/k_2} X}, and we obtain
by functoriality a canonical map
\abmap{\WRes_{k_2/k} X_{k_2}}{\WRes_{k_2/k}\WRes_{k/k_2} X \cong X}.
If \(X\) is a group sheaf, then, by definition, its Lie algebra is
the kernel of this map.

\begin{cor}
\label{cor:ind-Lie}
If \(\Gt_1\) is an fppf group \(\Gamma_1\)-sheaf, then
the sheaf isomorphism
\abmap
	{\uMor_{\Gamma_1}(\Gamma, \WRes_{k_2/k}\Gt_{1\,k_2})}
	{\WRes_{k_2/k}\uMor_{\Gamma_{1\,k_2}}(\Gamma_{k_2}, \Gt_{1\,k_2})}
from Remark \ref{rem:ind-bc+Res}
restricts to a Lie-algebra isomorphism of
\(\uMor_{\Gamma_1}(\Gamma, \uLie(\Gt_1))\)
onto
\(\uLie(\uMor_{\Gamma_1}(\Gamma, \Gt_1))\).
\end{cor}

\begin{proof}
It is clear that the sheaf isomorphism is also a group isomorphism.
That it carries
\(\uMor_{\Gamma_1}(\Gamma, \uLie(\Gt_1))\) onto
\(\uLie(\uMor_{\Gamma_1}(\Gamma, \Gt_1))\) follows from
exactness of \(\uMor_{\Gamma_1}(\Gamma, (\cdot))\)
(Corollary \ref{cor:ksep-ind-scheme}), applied to
the exact sequence
\(\xymatrix@1{
0 \ar[r] & \uLie(\Gt_1) \ar[r] & \WRes_{k_2/k}\Gt_{1\,k_2} \ar[r] & \Gt_1 \ar[r] & 1
}\).
Since the Lie-algebra structure on
\(\uLie(\Gt_1)\) is deduced from
the group-sheaf structure on \(\Gt_{1\,k_2}\)
\cite{demazure-gabriel:groupes-algebriques}*
	{Ch.~II, \S4, Proposition 4.5},
and analogously for
\(\uLie(\uMor_{\Gamma_1}(\Gamma, \Gt_1))\),
it follows that the restriction is
a Lie-algebra isomorphism.
\end{proof}

\begin{lem}
\label{lem:ind-Frob}
Write \(p\) for the characteristic exponent of \(k\),
\((\cdot)^{(p)}\) for the Frobenius twist,
and \(\operatorname{Frob}_{(\cdot)}\)
for the Frobenius natural transformation
\abmap{(\cdot)}{(\cdot)^{(p)}}
\cite{demazure-gabriel:groupes-algebriques}*
	{Ch.~II, \S7, 1.1}
(taken to be trivial if \(p\) equals \(1\)).
For every \(\Gamma_1\)-scheme \(\Xt_1\),
we have that the subfunctors
\(\uMor_{\Gamma_1}(\Gamma, \Xt_1)^{(p)}\),
\(\uMor_{\Gamma_1}(\Gamma^{(p)}, \Xt_1^{(p)})\), and
\(\uMor_{\Gamma_1^{(p)}}(\Gamma^{(p)}, \Xt_1^{(p)})\) of
\(\uMor(\Gamma, \Xt_1)^{(p)} =
\uMor(\Gamma^{(p)}, \Xt_1^{(p)})\)
are equal,
the functorial morphism
\abmap
	{\uMor_{\Gamma_1}(\Gamma^{(p)}, \Xt_1^{(p)})}
	{\uMor_{\Gamma_1}(\Gamma, \Xt_1^{(p)})}
is an isomorphism, and
the diagram
\[\xymatrix{
\uMor_{\Gamma_1}(\Gamma, \Xt_1) \ar[d]_{\Frob_{\Xt_1}\circ(\cdot)}\ar[r]^-\Frob & \uMor_{\Gamma_1}(\Gamma, \Xt_1)^{(p)} \ar@{=}[r] & \uMor_{\Gamma_1}(\Gamma^{(p)}, \Xt_1^{(p)}) \ar[dll]^{(\cdot)\circ\Frob_\Gamma} \\
\uMor_{\Gamma_1}(\Gamma, \Xt_1^{(p)})
}\]
commutes.
\end{lem}

\begin{proof}
We have that
\(\uMor_{\Gamma_1}(\Gamma, \Xt_1)^{(p)} =
(\uMor(\Gamma, \Xt_1)^{\Gamma_1})^{(p)}\)
equals
\((\uMor(\Gamma, \Xt_1)^{(p)})^{\Gamma_1} =
(\uMor(\Gamma^{(p)}, \Xt_1^{(p)}))^{\Gamma_1} =
\uMor_{\Gamma_1}(\Gamma^{(p)}, \Xt_1^{(p)})\).
Since
\(\Gamma_1\) acts on
\(\uMor(\Gamma^{(p)}, \Xt_1^{(p)})\)
\textit{via}
\abmap{\Gamma_1}{\Gamma_1^{(p)}}, which is a quotient map,
we have that
\(\uMor_{\Gamma_1}(\Gamma^{(p)}, \Xt_1^{(p)}) =
\uMor(\Gamma^{(p)}, \Xt_1^{(p)})^{\Gamma_1}\) equals
\(\uMor(\Gamma^{(p)}, \Xt_1^{(p)})^{\Gamma_1^{(p)}} =
\uMor_{\Gamma_1^{(p)}}(\Gamma^{(p)}, \Xt_1^{(p)})\).

We have a commutative diagram
\[\xymatrix{
1 \ar[r] & \Gamma_1 \ar@{->>}[d]_{\Frob_{\Gamma_1}}\ar[r] & \Gamma \ar@{->>}[d]_{\Frob_\Gamma}\ar[r] & \Gamma/\Gamma_1 \ar[d]^{\Frob_{\Gamma/\Gamma_1}}_*[@!270]{\sim}\ar[r] & 1 \\
1 \ar[r] & \Gamma_1^{(p)} \ar[r] & \Gamma^{(p)} \ar[r] & (\Gamma/\Gamma_1)^{(p)} \ar[r] & 1.
}\]
Since the top row is exact,
the left and middle arrows are quotient maps, and
the right arrow is an isomorphism,
it follows from the nine lemma
and the fact that
\(\ker(\Frob_{\Gamma_1})\) equals \(\ker(\Frob_\Gamma)\)
that the bottom row is exact.
That is, the functorial map
\abmap{\Gamma^{(p)}}{(\Gamma/\Gamma_1)^{(p)}}
factors through an isomorphism
\abbimap{\Gamma^{(p)}/\Gamma_1^{(p)}}{(\Gamma/\Gamma_1)^{(p)}}.

To show that
\map{\Frob_\Gamma}
	{\uMor_{\Gamma_1^{(p)}}(\Gamma^{(p)}, \Xt_1^{(p)})}
	{\uMor_{\Gamma_1}(\Gamma, \Xt_1^{(p)})}
is an isomorphism,
it suffices to show that it becomes one after fppf base change.
Thus, by Lemma \ref{lem:local-sections-separable}, we may, and do,
assume that \(\Gamma/\Gamma_1\) is constant, and that
there is a section
\map\sigma{\Gamma/\Gamma_1}\Gamma.
The composition of
the isomorphism
\abbimap{\Gamma^{(p)}/\Gamma_1^{(p)}}{(\Gamma/\Gamma_1)^{(p)}}
with
\map{\sigma^{(p)}}{(\Gamma/\Gamma_1)^{(p)}}{\Gamma^{(p)}}
is a section of the quotient map
\abmap{\Gamma^{(p)}}{\Gamma^{(p)}/\Gamma_1^{(p)}}.
Thus Lemma \ref{lem:induction-split-case} shows that
the choice of \(\sigma\) furnishes isomorphisms
\abbimap
	{\uMor_{\Gamma_1}(\Gamma, \Xt_1^{(p)})}
	{\prod_{(\Gamma/\Gamma_1)(k)} \Xt_1^{(p)}}
and
\abbimap
	{\uMor_{\Gamma_1}(\Gamma^{(p)}, \Xt_1^{(p)}) =
		\uMor_{\Gamma_1^{(p)}}(\Gamma^{(p)}, \Xt_1^{(p)})}
	{\prod_{(\Gamma^{(p)}/\Gamma_1^{(p)})(k)} \Xt_1^{(p)}}
such that
\[\xymatrix{
\uMor_{\Gamma_1}(\Gamma^{(p)}, \Xt_1^{(p)}) \ar[d]_{(\cdot)\circ\Frob_\Gamma}\ar[r] & \prod_{(\Gamma/\Gamma_1)(k)} \Xt_1^{(p)} \ar[d] \\
\uMor_{\Gamma_1}(\Gamma, \Xt_1^{(p)}) \ar[r] & \prod_{(\Gamma^{(p)}/\Gamma_1^{(p)})(k)} \Xt_1^{(p)}
}\]
commutes, where
the right-hand arrow comes from the identification of
the two indexing sets \textit{via} \(\Frob_{\Gamma/\Gamma_1}\).
The result follows.
\end{proof}

\begin{prop}
\label{prop:ind-scheme}
If $\Xt_1$ is an affine scheme over \(k\) with $\Gamma_1$-action, then
\(\uMor_{\Gamma_1}(\Gamma, \Xt_1)\) is an affine $k$-scheme as well.
The properties
``of finite type'',
``smooth'',
``\'etale'',
and
``geometrically connected'', and,
when restricted to linear algebraic \(k\)-group schemes,
``connected'' and
``of multiplicative type'',
are preserved by \(\uMor_{\Gamma_1}(\Gamma, (\cdot))\).
\end{prop}

\begin{proof}
	By Lemma~\ref{lem:local-sections-separable}, there is a finite separable extension $k'$ of $k$
such that $(\Gamma/\Gamma_1)_{k'}$ is finite and constant and
the map $\abmap{\Gamma_{k'}}{(\Gamma/\Gamma_1)_{k'}}$ admits a sheaf-theoretic section. By
Lemma~\ref{lem:induction-split-case}, the sheaf $\uMor_{\Gamma_1}(\Gamma, \Xt_1)_{k'}$ is
(representable by) a finite power of $\Xt_{1\,k'}$ and thus is an affine scheme over \(k\).
By Theorem~\ref{thm:fppf-descent}, the functor $\uMor_{\Gamma_1}(\Gamma, \Xt_1)$ is therefore
representable by some affine $k$-scheme,
which is of finite type, respectively smooth, if
\(\Xt_{1\,k'}\) is.
Since all of the properties claimed to be preserved
can be checked after passing to a field extension
(see
\cite{grothendieck:EGA-IV.4}*{Corollaire 17.7.3}
and
\cite{milne:algebraic-groups}*{A.59 and Definition 12.17}%
),
the result follows.
\end{proof}

\begin{rem}
\label{rem:concrete-descent}
The proof of Proposition \ref{prop:ind-scheme} shows,
with the notation there, that
\(\uMor_{\Gamma_1}(\Gamma, \Xt_1)\) is
a scheme by an abstract descent argument, as in
Theorem \ref{thm:fppf-descent}.
We discuss how to realise this descent concretely in our case.
For each \(\gamma \in \Gamma(\ks)\), we can define
\(k_\gamma\) to be the fixed field in \(\ks\) of
\(\stab_{\Gal(k)}(\gamma\Gamma_1(\ks))\), and then construct
a \(\Gamma_1(\ks)\)-valued cocycle \(c_\gamma\) on \(\Gal(k_\gamma)\) by
\abmapto\sigma{\gamma\inv\sigma(\gamma)}.
Then \(\gamma\inv\) provides a morphism
from the base-changed scheme \(\Xt_{1\,k_\gamma}\) to
its twist \(\Xt_{1\,k_\gamma\,c_\gamma}\) by \(c_\gamma\), hence
a morphism
\abmap{\Xt_1}{\WRes_{k_\gamma/k} \Xt_{1\,k_\gamma\,c_\gamma}},
where \(\WRes_{k_\gamma/k}\) is the Weil restriction.
Put
\(\Xt_{1\,\gamma} = \WRes_{k_\gamma/k} \Xt_{1\,k_\gamma\,c_\gamma}\),
and
let us abuse notation by writing again \(\gamma\inv\) for
the map \abmap{\Xt_1}{\Xt_{1\,\gamma}} constructed above.

For every \(\gamma_1 \in \Gamma_1(\ks)\),
we have that
\(k_{\gamma\gamma_1}\) equals \(k_\gamma\),
\(c_{\gamma\gamma_1}\) is cohomologous to \(c_\gamma\),
so
\(\Xt_{1\,k_\gamma\,c_\gamma}\) is isomorphic to
\(\Xt_{1\,k_{\gamma\gamma_1}\,c_{\gamma\gamma_1}}\), and
there is
an isomorphism \abbimap{\Xt_{1\,\gamma}}{\Xt_{1\,\gamma\gamma_1}} such that
the diagram
\[\xymatrix{
\Xt_1 \ar@{=}[d]\ar[r]^{\gamma\inv} & \Xt_{1\,\gamma} \ar[d]^[@]{\sim} \\
\Xt_1 \ar[r]^{\gamma\gamma_1\inv} & \Xt_{1\,\gamma\gamma_1}
}\]
commutes.  Thus we may, and do, regard not just
\(\Xt_{1\,\gamma}\), but also
the map \map{\gamma\inv}{\Xt_1}{\Xt_{1\,\gamma}}, as
depending only on
the coset \(\gamma\Gamma_1(\ks)\).
Similarly, replacing \(\gamma\) by a \(\Gal(k)\)-conjugate replaces
\(k_\gamma\) and
\(\Xt_{1\,k_\gamma\,c_\gamma}\)
by the corresponding \(\Gal(k)\)-conjugates, but
does not change the isomorphism type of
\(\Xt_{1\,\gamma}\), affording a commutative diagram as above.
Thus we may, and do, regard both \(\Xt_{1\,\gamma}\) and
\map{\gamma\inv}{\Xt_1}{\Xt_{1\,\gamma}} as depending only on
the \(\Gal(k)\)-orbit of \(\gamma\Gamma_1(\ks)\).
(So what we are really doing is considering, not just
\(\Xt_{1\,\gamma}\), but the injective limit
\(\injlim_{(\sigma, \gamma_1) \in \Gal(k) \ltimes \Gamma_1(\ks)}
	\sigma(\Xt_{1\,\gamma\gamma_1})\).)
Now consider the map
\[\abmap
	{\uMor_{\Gamma_1}(\Gamma, \Xt_1)}
	{\prod_\gamma
	\Xt_{1\,\gamma}},\]
where \(\gamma\) ranges over
the set of \(\Gal(k)\)-orbits on \(\Gamma(\ks)/\Gamma_1(\ks)\).
This map is an isomorphism, because,
upon base change to \(\ks\), it becomes
\[\abmap
	{\uMor_{\Gamma_{1\,\ks}}(\Gamma_\ks, \Xt_{1\,\ks})}
	{\prod_{\gamma \in \Gal(k)\backslash\Gamma(\ks)/\Gamma_1(\ks)}
		\prod_{\sigma \in \Gal(k)/\stab \gamma\Gamma_1(k)}
			\Xt_{1\,\ks}};
\]
and this latter is an isomorphism by
Lemma \ref{lem:induction-split-case}
(with \(S\) there being \(\Gamma(\ks)/\Gamma_1(\ks)\)).
\end{rem}

\begin{rem}
\label{rem:ind-p-Lie}
Let \(p\) be the characteristic exponent of \(k\).
Write \([p]\) for the trivial map, or
the \(p\)th-power map, on
the Lie algebra of a \(k\)-group scheme
\cite{demazure-gabriel:groupes-algebriques}*
	{Ch.~II, \S7, Proposition 3.4},
according as \(p\) does or does not equal \(1\).
If \(\Gt_1\) is an affine group \(\Gamma_1\)-scheme and
we identify
\(\uLie(\uMor_{\Gamma_1}(\Gamma, \Gt_1))\)
with
\(\uMor_{\Gamma_1}(\Gamma, \uLie(\Gt_1))\)
\textit{via}
Corollary \ref{cor:ind-Lie}, then
it follows from functoriality
\cite{demazure-gabriel:groupes-algebriques}*
	{Ch.~II, \S7, 1.1}
that
\[\xymatrix{
\Gamma \times \uLie(\uMor_{\Gamma_1}(\Gamma, \Gt_1)) \ar[d]\ar[r]^{1 \times [p]} & \Gamma \times \uLie(\uMor_{\Gamma_1}(\Gamma, \Gt_1)) \ar[d] \\
\uLie(\Gt_1) \ar[r]^{[p]} & \uLie(\Gt_1)
}\]
commutes, where the vertical maps are the evaluation maps.
\end{rem}

\begin{defn}
\label{defn:ind-first-section}
Since \(\Gamma_1\) is
an open subgroup scheme of the finite type $k$-group scheme \(\Gamma\),
it is an open and closed subscheme of \(\Gamma\).
If $\mathcal{F}$ is a $\Gamma_1$-sheaf, write \(\iota_{\mathcal{F}}\)
for the natural transformation
\abmap
	{\mathcal F}
	{\uMor_{\Gamma_1}(\Gamma, \mathcal F)}
coming from the identification of
\(\uMor_{\Gamma_1}(\Gamma_1, \mathcal F)\) with
\(\mathcal F\) by evaluation at the identity, followed by
``extension by zero'' to the complementary subscheme.
\end{defn}

For a ring $A$ and a prime ideal $\mathfrak{p} \in \Spec A$, let $k(\mathfrak{p})$ be the residue field
of $A$ at $\mathfrak{p}$, i.e., $k(\mathfrak{p}) = \operatorname{Frac}(A/\mathfrak{p})$. If $A$ is a
$k$-algebra and $\gamma, \gamma' \in \Gamma(A)$, then we will say that $\gamma$ and $\gamma'$
\textit{lie in different $\Gamma_1$-cosets} if, for all $\mathfrak{p} \in \Spec A$, the elements
$\gamma_{k(\mathfrak{p})\alg}$ and $\gamma'_{k(\mathfrak{p})\alg}$ lie in different cosets for the
right translation action of $\Gamma_1(k(\mathfrak{p})\alg)$ on $\Gamma(k(\mathfrak{p})\alg)$, where
\(k(\mathfrak p)\alg\) is an algebraic closure of \(k(\mathfrak p)\).

\begin{prop}
\label{prop:ind-second-adjoint}
Let
\(\Gt_1\) be a group \(\Gamma_1\)-sheaf, and
\(\Ht\) a group \(\Gamma\)-sheaf.
\begin{enumerate}[label=(\alph*), ref=\alph*]
\item\label{subprop:ind-second-adjointish}
The natural transformation
\map\alpha
	{\uHom
		(\uMor_{\Gamma_1}(\Gamma, \Gt_1), \Ht)}
	{\uMor_{\Gamma_1}
		(\Gamma, \uHom(\Gt_1, \Ht))}
given by
\abmapto
	\ell
	{\bigl(\abmapto\gamma
		{\bigl(\abmapto{\gt_1}{\ell(\gamma\inv\cdot(\iota_{\Gt_1}(\gt_1)))}\bigr)}
	\bigr)},
where \(\iota_{\Gt_1}\) is as in Definition \ref{defn:ind-first-section},
is a monomorphism, natural in both \(\Gt_1\) and \(\Ht\).
\item\label{subprop:ind-image-of-second}
For every \(k\)-algebra \(A'\), the set of \(A'\)-points of the
sheaf image
\(\alpha(
	\uHom
		(\uMor_{\Gamma_1}(\Gamma, \Gt_1),
		\Ht)
)\)
consists of those
\(\phi \in
	\uMor_{\Gamma_1}
		(\Gamma,
		\uHom(\Gt_1, \Ht))
	(A')\)
such that, for every \(A'\)-algebra \(A\),
the subsets
\(\phi_A(\gamma)(\Gt_1(A))\) and
\(\phi_A(\gamma')(\Gt_1(A))\) of
\(\Ht(A)\) commute whenever
\(\gamma, \gamma' \in \Gamma(A)\) lie in
different $\Gamma_1(A)$-cosets.
\end{enumerate}
\end{prop}

\begin{proof}
	Let $k'/k$ be a finite separable extension of $k$ such that
$(\Gamma/\Gamma_1)_{k'}$ is constant and the map
$\abmap{\Gamma_{k'}}{(\Gamma/\Gamma_1)_{k'}}$
admits a section $\sigma$, as we may by Lemma~\ref{lem:local-sections-separable}.
In general, to check that a natural transformation of fppf sheaves over $k$ is an isomorphism, it
suffices to pass to an fppf cover of $k$, so to
prove (\ref{subprop:ind-second-adjointish}) we may and do pass from $k$ to $k'$. By
Lemma~\ref{lem:induction-split-case}, there is a natural isomorphism
$\bimap{\epsilon_{\sigma}}
	{\uMor_{\Gamma_1}(\Gamma, \Gt_1)}
	{\prod_{(\Gamma/\Gamma_1)(k)} \Gt_1}$
defined by $\epsilon_{\sigma}(\varphi) = (\varphi(\sigma(s)))_{s \in (\Gamma/\Gamma_1)(k)}$.
We have a similar isomorphism
\abbimap
	{\uMor_{\Gamma_1}(\Gamma, \uHom(\Gt_1, \Ht))}
	{\prod_{(\Gamma/\Gamma_1)(k)} \uHom(\Gt_1, \Ht)}.
Furthermore, the morphism $\iota_{\Gt_1}$ identifies with the inclusion into one factor of
$\prod_{(\Gamma/\Gamma_1)(k)} \Gt_1$, so the first point is simply the statement that
a homomorphism $\abmap{\prod_{(\Gamma/\Gamma_1)(k)} \Gt_1}\Ht$ is
determined by the induced \((\Gamma/\Gamma_1)(k)\)-tuple of
homomorphisms $\abmap{\Gt_1}\Ht$.

For (\ref{subprop:ind-image-of-second}), let $k'/k$ be as above.
Let $(\Gamma/\Gamma_1)(k') = \{\gamma_1, \dots, \gamma_n\}$.
If $\mathcal{F}$ is an fppf sheaf over $k$,
then the Weil restriction $\WRes_{k'/k}(\mathcal{F})$
(whose $A$-points are $\mathcal{F}(A \otimes_k k')$)
is also an fppf sheaf.
Using this, we note that the proposed presheaf image $I_\alpha$ of
$\alpha$ is already a sheaf: indeed, there is a map
\[
    \map\Phi
		{\uHom(\Gt_1, \Ht)}
		{\prod_{i=1}^n \WRes_{k'/k}(\SubSh(\Ht_{k'}))^2}
\]
given by
\[
   \abmapto\phi{\left(\im(\phi_{k'}(\sigma(\gamma_i))), C_{\Ht_{k'}}(\im(\phi_{k'}(\sigma(\gamma_i))))\right)_i}.
\]
The subpresheaf of
$\prod_{i=1}^n \WRes_{k'/k}(\SubSh(\Ht_{k'}))^2$
consisting of those $(\mathcal{G}_1, \mathcal{G}'_1,\dots, \mathcal{G}_n, \mathcal{G}'_n)$
such that $\mathcal{G}_i \subseteq \mathcal{G}'_j$ for all $i \neq j$
is a sheaf
by Lemma~\ref{lemma:incidence-sheaf}, and $I_\alpha$ is
its preimage under
$\Phi$, since to check that $\phi \in \uMor_{\Gamma_1}(\Gamma, \uHom(\Gt_1, \Ht))(A')$ lies in
$I_{\alpha}(A')$ it is enough to check that for every $A' \otimes_k k'$-algebra $A$ and every
$i \neq j$, the subsheaves $\im(\phi_{k'}(\sigma(\gamma_i)))_{A'}$ and
$\im(\phi_{k'}(\sigma(\gamma_j)))_{A'}$ commute. Thus in particular $I_\alpha$ is a sheaf.

It is straightforward to check that $\alpha$ factors through $I_\alpha$
and to check that \(\alpha\) is an isomorphism onto \(I_\alpha\),
it suffices to pass
from $k$ to $k'$. In that case, we use $\sigma$ to identify
$\uMor_{\Gamma_1}(\Gamma, \Gt_1)$ with $\prod_{(\Gamma/\Gamma_1)(k)} \Gt_1$ and
\(\uMor_{\Gamma_1}(\Gamma, \uHom(\Gt_1, \Ht))\) with
	\(\prod_{(\Gamma/\Gamma_1)(k)} \uHom(\Gt_1, \Ht)\).
In this case, $\alpha$ is identified with the map
$\abmap
	{\uHom(\prod_{(\Gamma/\Gamma_1)(k)} \Gt_1, \Ht)}
	{\prod_{(\Gamma/\Gamma_1)(k)} \uHom(\Gt_1, \Ht)}$
whose composition with projection on the factor corresponding to
\(\gamma_i\) is
\abmapto
	\ell
	{\bigl(\abmapto{\gt_1}{\ell(\gamma_i\inv\cdot(\iota_{\Gt_1}(\gt_1)))}\bigr)}.
Thus the result is clear.
\end{proof}

\begin{cor}\label{cor:was-subprop:ind-image-of-second-smooth}
Suppose that
\(\Gt_1\) is a smooth affine group \(\Gamma_1\)-scheme
over \(k\) and $\Ht$ is a group $\Gamma$-scheme.
Let
\(\phi\) be an element of
\(\Mor_{\Gamma_1}(\Gamma, \uHom(\Gt_1, \Ht))\)
such that
\(\phi_\ks(\gamma)(\Gt_1(\ks))\) and
\(\phi_\ks(\gamma')(\Gt_1(\ks))\) commute whenever
\(\gamma, \gamma' \in \Gamma(\ks)\) lie in
different cosets for
the right-translation action of \(\Gamma_1(\ks)\).
Then \(\phi\) is the image under
\(\alpha\) of
the unique element \(\ell\) of
\(\Hom(\uMor_{\Gamma_1}(\Gamma, \Gt_1), \Ht)\)
such that
\begin{equation}
\tag{$*$}
\label{eq:ind-image-of-second}
\ell_\ks(\ft) \quad\text{equals}\quad
\prod_{\gamma \in \Gamma(\ks)/\Gamma_1(\ks)}
	\phi(\gamma)(\ft(\gamma))
\end{equation}
for all
\(\ft \in \uMor_{\Gamma_1}(\Gamma, \Gt_1)(\ks)\).
\end{cor}

\begin{proof}
We use repeatedly that certain schemes are smooth, and that
the rational points of a smooth scheme valued in
a separably closed field are
Zariski dense.

For example, Proposition \ref{prop:ind-scheme} shows that
\(\uMor_{\Gamma_1}(\Gamma, \Gt_1)\) is a smooth scheme, so
\eqref{eq:ind-image-of-second}
does indeed determine a unique element of
\(\uHom(\uMor_{\Gamma_1}(\Gamma, \Gt_1), \Ht)(\ks)\).
Since the proposed element is fixed by \(\Gal(k)\),
it comes from an element of
\(\Hom(\uMor_{\Gamma_1}(\Gamma, \Gt_1), \Ht)\).
Thus, since \(\alpha_\ks\) is a monomorphism, we may, and do,
assume, after replacing \(k\) by \(\ks\), that
\(k\) is separably closed.

First fix \(\gamma, \gamma' \in \Gamma(k)\) in
different \(\Gamma_1(k)\)-cosets.
Then, since \(\phi(\gamma)(\Gt_1)\) and
\(\phi(\gamma')(\Gt_1)\) are
the Zariski closures of
\(\phi(\gamma)(\Gt_1(k))\) and
\(\phi(\gamma')(\Gt_1(k))\), they
commute.

Now discard the fixed element
\(\gamma' \in \Gamma(k)\), and fix only
\(\gamma \in \Gamma(k)\).
Since the complementary subscheme to
the open subscheme \(\gamma\Gamma_1\) of
\(\Gamma\) is
smooth, its set
\(\Gamma(k) \setminus \gamma\Gamma_1(k)\) of
\(k\)-points is Zariski dense. Moreover,
the closed subscheme of \(\Gamma \setminus \gamma\Gamma_1\) whose
\(A\)-points, for every \(k\)-algebra \(A\), are given by
\[
\smashsett
	{\gamma' \in \Gamma(A) \setminus \gamma\Gamma_1(A)}
	{\(\phi_A(\gamma')(\Gt_1(A))\) commutes with
		\(\phi(\gamma)_A(\Gt_1(A))\)}
\]
contains
\(\Gamma(k) \setminus \gamma\Gamma_1(k)\).
By smoothness, we have for every such \(A\) and every
\(\gamma' \in \Gamma(A) \setminus \gamma\Gamma_1(A)\) that
\(\phi(\gamma)_A(\Gt_1(A))\)
commutes with
\(\phi_A(\gamma')(\Gt_1(A))\).

Finally, discard the fixed element \(\gamma \in \Gamma(k)\)
(as well as \(\gamma'\)).
Since \(\Gamma\) is smooth and
the closed subscheme whose \(A\)-points,
for every \(k\)-algebra \(A\), are given by
\[
\smashsett
	{\gamma \in \Gamma(A)}
	{\(\phi_A(\gamma')(\Gt_1(A))\) commutes with
		\(\phi_A(\gamma)(\Gt_1(A))\) for every
	\(\gamma' \in \Gamma(A) \setminus \gamma\Gamma_1(A)\)},
\]
contains \(\Gamma(k)\),
we have that the criterion in
Proposition~\ref{prop:ind-second-adjoint}%
	(\ref{subprop:ind-image-of-second})
for belonging to the image of \(\alpha\)
is satisfied.
It follows from that result that
there is some
\(\ell \in \Hom(\uMor_{\Gamma_1}(\Gamma, \Gt_1), \Ht)\)
such that
\(\alpha(\ell)\) equals \(\phi\).
Now fix \(\ft \in \Mor_{\Gamma_1}(\Gamma, \Gt_1)\).
Since
\(\ft\) and
\(\prod_{\gamma \in \Gamma(k)/\Gamma_1(k)}
	\gamma\inv\cdot\iota_{\Gt_1}(\ft(\gamma))\)
agree on \(\Gamma(k)\), they are equal.
Thus
\(\ell(\ft)\) equals
\(\prod_{\gamma \in \Gamma(k)/\Gamma_1(k)}
	\ell(\gamma\inv\cdot\iota_{\Gt_1}(\ft(\gamma)))\),
which, by definition, equals
\[
\prod_{\gamma \in \Gamma(k)/\Gamma_1(k)}
	\alpha(\ell)(\gamma)(\ft(\gamma)) =
\prod_{\gamma \in \Gamma(k)/\Gamma_1(k)}
	\phi(\gamma)(\ft(\gamma)).\qedhere
\]
\end{proof}

\begin{rem}
\label{rem:ind-char}
Suppose that \(\Tt_1\) is a group of multiplicative type, i.e., that
there is a \(\Z[\Gal(k)]\)-module \(\Mt_1\) such that
\(\Tt_1(A)\) equals \(\Hom_{\Z[\Gal(k)]}(\Mt_1, (\ks \otimes_k A)^\times)\)
for every \(k\)-algebra \(A\).
Then \(\Mt_1\) is
isomorphic, as a \(\Z[\Gal(k)]\)-module, to
\(\Hom(\Tt_{1\,\ks}, \GL_{1, \ks}) = \bX^*(\Tt_{1\,\ks})\).
Suppose moreover that
\(\Tt_1\) is equipped with an action of \(\Gamma_1\).

Put \(\Tt = \uMor_{\Gamma_1}(\Gamma, \Tt_1)\).
Proposition \ref{prop:ind-scheme} already shows that
\(\Tt\) is a group of multiplicative type, but
we can say more.
Namely, Proposition \ref{prop:ind-second-adjoint} provides
a natural isomorphism
\bimap\alpha
	{\uHom(\Tt, \GL_{1, k})}
	{\uMor_{\Gamma_1}
		(\Gamma, \uHom(\Tt_1, \GL_{1, k}))},
hence a map
\abbimap
	{\bX^*(\Tt_\ks) = \Hom(\Tt_\ks, \GL_{1\,\ks})}
	{\Mor_{\Gamma_{1\,\ks}}
		(\Gamma_\ks, \uHom(\Tt_{1\,\ks}, \GL_{1\,\ks}))}
on \(\ks\)-points.
We describe the inverse of this map concretely.

First,
the (\(\Gal(k) \ltimes \Gamma_1(\ks)\))-equivariant map
\abmap
	{\bX^*(\Tt_{1\,\ks})}
	{\bX^*(\Tt_\ks)}
coming from the co-unit
\abmap{\Tt = \uMor_{\Gamma_1}(\Gamma, \Tt_1)}{\Tt_1}
of the adjunction in Lemma \ref{lem:ind-first-adjoint}
extends uniquely to a
(\(\Gal(k) \ltimes \Gamma(\ks)\))-equivariant map
\abmap
	{\Z[\Gamma(\ks)] \otimes_{\Z[\Gamma_1(\ks)]}
		\bX^*(\Tt_{1\,\ks})}
	{\bX^*(\Tt_\ks)}.
Now choose
a set \(S\)
of representatives for
the cosets of \(\Gamma_1(\ks)\) in \(\Gamma(\ks)\).
Each element of \(S\) gives a map
\abmap
	{\Z[\Gamma(\ks)] \otimes_{\Z[\Gamma_1(\ks)]} \bX^*(\Tt_{1\,\ks})}
	{\bX^*(\Tt_{1\,\ks})}, and
these maps assemble to an isomorphism
\abbimap
	{\Z[\Gamma(\ks)] \otimes_{\Z[\Gamma_1(\ks)]} \bX^*(\Tt_{1\,\ks})}
	{\prod_S \bX^*(\Tt_{1\,\ks})}.
We also have an evaluation morphism
\abmap
	{\Mor_{\Gamma_{1\,\ks}}
		(\Gamma_\ks, \uHom(\Tt_{1\,\ks}, \GL_{1, \ks}))}
	{\prod_S \Hom(\Tt_{1\,\ks}, \GL_{1\,\ks})}.
These maps fit together into a commutative diagram
\[\xymatrix{
\Mor_{\Gamma_{1\,\ks}}(\Gamma_\ks, \uHom(\Tt_{1\,\ks}, \GL_{1\,\ks})) \ar[dr]
\ar@{-->}[r] &
\Z[\Gamma(\ks)] \otimes_{\Z[\Gamma_1(\ks)]} \bX^*(\Tt_{1\,\ks}) \ar[d]^*[@!90]{\sim}\ar[r] &
\bX^*(\Tt_\ks), \\
& \prod_S \bX^*(\Tt_{1\,\ks})
}\]
and the composition across the top row is the promised inverse of
the map on \(\ks\)-rational points coming from \(\alpha\).
\end{rem}

\end{document}